%% file: main.tex
\documentclass[a4paper,11pt]{article}

\usepackage[a4paper,margin=30mm]{geometry} 

\usepackage{graphicx} 
\usepackage{subcaption}

\usepackage[colorlinks=true,linkcolor=blue]{hyperref}
\usepackage{amsthm, amsmath, amssymb}
\usepackage{mathtools}

\usepackage{pgfplots}
\usepackage{pgfplotstable}
\pgfplotsset{compat=1.18}
\newcommand{\mathdefault}[1][]{}{} 

\usepackage{booktabs}
\usepackage{multirow}

\usepackage[algoruled, noline, noend, linesnumbered]{algorithm2e}

\usepackage{cleveref}

\newcommand{\calA}{\mathcal A}
\newcommand{\calK}{\mathcal K}
\newcommand{\calO}{\mathcal O}
\newcommand{\R}{\mathbb R}
\newcommand{\C}{\mathbb C}
\newcommand{\vc}[1]{\mathbf{#1}}
\newcommand{\mrm}[1]{\mathrm{#1}}

\newtheorem{theorem}{Theorem}[section] 
\newtheorem{lemma}[theorem]{Lemma}     

\newtheorem{example}[theorem]{Example}

\title{A novel Krylov subspace method for approximating \\  Fréchet derivatives of large-scale matrix functions}

\author{Daniel Kressner\thanks{Institute of Mathematics, EPFL, Switzerland. {\tt oehme.pb@gmail.com, daniel.kressner@epfl.ch }} \and Peter Oehme\footnotemark[1]}

\begin{document}
    \maketitle
    
    \begin{abstract}
        We present a novel Krylov subspace method for approximating $L_f(A, E) \vc{b}$, the matrix-vector product of the Fréchet derivative $L_f(A, E)$ of a large-scale matrix function $f(A)$ in direction $E$, a
        task that arises naturally in the sensitivity analysis of quantities involving matrix functions, such as centrality measures for networks. 
        It also arises in the context of gradient-based methods for optimization problems that feature matrix functions, e.g., when fitting an evolution equation to an observed solution trajectory.
        In principle, the well-known identity
         \[
            f\left( \begin{bmatrix}
                A & E \\ 0 & A
            \end{bmatrix} \right) \begin{bmatrix}
                0 \\ \vc{b}
            \end{bmatrix} = \begin{bmatrix}
                L_f(A, E) \vc{b} \\ f(A) \vc{b}
            \end{bmatrix},
        \]
        allows one to directly apply any standard Krylov subspace method, such as the Arnoldi algorithm, to address this task.
        However, this comes with the major disadvantage that the involved block triangular matrix has unfavorable spectral properties, which impede the convergence analysis and, to a certain extent, also the observed convergence. To avoid these difficulties, we propose a novel modification of the Arnoldi algorithm that aims at better preserving the block triangular structure. In turn, this allows one to bound the convergence of the modified method by the best polynomial approximation of the derivative $f^\prime$ on the numerical range of $A$. Several numerical experiments illustrate our findings.
    \end{abstract}
   
    \input{intro}

    \input{alg}
    \input{convan}
    \input{exps}

    \bibliographystyle{plain}
    \bibliography{lit}
\end{document}

%% file: intro.tex
\section{Introduction} \label{sec:intro}

This work is concerned with matrix functions $f(A)$, the title motif of Higham's monumental monograph~\cite{Higham2008}, where $A \in \C^{n\times n}$ is a matrix and $f$ is a scalar function analytic in an open region containing the eigenvalues of $A$. Well-known examples include the matrix exponential, matrix square root, matrix logarithm, and the inverse of a matrix.
Consider the Fr\'echet derivative of $f$ at $A$, defined as the unique linear operator $L_f(A,\cdot): \C^{n\times n} \to \C^{n\times n}$ such that
\begin{equation} \label{eq:derdef}
f(A+ \varepsilon E) - f(A) = \varepsilon L_f(A,E) + \calO(\varepsilon^2)
\end{equation}
holds for every $E \in \C^{n\times n}$ and sufficiently small $\varepsilon > 0$. Such derivatives appear naturally in the sensitivity analysis of quantities involving matrix functions, such as centrality and communicability measures for networks~\cite{Benzi2020,delacruz2022,Pozza2018,Schweitzer2023}.
They also arise when using gradient-based methods for optimization problems that feature matrix functions; see~\cite{Didier2024,monti2025,Thanou2017,Zeilmann2023} for instances of this situation in different application fields.

For larger $n$, the sheer size of their $n^2 \times n^2$ matrix representation makes Fr\'echet derivatives much too expensive to compute and work with.
On the other hand, none of the above applications actually requires full access to this quantity. Typically, only the derivative evaluated in a specific direction $E$ applied to a vector $\vc{b} \in \C^n$ is needed:
\begin{equation} \label{eq:derivativeb}
 L_f(A,E) \cdot \vc{b},
\end{equation}
which coincides with the directional derivative of $f(A) \vc{b}$. Moreover, the matrices $A$ and $E$ usually feature some (data) sparsity that makes them cheap to multiply with vectors, prompting the use of Krylov subspace methods.
While the development and analysis of Krylov subspace methods for approximating large-scale matrix functions
$f(A) \vc{b}$ has been an active area of research during the last decades, significantly less is known about such methods for approximating~\eqref{eq:derivativeb}.

Before discussing numerical methods for~\eqref{eq:derivativeb}, it is useful to recall existing approaches for computing the whole $n\times n$ matrix $L_f(A,E)$. 
A common starting point of many methods is the identity
\begin{equation} \label{eq:blocktriangular}
 f(\calA) = \begin{bmatrix} f(A) & L_f(A,E) \\
          0 & f(A)´
         \end{bmatrix}
 \quad \text{with} \quad \calA = \begin{bmatrix} A & E \\
          0 & A
         \end{bmatrix},
\end{equation}
which is
attributed to Najfeld and Havel~\cite{Najfeld1995} in~\cite{Higham2008}. This allows one to straightforwardly apply existing robust methods for computing matrix functions, such as scaling-and-squaring for matrix exponentials~\cite{zbMATH05772935}, to $\calA$ in order to compute $L_f(A,E)$. However, it can be computationally beneficial to modify such a method by, e.g., exploiting the block triangular structure of $\calA$ or targeting $L_f(A,E)$ directly; see~\cite{Almohy2008,Almohy2013,Higham2013} for examples. Alternatively, the definition~\eqref{eq:derdef} can be used to proceed via  finite-difference approximations such as $L_f(A,E) \approx \big( f(A+ \varepsilon E) - f(A) \big) / \varepsilon$ for very small $\varepsilon$, but this limits the attainable accuracy due to numerical cancellation~\cite[P. 68]{Higham2008}. This effect can often be mitigated by using the 
complex step approximation $L_f(A,E) \approx \mathrm{Im}\big( f(A+ \mathrm{i} \varepsilon E) - f(A) \big) / \varepsilon$, see~\cite{zbMATH07363165}, but this introduces complex arithmetic for real data and does not easily generalize when $A$ itself is complex.

Each of the approaches above can be used to employ standard Krylov subspace methods, such as the Arnoldi method, for matrix functions to approximate $L_f(A,E)\vc{b}$:
\begin{description}
 \item[Arnoldi for $\calA$:]
 The identity~\eqref{eq:blocktriangular} implies
\begin{equation} \label{eq:blocktriangulartimesb}
 f(\calA) \begin{bmatrix} 
          0 \\ \vc{b}
         \end{bmatrix}= \begin{bmatrix} L_f(A,E) \vc{b} \\
          f(A) \vc{b}
         \end{bmatrix},
\end{equation}
which suggest to apply the Arnoldi method with the matrix $\calA$ and starting vector $\begin{bmatrix} 
          0 \\ \vc{b}
         \end{bmatrix}$ to compute $L_f(A,E) \vc{b}$ and $f(A) \vc{b}$ simultaneously;
see~\cite[Sec. 4.1]{delacruz2022} for an example of this approach. A major disadvantage is that $\calA$ typically has much less favorable spectral properties than $A$. For example, it simple to verify that $\calA$  generically (with respect to $E$) has Jordan blocks of size at least $2$, even when $A$ is symmetric! In turn, convergence bounds based on the numerical range, such as~\cite{zbMATH05815158}, can deteriorate significantly; see~\cite{Beckermann2018} for a related situation. Also one can no longer use the less expensive Lanczos method when $A$ is symmetric.
 \item[Finite differences with Arnoldi:]
 Given a finite difference approximation like \[L_f(A,E) \vc{b} \approx \big( f(A+ \varepsilon E) \vc{b} - f(A) \vc{b} \big) / \varepsilon\] for small $\varepsilon > 0$,
 one can apply the Arnoldi method to approximate each of the individual terms $f(A) \vc{b}$ and $f(A+ \varepsilon E) \vc{b}$, as suggested in~\cite[Sec. 4.1]{delacruz2022}.
 As discussed above, cancellation is again a concern. In order to attain an error tolerance $\mathsf{tol}$, one needs to choose $\varepsilon \lesssim \sqrt{\mathsf{tol}}$. Because of the appearance of $\varepsilon$ in the denominator, a simple application of the triangle inequality \emph{suggests} that each of the terms needs to be approximated with significantly higher accuracy, $\mathsf{tol}^{3/2}$ instead of $\mathsf{tol}$. In other words, convergence might be adversely affected by cancellation of the Arnoldi approximations.
 
 \item[Complex step with Arnoldi:] Combining the complex step approximation \[L_f(A,E) \vc{b} \approx  \mathrm{Im}\big( f(A+ \mathrm{i} \varepsilon E) \vc{b} - f(A) \vc{b} \big) / \varepsilon\] with Arnoldi comes (again) with the disadvantage that complex arithmetic is needed for real data. Also, the presence of $\varepsilon$ is again a concern; at least we are not aware of theoretical results ruling out the cancellation effect of the Arnoldi approximations mentioned above.
\end{description}

In this work, we propose and analyze a modification of the Arnoldi method for $\calA$ that better preserves the block triangular structure of $\calA$. This modification is inspired by existing Arnoldi methods for quadratic eigenvalue problems~\cite{zbMATH02206295}, and it allows us to mitigate the theoretical difficulties associated with the unfavorable spectral properties of $\calA$. In particular, the results of~\Cref{sec:convanalysis} bound the convergence of the modified Arnoldi method to $L_f(A,E) \vc{b}$ by the best polynomial approximation of $f^\prime$ on the numerical range of $A$ (instead of $\calA$). This parallels existing results~\cite{Beckermann2018,zbMATH01116284} for $f(A) \vc{b}$ that bound the convergence of Arnoldi by the best polynomial approximation of $f$ on the numerical range of $A$.

\begin{example} \label{ex:sqrtm}
To provide some first insight into the numerical performance of the different methods for approximating $L_f(A,E) \vc{b}$, we consider 
$f(z) = \sqrt{z}$, $A = \mathrm{diag}(1,\ldots,500)$, and random $E, \vc{b}$, generated using \texttt{randn} in Matlab. \Cref{fig:sqrtm} shows the error in the Euclidean norm
versus the number of steps of the Arnoldi methods used for the approximation. Interestingly, the difficulties with the convergence theory mentioned above for existing methods show up in a much less pronounced way than what one may expect. Still, our modified method converges fastest and at a rate comparable to the convergence of the Arnoldi method for approximating $f(A) \vc{b}$.
\end{example}

\begin{figure}
    \centering
    \resizebox{0.6\textwidth}{!}{\input{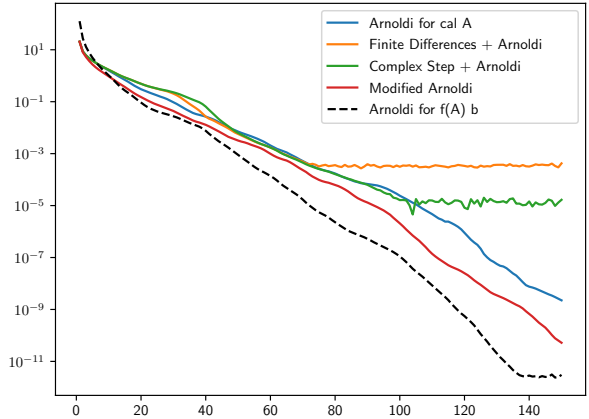}}
    \caption{\label{fig:sqrtm}Error vs. iterations of four different methods to approximate $L_f(A, E) \vc{b}$ for~\Cref{ex:sqrtm}. The dashed curve corresponds to the Arnoldi method for approximating $f(A) \vc{b}$ and is shown for reference only.}
\end{figure}

\subsection{Low-rank directions}

When the direction $E$ has rank one or, more generally, low rank, there are specialized methods for approximating the whole matrix $L_f(A, E)$ by a low-rank matrix. A (tensorized) Krylov subspace method has been proposed and analyzed in~\cite{zbMATH07478627,zbMATH07215841}.
The analysis in~\cite{Crouzeix2025,zbMATH07215841} relates the convergence of this method also to the best polynomial approximation of $f^\prime$ on the numerical range of $A$, analogous to our analysis in~\Cref{sec:convanalysis}.
The close connection between~\cite{zbMATH07478627,zbMATH07215841} and our work is also reflected by the following theorem, which shows that the computation of bilinear forms involving $L_f(A, E)$ can be effected by evaluating the Fr\'echet derivatives for rank-one directions. Note that $\langle \cdot, \cdot \rangle$ denotes the 
usual trace inner product defined by $\langle X, Y \rangle = \mathrm{trace}(X^* Y)$.

\begin{theorem} \label{thm:schweitzer}
Given $A \in \C^{n\times n}$ such that  $L_f(A, \cdot)$ is well defined, $\vc{b}, \vc{c} \in \C^n$ and $E \in \C^{n\times n}$, it holds that
 $\vc{c}^* L_f(A, E) \vc{b} = \langle L_f(A,\vc{c} \vc{b}^*), E \rangle$.
\end{theorem}
\begin{proof}
It is well known that $L_f(A^*, \cdot)$ is the adjoint of the linear operator $L_f(A,\cdot): \C^{n\times n} \to \C^{n\times n}$ with respect to the trace inner product; e.g.,~\cite[P. 66]{Higham2008}. Using the defining property of an adjoint, this implies the result:
\begin{align*}
 \vc{c}^* L_f(A, E) \vc{b} = \mathrm{trace}(  \vc{b} \vc{c}^* L_f(A, E)  ) =  \langle \vc{c} \vc{b^*}, L_f(A, E) \rangle   
 = \langle L_f(A^*, \vc{c} \vc{b^*}), E \rangle.
\end{align*}
\end{proof}

For the special case of a rank-one matrix $E = \vc{x} \vc{y}^*$, the result of \Cref{thm:schweitzer} yields
\begin{equation} \label{eq:cortheorem}
  \vc{c}^* L_f(A, \vc{x} \vc{y}^*) \vc{b} = \vc{y}^* L_f(A,\vc{c} \vc{b}^*)^* \vc{x} = \overline{ \vc{x}^* L_f(A,\vc{c} \vc{b}^*) \vc{y} } 
\end{equation}
This matches the result of Theorem 2.3 in~\cite{Schweitzer2023}, which has been shown using vectorization and properties of Kronecker products.

%% file: alg.tex
\section{Krylov subspace approximation to $L_f(A,E) \vc{b}$}

\subsection{Preliminaries}

We let $\calK_k(A, \vc{b}) = \mathrm{span}\{\vc{b}, A \vc{b}, \ldots, A^{k-1} \vc{b}\}$ denote the usual Krylov subspace and, to simply the discussion, we will always assume that it assumes maximal dimension $k$.
Let $U_k \in \R^{n\times k}$ contain an orthonormal basis of $\calK_k(A, \vc{b})$. Then the Arnoldi method for matrix functions returns the approximation
\begin{equation} \label{eq:fabapprox}
 f(A) \vc{b} \approx U_k f(U_k^T A U_k) U_k^T \vc{b}.
\end{equation}
One usually has $k \ll n$ and, thus, the evaluation of the $k\times k$ matrix function $f(U_k^T A U_k)$ is relatively cheap.
The convergence of~\eqref{eq:fabapprox} is closely related to the polynomial approximation of $f$. In fact, when $f$ itself is a polynomial of degree at most $k-1$ then the approximation~\eqref{eq:fabapprox} is exact~\cite{Ericsson1990, Saad1992}. The following lemma provides a slight variation of this polynomial exactness result, allowing the subspace to be larger than $\calK_k(A, \vc{b})$.
\begin{lemma}\label{lem:krylov-exactness}
    Let $p$ be a polynomial of degree at most $k - 1$, and let $U$ be an orthonormal basis of a subspace $\mathcal{U}$ such that $\mathcal{K}_k(A, \vc{b}) \subseteq \mathcal{U}$. Then it holds that
    \[
        p(A) \vc{b} = U p(U^* A U) U^* \vc{b}.
    \]
\end{lemma}

\begin{proof}
    It suffices to show the claim for the monomials $x^j$, $j = 0, 1, \dots, k - 1$. We observe that
    \begin{equation} \label{eq:fact1}
             U {(U^* A U)}^j U^* \vc{b} = {(U U^* A)}^j U U^* \vc{b}.
    \end{equation}
    By assumption we know that $A^i \vc{b} \in \mathcal{U}$, whence $U U^* A^i \vc{b} = A^i \vc{b}$ for every 
    $0 \le i \le j$. In turn,
    \begin{align*}
     {(U U^* A)}^j U U^* \vc{b} & = {(U U^* A)}^j \vc{b} = {(U U^* A)}^{j-1} UU^* A \vc{b} = {(U U^* A)}^{j-1} A \vc{b} \\
     & = {(U U^* A)}^{j-2} A^2 \vc{b} = \cdots = A^j b.
    \end{align*}
Combined with~\eqref{eq:fact1}, this proves the claim $U {(U^* A U)}^j U^* \vc{b} = A^j \vc{b}$.
\end{proof}

\subsection{Modified Arnoldi algorithm: Basic version}

As explained in \Cref{sec:intro}, directly applying the approximation~\eqref{eq:fabapprox} to the block triangular embedding~\eqref{eq:blocktriangular} in order to approximate $L_f(A,E) \vc{b}$ comes with theoretical difficulties. To address this issue, we consider a basis $\mathcal V_k \in \C^{2n\times k}$ for the relevant Krylov subspace $\mathcal{K}_k\left( \mathcal{A}, \begin{bmatrix} 
    0 \\ \vc{b}
\end{bmatrix} \right)$ and tear it into two parts, letting $U_k \in \C^{n\times k}$ and $V_k \in \C^{n\times k}$ contain orthonormal bases for the top $n$ rows and bottom $n$ rows of $\mathcal V_k$, respectively. Trivially,
\[
\mathcal{K}_k\left( \mathcal{A}, \begin{bmatrix} 
    0 \\ \vc{b}
\end{bmatrix} \right) \subset \mathrm{span}(\mathcal{U}_k), \quad 
\text{with} \quad \mathcal{U}_k = \begin{bmatrix}
        U_k & 0 \\ 0 & V_k
    \end{bmatrix}.
\]
By \Cref{lem:krylov-exactness}, the corresponding approximation
\begin{equation} \label{eq:approxfk}
  \vc{f}_k = \mathcal{U}_k f(\mathcal{U}_k^* \calA \mathcal{U}_k) \mathcal{U}_k^* \begin{bmatrix}
        0 \\ \vc{b}
        \end{bmatrix}
\end{equation}
still satisfies polynomial exactness, and it comes with the advantage that the compressed matrix
\begin{equation} \label{eq:blockstructure}
    \mathcal{U}_k^* \mathcal{A} \mathcal{U}_k = \begin{bmatrix}
        U_k^* A U_k & U_k^* E V_k \\ 0 & V_k^* A V_k
    \end{bmatrix}
\end{equation}
preserves the block triangular structure of $\mathcal A$. By~\eqref{eq:blocktriangular}, the first $n$ 
entries of $\vc{f}_k$ approximate $L_f(A,E) \vc{b}$, while its bottom $n$ entries approximate $f(A)\vc{b}$.
\begin{algorithm}
    \caption{Modified Arnoldi Algorithm (Basic Version) \label{alg:mod-arnoldi}}
    \KwData{$A, E \in \mathbb{C}^{n \times n}, \vc{b} \in \mathbb{C}^n, k$}
    \KwResult{Approximation $\vc{f}_k \in \mathbb{C}^{2n}$ defined in~\eqref{eq:approxfk}}
    Compute a basis $\mathcal{V}_k$ of $\mathcal{K}_k\left( \calA, \begin{bmatrix}
        0 \\ \vc{b}
    \end{bmatrix} \right)$\;
    Split $\mathcal{V}_k = \begin{bmatrix}
        C \\ D
    \end{bmatrix}$ and orthogonalize $U_k = \mrm{orth}(C), V_k = \mrm{orth}(D)$\;\label{algline:naive-reorth}
    Form $\mathcal{U}_k = \begin{bmatrix}
        U_k & 0 \\ 0 & V_k
    \end{bmatrix}$ and compute $\vc{f}_k = \mathcal{U}_k f(\mathcal{U}_k^* \calA \mathcal{U}_k) \mathcal{U}_k^* \begin{bmatrix}
        0 \\ \vc{b}
    \end{bmatrix}$\;\label{algline:matfunc-proj}
\end{algorithm}

\Cref{alg:mod-arnoldi} summarizes the procedure described above. In this basic form, \Cref{alg:mod-arnoldi} clearly  suffers from inefficiencies because the (stable) computation of $\mathcal{V}_k$ already requires orthogonalization and then its bottom and top parts need to be orthogonalized again. Also, the explicit computation of the compressed matrix $\mathcal{U}_k^* \mathcal{A} \mathcal{U}_k$  
requires additional matrix-vector products with $A$ and $E$.

\subsection{Modified Arnoldi algorithm: Separate orthogonalization}

To reduce the expense of orthogonalization, we recursively construct (non-orthonormal) Krylov subspace bases of the form
\begin{equation} \label{eq:structbasis}
 \begin{bmatrix}
    U_i R_i \\ V_i
\end{bmatrix} \in \mathbb{C}^{2 n \times (i + 1)},
\end{equation}
such that $U_i \in \mathbb{C}^{n \times i}$, $V_i \in \mathbb{C}^{n \times (i + 1)}$ are orthonormal, and $R_i \in \mathbb{C}^{i \times (i + 1)}$
is strictly upper triangular.
In the following, we explain how the Arnoldi method applied to the Krylov subspace
$\mathcal{K}_k\left( \calA, \begin{bmatrix}
    0 \\ \vc{b}
\end{bmatrix} \right)$ is modified to yield bases of the form~\eqref{eq:structbasis} for $i = 0,1,\ldots,k$.
\begin{description}
    \item[Initialization:] For $i = 0$, the Krylov subspace is spanned by the vector $\begin{bmatrix}
        0 \\ \vc{b}
    \end{bmatrix}$ and a basis of the form~\eqref{eq:structbasis} is obtained by setting
    $V_0 = [ \vc{v}_0 ]$ with $\vc{v}_0 = \vc{b} / \Vert \vc{b} \Vert_2$. Note that $U_0 \in \C^{n\times 0}$ and $R_0 \in \C^{0\times 1}$ are empty matrices.
    \item[Iteration from $i-1$ to $i$:] Given a basis of the form~\eqref{eq:structbasis} for $i-1$, we partition its last vector into the two components $\begin{bmatrix}
        \vc{w}_{i-1} \\ \vc{v}_{i-1}
    \end{bmatrix}$. Note that $\vc{w}_{i-1}$ is not available explicitly, but needs to be extracted from the matrix product $U_{i-1} R_{i-1}$.
    Multiplying the last basis vector with $\mathcal{A}$,
    \begin{equation} \label{eq:matvecprod}
        \begin{bmatrix}
            A & E \\ 0 & A
        \end{bmatrix} \begin{bmatrix}
            \vc{w}_{i - 1} \\ \vc{v}_{i - 1}
        \end{bmatrix} = \begin{bmatrix}
            A \vc{w}_{i - 1} + E \vc{v}_{i - 1} \\ A \vc{v}_{i - 1}
        \end{bmatrix},
    \end{equation}
    and appending it to the current basis yields a basis 
    for the next Krylov subspace, provided that no breakdown in the Arnoldi process occurs.
    We now aim at transforming this new basis into the desired form
    \begin{equation} \label{eq:nextbasis}
     \begin{bmatrix}
      U_{i} R_i \\
      V_i
     \end{bmatrix} = 
    \begin{bmatrix}
        \begin{bmatrix}
            U_{i - 1} & \vc{u}_i
        \end{bmatrix} R_i \\
        \begin{bmatrix}
            V_{i - 1} & \vc{v}_i
        \end{bmatrix}
    \end{bmatrix}.
    \end{equation}
    One step of Gram-Schmidt is used to orthonormalize the bottom part of~\eqref{eq:matvecprod}, $\tilde{\vc{v}}_i = A \vc{v}_{i - 1}$, versus the existing basis $V_{i - 1}$: \[
\vc{h} = V_{i - 1}^* \tilde{\vc{v}}_i,\quad \tilde{\vc{v}}_i \gets \tilde{\vc{v}}_i - V_{i - 1} \vc{h}, \quad \beta = \Vert \tilde{\vc{v}}_i \Vert_2, \quad \vc{v}_i = \tilde{\vc{v}}_i / \beta. 
\]
This amounts to the transformation
\[
 \begin{bmatrix}
        V_{i - 1} & \vc{v}_i
    \end{bmatrix} = \begin{bmatrix}
            V_{i - 1} & A \vc{v}_{i - 1}
        \end{bmatrix}  \begin{bmatrix}
            \mathrm{id} & -\vc{h} / \beta \\ 0 & 1 / \beta
        \end{bmatrix}.
\]
In order to preserve the basis property for the Krylov subspace, we need to apply the same transformation to the top part, absorbing it into the triangular factor:
    \begin{equation*}
        \tilde{R}_i = \begin{bmatrix}
            R_{i - 1} & 0 \\ 0 & 1
        \end{bmatrix} \begin{bmatrix}
            \mathrm{id} & -\vc{h} / \beta \\ 0 & 1 / \beta
        \end{bmatrix} = \begin{bmatrix}
            R_{i - 1} & -R_{i - 1} \vc{h} / \beta \\ 0 & 1 / \beta
        \end{bmatrix}.
    \end{equation*}

    Similarly, we orthonormalize $\tilde{\vc{u}}_i = A \vc{w}_{i - 1} + E \vc{v}_{i - 1}$ versus $U_{i - 1}$ by computing
    \[ \vc{g} = U_{i - 1}^* \tilde{\vc{u}}_i,\quad  \tilde{\vc{u}}_i \gets \tilde{\vc{u}}_i - U_{i - 1} \vc{g} \quad \alpha = \Vert \tilde{\vc{u}}_i \Vert_2,\quad \vc{u}_i = \tilde{\vc{u}}_i / \alpha.\]
    This time, the corresponding transformation
     \[
 \begin{bmatrix}
        U_{i - 1} & \vc{u}_i
    \end{bmatrix} = \begin{bmatrix}
            U_{i - 1} & A \vc{w}_{i - 1} + E \vc{v}_{i - 1}
        \end{bmatrix}  \begin{bmatrix}
            \mathrm{id} & -\vc{g} / \alpha \\ 0 & 1 / \alpha
        \end{bmatrix}
\]
cannot be applied to the other part, because it would destroy the orthonormality of $V_i$. Instead, we compensate it by applying the inverse transformation to the triangular factor:
    \begin{equation*}
        R_i = \begin{bmatrix}
            \mathrm{id} & \vc{g} \\ 0 & \alpha
        \end{bmatrix} \tilde{R}_i = \begin{bmatrix}
            R_{i - 1} & -R_{i - 1} \vc{h} / \beta + \vc{g} / \beta \\ 0 & \alpha / \beta
        \end{bmatrix}.
    \end{equation*}
    In turn, we have constructed a Krylov subspace basis taking the form~\eqref{eq:nextbasis}. 
\end{description}
Algorithm~\ref{alg:orth} summarizes the described procedure. Note that for $i = 1$, line~\ref{algline:orth-u} simplifies to
$\tilde{\vc{u}}_1 = E \vc{v}_{0}$, $\vc{g}$ remains void, and line~\ref{algline:update} simplifies to
$R_1 = \begin{bmatrix}
        0 & \alpha / \beta
    \end{bmatrix}$.

\begin{algorithm} 
    \caption{Separate Orthonormalization}\label{alg:orth}
    \KwData{$A, E \in \mathbb{C}^{n \times n}, \vc{b} \in \mathbb{C}^n, k$}
    \KwResult{$U_k \in \mathbb{C}^{n \times k}, V_k \in \mathbb{C}^{n \times (k + 1)}, R_k \in \mathbb{C}^{k \times (k + 1)}$}
    Compute $\beta = \Vert \vc{b} \Vert_2, \vc{v}_0 = \vc{b} / \beta$\;
    Set $V_0 = [ \vc{v}_0], U_0 = [\,], R_0 = [\,]$\;
    \For{$i = 1, 2, 3, \dots, k$}{
        \tcp{Orthonormalization for $A \vc{v}_{i - 1}$}
        Compute $\tilde{\vc{v}}_i = A \vc{v}_{i - 1}$ and $\vc{h} = V_{i - 1}^* \tilde{\vc{v}}_i$\;
        Orthogonalize $\tilde{\vc{v}}_i \gets \tilde{\vc{v}}_i - V_{i - 1} \vc{h}$\;
        Set $\beta = \Vert \tilde{\vc{v}}_i \Vert_2, \vc{v}_i = \tilde{\vc{v}}_i / \beta, V_i = [ V_{i - 1}, \vc{v}_i ]$\;
        
        \tcp{Orthonormalization for $A U_{i - 1} R_{i - 1}[\colon, -1] + E \vc{v}_{i - 1}$}
        Compute $\tilde{\vc{u}}_i = A U_{i - 1} R_{i - 1}[\colon, -1] + E \vc{v}_{i - 1}$ and $\vc{g} = U_{i - 1}^* \tilde{\vc{u}}_i$\;\label{algline:orth-u}
        Orthogonalize $\tilde{\vc{u}}_i \gets \tilde{\vc{u}}_i - U_{i - 1} \vc{g}$\;
        Set $\alpha = \Vert \tilde{\vc{u}}_i \Vert_2, \vc{u}_i = \tilde{\vc{u}}_i / \alpha, U_i = [ U_{i - 1}, \vc{u}_i ]$\;
        Update $R_i = \begin{bmatrix}
            R_{i - 1} & -R_{i - 1} \vc{h} / \beta + \vc{g} / \alpha \\
            0 & \alpha / \beta
        \end{bmatrix}$\; \label{algline:update}
    }
\end{algorithm}

Note that Algorithm~\ref{alg:orth} already carries out all the matrix-vector products with $A$ and $E$ required for forming the compressed matrix 
$\mathcal{U}_k^* \mathcal{A} \mathcal{U}_k$; see~\eqref{eq:blockstructure}. We can thus avoid the need for extra large-scale 
matrix-vector products by letting Algorithm~\ref{alg:orth} store and return the matrices $A U_k, V_k^* A V_k,$ and $E V_k$. Compared to the naïve orthogonalization in Algorithm~\ref{alg:mod-arnoldi}, we have thus not only removed the additional orthogonalizations in line~\ref{algline:naive-reorth} but we have also avoided the block matrix multiplication in line~\ref{algline:matfunc-proj}. This mitigates the computational inefficiencies discussed above, at the expense of $\mathcal O(nk)$ additional memory for storing $A U_k$ and $E V_k$. 

%% file: convan.tex
\subsection{Convergence analysis} \label{sec:convanalysis}

\Cref{lem:krylov-exactness} combined with the block structure preservation~\eqref{eq:blockstructure} allow us to analyze the convergence of \Cref{alg:mod-arnoldi} (with or without the separate orthogonalization performed by \Cref{alg:orth}). The following theorem is the main theoretical result of this work, relating convergence to the best polynomial approximation of the derivative of $f$ on the numerical range $W(A) = \{ x^* A x : \|x\|_2 = 1\}$ of $A$.

\begin{theorem} \label{thm:main}
Given an analytic function $f: \Omega \to \C$ for some domain $\Omega \in \C$ and $A \in \C^{n \times n}$, assume that $W(A) \subset \Omega$.
For some $E \in \C^{n\times n}$, $\vc{b} \in \C^n$, partition the output of \Cref{alg:mod-arnoldi} into $\vc{f}_k = \begin{bmatrix}
        \vc{v}_1 \\ \vc{v}_2
    \end{bmatrix}$ such that $\vc{v}_1, \vc{v}_2 \in \mathbb{C}^n$.
Letting $\Pi_{k-2}$ denote the set of all polynomials of degree at most $k - 2$, it holds that
    \[
        \Vert L_f(A, E) \vc{b} - \vc{v}_1 \Vert_2 \leq 2 C \cdot \Vert \vc{b} \Vert_2 \Vert E \Vert_F \cdot \inf_{q \in \Pi_{k-2}} \sup\limits_{z \in W(A)} \vert f^\prime(z) - q(z) \vert,
    \]
    with
    $C = 1$ if $A$ is normal and $C = (1 + \sqrt{2})^2$ otherwise.
\end{theorem}

\begin{proof}
    For arbitrary $p \in \Pi_{k - 1}$, let us denote $\vc{p}_k = \mathcal{U}_k p(\mathcal{U}_k^* \mathcal{A} \mathcal{U}_k) \mathcal{U}_k^* \begin{bmatrix}
        0 \\ \vc{b}
    \end{bmatrix}$, where $\mathcal{U}_k$ is the block diagonal basis constructed by \Cref{alg:mod-arnoldi}.   Using Lemma~\ref{lem:krylov-exactness}, we may write
    \begin{equation} \label{eq:errorform}
        f(\mathcal{A}) \begin{bmatrix}
            0 \\ \vc{b}
        \end{bmatrix} - \vc{f}_k = e(\mathcal{A}) \begin{bmatrix}
            0 \\ \vc{b}
        \end{bmatrix} - \vc{e}_k,
    \end{equation}
    with $e = f - p$ and $\vc{e}_k = \vc{f}_k - \vc{p}_k$. Recalling~\eqref{eq:blocktriangular}, it holds that
    \[
        e(\mathcal{A}) = \begin{bmatrix}
            e(A) & L_e(A, E) \\ 0 & e(A)
        \end{bmatrix}.
    \]
    Corollary~5.1 in~\cite{Crouzeix2025} gives the upper bound
    \[
        \Vert L_e(A, E) \Vert_F \leq C \|E \|_F\cdot \sup\limits_{z \in W(A)} \vert e^\prime(z) \vert.
    \]
    To address the second term in~\eqref{eq:errorform}, we first note that
    \[
    e(\mathcal{U}_k^* \mathcal{A} \mathcal{U}_k) = 
e\left( \begin{bmatrix}
            U_k^* A U_k & U_k^* E V_k \\ 0 & V_k ^* A V_k
        \end{bmatrix} \right) = 
        \begin{bmatrix}
            e(U_k^* A U_k) & E_{12} \\ 0 & e(V_k ^* A V_k)
        \end{bmatrix}
    \]
    for some $E_{12} \in \C^{n\times n}$; see, e.g.,~\cite{Kenney1998} and~\cite{Higham2008}. Because the numerical ranges of $U_k^* A U_k$ and $V_k^* A V_k$ are both contained in $W(A)$, we can apply Lemma 4.1 from~\cite{MR4257876}, which gives
    \[
     \Vert E_{12} \Vert_F \leq C \|E \|_F\cdot \sup\limits_{z \in W(A)} \vert e^\prime(z) \vert.
    \]
    Inserting these relations into~\eqref{eq:errorform} gives 
    \[
    L_f(A, E) \vc{b} - \vc{v}_1 = L_e(A, E) \vc{b} - U_k E_{12} V_k^* \vc{b},
    \]
    and the bound
    \[
     \|L_f(A, E) \vc{b} - \vc{v}_1\|_2 \le  ( \|L_e(A, E)\|_F + \|E_{12}\|_F )  \|\vc{b}\|_2 \le 
     2 C \cdot \Vert \vc{b} \Vert_2 \Vert E \Vert_F \cdot \sup\limits_{z \in W(A)} \vert e^\prime(z) \vert.
    \]
    Noting that $e^\prime = f^\prime - q^\prime$ and setting $q = p^\prime$ concludes the proof.
\end{proof}

%% file: exps.tex
\section{Numerical Experiments}

To test our newly proposed algorithm, all numerical experiments were run on a MacBook Air with 16 GB of RAM, an Apple M2 processor, and version 26.1 of macOS. The Python source code is publicly available at~\href{https://github.com/peoe/fAb-frechet}{github.com/peoe/fAb-frechet}.

\subsection{Centrality Measure Sensitivity in Network Analysis}

Let $A$ be the adjacency matrix of a graph. There are multiple centrality measures for individual graph nodes and edges, some of which can be expressed as matrix functions. We focus on the following functions:
\begin{description}
    \item[Total Network Communicability:] $f_{\mrm{TN}}(A) \coloneqq \vc{1}^T \exp{(A)} \vc{1}$
    \item[Subgraph Centrality:] $f_{\mrm{SC}}(A, \ell) \coloneqq \vc{e}_\ell^T \exp{(A)} \vc{e}_\ell$
    \item[Estrada Index:] $f_{\mrm{EI}}(A) \coloneqq \mrm{tr}(\exp{(A)})$
\end{description}
Most applications of these quantities focus on their sensitivities w.r.t.\ changes of the adjacency matrix $A$. Thus, one wishes to compute the Fréchet derivative and the subsequent sensitivity matrices $S^{(f)}$ as
\begin{align*}
    S^{(TN)}_{ij}(A) &\coloneqq \vc{1}^T L_{\exp}(A, \vc{e}_i \vc{e}_j^T) \vc{1}, \\
    S^{(SC)}_{ij}(A, \ell) &\coloneqq \vc{e}_\ell^T L_{\exp}(A, \vc{e}_i \vc{e}_j^T) \vc{e}_\ell, \\
    S^{(EI)}_{ij}(A) &\coloneqq \mrm{tr}(L_{\exp}(A, \vc{e}_i \vc{e}_j^T)).
\end{align*}
Applying the result~\eqref{eq:cortheorem}, which is a consequence of Theorem~\ref{thm:schweitzer} and~\cite[Corollary~4.5]{Schweitzer2023}, we can reduce these to
\begin{align*}
    S^{(TN)}_{ij}(A) &= \vc{e}_i^T L_{\exp}(A^T, \vc{1} \vc{1}^T) \vc{e}_j, \\
    S^{(SC)}_{ij}(A, \ell) &= \vc{e}_i^T L_{\exp}(A^T, \vc{e}_\ell \vc{e}_\ell^T) \vc{e}_j, \\
    S^{(EI)}_{ij}(A) &= \vc{e}_i^T L_{\exp}(A^T, \mrm{id}) \vc{e}_j.
\end{align*}

We will compute $S^{(TN)}_{ij}$ for the adjacency matrices from the following datasets:
\begin{description}
    \item[\texttt{Air500}~\cite{DynConnectLabData}:] The dataset from~\cite{Marcelino2012} contains the top $500$ airports based on passenger volume between July 2007 and June 2008 in the world as nodes. The $24009$ edges in this network represent flight connections, resulting in a directed unweighted graph.
    \item[\texttt{Autobahn}~\cite{DynConnectLabData}:] The dataset from~\cite{Kaiser2004} represents the German highway network in 2002. The graph has $1168$ vertices and represents the $1243$ connections between highway segments as $2486$ unweighted edges.
    \item[\texttt{Pajek/USPowerGrid}~\cite{SuiteSparseData}:] The dataset contains $4941$ nodes and $13188$ edges making up parts of the connections in the US power grid.
    \item[\texttt{SNAP/as-735}:] The dataset contains $7716$ nodes and $26467$ edges representing the autonomous systems of Internet router subgraphs from November 1997 to January 2000.
    \item[\texttt{SNAP/ca-HepTh}~\cite{SuiteSparseData}:] The dataset of arXiv collaborations in the High Energy Physics - Theory category from January 1993 to April 2003 contains $9877$ vertices and $51971$ edges.
\end{description}

To get a comparison between the methods of direct computation, low-rank updating from~\cite{Beckermann2018}, and our new method we only compute $\vc{1} L_{\exp}(A^T, E_{ij}) \vc{1}$ for a single pair of $i$ and $j$. These indices are chosen according to Table~\ref{tab:mat-inds}. Because the direction of change in the Fréchet derivative always has rank $1$, we can use the methods proposed in~\cite{delacruz2022}, \cite{Schweitzer2023}, the Arnoldi method for $\mathcal{A}$, and our modified Arnoldi algorithm combining Algorithms~\ref{alg:mod-arnoldi} and~\ref{alg:orth}. In Figures~\ref{fig:air-conv}, \ref{fig:autobahn-conv}, \ref{fig:uspower-conv}, \ref{fig:as-conv}, and \ref{fig:hepth-conv} we display the convergence of these four methods for all networks.

In addition to the convergence plots above we also compare the computational times for $k = 50$ steps of the four methods. These are only preliminary metrics, but they suggest that the additional costs from the computations are acceptable. In particular, the runtime of our algorithm matches that of the Arnoldi method for $\mathcal{A}$ closely. We list detailed runtimes in Table~\ref{tab:runtimes}.

\begin{table}
    \centering
    \begin{tabular}{ccc}
        \toprule
        Network & Row index $i$ & Column index $j$ \\
        \midrule
        \texttt{Air500} & 256 & 123 \\
        \texttt{Autobahn} & 218 & 605 \\
        \texttt{Pajek/USPowerGrid} & 3579 & 2400 \\
        \texttt{SNAP/as-735} & 3105 & 5000 \\
        \texttt{SNAP/ca-HepTh} & 7200 & 6969 \\
        \bottomrule
    \end{tabular}
    \caption{Row and column indices used in the computation of $S^{(TN)}_{ij}$.}
    \label{tab:mat-inds}
\end{table}

\begin{figure}
    \centering
    \begin{minipage}{.47\textwidth}
        \centering
        \resizebox{1\textwidth}{!}{\input{figs/network_conv_Air500.pgf}}
        \captionof{figure}{Convergence of $S^{(TN)}_{ij}$ for \texttt{Air500}}
        \label{fig:air-conv}
    \end{minipage}%
    \hfill%
    \begin{minipage}{.47\textwidth}
        \centering
        \resizebox{1\textwidth}{!}{\input{figs/network_conv_Autobahn.pgf}}
        \captionof{figure}{Convergence of $S^{(TN)}_{ij}$ for \texttt{Autobahn}}
        \label{fig:autobahn-conv}
    \end{minipage}
\end{figure}

\begin{figure}
    \centering
    \begin{minipage}{.47\textwidth}
        \centering
        \resizebox{1\textwidth}{!}{\input{figs/network_conv_USPowerGrid.pgf}}
        \captionof{figure}{Convergence of $S^{(TN)}_{ij}$ for \texttt{Pajek/USPowerGrid}}
        \label{fig:uspower-conv}
    \end{minipage}%
    \hfill%
    \begin{minipage}{.47\textwidth}
        \centering
        \resizebox{1\textwidth}{!}{\input{figs/network_conv_as-735.pgf}}
        \captionof{figure}{Convergence of $S^{(TN)}_{ij}$ for \texttt{SNAP/as-735}}
        \label{fig:as-conv}
    \end{minipage}
\end{figure}

\begin{figure}
    \centering
    \resizebox{0.47\textwidth}{!}{\input{figs/network_conv_HepTh.pgf}}
    \caption{Convergence of $S^{(TN)}_{ij}$ for \texttt{SNAP/ca-HepTh}}
    \label{fig:hepth-conv}
\end{figure}


\begin{table}[!ht]
    \centering
    \footnotesize
    \begin{filecontents*}{figs/runtimes.csv}
network,arnoldi,noschese,schweitzer,ours
Air500,0.015106916427612305,0.034232139587402344,0.02541518211364746,0.007349252700805664
Autobahn,0.02712559700012207,0.03959989547729492,0.03605008125305176,0.010908842086791992
USPowerGrid,0.09600663185119629,0.1645030975341797,0.13300108909606934,0.03422188758850098
as-735,0.14859986305236816,0.27531957626342773,0.21540236473083496,0.06691765785217285
HepTh,0.18254423141479492,0.40462565422058105,0.2805747985839844,0.06598472595214844
\end{filecontents*}
\pgfplotstabletypeset[
        col sep=comma,
        header=true,
        display columns/0/.style={
            string type,
            column name={Network},
            postproc cell content/.style={@cell content={\texttt{##1}}}
        },
        display columns/1/.style={
            fixed,
            precision=5,
            column name={Arnoldi for $\mathcal{A}$}
        },
        display columns/2/.style={
            fixed,
            precision=5,
            column name={Noschese}
        },
        display columns/3/.style={
            fixed,
            precision=5,
            column name={Schweitzer}
        },
        display columns/4/.style={
            fixed,
            precision=5,
            column name={Modified Arnoldi}
        },
        every head row/.style={
            before row=\toprule,
            after row=\midrule
        },
        every last row/.style={
            after row=\bottomrule
        },
    ]{figs/runtimes.csv}
    \caption{Runtimes in seconds of the four methods computing $S^{(TN)}_{ij}$ for the different networks with $k = 50$ steps.}
    \label{tab:runtimes}
\end{table}

To demonstrate the capability of our algorithm in a setting where the direction $E$ of the Fréchet derivative $L_{\exp}(A, E)$ is not low-rank we take the quantity $\vc{1}^T L_{\exp}(A, \vc{e}_i \vc{e}_j^T) \vc{1}$ and instead compute $S^{(FR-TN)}_{ij} = \vc{1}^T L_{\exp}(A, E) \vc{1}$. The matrix $E$ is chosen such that $E_{ij} = 1$ if $A_{ij} \neq 0$, which generally results in $E$ not low-rank. In this case the methods of~\cite{delacruz2022} and~\cite{Schweitzer2023} do not apply, and thus we only show the convergence of the Arnoldi method for $\mathcal{A}$ and our modified Arnoldi algorithm combining Algorithms~\ref{alg:mod-arnoldi} and~\ref{alg:orth} in Figures~\ref{fig:ei-air-conv}, \ref{fig:ei-autobahn-conv}, \ref{fig:ei-uspower-conv}, \ref{fig:ei-as-conv}, and \ref{fig:ei-hepth-conv}.

\begin{figure}
    \centering
    \begin{minipage}{.47\textwidth}
        \centering
        \resizebox{1\textwidth}{!}{\input{figs/network_conv_ei_Air500.pgf}}
        \captionof{figure}{Convergence of $S^{(FR-TN)}_{ij}$ for \texttt{Air500}}
        \label{fig:ei-air-conv}
    \end{minipage}%
    \hfill%
    \begin{minipage}{.47\textwidth}
        \centering
        \resizebox{1\textwidth}{!}{\input{figs/network_conv_ei_Autobahn.pgf}}
        \captionof{figure}{Convergence of $S^{(FR-TN)}_{ij}$ for \texttt{Autobahn}}
        \label{fig:ei-autobahn-conv}
    \end{minipage}
\end{figure}

\begin{figure}
    \centering
    \begin{minipage}{.47\textwidth}
        \centering
        \resizebox{1\textwidth}{!}{\input{figs/network_conv_ei_USPowerGrid.pgf}}
        \captionof{figure}{Convergence of $S^{(FR-TN)}_{ij}$ for \texttt{Pajek/USPowerGrid}}
        \label{fig:ei-uspower-conv}
    \end{minipage}%
    \hfill%
    \begin{minipage}{.47\textwidth}
        \centering
        \resizebox{1\textwidth}{!}{\input{figs/network_conv_ei_as-735.pgf}}
        \captionof{figure}{Convergence of $S^{(FR-TN)}_{ij}$ for \texttt{SNAP/as-735}}
        \label{fig:ei-as-conv}
    \end{minipage}
\end{figure}

\begin{figure}
    \centering
    \resizebox{0.47\textwidth}{!}{\input{figs/network_conv_ei_HepTh.pgf}}
    \caption{Convergence of $S^{(FR-TN)}_{ij}$ for \texttt{SNAP/ca-HepTh}}
    \label{fig:ei-hepth-conv}
\end{figure}

\subsection{Parameter Fitting for the 2D Parametric Heat Equation}

We demonstrate the application of Fréchet derivatives in the setting of parameter fitting. The experiment we conduct here serves as a simplified version of the more advanced examples used in machine learning, which would go beyond the scope of this paper.

Let $\Omega = {[-1, 1]}^2$ be a 2D domain and $T = 1$ the final time. The parametric heat equation with homogeneous Dirichlet boundary conditions is given by
\begin{equation}\label{eq:param-heat-eq}
    \begin{aligned}
        \partial_t u(t, x, y) &= \sigma \Delta u(t, x, y),& &\forall t \in [0, T], x, y \in \Omega \\
        u(t, x, y) &= 0,& &\forall t \in [0, T], x, y \in \partial \Omega, \\
        u(0, x, y) &= u_0(x, y),& &\forall x, y \in \Omega,
    \end{aligned}
\end{equation}
where $\sigma$ is a scalar parameter determining the thermal conductivity of the domain $\Omega$.

We discretize~\eqref{eq:param-heat-eq} on an equispaced $75 \times 75$ grid using second order central finite differences, resulting in a system matrix $A(\sigma)$ of size $5625 \times 5625$. The solution of the discretized equation~\eqref{eq:param-heat-eq} is given by the solution operator $\vc{s}(\sigma) \coloneqq \exp{(T A(\sigma))} \vc{u}_0$. This system serves as a simple model in parameter fitting where the objective is to find a parameter value $\sigma^*$ which, given a reference solution $\vc{s}_\text{ref}$ minimizes
\[
    f(\sigma) = \Vert \vc{s}(\sigma) - \vc{s}_\text{ref} \Vert_2^2.
\]
To solve this problem we employ a gradient descent method, requiring the evaluation of $f(\sigma_k)$ and $f'(\sigma_k)$ for multiple parameter values $\sigma_1, \sigma_2, \dots$ This is expensive because every evaluation of $f$ requires the computation of the matrix function $\exp{(T A(\sigma)} \vc{u}_0$, and approximations of $f'$ via techniques such as finite differences require additional evaluations of $f$. Instead, we will use Algorithm~\ref{alg:mod-arnoldi} to compute the matrix function action $\exp{(T A(\sigma_k))} \vc{u}_0$ and its Fréchet derivative action $L_{\exp}(T A(\sigma_k), \partial_\sigma A(\sigma_k) \vc{u}_0$ at the same time for a right-hand side.

In Figure~\ref{fig:param-fitting-conv} we plot the value of $f(\sigma)$ over the trajectory of $\sigma$ computed in the gradient descent algorithm. We used an initial step size of $1/2$ and an absolute stopping tolerance of $10^{-8}$. In Figure~\ref{fig:param-traj} we additionally show the trajectory of the parameter $\sigma$ from the initial guess $\sigma = 1$ to the optimal value of $\sigma = 0.85$.

\begin{figure}
    \centering
    \begin{minipage}{.47\textwidth}
        \centering
        \resizebox{1\textwidth}{!}{\input{figs/heat_conv.pgf}}
        \captionof{figure}{Convergence of the function value $f(\sigma)$ during the gradient descent}
        \label{fig:param-fitting-conv}
    \end{minipage}%
    \hfill%
    \begin{minipage}{.47\textwidth}
        \centering
        \resizebox{1\textwidth}{!}{\input{figs/heat_param.pgf}}
        \captionof{figure}{Convergence of the parameter value $\sigma$ during the gradient descent}
        \label{fig:param-traj}
    \end{minipage}
\end{figure}

%% file: figs/network_conv_Air500.pgf
\begingroup%
\makeatletter%
\begin{pgfpicture}%
\pgfpathrectangle{\pgfpointorigin}{\pgfqpoint{5.825290in}{4.278899in}}%
\pgfusepath{use as bounding box, clip}%
\begin{pgfscope}%
\pgfsetbuttcap%
\pgfsetmiterjoin%
\definecolor{currentfill}{rgb}{1.000000,1.000000,1.000000}%
\pgfsetfillcolor{currentfill}%
\pgfsetlinewidth{0.000000pt}%
\definecolor{currentstroke}{rgb}{1.000000,1.000000,1.000000}%
\pgfsetstrokecolor{currentstroke}%
\pgfsetdash{}{0pt}%
\pgfpathmoveto{\pgfqpoint{0.000000in}{0.000000in}}%
\pgfpathlineto{\pgfqpoint{5.825290in}{0.000000in}}%
\pgfpathlineto{\pgfqpoint{5.825290in}{4.278899in}}%
\pgfpathlineto{\pgfqpoint{0.000000in}{4.278899in}}%
\pgfpathlineto{\pgfqpoint{0.000000in}{0.000000in}}%
\pgfpathclose%
\pgfusepath{fill}%
\end{pgfscope}%
\begin{pgfscope}%
\pgfsetbuttcap%
\pgfsetmiterjoin%
\definecolor{currentfill}{rgb}{1.000000,1.000000,1.000000}%
\pgfsetfillcolor{currentfill}%
\pgfsetlinewidth{0.000000pt}%
\definecolor{currentstroke}{rgb}{0.000000,0.000000,0.000000}%
\pgfsetstrokecolor{currentstroke}%
\pgfsetstrokeopacity{0.000000}%
\pgfsetdash{}{0pt}%
\pgfpathmoveto{\pgfqpoint{0.865290in}{0.582899in}}%
\pgfpathlineto{\pgfqpoint{5.825290in}{0.582899in}}%
\pgfpathlineto{\pgfqpoint{5.825290in}{4.278899in}}%
\pgfpathlineto{\pgfqpoint{0.865290in}{4.278899in}}%
\pgfpathlineto{\pgfqpoint{0.865290in}{0.582899in}}%
\pgfpathclose%
\pgfusepath{fill}%
\end{pgfscope}%
\begin{pgfscope}%
\pgfsetbuttcap%
\pgfsetroundjoin%
\definecolor{currentfill}{rgb}{0.000000,0.000000,0.000000}%
\pgfsetfillcolor{currentfill}%
\pgfsetlinewidth{0.803000pt}%
\definecolor{currentstroke}{rgb}{0.000000,0.000000,0.000000}%
\pgfsetstrokecolor{currentstroke}%
\pgfsetdash{}{0pt}%
\pgfsys@defobject{currentmarker}{\pgfqpoint{0.000000in}{-0.048611in}}{\pgfqpoint{0.000000in}{0.000000in}}{%
\pgfpathmoveto{\pgfqpoint{0.000000in}{0.000000in}}%
\pgfpathlineto{\pgfqpoint{0.000000in}{-0.048611in}}%
\pgfusepath{stroke,fill}%
}%
\begin{pgfscope}%
\pgfsys@transformshift{1.372563in}{0.582899in}%
\pgfsys@useobject{currentmarker}{}%
\end{pgfscope}%
\end{pgfscope}%
\begin{pgfscope}%
\definecolor{textcolor}{rgb}{0.000000,0.000000,0.000000}%
\pgfsetstrokecolor{textcolor}%
\pgfsetfillcolor{textcolor}%
\pgftext[x=1.372563in,y=0.485677in,,top]{\color{textcolor}{\sffamily\fontsize{16.000000}{19.200000}\selectfont\catcode`\^=\active\def^{\ifmmode\sp\else\^{}\fi}\catcode`\%=\active\def
\end{pgfscope}%
\begin{pgfscope}%
\pgfsetbuttcap%
\pgfsetroundjoin%
\definecolor{currentfill}{rgb}{0.000000,0.000000,0.000000}%
\pgfsetfillcolor{currentfill}%
\pgfsetlinewidth{0.803000pt}%
\definecolor{currentstroke}{rgb}{0.000000,0.000000,0.000000}%
\pgfsetstrokecolor{currentstroke}%
\pgfsetdash{}{0pt}%
\pgfsys@defobject{currentmarker}{\pgfqpoint{0.000000in}{-0.048611in}}{\pgfqpoint{0.000000in}{0.000000in}}{%
\pgfpathmoveto{\pgfqpoint{0.000000in}{0.000000in}}%
\pgfpathlineto{\pgfqpoint{0.000000in}{-0.048611in}}%
\pgfusepath{stroke,fill}%
}%
\begin{pgfscope}%
\pgfsys@transformshift{1.842260in}{0.582899in}%
\pgfsys@useobject{currentmarker}{}%
\end{pgfscope}%
\end{pgfscope}%
\begin{pgfscope}%
\definecolor{textcolor}{rgb}{0.000000,0.000000,0.000000}%
\pgfsetstrokecolor{textcolor}%
\pgfsetfillcolor{textcolor}%
\pgftext[x=1.842260in,y=0.485677in,,top]{\color{textcolor}{\sffamily\fontsize{16.000000}{19.200000}\selectfont\catcode`\^=\active\def^{\ifmmode\sp\else\^{}\fi}\catcode`\%=\active\def
\end{pgfscope}%
\begin{pgfscope}%
\pgfsetbuttcap%
\pgfsetroundjoin%
\definecolor{currentfill}{rgb}{0.000000,0.000000,0.000000}%
\pgfsetfillcolor{currentfill}%
\pgfsetlinewidth{0.803000pt}%
\definecolor{currentstroke}{rgb}{0.000000,0.000000,0.000000}%
\pgfsetstrokecolor{currentstroke}%
\pgfsetdash{}{0pt}%
\pgfsys@defobject{currentmarker}{\pgfqpoint{0.000000in}{-0.048611in}}{\pgfqpoint{0.000000in}{0.000000in}}{%
\pgfpathmoveto{\pgfqpoint{0.000000in}{0.000000in}}%
\pgfpathlineto{\pgfqpoint{0.000000in}{-0.048611in}}%
\pgfusepath{stroke,fill}%
}%
\begin{pgfscope}%
\pgfsys@transformshift{2.311957in}{0.582899in}%
\pgfsys@useobject{currentmarker}{}%
\end{pgfscope}%
\end{pgfscope}%
\begin{pgfscope}%
\definecolor{textcolor}{rgb}{0.000000,0.000000,0.000000}%
\pgfsetstrokecolor{textcolor}%
\pgfsetfillcolor{textcolor}%
\pgftext[x=2.311957in,y=0.485677in,,top]{\color{textcolor}{\sffamily\fontsize{16.000000}{19.200000}\selectfont\catcode`\^=\active\def^{\ifmmode\sp\else\^{}\fi}\catcode`\%=\active\def
\end{pgfscope}%
\begin{pgfscope}%
\pgfsetbuttcap%
\pgfsetroundjoin%
\definecolor{currentfill}{rgb}{0.000000,0.000000,0.000000}%
\pgfsetfillcolor{currentfill}%
\pgfsetlinewidth{0.803000pt}%
\definecolor{currentstroke}{rgb}{0.000000,0.000000,0.000000}%
\pgfsetstrokecolor{currentstroke}%
\pgfsetdash{}{0pt}%
\pgfsys@defobject{currentmarker}{\pgfqpoint{0.000000in}{-0.048611in}}{\pgfqpoint{0.000000in}{0.000000in}}{%
\pgfpathmoveto{\pgfqpoint{0.000000in}{0.000000in}}%
\pgfpathlineto{\pgfqpoint{0.000000in}{-0.048611in}}%
\pgfusepath{stroke,fill}%
}%
\begin{pgfscope}%
\pgfsys@transformshift{2.781654in}{0.582899in}%
\pgfsys@useobject{currentmarker}{}%
\end{pgfscope}%
\end{pgfscope}%
\begin{pgfscope}%
\definecolor{textcolor}{rgb}{0.000000,0.000000,0.000000}%
\pgfsetstrokecolor{textcolor}%
\pgfsetfillcolor{textcolor}%
\pgftext[x=2.781654in,y=0.485677in,,top]{\color{textcolor}{\sffamily\fontsize{16.000000}{19.200000}\selectfont\catcode`\^=\active\def^{\ifmmode\sp\else\^{}\fi}\catcode`\%=\active\def
\end{pgfscope}%
\begin{pgfscope}%
\pgfsetbuttcap%
\pgfsetroundjoin%
\definecolor{currentfill}{rgb}{0.000000,0.000000,0.000000}%
\pgfsetfillcolor{currentfill}%
\pgfsetlinewidth{0.803000pt}%
\definecolor{currentstroke}{rgb}{0.000000,0.000000,0.000000}%
\pgfsetstrokecolor{currentstroke}%
\pgfsetdash{}{0pt}%
\pgfsys@defobject{currentmarker}{\pgfqpoint{0.000000in}{-0.048611in}}{\pgfqpoint{0.000000in}{0.000000in}}{%
\pgfpathmoveto{\pgfqpoint{0.000000in}{0.000000in}}%
\pgfpathlineto{\pgfqpoint{0.000000in}{-0.048611in}}%
\pgfusepath{stroke,fill}%
}%
\begin{pgfscope}%
\pgfsys@transformshift{3.251351in}{0.582899in}%
\pgfsys@useobject{currentmarker}{}%
\end{pgfscope}%
\end{pgfscope}%
\begin{pgfscope}%
\definecolor{textcolor}{rgb}{0.000000,0.000000,0.000000}%
\pgfsetstrokecolor{textcolor}%
\pgfsetfillcolor{textcolor}%
\pgftext[x=3.251351in,y=0.485677in,,top]{\color{textcolor}{\sffamily\fontsize{16.000000}{19.200000}\selectfont\catcode`\^=\active\def^{\ifmmode\sp\else\^{}\fi}\catcode`\%=\active\def
\end{pgfscope}%
\begin{pgfscope}%
\pgfsetbuttcap%
\pgfsetroundjoin%
\definecolor{currentfill}{rgb}{0.000000,0.000000,0.000000}%
\pgfsetfillcolor{currentfill}%
\pgfsetlinewidth{0.803000pt}%
\definecolor{currentstroke}{rgb}{0.000000,0.000000,0.000000}%
\pgfsetstrokecolor{currentstroke}%
\pgfsetdash{}{0pt}%
\pgfsys@defobject{currentmarker}{\pgfqpoint{0.000000in}{-0.048611in}}{\pgfqpoint{0.000000in}{0.000000in}}{%
\pgfpathmoveto{\pgfqpoint{0.000000in}{0.000000in}}%
\pgfpathlineto{\pgfqpoint{0.000000in}{-0.048611in}}%
\pgfusepath{stroke,fill}%
}%
\begin{pgfscope}%
\pgfsys@transformshift{3.721048in}{0.582899in}%
\pgfsys@useobject{currentmarker}{}%
\end{pgfscope}%
\end{pgfscope}%
\begin{pgfscope}%
\definecolor{textcolor}{rgb}{0.000000,0.000000,0.000000}%
\pgfsetstrokecolor{textcolor}%
\pgfsetfillcolor{textcolor}%
\pgftext[x=3.721048in,y=0.485677in,,top]{\color{textcolor}{\sffamily\fontsize{16.000000}{19.200000}\selectfont\catcode`\^=\active\def^{\ifmmode\sp\else\^{}\fi}\catcode`\%=\active\def
\end{pgfscope}%
\begin{pgfscope}%
\pgfsetbuttcap%
\pgfsetroundjoin%
\definecolor{currentfill}{rgb}{0.000000,0.000000,0.000000}%
\pgfsetfillcolor{currentfill}%
\pgfsetlinewidth{0.803000pt}%
\definecolor{currentstroke}{rgb}{0.000000,0.000000,0.000000}%
\pgfsetstrokecolor{currentstroke}%
\pgfsetdash{}{0pt}%
\pgfsys@defobject{currentmarker}{\pgfqpoint{0.000000in}{-0.048611in}}{\pgfqpoint{0.000000in}{0.000000in}}{%
\pgfpathmoveto{\pgfqpoint{0.000000in}{0.000000in}}%
\pgfpathlineto{\pgfqpoint{0.000000in}{-0.048611in}}%
\pgfusepath{stroke,fill}%
}%
\begin{pgfscope}%
\pgfsys@transformshift{4.190745in}{0.582899in}%
\pgfsys@useobject{currentmarker}{}%
\end{pgfscope}%
\end{pgfscope}%
\begin{pgfscope}%
\definecolor{textcolor}{rgb}{0.000000,0.000000,0.000000}%
\pgfsetstrokecolor{textcolor}%
\pgfsetfillcolor{textcolor}%
\pgftext[x=4.190745in,y=0.485677in,,top]{\color{textcolor}{\sffamily\fontsize{16.000000}{19.200000}\selectfont\catcode`\^=\active\def^{\ifmmode\sp\else\^{}\fi}\catcode`\%=\active\def
\end{pgfscope}%
\begin{pgfscope}%
\pgfsetbuttcap%
\pgfsetroundjoin%
\definecolor{currentfill}{rgb}{0.000000,0.000000,0.000000}%
\pgfsetfillcolor{currentfill}%
\pgfsetlinewidth{0.803000pt}%
\definecolor{currentstroke}{rgb}{0.000000,0.000000,0.000000}%
\pgfsetstrokecolor{currentstroke}%
\pgfsetdash{}{0pt}%
\pgfsys@defobject{currentmarker}{\pgfqpoint{0.000000in}{-0.048611in}}{\pgfqpoint{0.000000in}{0.000000in}}{%
\pgfpathmoveto{\pgfqpoint{0.000000in}{0.000000in}}%
\pgfpathlineto{\pgfqpoint{0.000000in}{-0.048611in}}%
\pgfusepath{stroke,fill}%
}%
\begin{pgfscope}%
\pgfsys@transformshift{4.660442in}{0.582899in}%
\pgfsys@useobject{currentmarker}{}%
\end{pgfscope}%
\end{pgfscope}%
\begin{pgfscope}%
\definecolor{textcolor}{rgb}{0.000000,0.000000,0.000000}%
\pgfsetstrokecolor{textcolor}%
\pgfsetfillcolor{textcolor}%
\pgftext[x=4.660442in,y=0.485677in,,top]{\color{textcolor}{\sffamily\fontsize{16.000000}{19.200000}\selectfont\catcode`\^=\active\def^{\ifmmode\sp\else\^{}\fi}\catcode`\%=\active\def
\end{pgfscope}%
\begin{pgfscope}%
\pgfsetbuttcap%
\pgfsetroundjoin%
\definecolor{currentfill}{rgb}{0.000000,0.000000,0.000000}%
\pgfsetfillcolor{currentfill}%
\pgfsetlinewidth{0.803000pt}%
\definecolor{currentstroke}{rgb}{0.000000,0.000000,0.000000}%
\pgfsetstrokecolor{currentstroke}%
\pgfsetdash{}{0pt}%
\pgfsys@defobject{currentmarker}{\pgfqpoint{0.000000in}{-0.048611in}}{\pgfqpoint{0.000000in}{0.000000in}}{%
\pgfpathmoveto{\pgfqpoint{0.000000in}{0.000000in}}%
\pgfpathlineto{\pgfqpoint{0.000000in}{-0.048611in}}%
\pgfusepath{stroke,fill}%
}%
\begin{pgfscope}%
\pgfsys@transformshift{5.130139in}{0.582899in}%
\pgfsys@useobject{currentmarker}{}%
\end{pgfscope}%
\end{pgfscope}%
\begin{pgfscope}%
\definecolor{textcolor}{rgb}{0.000000,0.000000,0.000000}%
\pgfsetstrokecolor{textcolor}%
\pgfsetfillcolor{textcolor}%
\pgftext[x=5.130139in,y=0.485677in,,top]{\color{textcolor}{\sffamily\fontsize{16.000000}{19.200000}\selectfont\catcode`\^=\active\def^{\ifmmode\sp\else\^{}\fi}\catcode`\%=\active\def
\end{pgfscope}%
\begin{pgfscope}%
\pgfsetbuttcap%
\pgfsetroundjoin%
\definecolor{currentfill}{rgb}{0.000000,0.000000,0.000000}%
\pgfsetfillcolor{currentfill}%
\pgfsetlinewidth{0.803000pt}%
\definecolor{currentstroke}{rgb}{0.000000,0.000000,0.000000}%
\pgfsetstrokecolor{currentstroke}%
\pgfsetdash{}{0pt}%
\pgfsys@defobject{currentmarker}{\pgfqpoint{0.000000in}{-0.048611in}}{\pgfqpoint{0.000000in}{0.000000in}}{%
\pgfpathmoveto{\pgfqpoint{0.000000in}{0.000000in}}%
\pgfpathlineto{\pgfqpoint{0.000000in}{-0.048611in}}%
\pgfusepath{stroke,fill}%
}%
\begin{pgfscope}%
\pgfsys@transformshift{5.599836in}{0.582899in}%
\pgfsys@useobject{currentmarker}{}%
\end{pgfscope}%
\end{pgfscope}%
\begin{pgfscope}%
\definecolor{textcolor}{rgb}{0.000000,0.000000,0.000000}%
\pgfsetstrokecolor{textcolor}%
\pgfsetfillcolor{textcolor}%
\pgftext[x=5.599836in,y=0.485677in,,top]{\color{textcolor}{\sffamily\fontsize{16.000000}{19.200000}\selectfont\catcode`\^=\active\def^{\ifmmode\sp\else\^{}\fi}\catcode`\%=\active\def
\end{pgfscope}%
\begin{pgfscope}%
\definecolor{textcolor}{rgb}{0.000000,0.000000,0.000000}%
\pgfsetstrokecolor{textcolor}%
\pgfsetfillcolor{textcolor}%
\pgftext[x=3.345290in,y=0.215061in,,top]{\color{textcolor}{\sffamily\fontsize{16.000000}{19.200000}\selectfont\catcode`\^=\active\def^{\ifmmode\sp\else\^{}\fi}\catcode`\%=\active\def
\end{pgfscope}%
\begin{pgfscope}%
\pgfsetbuttcap%
\pgfsetroundjoin%
\definecolor{currentfill}{rgb}{0.000000,0.000000,0.000000}%
\pgfsetfillcolor{currentfill}%
\pgfsetlinewidth{0.803000pt}%
\definecolor{currentstroke}{rgb}{0.000000,0.000000,0.000000}%
\pgfsetstrokecolor{currentstroke}%
\pgfsetdash{}{0pt}%
\pgfsys@defobject{currentmarker}{\pgfqpoint{-0.048611in}{0.000000in}}{\pgfqpoint{-0.000000in}{0.000000in}}{%
\pgfpathmoveto{\pgfqpoint{-0.000000in}{0.000000in}}%
\pgfpathlineto{\pgfqpoint{-0.048611in}{0.000000in}}%
\pgfusepath{stroke,fill}%
}%
\begin{pgfscope}%
\pgfsys@transformshift{0.865290in}{1.051809in}%
\pgfsys@useobject{currentmarker}{}%
\end{pgfscope}%
\end{pgfscope}%
\begin{pgfscope}%
\definecolor{textcolor}{rgb}{0.000000,0.000000,0.000000}%
\pgfsetstrokecolor{textcolor}%
\pgfsetfillcolor{textcolor}%
\pgftext[x=0.270616in, y=0.967391in, left, base]{\color{textcolor}{\sffamily\fontsize{16.000000}{19.200000}\selectfont\catcode`\^=\active\def^{\ifmmode\sp\else\^{}\fi}\catcode`\%=\active\def
\end{pgfscope}%
\begin{pgfscope}%
\pgfsetbuttcap%
\pgfsetroundjoin%
\definecolor{currentfill}{rgb}{0.000000,0.000000,0.000000}%
\pgfsetfillcolor{currentfill}%
\pgfsetlinewidth{0.803000pt}%
\definecolor{currentstroke}{rgb}{0.000000,0.000000,0.000000}%
\pgfsetstrokecolor{currentstroke}%
\pgfsetdash{}{0pt}%
\pgfsys@defobject{currentmarker}{\pgfqpoint{-0.048611in}{0.000000in}}{\pgfqpoint{-0.000000in}{0.000000in}}{%
\pgfpathmoveto{\pgfqpoint{-0.000000in}{0.000000in}}%
\pgfpathlineto{\pgfqpoint{-0.048611in}{0.000000in}}%
\pgfusepath{stroke,fill}%
}%
\begin{pgfscope}%
\pgfsys@transformshift{0.865290in}{1.561657in}%
\pgfsys@useobject{currentmarker}{}%
\end{pgfscope}%
\end{pgfscope}%
\begin{pgfscope}%
\definecolor{textcolor}{rgb}{0.000000,0.000000,0.000000}%
\pgfsetstrokecolor{textcolor}%
\pgfsetfillcolor{textcolor}%
\pgftext[x=0.270616in, y=1.477239in, left, base]{\color{textcolor}{\sffamily\fontsize{16.000000}{19.200000}\selectfont\catcode`\^=\active\def^{\ifmmode\sp\else\^{}\fi}\catcode`\%=\active\def
\end{pgfscope}%
\begin{pgfscope}%
\pgfsetbuttcap%
\pgfsetroundjoin%
\definecolor{currentfill}{rgb}{0.000000,0.000000,0.000000}%
\pgfsetfillcolor{currentfill}%
\pgfsetlinewidth{0.803000pt}%
\definecolor{currentstroke}{rgb}{0.000000,0.000000,0.000000}%
\pgfsetstrokecolor{currentstroke}%
\pgfsetdash{}{0pt}%
\pgfsys@defobject{currentmarker}{\pgfqpoint{-0.048611in}{0.000000in}}{\pgfqpoint{-0.000000in}{0.000000in}}{%
\pgfpathmoveto{\pgfqpoint{-0.000000in}{0.000000in}}%
\pgfpathlineto{\pgfqpoint{-0.048611in}{0.000000in}}%
\pgfusepath{stroke,fill}%
}%
\begin{pgfscope}%
\pgfsys@transformshift{0.865290in}{2.071506in}%
\pgfsys@useobject{currentmarker}{}%
\end{pgfscope}%
\end{pgfscope}%
\begin{pgfscope}%
\definecolor{textcolor}{rgb}{0.000000,0.000000,0.000000}%
\pgfsetstrokecolor{textcolor}%
\pgfsetfillcolor{textcolor}%
\pgftext[x=0.346658in, y=1.987087in, left, base]{\color{textcolor}{\sffamily\fontsize{16.000000}{19.200000}\selectfont\catcode`\^=\active\def^{\ifmmode\sp\else\^{}\fi}\catcode`\%=\active\def
\end{pgfscope}%
\begin{pgfscope}%
\pgfsetbuttcap%
\pgfsetroundjoin%
\definecolor{currentfill}{rgb}{0.000000,0.000000,0.000000}%
\pgfsetfillcolor{currentfill}%
\pgfsetlinewidth{0.803000pt}%
\definecolor{currentstroke}{rgb}{0.000000,0.000000,0.000000}%
\pgfsetstrokecolor{currentstroke}%
\pgfsetdash{}{0pt}%
\pgfsys@defobject{currentmarker}{\pgfqpoint{-0.048611in}{0.000000in}}{\pgfqpoint{-0.000000in}{0.000000in}}{%
\pgfpathmoveto{\pgfqpoint{-0.000000in}{0.000000in}}%
\pgfpathlineto{\pgfqpoint{-0.048611in}{0.000000in}}%
\pgfusepath{stroke,fill}%
}%
\begin{pgfscope}%
\pgfsys@transformshift{0.865290in}{2.581354in}%
\pgfsys@useobject{currentmarker}{}%
\end{pgfscope}%
\end{pgfscope}%
\begin{pgfscope}%
\definecolor{textcolor}{rgb}{0.000000,0.000000,0.000000}%
\pgfsetstrokecolor{textcolor}%
\pgfsetfillcolor{textcolor}%
\pgftext[x=0.346658in, y=2.496936in, left, base]{\color{textcolor}{\sffamily\fontsize{16.000000}{19.200000}\selectfont\catcode`\^=\active\def^{\ifmmode\sp\else\^{}\fi}\catcode`\%=\active\def
\end{pgfscope}%
\begin{pgfscope}%
\pgfsetbuttcap%
\pgfsetroundjoin%
\definecolor{currentfill}{rgb}{0.000000,0.000000,0.000000}%
\pgfsetfillcolor{currentfill}%
\pgfsetlinewidth{0.803000pt}%
\definecolor{currentstroke}{rgb}{0.000000,0.000000,0.000000}%
\pgfsetstrokecolor{currentstroke}%
\pgfsetdash{}{0pt}%
\pgfsys@defobject{currentmarker}{\pgfqpoint{-0.048611in}{0.000000in}}{\pgfqpoint{-0.000000in}{0.000000in}}{%
\pgfpathmoveto{\pgfqpoint{-0.000000in}{0.000000in}}%
\pgfpathlineto{\pgfqpoint{-0.048611in}{0.000000in}}%
\pgfusepath{stroke,fill}%
}%
\begin{pgfscope}%
\pgfsys@transformshift{0.865290in}{3.091202in}%
\pgfsys@useobject{currentmarker}{}%
\end{pgfscope}%
\end{pgfscope}%
\begin{pgfscope}%
\definecolor{textcolor}{rgb}{0.000000,0.000000,0.000000}%
\pgfsetstrokecolor{textcolor}%
\pgfsetfillcolor{textcolor}%
\pgftext[x=0.346658in, y=3.006784in, left, base]{\color{textcolor}{\sffamily\fontsize{16.000000}{19.200000}\selectfont\catcode`\^=\active\def^{\ifmmode\sp\else\^{}\fi}\catcode`\%=\active\def
\end{pgfscope}%
\begin{pgfscope}%
\pgfsetbuttcap%
\pgfsetroundjoin%
\definecolor{currentfill}{rgb}{0.000000,0.000000,0.000000}%
\pgfsetfillcolor{currentfill}%
\pgfsetlinewidth{0.803000pt}%
\definecolor{currentstroke}{rgb}{0.000000,0.000000,0.000000}%
\pgfsetstrokecolor{currentstroke}%
\pgfsetdash{}{0pt}%
\pgfsys@defobject{currentmarker}{\pgfqpoint{-0.048611in}{0.000000in}}{\pgfqpoint{-0.000000in}{0.000000in}}{%
\pgfpathmoveto{\pgfqpoint{-0.000000in}{0.000000in}}%
\pgfpathlineto{\pgfqpoint{-0.048611in}{0.000000in}}%
\pgfusepath{stroke,fill}%
}%
\begin{pgfscope}%
\pgfsys@transformshift{0.865290in}{3.601051in}%
\pgfsys@useobject{currentmarker}{}%
\end{pgfscope}%
\end{pgfscope}%
\begin{pgfscope}%
\definecolor{textcolor}{rgb}{0.000000,0.000000,0.000000}%
\pgfsetstrokecolor{textcolor}%
\pgfsetfillcolor{textcolor}%
\pgftext[x=0.346658in, y=3.516632in, left, base]{\color{textcolor}{\sffamily\fontsize{16.000000}{19.200000}\selectfont\catcode`\^=\active\def^{\ifmmode\sp\else\^{}\fi}\catcode`\%=\active\def
\end{pgfscope}%
\begin{pgfscope}%
\pgfsetbuttcap%
\pgfsetroundjoin%
\definecolor{currentfill}{rgb}{0.000000,0.000000,0.000000}%
\pgfsetfillcolor{currentfill}%
\pgfsetlinewidth{0.803000pt}%
\definecolor{currentstroke}{rgb}{0.000000,0.000000,0.000000}%
\pgfsetstrokecolor{currentstroke}%
\pgfsetdash{}{0pt}%
\pgfsys@defobject{currentmarker}{\pgfqpoint{-0.048611in}{0.000000in}}{\pgfqpoint{-0.000000in}{0.000000in}}{%
\pgfpathmoveto{\pgfqpoint{-0.000000in}{0.000000in}}%
\pgfpathlineto{\pgfqpoint{-0.048611in}{0.000000in}}%
\pgfusepath{stroke,fill}%
}%
\begin{pgfscope}%
\pgfsys@transformshift{0.865290in}{4.110899in}%
\pgfsys@useobject{currentmarker}{}%
\end{pgfscope}%
\end{pgfscope}%
\begin{pgfscope}%
\definecolor{textcolor}{rgb}{0.000000,0.000000,0.000000}%
\pgfsetstrokecolor{textcolor}%
\pgfsetfillcolor{textcolor}%
\pgftext[x=0.464945in, y=4.026480in, left, base]{\color{textcolor}{\sffamily\fontsize{16.000000}{19.200000}\selectfont\catcode`\^=\active\def^{\ifmmode\sp\else\^{}\fi}\catcode`\%=\active\def
\end{pgfscope}%
\begin{pgfscope}%
\definecolor{textcolor}{rgb}{0.000000,0.000000,0.000000}%
\pgfsetstrokecolor{textcolor}%
\pgfsetfillcolor{textcolor}%
\pgftext[x=0.215061in,y=2.430899in,,bottom,rotate=90.000000]{\color{textcolor}{\sffamily\fontsize{16.000000}{19.200000}\selectfont\catcode`\^=\active\def^{\ifmmode\sp\else\^{}\fi}\catcode`\%=\active\def
\end{pgfscope}%
\begin{pgfscope}%
\pgfpathrectangle{\pgfqpoint{0.865290in}{0.582899in}}{\pgfqpoint{4.960000in}{3.696000in}}%
\pgfusepath{clip}%
\pgfsetrectcap%
\pgfsetroundjoin%
\pgfsetlinewidth{1.505625pt}%
\definecolor{currentstroke}{rgb}{0.121569,0.466667,0.705882}%
\pgfsetstrokecolor{currentstroke}%
\pgfsetdash{}{0pt}%
\pgfpathmoveto{\pgfqpoint{1.090745in}{4.110879in}}%
\pgfpathlineto{\pgfqpoint{1.278623in}{4.106979in}}%
\pgfpathlineto{\pgfqpoint{1.466502in}{4.101297in}}%
\pgfpathlineto{\pgfqpoint{1.654381in}{4.097537in}}%
\pgfpathlineto{\pgfqpoint{1.842260in}{3.885486in}}%
\pgfpathlineto{\pgfqpoint{2.030139in}{3.465627in}}%
\pgfpathlineto{\pgfqpoint{2.218017in}{3.126829in}}%
\pgfpathlineto{\pgfqpoint{2.405896in}{2.357682in}}%
\pgfpathlineto{\pgfqpoint{2.593775in}{1.942322in}}%
\pgfpathlineto{\pgfqpoint{2.781654in}{1.554688in}}%
\pgfpathlineto{\pgfqpoint{2.969533in}{1.143772in}}%
\pgfpathlineto{\pgfqpoint{3.157411in}{1.018573in}}%
\pgfpathlineto{\pgfqpoint{3.345290in}{1.016375in}}%
\pgfpathlineto{\pgfqpoint{3.533169in}{1.016163in}}%
\pgfpathlineto{\pgfqpoint{3.721048in}{1.016142in}}%
\pgfpathlineto{\pgfqpoint{3.908927in}{1.016100in}}%
\pgfpathlineto{\pgfqpoint{4.096805in}{1.016100in}}%
\pgfpathlineto{\pgfqpoint{4.284684in}{1.016100in}}%
\pgfpathlineto{\pgfqpoint{4.472563in}{1.016100in}}%
\pgfpathlineto{\pgfqpoint{4.660442in}{1.016100in}}%
\pgfpathlineto{\pgfqpoint{4.848320in}{1.016100in}}%
\pgfpathlineto{\pgfqpoint{5.036199in}{1.016100in}}%
\pgfpathlineto{\pgfqpoint{5.224078in}{1.016100in}}%
\pgfpathlineto{\pgfqpoint{5.411957in}{1.016142in}}%
\pgfpathlineto{\pgfqpoint{5.599836in}{1.016142in}}%
\pgfusepath{stroke}%
\end{pgfscope}%
\begin{pgfscope}%
\pgfpathrectangle{\pgfqpoint{0.865290in}{0.582899in}}{\pgfqpoint{4.960000in}{3.696000in}}%
\pgfusepath{clip}%
\pgfsetrectcap%
\pgfsetroundjoin%
\pgfsetlinewidth{1.505625pt}%
\definecolor{currentstroke}{rgb}{1.000000,0.498039,0.054902}%
\pgfsetstrokecolor{currentstroke}%
\pgfsetdash{}{0pt}%
\pgfpathmoveto{\pgfqpoint{1.090745in}{4.110899in}}%
\pgfpathlineto{\pgfqpoint{1.278623in}{4.055037in}}%
\pgfpathlineto{\pgfqpoint{1.466502in}{3.487104in}}%
\pgfpathlineto{\pgfqpoint{1.654381in}{2.715377in}}%
\pgfpathlineto{\pgfqpoint{1.842260in}{2.555120in}}%
\pgfpathlineto{\pgfqpoint{2.030139in}{1.827902in}}%
\pgfpathlineto{\pgfqpoint{2.218017in}{1.613889in}}%
\pgfpathlineto{\pgfqpoint{2.405896in}{1.157509in}}%
\pgfpathlineto{\pgfqpoint{2.593775in}{0.903386in}}%
\pgfpathlineto{\pgfqpoint{2.781654in}{0.751594in}}%
\pgfpathlineto{\pgfqpoint{2.969533in}{0.750899in}}%
\pgfpathlineto{\pgfqpoint{3.157411in}{0.752744in}}%
\pgfpathlineto{\pgfqpoint{3.345290in}{0.750899in}}%
\pgfpathlineto{\pgfqpoint{3.533169in}{0.752744in}}%
\pgfpathlineto{\pgfqpoint{3.721048in}{0.750899in}}%
\pgfpathlineto{\pgfqpoint{3.908927in}{0.752744in}}%
\pgfpathlineto{\pgfqpoint{4.096805in}{0.750899in}}%
\pgfpathlineto{\pgfqpoint{4.284684in}{0.752744in}}%
\pgfpathlineto{\pgfqpoint{4.472563in}{0.750899in}}%
\pgfpathlineto{\pgfqpoint{4.660442in}{0.752744in}}%
\pgfpathlineto{\pgfqpoint{4.848320in}{0.750899in}}%
\pgfpathlineto{\pgfqpoint{5.036199in}{0.752744in}}%
\pgfpathlineto{\pgfqpoint{5.224078in}{0.750899in}}%
\pgfpathlineto{\pgfqpoint{5.411957in}{0.752744in}}%
\pgfpathlineto{\pgfqpoint{5.599836in}{0.750899in}}%
\pgfusepath{stroke}%
\end{pgfscope}%
\begin{pgfscope}%
\pgfpathrectangle{\pgfqpoint{0.865290in}{0.582899in}}{\pgfqpoint{4.960000in}{3.696000in}}%
\pgfusepath{clip}%
\pgfsetrectcap%
\pgfsetroundjoin%
\pgfsetlinewidth{1.505625pt}%
\definecolor{currentstroke}{rgb}{0.172549,0.627451,0.172549}%
\pgfsetstrokecolor{currentstroke}%
\pgfsetdash{}{0pt}%
\pgfpathmoveto{\pgfqpoint{1.090745in}{4.106574in}}%
\pgfpathlineto{\pgfqpoint{1.278623in}{3.927241in}}%
\pgfpathlineto{\pgfqpoint{1.466502in}{3.205415in}}%
\pgfpathlineto{\pgfqpoint{1.654381in}{2.972756in}}%
\pgfpathlineto{\pgfqpoint{1.842260in}{2.554637in}}%
\pgfpathlineto{\pgfqpoint{2.030139in}{2.046522in}}%
\pgfpathlineto{\pgfqpoint{2.218017in}{1.576792in}}%
\pgfpathlineto{\pgfqpoint{2.405896in}{1.041283in}}%
\pgfpathlineto{\pgfqpoint{2.593775in}{0.794070in}}%
\pgfpathlineto{\pgfqpoint{2.781654in}{0.805860in}}%
\pgfpathlineto{\pgfqpoint{2.969533in}{0.805860in}}%
\pgfpathlineto{\pgfqpoint{3.157411in}{0.805860in}}%
\pgfpathlineto{\pgfqpoint{3.345290in}{0.805860in}}%
\pgfpathlineto{\pgfqpoint{3.533169in}{0.805860in}}%
\pgfpathlineto{\pgfqpoint{3.721048in}{0.805435in}}%
\pgfpathlineto{\pgfqpoint{3.908927in}{0.805860in}}%
\pgfpathlineto{\pgfqpoint{4.096805in}{0.805860in}}%
\pgfpathlineto{\pgfqpoint{4.284684in}{0.805860in}}%
\pgfpathlineto{\pgfqpoint{4.472563in}{0.805435in}}%
\pgfpathlineto{\pgfqpoint{4.660442in}{0.805860in}}%
\pgfpathlineto{\pgfqpoint{4.848320in}{0.805860in}}%
\pgfpathlineto{\pgfqpoint{5.036199in}{0.805860in}}%
\pgfpathlineto{\pgfqpoint{5.224078in}{0.805435in}}%
\pgfpathlineto{\pgfqpoint{5.411957in}{0.805860in}}%
\pgfpathlineto{\pgfqpoint{5.599836in}{0.805860in}}%
\pgfusepath{stroke}%
\end{pgfscope}%
\begin{pgfscope}%
\pgfpathrectangle{\pgfqpoint{0.865290in}{0.582899in}}{\pgfqpoint{4.960000in}{3.696000in}}%
\pgfusepath{clip}%
\pgfsetrectcap%
\pgfsetroundjoin%
\pgfsetlinewidth{1.505625pt}%
\definecolor{currentstroke}{rgb}{0.839216,0.152941,0.156863}%
\pgfsetstrokecolor{currentstroke}%
\pgfsetdash{}{0pt}%
\pgfpathmoveto{\pgfqpoint{1.090745in}{4.110675in}}%
\pgfpathlineto{\pgfqpoint{1.278623in}{4.029128in}}%
\pgfpathlineto{\pgfqpoint{1.466502in}{3.378283in}}%
\pgfpathlineto{\pgfqpoint{1.654381in}{2.947099in}}%
\pgfpathlineto{\pgfqpoint{1.842260in}{2.621535in}}%
\pgfpathlineto{\pgfqpoint{2.030139in}{2.204535in}}%
\pgfpathlineto{\pgfqpoint{2.218017in}{1.744117in}}%
\pgfpathlineto{\pgfqpoint{2.405896in}{1.011777in}}%
\pgfpathlineto{\pgfqpoint{2.593775in}{1.040077in}}%
\pgfpathlineto{\pgfqpoint{2.781654in}{1.017804in}}%
\pgfpathlineto{\pgfqpoint{2.969533in}{1.017824in}}%
\pgfpathlineto{\pgfqpoint{3.157411in}{1.017824in}}%
\pgfpathlineto{\pgfqpoint{3.345290in}{1.017824in}}%
\pgfpathlineto{\pgfqpoint{3.533169in}{1.017824in}}%
\pgfpathlineto{\pgfqpoint{3.721048in}{1.017824in}}%
\pgfpathlineto{\pgfqpoint{3.908927in}{1.017824in}}%
\pgfpathlineto{\pgfqpoint{4.096805in}{1.017824in}}%
\pgfpathlineto{\pgfqpoint{4.284684in}{1.017824in}}%
\pgfpathlineto{\pgfqpoint{4.472563in}{1.017824in}}%
\pgfpathlineto{\pgfqpoint{4.660442in}{1.017824in}}%
\pgfpathlineto{\pgfqpoint{4.848320in}{1.017824in}}%
\pgfpathlineto{\pgfqpoint{5.036199in}{1.017824in}}%
\pgfpathlineto{\pgfqpoint{5.224078in}{1.017824in}}%
\pgfpathlineto{\pgfqpoint{5.411957in}{1.017824in}}%
\pgfpathlineto{\pgfqpoint{5.599836in}{1.017824in}}%
\pgfusepath{stroke}%
\end{pgfscope}%
\begin{pgfscope}%
\pgfsetrectcap%
\pgfsetmiterjoin%
\pgfsetlinewidth{0.803000pt}%
\definecolor{currentstroke}{rgb}{0.000000,0.000000,0.000000}%
\pgfsetstrokecolor{currentstroke}%
\pgfsetdash{}{0pt}%
\pgfpathmoveto{\pgfqpoint{0.865290in}{0.582899in}}%
\pgfpathlineto{\pgfqpoint{0.865290in}{4.278899in}}%
\pgfusepath{stroke}%
\end{pgfscope}%
\begin{pgfscope}%
\pgfsetrectcap%
\pgfsetmiterjoin%
\pgfsetlinewidth{0.803000pt}%
\definecolor{currentstroke}{rgb}{0.000000,0.000000,0.000000}%
\pgfsetstrokecolor{currentstroke}%
\pgfsetdash{}{0pt}%
\pgfpathmoveto{\pgfqpoint{5.825290in}{0.582899in}}%
\pgfpathlineto{\pgfqpoint{5.825290in}{4.278899in}}%
\pgfusepath{stroke}%
\end{pgfscope}%
\begin{pgfscope}%
\pgfsetrectcap%
\pgfsetmiterjoin%
\pgfsetlinewidth{0.803000pt}%
\definecolor{currentstroke}{rgb}{0.000000,0.000000,0.000000}%
\pgfsetstrokecolor{currentstroke}%
\pgfsetdash{}{0pt}%
\pgfpathmoveto{\pgfqpoint{0.865290in}{0.582899in}}%
\pgfpathlineto{\pgfqpoint{5.825290in}{0.582899in}}%
\pgfusepath{stroke}%
\end{pgfscope}%
\begin{pgfscope}%
\pgfsetrectcap%
\pgfsetmiterjoin%
\pgfsetlinewidth{0.803000pt}%
\definecolor{currentstroke}{rgb}{0.000000,0.000000,0.000000}%
\pgfsetstrokecolor{currentstroke}%
\pgfsetdash{}{0pt}%
\pgfpathmoveto{\pgfqpoint{0.865290in}{4.278899in}}%
\pgfpathlineto{\pgfqpoint{5.825290in}{4.278899in}}%
\pgfusepath{stroke}%
\end{pgfscope}%
\begin{pgfscope}%
\pgfsetbuttcap%
\pgfsetmiterjoin%
\definecolor{currentfill}{rgb}{1.000000,1.000000,1.000000}%
\pgfsetfillcolor{currentfill}%
\pgfsetfillopacity{0.800000}%
\pgfsetlinewidth{1.003750pt}%
\definecolor{currentstroke}{rgb}{0.800000,0.800000,0.800000}%
\pgfsetstrokecolor{currentstroke}%
\pgfsetstrokeopacity{0.800000}%
\pgfsetdash{}{0pt}%
\pgfpathmoveto{\pgfqpoint{2.158059in}{2.796434in}}%
\pgfpathlineto{\pgfqpoint{5.669735in}{2.796434in}}%
\pgfpathquadraticcurveto{\pgfqpoint{5.714179in}{2.796434in}}{\pgfqpoint{5.714179in}{2.840879in}}%
\pgfpathlineto{\pgfqpoint{5.714179in}{4.123343in}}%
\pgfpathquadraticcurveto{\pgfqpoint{5.714179in}{4.167788in}}{\pgfqpoint{5.669735in}{4.167788in}}%
\pgfpathlineto{\pgfqpoint{2.158059in}{4.167788in}}%
\pgfpathquadraticcurveto{\pgfqpoint{2.113615in}{4.167788in}}{\pgfqpoint{2.113615in}{4.123343in}}%
\pgfpathlineto{\pgfqpoint{2.113615in}{2.840879in}}%
\pgfpathquadraticcurveto{\pgfqpoint{2.113615in}{2.796434in}}{\pgfqpoint{2.158059in}{2.796434in}}%
\pgfpathlineto{\pgfqpoint{2.158059in}{2.796434in}}%
\pgfpathclose%
\pgfusepath{stroke,fill}%
\end{pgfscope}%
\begin{pgfscope}%
\pgfsetrectcap%
\pgfsetroundjoin%
\pgfsetlinewidth{1.505625pt}%
\definecolor{currentstroke}{rgb}{0.121569,0.466667,0.705882}%
\pgfsetstrokecolor{currentstroke}%
\pgfsetdash{}{0pt}%
\pgfpathmoveto{\pgfqpoint{2.202504in}{3.987840in}}%
\pgfpathlineto{\pgfqpoint{2.424726in}{3.987840in}}%
\pgfpathlineto{\pgfqpoint{2.646948in}{3.987840in}}%
\pgfusepath{stroke}%
\end{pgfscope}%
\begin{pgfscope}%
\definecolor{textcolor}{rgb}{0.000000,0.000000,0.000000}%
\pgfsetstrokecolor{textcolor}%
\pgfsetfillcolor{textcolor}%
\pgftext[x=2.824726in,y=3.910062in,left,base]{\color{textcolor}{\sffamily\fontsize{16.000000}{19.200000}\selectfont\catcode`\^=\active\def^{\ifmmode\sp\else\^{}\fi}\catcode`\%=\active\def
\end{pgfscope}%
\begin{pgfscope}%
\pgfsetrectcap%
\pgfsetroundjoin%
\pgfsetlinewidth{1.505625pt}%
\definecolor{currentstroke}{rgb}{1.000000,0.498039,0.054902}%
\pgfsetstrokecolor{currentstroke}%
\pgfsetdash{}{0pt}%
\pgfpathmoveto{\pgfqpoint{2.202504in}{3.661668in}}%
\pgfpathlineto{\pgfqpoint{2.424726in}{3.661668in}}%
\pgfpathlineto{\pgfqpoint{2.646948in}{3.661668in}}%
\pgfusepath{stroke}%
\end{pgfscope}%
\begin{pgfscope}%
\definecolor{textcolor}{rgb}{0.000000,0.000000,0.000000}%
\pgfsetstrokecolor{textcolor}%
\pgfsetfillcolor{textcolor}%
\pgftext[x=2.824726in,y=3.583890in,left,base]{\color{textcolor}{\sffamily\fontsize{16.000000}{19.200000}\selectfont\catcode`\^=\active\def^{\ifmmode\sp\else\^{}\fi}\catcode`\%=\active\def
\end{pgfscope}%
\begin{pgfscope}%
\pgfsetrectcap%
\pgfsetroundjoin%
\pgfsetlinewidth{1.505625pt}%
\definecolor{currentstroke}{rgb}{0.172549,0.627451,0.172549}%
\pgfsetstrokecolor{currentstroke}%
\pgfsetdash{}{0pt}%
\pgfpathmoveto{\pgfqpoint{2.202504in}{3.335497in}}%
\pgfpathlineto{\pgfqpoint{2.424726in}{3.335497in}}%
\pgfpathlineto{\pgfqpoint{2.646948in}{3.335497in}}%
\pgfusepath{stroke}%
\end{pgfscope}%
\begin{pgfscope}%
\definecolor{textcolor}{rgb}{0.000000,0.000000,0.000000}%
\pgfsetstrokecolor{textcolor}%
\pgfsetfillcolor{textcolor}%
\pgftext[x=2.824726in,y=3.257719in,left,base]{\color{textcolor}{\sffamily\fontsize{16.000000}{19.200000}\selectfont\catcode`\^=\active\def^{\ifmmode\sp\else\^{}\fi}\catcode`\%=\active\def
\end{pgfscope}%
\begin{pgfscope}%
\pgfsetrectcap%
\pgfsetroundjoin%
\pgfsetlinewidth{1.505625pt}%
\definecolor{currentstroke}{rgb}{0.839216,0.152941,0.156863}%
\pgfsetstrokecolor{currentstroke}%
\pgfsetdash{}{0pt}%
\pgfpathmoveto{\pgfqpoint{2.202504in}{3.009325in}}%
\pgfpathlineto{\pgfqpoint{2.424726in}{3.009325in}}%
\pgfpathlineto{\pgfqpoint{2.646948in}{3.009325in}}%
\pgfusepath{stroke}%
\end{pgfscope}%
\begin{pgfscope}%
\definecolor{textcolor}{rgb}{0.000000,0.000000,0.000000}%
\pgfsetstrokecolor{textcolor}%
\pgfsetfillcolor{textcolor}%
\pgftext[x=2.824726in,y=2.931547in,left,base]{\color{textcolor}{\sffamily\fontsize{16.000000}{19.200000}\selectfont\catcode`\^=\active\def^{\ifmmode\sp\else\^{}\fi}\catcode`\%=\active\def
\end{pgfscope}%
\end{pgfpicture}%
\makeatother%
\endgroup%

%% file: figs/network_conv_Autobahn.pgf
\begingroup%
\makeatletter%
\begin{pgfpicture}%
\pgfpathrectangle{\pgfpointorigin}{\pgfqpoint{5.825290in}{4.278899in}}%
\pgfusepath{use as bounding box, clip}%
\begin{pgfscope}%
\pgfsetbuttcap%
\pgfsetmiterjoin%
\definecolor{currentfill}{rgb}{1.000000,1.000000,1.000000}%
\pgfsetfillcolor{currentfill}%
\pgfsetlinewidth{0.000000pt}%
\definecolor{currentstroke}{rgb}{1.000000,1.000000,1.000000}%
\pgfsetstrokecolor{currentstroke}%
\pgfsetdash{}{0pt}%
\pgfpathmoveto{\pgfqpoint{0.000000in}{0.000000in}}%
\pgfpathlineto{\pgfqpoint{5.825290in}{0.000000in}}%
\pgfpathlineto{\pgfqpoint{5.825290in}{4.278899in}}%
\pgfpathlineto{\pgfqpoint{0.000000in}{4.278899in}}%
\pgfpathlineto{\pgfqpoint{0.000000in}{0.000000in}}%
\pgfpathclose%
\pgfusepath{fill}%
\end{pgfscope}%
\begin{pgfscope}%
\pgfsetbuttcap%
\pgfsetmiterjoin%
\definecolor{currentfill}{rgb}{1.000000,1.000000,1.000000}%
\pgfsetfillcolor{currentfill}%
\pgfsetlinewidth{0.000000pt}%
\definecolor{currentstroke}{rgb}{0.000000,0.000000,0.000000}%
\pgfsetstrokecolor{currentstroke}%
\pgfsetstrokeopacity{0.000000}%
\pgfsetdash{}{0pt}%
\pgfpathmoveto{\pgfqpoint{0.865290in}{0.582899in}}%
\pgfpathlineto{\pgfqpoint{5.825290in}{0.582899in}}%
\pgfpathlineto{\pgfqpoint{5.825290in}{4.278899in}}%
\pgfpathlineto{\pgfqpoint{0.865290in}{4.278899in}}%
\pgfpathlineto{\pgfqpoint{0.865290in}{0.582899in}}%
\pgfpathclose%
\pgfusepath{fill}%
\end{pgfscope}%
\begin{pgfscope}%
\pgfsetbuttcap%
\pgfsetroundjoin%
\definecolor{currentfill}{rgb}{0.000000,0.000000,0.000000}%
\pgfsetfillcolor{currentfill}%
\pgfsetlinewidth{0.803000pt}%
\definecolor{currentstroke}{rgb}{0.000000,0.000000,0.000000}%
\pgfsetstrokecolor{currentstroke}%
\pgfsetdash{}{0pt}%
\pgfsys@defobject{currentmarker}{\pgfqpoint{0.000000in}{-0.048611in}}{\pgfqpoint{0.000000in}{0.000000in}}{%
\pgfpathmoveto{\pgfqpoint{0.000000in}{0.000000in}}%
\pgfpathlineto{\pgfqpoint{0.000000in}{-0.048611in}}%
\pgfusepath{stroke,fill}%
}%
\begin{pgfscope}%
\pgfsys@transformshift{1.372563in}{0.582899in}%
\pgfsys@useobject{currentmarker}{}%
\end{pgfscope}%
\end{pgfscope}%
\begin{pgfscope}%
\definecolor{textcolor}{rgb}{0.000000,0.000000,0.000000}%
\pgfsetstrokecolor{textcolor}%
\pgfsetfillcolor{textcolor}%
\pgftext[x=1.372563in,y=0.485677in,,top]{\color{textcolor}{\sffamily\fontsize{16.000000}{19.200000}\selectfont\catcode`\^=\active\def^{\ifmmode\sp\else\^{}\fi}\catcode`\%=\active\def
\end{pgfscope}%
\begin{pgfscope}%
\pgfsetbuttcap%
\pgfsetroundjoin%
\definecolor{currentfill}{rgb}{0.000000,0.000000,0.000000}%
\pgfsetfillcolor{currentfill}%
\pgfsetlinewidth{0.803000pt}%
\definecolor{currentstroke}{rgb}{0.000000,0.000000,0.000000}%
\pgfsetstrokecolor{currentstroke}%
\pgfsetdash{}{0pt}%
\pgfsys@defobject{currentmarker}{\pgfqpoint{0.000000in}{-0.048611in}}{\pgfqpoint{0.000000in}{0.000000in}}{%
\pgfpathmoveto{\pgfqpoint{0.000000in}{0.000000in}}%
\pgfpathlineto{\pgfqpoint{0.000000in}{-0.048611in}}%
\pgfusepath{stroke,fill}%
}%
\begin{pgfscope}%
\pgfsys@transformshift{1.842260in}{0.582899in}%
\pgfsys@useobject{currentmarker}{}%
\end{pgfscope}%
\end{pgfscope}%
\begin{pgfscope}%
\definecolor{textcolor}{rgb}{0.000000,0.000000,0.000000}%
\pgfsetstrokecolor{textcolor}%
\pgfsetfillcolor{textcolor}%
\pgftext[x=1.842260in,y=0.485677in,,top]{\color{textcolor}{\sffamily\fontsize{16.000000}{19.200000}\selectfont\catcode`\^=\active\def^{\ifmmode\sp\else\^{}\fi}\catcode`\%=\active\def
\end{pgfscope}%
\begin{pgfscope}%
\pgfsetbuttcap%
\pgfsetroundjoin%
\definecolor{currentfill}{rgb}{0.000000,0.000000,0.000000}%
\pgfsetfillcolor{currentfill}%
\pgfsetlinewidth{0.803000pt}%
\definecolor{currentstroke}{rgb}{0.000000,0.000000,0.000000}%
\pgfsetstrokecolor{currentstroke}%
\pgfsetdash{}{0pt}%
\pgfsys@defobject{currentmarker}{\pgfqpoint{0.000000in}{-0.048611in}}{\pgfqpoint{0.000000in}{0.000000in}}{%
\pgfpathmoveto{\pgfqpoint{0.000000in}{0.000000in}}%
\pgfpathlineto{\pgfqpoint{0.000000in}{-0.048611in}}%
\pgfusepath{stroke,fill}%
}%
\begin{pgfscope}%
\pgfsys@transformshift{2.311957in}{0.582899in}%
\pgfsys@useobject{currentmarker}{}%
\end{pgfscope}%
\end{pgfscope}%
\begin{pgfscope}%
\definecolor{textcolor}{rgb}{0.000000,0.000000,0.000000}%
\pgfsetstrokecolor{textcolor}%
\pgfsetfillcolor{textcolor}%
\pgftext[x=2.311957in,y=0.485677in,,top]{\color{textcolor}{\sffamily\fontsize{16.000000}{19.200000}\selectfont\catcode`\^=\active\def^{\ifmmode\sp\else\^{}\fi}\catcode`\%=\active\def
\end{pgfscope}%
\begin{pgfscope}%
\pgfsetbuttcap%
\pgfsetroundjoin%
\definecolor{currentfill}{rgb}{0.000000,0.000000,0.000000}%
\pgfsetfillcolor{currentfill}%
\pgfsetlinewidth{0.803000pt}%
\definecolor{currentstroke}{rgb}{0.000000,0.000000,0.000000}%
\pgfsetstrokecolor{currentstroke}%
\pgfsetdash{}{0pt}%
\pgfsys@defobject{currentmarker}{\pgfqpoint{0.000000in}{-0.048611in}}{\pgfqpoint{0.000000in}{0.000000in}}{%
\pgfpathmoveto{\pgfqpoint{0.000000in}{0.000000in}}%
\pgfpathlineto{\pgfqpoint{0.000000in}{-0.048611in}}%
\pgfusepath{stroke,fill}%
}%
\begin{pgfscope}%
\pgfsys@transformshift{2.781654in}{0.582899in}%
\pgfsys@useobject{currentmarker}{}%
\end{pgfscope}%
\end{pgfscope}%
\begin{pgfscope}%
\definecolor{textcolor}{rgb}{0.000000,0.000000,0.000000}%
\pgfsetstrokecolor{textcolor}%
\pgfsetfillcolor{textcolor}%
\pgftext[x=2.781654in,y=0.485677in,,top]{\color{textcolor}{\sffamily\fontsize{16.000000}{19.200000}\selectfont\catcode`\^=\active\def^{\ifmmode\sp\else\^{}\fi}\catcode`\%=\active\def
\end{pgfscope}%
\begin{pgfscope}%
\pgfsetbuttcap%
\pgfsetroundjoin%
\definecolor{currentfill}{rgb}{0.000000,0.000000,0.000000}%
\pgfsetfillcolor{currentfill}%
\pgfsetlinewidth{0.803000pt}%
\definecolor{currentstroke}{rgb}{0.000000,0.000000,0.000000}%
\pgfsetstrokecolor{currentstroke}%
\pgfsetdash{}{0pt}%
\pgfsys@defobject{currentmarker}{\pgfqpoint{0.000000in}{-0.048611in}}{\pgfqpoint{0.000000in}{0.000000in}}{%
\pgfpathmoveto{\pgfqpoint{0.000000in}{0.000000in}}%
\pgfpathlineto{\pgfqpoint{0.000000in}{-0.048611in}}%
\pgfusepath{stroke,fill}%
}%
\begin{pgfscope}%
\pgfsys@transformshift{3.251351in}{0.582899in}%
\pgfsys@useobject{currentmarker}{}%
\end{pgfscope}%
\end{pgfscope}%
\begin{pgfscope}%
\definecolor{textcolor}{rgb}{0.000000,0.000000,0.000000}%
\pgfsetstrokecolor{textcolor}%
\pgfsetfillcolor{textcolor}%
\pgftext[x=3.251351in,y=0.485677in,,top]{\color{textcolor}{\sffamily\fontsize{16.000000}{19.200000}\selectfont\catcode`\^=\active\def^{\ifmmode\sp\else\^{}\fi}\catcode`\%=\active\def
\end{pgfscope}%
\begin{pgfscope}%
\pgfsetbuttcap%
\pgfsetroundjoin%
\definecolor{currentfill}{rgb}{0.000000,0.000000,0.000000}%
\pgfsetfillcolor{currentfill}%
\pgfsetlinewidth{0.803000pt}%
\definecolor{currentstroke}{rgb}{0.000000,0.000000,0.000000}%
\pgfsetstrokecolor{currentstroke}%
\pgfsetdash{}{0pt}%
\pgfsys@defobject{currentmarker}{\pgfqpoint{0.000000in}{-0.048611in}}{\pgfqpoint{0.000000in}{0.000000in}}{%
\pgfpathmoveto{\pgfqpoint{0.000000in}{0.000000in}}%
\pgfpathlineto{\pgfqpoint{0.000000in}{-0.048611in}}%
\pgfusepath{stroke,fill}%
}%
\begin{pgfscope}%
\pgfsys@transformshift{3.721048in}{0.582899in}%
\pgfsys@useobject{currentmarker}{}%
\end{pgfscope}%
\end{pgfscope}%
\begin{pgfscope}%
\definecolor{textcolor}{rgb}{0.000000,0.000000,0.000000}%
\pgfsetstrokecolor{textcolor}%
\pgfsetfillcolor{textcolor}%
\pgftext[x=3.721048in,y=0.485677in,,top]{\color{textcolor}{\sffamily\fontsize{16.000000}{19.200000}\selectfont\catcode`\^=\active\def^{\ifmmode\sp\else\^{}\fi}\catcode`\%=\active\def
\end{pgfscope}%
\begin{pgfscope}%
\pgfsetbuttcap%
\pgfsetroundjoin%
\definecolor{currentfill}{rgb}{0.000000,0.000000,0.000000}%
\pgfsetfillcolor{currentfill}%
\pgfsetlinewidth{0.803000pt}%
\definecolor{currentstroke}{rgb}{0.000000,0.000000,0.000000}%
\pgfsetstrokecolor{currentstroke}%
\pgfsetdash{}{0pt}%
\pgfsys@defobject{currentmarker}{\pgfqpoint{0.000000in}{-0.048611in}}{\pgfqpoint{0.000000in}{0.000000in}}{%
\pgfpathmoveto{\pgfqpoint{0.000000in}{0.000000in}}%
\pgfpathlineto{\pgfqpoint{0.000000in}{-0.048611in}}%
\pgfusepath{stroke,fill}%
}%
\begin{pgfscope}%
\pgfsys@transformshift{4.190745in}{0.582899in}%
\pgfsys@useobject{currentmarker}{}%
\end{pgfscope}%
\end{pgfscope}%
\begin{pgfscope}%
\definecolor{textcolor}{rgb}{0.000000,0.000000,0.000000}%
\pgfsetstrokecolor{textcolor}%
\pgfsetfillcolor{textcolor}%
\pgftext[x=4.190745in,y=0.485677in,,top]{\color{textcolor}{\sffamily\fontsize{16.000000}{19.200000}\selectfont\catcode`\^=\active\def^{\ifmmode\sp\else\^{}\fi}\catcode`\%=\active\def
\end{pgfscope}%
\begin{pgfscope}%
\pgfsetbuttcap%
\pgfsetroundjoin%
\definecolor{currentfill}{rgb}{0.000000,0.000000,0.000000}%
\pgfsetfillcolor{currentfill}%
\pgfsetlinewidth{0.803000pt}%
\definecolor{currentstroke}{rgb}{0.000000,0.000000,0.000000}%
\pgfsetstrokecolor{currentstroke}%
\pgfsetdash{}{0pt}%
\pgfsys@defobject{currentmarker}{\pgfqpoint{0.000000in}{-0.048611in}}{\pgfqpoint{0.000000in}{0.000000in}}{%
\pgfpathmoveto{\pgfqpoint{0.000000in}{0.000000in}}%
\pgfpathlineto{\pgfqpoint{0.000000in}{-0.048611in}}%
\pgfusepath{stroke,fill}%
}%
\begin{pgfscope}%
\pgfsys@transformshift{4.660442in}{0.582899in}%
\pgfsys@useobject{currentmarker}{}%
\end{pgfscope}%
\end{pgfscope}%
\begin{pgfscope}%
\definecolor{textcolor}{rgb}{0.000000,0.000000,0.000000}%
\pgfsetstrokecolor{textcolor}%
\pgfsetfillcolor{textcolor}%
\pgftext[x=4.660442in,y=0.485677in,,top]{\color{textcolor}{\sffamily\fontsize{16.000000}{19.200000}\selectfont\catcode`\^=\active\def^{\ifmmode\sp\else\^{}\fi}\catcode`\%=\active\def
\end{pgfscope}%
\begin{pgfscope}%
\pgfsetbuttcap%
\pgfsetroundjoin%
\definecolor{currentfill}{rgb}{0.000000,0.000000,0.000000}%
\pgfsetfillcolor{currentfill}%
\pgfsetlinewidth{0.803000pt}%
\definecolor{currentstroke}{rgb}{0.000000,0.000000,0.000000}%
\pgfsetstrokecolor{currentstroke}%
\pgfsetdash{}{0pt}%
\pgfsys@defobject{currentmarker}{\pgfqpoint{0.000000in}{-0.048611in}}{\pgfqpoint{0.000000in}{0.000000in}}{%
\pgfpathmoveto{\pgfqpoint{0.000000in}{0.000000in}}%
\pgfpathlineto{\pgfqpoint{0.000000in}{-0.048611in}}%
\pgfusepath{stroke,fill}%
}%
\begin{pgfscope}%
\pgfsys@transformshift{5.130139in}{0.582899in}%
\pgfsys@useobject{currentmarker}{}%
\end{pgfscope}%
\end{pgfscope}%
\begin{pgfscope}%
\definecolor{textcolor}{rgb}{0.000000,0.000000,0.000000}%
\pgfsetstrokecolor{textcolor}%
\pgfsetfillcolor{textcolor}%
\pgftext[x=5.130139in,y=0.485677in,,top]{\color{textcolor}{\sffamily\fontsize{16.000000}{19.200000}\selectfont\catcode`\^=\active\def^{\ifmmode\sp\else\^{}\fi}\catcode`\%=\active\def
\end{pgfscope}%
\begin{pgfscope}%
\pgfsetbuttcap%
\pgfsetroundjoin%
\definecolor{currentfill}{rgb}{0.000000,0.000000,0.000000}%
\pgfsetfillcolor{currentfill}%
\pgfsetlinewidth{0.803000pt}%
\definecolor{currentstroke}{rgb}{0.000000,0.000000,0.000000}%
\pgfsetstrokecolor{currentstroke}%
\pgfsetdash{}{0pt}%
\pgfsys@defobject{currentmarker}{\pgfqpoint{0.000000in}{-0.048611in}}{\pgfqpoint{0.000000in}{0.000000in}}{%
\pgfpathmoveto{\pgfqpoint{0.000000in}{0.000000in}}%
\pgfpathlineto{\pgfqpoint{0.000000in}{-0.048611in}}%
\pgfusepath{stroke,fill}%
}%
\begin{pgfscope}%
\pgfsys@transformshift{5.599836in}{0.582899in}%
\pgfsys@useobject{currentmarker}{}%
\end{pgfscope}%
\end{pgfscope}%
\begin{pgfscope}%
\definecolor{textcolor}{rgb}{0.000000,0.000000,0.000000}%
\pgfsetstrokecolor{textcolor}%
\pgfsetfillcolor{textcolor}%
\pgftext[x=5.599836in,y=0.485677in,,top]{\color{textcolor}{\sffamily\fontsize{16.000000}{19.200000}\selectfont\catcode`\^=\active\def^{\ifmmode\sp\else\^{}\fi}\catcode`\%=\active\def
\end{pgfscope}%
\begin{pgfscope}%
\definecolor{textcolor}{rgb}{0.000000,0.000000,0.000000}%
\pgfsetstrokecolor{textcolor}%
\pgfsetfillcolor{textcolor}%
\pgftext[x=3.345290in,y=0.215061in,,top]{\color{textcolor}{\sffamily\fontsize{16.000000}{19.200000}\selectfont\catcode`\^=\active\def^{\ifmmode\sp\else\^{}\fi}\catcode`\%=\active\def
\end{pgfscope}%
\begin{pgfscope}%
\pgfsetbuttcap%
\pgfsetroundjoin%
\definecolor{currentfill}{rgb}{0.000000,0.000000,0.000000}%
\pgfsetfillcolor{currentfill}%
\pgfsetlinewidth{0.803000pt}%
\definecolor{currentstroke}{rgb}{0.000000,0.000000,0.000000}%
\pgfsetstrokecolor{currentstroke}%
\pgfsetdash{}{0pt}%
\pgfsys@defobject{currentmarker}{\pgfqpoint{-0.048611in}{0.000000in}}{\pgfqpoint{-0.000000in}{0.000000in}}{%
\pgfpathmoveto{\pgfqpoint{-0.000000in}{0.000000in}}%
\pgfpathlineto{\pgfqpoint{-0.048611in}{0.000000in}}%
\pgfusepath{stroke,fill}%
}%
\begin{pgfscope}%
\pgfsys@transformshift{0.865290in}{0.980811in}%
\pgfsys@useobject{currentmarker}{}%
\end{pgfscope}%
\end{pgfscope}%
\begin{pgfscope}%
\definecolor{textcolor}{rgb}{0.000000,0.000000,0.000000}%
\pgfsetstrokecolor{textcolor}%
\pgfsetfillcolor{textcolor}%
\pgftext[x=0.270616in, y=0.896393in, left, base]{\color{textcolor}{\sffamily\fontsize{16.000000}{19.200000}\selectfont\catcode`\^=\active\def^{\ifmmode\sp\else\^{}\fi}\catcode`\%=\active\def
\end{pgfscope}%
\begin{pgfscope}%
\pgfsetbuttcap%
\pgfsetroundjoin%
\definecolor{currentfill}{rgb}{0.000000,0.000000,0.000000}%
\pgfsetfillcolor{currentfill}%
\pgfsetlinewidth{0.803000pt}%
\definecolor{currentstroke}{rgb}{0.000000,0.000000,0.000000}%
\pgfsetstrokecolor{currentstroke}%
\pgfsetdash{}{0pt}%
\pgfsys@defobject{currentmarker}{\pgfqpoint{-0.048611in}{0.000000in}}{\pgfqpoint{-0.000000in}{0.000000in}}{%
\pgfpathmoveto{\pgfqpoint{-0.000000in}{0.000000in}}%
\pgfpathlineto{\pgfqpoint{-0.048611in}{0.000000in}}%
\pgfusepath{stroke,fill}%
}%
\begin{pgfscope}%
\pgfsys@transformshift{0.865290in}{1.428727in}%
\pgfsys@useobject{currentmarker}{}%
\end{pgfscope}%
\end{pgfscope}%
\begin{pgfscope}%
\definecolor{textcolor}{rgb}{0.000000,0.000000,0.000000}%
\pgfsetstrokecolor{textcolor}%
\pgfsetfillcolor{textcolor}%
\pgftext[x=0.270616in, y=1.344308in, left, base]{\color{textcolor}{\sffamily\fontsize{16.000000}{19.200000}\selectfont\catcode`\^=\active\def^{\ifmmode\sp\else\^{}\fi}\catcode`\%=\active\def
\end{pgfscope}%
\begin{pgfscope}%
\pgfsetbuttcap%
\pgfsetroundjoin%
\definecolor{currentfill}{rgb}{0.000000,0.000000,0.000000}%
\pgfsetfillcolor{currentfill}%
\pgfsetlinewidth{0.803000pt}%
\definecolor{currentstroke}{rgb}{0.000000,0.000000,0.000000}%
\pgfsetstrokecolor{currentstroke}%
\pgfsetdash{}{0pt}%
\pgfsys@defobject{currentmarker}{\pgfqpoint{-0.048611in}{0.000000in}}{\pgfqpoint{-0.000000in}{0.000000in}}{%
\pgfpathmoveto{\pgfqpoint{-0.000000in}{0.000000in}}%
\pgfpathlineto{\pgfqpoint{-0.048611in}{0.000000in}}%
\pgfusepath{stroke,fill}%
}%
\begin{pgfscope}%
\pgfsys@transformshift{0.865290in}{1.876642in}%
\pgfsys@useobject{currentmarker}{}%
\end{pgfscope}%
\end{pgfscope}%
\begin{pgfscope}%
\definecolor{textcolor}{rgb}{0.000000,0.000000,0.000000}%
\pgfsetstrokecolor{textcolor}%
\pgfsetfillcolor{textcolor}%
\pgftext[x=0.270616in, y=1.792224in, left, base]{\color{textcolor}{\sffamily\fontsize{16.000000}{19.200000}\selectfont\catcode`\^=\active\def^{\ifmmode\sp\else\^{}\fi}\catcode`\%=\active\def
\end{pgfscope}%
\begin{pgfscope}%
\pgfsetbuttcap%
\pgfsetroundjoin%
\definecolor{currentfill}{rgb}{0.000000,0.000000,0.000000}%
\pgfsetfillcolor{currentfill}%
\pgfsetlinewidth{0.803000pt}%
\definecolor{currentstroke}{rgb}{0.000000,0.000000,0.000000}%
\pgfsetstrokecolor{currentstroke}%
\pgfsetdash{}{0pt}%
\pgfsys@defobject{currentmarker}{\pgfqpoint{-0.048611in}{0.000000in}}{\pgfqpoint{-0.000000in}{0.000000in}}{%
\pgfpathmoveto{\pgfqpoint{-0.000000in}{0.000000in}}%
\pgfpathlineto{\pgfqpoint{-0.048611in}{0.000000in}}%
\pgfusepath{stroke,fill}%
}%
\begin{pgfscope}%
\pgfsys@transformshift{0.865290in}{2.324558in}%
\pgfsys@useobject{currentmarker}{}%
\end{pgfscope}%
\end{pgfscope}%
\begin{pgfscope}%
\definecolor{textcolor}{rgb}{0.000000,0.000000,0.000000}%
\pgfsetstrokecolor{textcolor}%
\pgfsetfillcolor{textcolor}%
\pgftext[x=0.346658in, y=2.240140in, left, base]{\color{textcolor}{\sffamily\fontsize{16.000000}{19.200000}\selectfont\catcode`\^=\active\def^{\ifmmode\sp\else\^{}\fi}\catcode`\%=\active\def
\end{pgfscope}%
\begin{pgfscope}%
\pgfsetbuttcap%
\pgfsetroundjoin%
\definecolor{currentfill}{rgb}{0.000000,0.000000,0.000000}%
\pgfsetfillcolor{currentfill}%
\pgfsetlinewidth{0.803000pt}%
\definecolor{currentstroke}{rgb}{0.000000,0.000000,0.000000}%
\pgfsetstrokecolor{currentstroke}%
\pgfsetdash{}{0pt}%
\pgfsys@defobject{currentmarker}{\pgfqpoint{-0.048611in}{0.000000in}}{\pgfqpoint{-0.000000in}{0.000000in}}{%
\pgfpathmoveto{\pgfqpoint{-0.000000in}{0.000000in}}%
\pgfpathlineto{\pgfqpoint{-0.048611in}{0.000000in}}%
\pgfusepath{stroke,fill}%
}%
\begin{pgfscope}%
\pgfsys@transformshift{0.865290in}{2.772474in}%
\pgfsys@useobject{currentmarker}{}%
\end{pgfscope}%
\end{pgfscope}%
\begin{pgfscope}%
\definecolor{textcolor}{rgb}{0.000000,0.000000,0.000000}%
\pgfsetstrokecolor{textcolor}%
\pgfsetfillcolor{textcolor}%
\pgftext[x=0.346658in, y=2.688055in, left, base]{\color{textcolor}{\sffamily\fontsize{16.000000}{19.200000}\selectfont\catcode`\^=\active\def^{\ifmmode\sp\else\^{}\fi}\catcode`\%=\active\def
\end{pgfscope}%
\begin{pgfscope}%
\pgfsetbuttcap%
\pgfsetroundjoin%
\definecolor{currentfill}{rgb}{0.000000,0.000000,0.000000}%
\pgfsetfillcolor{currentfill}%
\pgfsetlinewidth{0.803000pt}%
\definecolor{currentstroke}{rgb}{0.000000,0.000000,0.000000}%
\pgfsetstrokecolor{currentstroke}%
\pgfsetdash{}{0pt}%
\pgfsys@defobject{currentmarker}{\pgfqpoint{-0.048611in}{0.000000in}}{\pgfqpoint{-0.000000in}{0.000000in}}{%
\pgfpathmoveto{\pgfqpoint{-0.000000in}{0.000000in}}%
\pgfpathlineto{\pgfqpoint{-0.048611in}{0.000000in}}%
\pgfusepath{stroke,fill}%
}%
\begin{pgfscope}%
\pgfsys@transformshift{0.865290in}{3.220390in}%
\pgfsys@useobject{currentmarker}{}%
\end{pgfscope}%
\end{pgfscope}%
\begin{pgfscope}%
\definecolor{textcolor}{rgb}{0.000000,0.000000,0.000000}%
\pgfsetstrokecolor{textcolor}%
\pgfsetfillcolor{textcolor}%
\pgftext[x=0.346658in, y=3.135971in, left, base]{\color{textcolor}{\sffamily\fontsize{16.000000}{19.200000}\selectfont\catcode`\^=\active\def^{\ifmmode\sp\else\^{}\fi}\catcode`\%=\active\def
\end{pgfscope}%
\begin{pgfscope}%
\pgfsetbuttcap%
\pgfsetroundjoin%
\definecolor{currentfill}{rgb}{0.000000,0.000000,0.000000}%
\pgfsetfillcolor{currentfill}%
\pgfsetlinewidth{0.803000pt}%
\definecolor{currentstroke}{rgb}{0.000000,0.000000,0.000000}%
\pgfsetstrokecolor{currentstroke}%
\pgfsetdash{}{0pt}%
\pgfsys@defobject{currentmarker}{\pgfqpoint{-0.048611in}{0.000000in}}{\pgfqpoint{-0.000000in}{0.000000in}}{%
\pgfpathmoveto{\pgfqpoint{-0.000000in}{0.000000in}}%
\pgfpathlineto{\pgfqpoint{-0.048611in}{0.000000in}}%
\pgfusepath{stroke,fill}%
}%
\begin{pgfscope}%
\pgfsys@transformshift{0.865290in}{3.668305in}%
\pgfsys@useobject{currentmarker}{}%
\end{pgfscope}%
\end{pgfscope}%
\begin{pgfscope}%
\definecolor{textcolor}{rgb}{0.000000,0.000000,0.000000}%
\pgfsetstrokecolor{textcolor}%
\pgfsetfillcolor{textcolor}%
\pgftext[x=0.346658in, y=3.583887in, left, base]{\color{textcolor}{\sffamily\fontsize{16.000000}{19.200000}\selectfont\catcode`\^=\active\def^{\ifmmode\sp\else\^{}\fi}\catcode`\%=\active\def
\end{pgfscope}%
\begin{pgfscope}%
\pgfsetbuttcap%
\pgfsetroundjoin%
\definecolor{currentfill}{rgb}{0.000000,0.000000,0.000000}%
\pgfsetfillcolor{currentfill}%
\pgfsetlinewidth{0.803000pt}%
\definecolor{currentstroke}{rgb}{0.000000,0.000000,0.000000}%
\pgfsetstrokecolor{currentstroke}%
\pgfsetdash{}{0pt}%
\pgfsys@defobject{currentmarker}{\pgfqpoint{-0.048611in}{0.000000in}}{\pgfqpoint{-0.000000in}{0.000000in}}{%
\pgfpathmoveto{\pgfqpoint{-0.000000in}{0.000000in}}%
\pgfpathlineto{\pgfqpoint{-0.048611in}{0.000000in}}%
\pgfusepath{stroke,fill}%
}%
\begin{pgfscope}%
\pgfsys@transformshift{0.865290in}{4.116221in}%
\pgfsys@useobject{currentmarker}{}%
\end{pgfscope}%
\end{pgfscope}%
\begin{pgfscope}%
\definecolor{textcolor}{rgb}{0.000000,0.000000,0.000000}%
\pgfsetstrokecolor{textcolor}%
\pgfsetfillcolor{textcolor}%
\pgftext[x=0.464945in, y=4.031803in, left, base]{\color{textcolor}{\sffamily\fontsize{16.000000}{19.200000}\selectfont\catcode`\^=\active\def^{\ifmmode\sp\else\^{}\fi}\catcode`\%=\active\def
\end{pgfscope}%
\begin{pgfscope}%
\definecolor{textcolor}{rgb}{0.000000,0.000000,0.000000}%
\pgfsetstrokecolor{textcolor}%
\pgfsetfillcolor{textcolor}%
\pgftext[x=0.215061in,y=2.430899in,,bottom,rotate=90.000000]{\color{textcolor}{\sffamily\fontsize{16.000000}{19.200000}\selectfont\catcode`\^=\active\def^{\ifmmode\sp\else\^{}\fi}\catcode`\%=\active\def
\end{pgfscope}%
\begin{pgfscope}%
\pgfpathrectangle{\pgfqpoint{0.865290in}{0.582899in}}{\pgfqpoint{4.960000in}{3.696000in}}%
\pgfusepath{clip}%
\pgfsetrectcap%
\pgfsetroundjoin%
\pgfsetlinewidth{1.505625pt}%
\definecolor{currentstroke}{rgb}{0.121569,0.466667,0.705882}%
\pgfsetstrokecolor{currentstroke}%
\pgfsetdash{}{0pt}%
\pgfpathmoveto{\pgfqpoint{1.090745in}{4.110899in}}%
\pgfpathlineto{\pgfqpoint{1.278623in}{4.005232in}}%
\pgfpathlineto{\pgfqpoint{1.466502in}{3.730154in}}%
\pgfpathlineto{\pgfqpoint{1.654381in}{3.229853in}}%
\pgfpathlineto{\pgfqpoint{1.842260in}{2.945921in}}%
\pgfpathlineto{\pgfqpoint{2.030139in}{2.746012in}}%
\pgfpathlineto{\pgfqpoint{2.218017in}{2.259223in}}%
\pgfpathlineto{\pgfqpoint{2.405896in}{1.923013in}}%
\pgfpathlineto{\pgfqpoint{2.593775in}{1.407099in}}%
\pgfpathlineto{\pgfqpoint{2.781654in}{1.401433in}}%
\pgfpathlineto{\pgfqpoint{2.969533in}{1.404081in}}%
\pgfpathlineto{\pgfqpoint{3.157411in}{1.404081in}}%
\pgfpathlineto{\pgfqpoint{3.345290in}{1.404081in}}%
\pgfpathlineto{\pgfqpoint{3.533169in}{1.404081in}}%
\pgfpathlineto{\pgfqpoint{3.721048in}{1.404081in}}%
\pgfpathlineto{\pgfqpoint{3.908927in}{1.404081in}}%
\pgfpathlineto{\pgfqpoint{4.096805in}{1.404081in}}%
\pgfpathlineto{\pgfqpoint{4.284684in}{1.404081in}}%
\pgfpathlineto{\pgfqpoint{4.472563in}{1.404081in}}%
\pgfpathlineto{\pgfqpoint{4.660442in}{1.404081in}}%
\pgfpathlineto{\pgfqpoint{4.848320in}{1.404081in}}%
\pgfpathlineto{\pgfqpoint{5.036199in}{1.404081in}}%
\pgfpathlineto{\pgfqpoint{5.224078in}{1.404081in}}%
\pgfpathlineto{\pgfqpoint{5.411957in}{1.404081in}}%
\pgfpathlineto{\pgfqpoint{5.599836in}{1.404081in}}%
\pgfusepath{stroke}%
\end{pgfscope}%
\begin{pgfscope}%
\pgfpathrectangle{\pgfqpoint{0.865290in}{0.582899in}}{\pgfqpoint{4.960000in}{3.696000in}}%
\pgfusepath{clip}%
\pgfsetrectcap%
\pgfsetroundjoin%
\pgfsetlinewidth{1.505625pt}%
\definecolor{currentstroke}{rgb}{1.000000,0.498039,0.054902}%
\pgfsetstrokecolor{currentstroke}%
\pgfsetdash{}{0pt}%
\pgfpathmoveto{\pgfqpoint{1.090745in}{4.067717in}}%
\pgfpathlineto{\pgfqpoint{1.278623in}{3.779410in}}%
\pgfpathlineto{\pgfqpoint{1.466502in}{3.395334in}}%
\pgfpathlineto{\pgfqpoint{1.654381in}{3.114261in}}%
\pgfpathlineto{\pgfqpoint{1.842260in}{2.821650in}}%
\pgfpathlineto{\pgfqpoint{2.030139in}{2.403055in}}%
\pgfpathlineto{\pgfqpoint{2.218017in}{1.802328in}}%
\pgfpathlineto{\pgfqpoint{2.405896in}{1.306905in}}%
\pgfpathlineto{\pgfqpoint{2.593775in}{0.949013in}}%
\pgfpathlineto{\pgfqpoint{2.781654in}{0.750899in}}%
\pgfpathlineto{\pgfqpoint{2.969533in}{0.750899in}}%
\pgfpathlineto{\pgfqpoint{3.157411in}{0.765892in}}%
\pgfpathlineto{\pgfqpoint{3.345290in}{0.750899in}}%
\pgfpathlineto{\pgfqpoint{3.533169in}{0.765892in}}%
\pgfpathlineto{\pgfqpoint{3.721048in}{0.750899in}}%
\pgfpathlineto{\pgfqpoint{3.908927in}{0.765892in}}%
\pgfpathlineto{\pgfqpoint{4.096805in}{0.750899in}}%
\pgfpathlineto{\pgfqpoint{4.284684in}{0.765892in}}%
\pgfpathlineto{\pgfqpoint{4.472563in}{0.750899in}}%
\pgfpathlineto{\pgfqpoint{4.660442in}{0.765892in}}%
\pgfpathlineto{\pgfqpoint{4.848320in}{0.750899in}}%
\pgfpathlineto{\pgfqpoint{5.036199in}{0.765892in}}%
\pgfpathlineto{\pgfqpoint{5.224078in}{0.750899in}}%
\pgfpathlineto{\pgfqpoint{5.411957in}{0.765892in}}%
\pgfpathlineto{\pgfqpoint{5.599836in}{0.750899in}}%
\pgfusepath{stroke}%
\end{pgfscope}%
\begin{pgfscope}%
\pgfpathrectangle{\pgfqpoint{0.865290in}{0.582899in}}{\pgfqpoint{4.960000in}{3.696000in}}%
\pgfusepath{clip}%
\pgfsetrectcap%
\pgfsetroundjoin%
\pgfsetlinewidth{1.505625pt}%
\definecolor{currentstroke}{rgb}{0.172549,0.627451,0.172549}%
\pgfsetstrokecolor{currentstroke}%
\pgfsetdash{}{0pt}%
\pgfpathmoveto{\pgfqpoint{1.090745in}{4.037804in}}%
\pgfpathlineto{\pgfqpoint{1.278623in}{3.770623in}}%
\pgfpathlineto{\pgfqpoint{1.466502in}{3.508937in}}%
\pgfpathlineto{\pgfqpoint{1.654381in}{3.199972in}}%
\pgfpathlineto{\pgfqpoint{1.842260in}{2.639362in}}%
\pgfpathlineto{\pgfqpoint{2.030139in}{2.362829in}}%
\pgfpathlineto{\pgfqpoint{2.218017in}{1.910587in}}%
\pgfpathlineto{\pgfqpoint{2.405896in}{1.534229in}}%
\pgfpathlineto{\pgfqpoint{2.593775in}{0.927837in}}%
\pgfpathlineto{\pgfqpoint{2.781654in}{0.840021in}}%
\pgfpathlineto{\pgfqpoint{2.969533in}{0.868002in}}%
\pgfpathlineto{\pgfqpoint{3.157411in}{0.857754in}}%
\pgfpathlineto{\pgfqpoint{3.345290in}{0.857754in}}%
\pgfpathlineto{\pgfqpoint{3.533169in}{0.852195in}}%
\pgfpathlineto{\pgfqpoint{3.721048in}{0.863013in}}%
\pgfpathlineto{\pgfqpoint{3.908927in}{0.857754in}}%
\pgfpathlineto{\pgfqpoint{4.096805in}{0.857754in}}%
\pgfpathlineto{\pgfqpoint{4.284684in}{0.852195in}}%
\pgfpathlineto{\pgfqpoint{4.472563in}{0.863013in}}%
\pgfpathlineto{\pgfqpoint{4.660442in}{0.857754in}}%
\pgfpathlineto{\pgfqpoint{4.848320in}{0.857754in}}%
\pgfpathlineto{\pgfqpoint{5.036199in}{0.852195in}}%
\pgfpathlineto{\pgfqpoint{5.224078in}{0.863013in}}%
\pgfpathlineto{\pgfqpoint{5.411957in}{0.857754in}}%
\pgfpathlineto{\pgfqpoint{5.599836in}{0.857754in}}%
\pgfusepath{stroke}%
\end{pgfscope}%
\begin{pgfscope}%
\pgfpathrectangle{\pgfqpoint{0.865290in}{0.582899in}}{\pgfqpoint{4.960000in}{3.696000in}}%
\pgfusepath{clip}%
\pgfsetrectcap%
\pgfsetroundjoin%
\pgfsetlinewidth{1.505625pt}%
\definecolor{currentstroke}{rgb}{0.839216,0.152941,0.156863}%
\pgfsetstrokecolor{currentstroke}%
\pgfsetdash{}{0pt}%
\pgfpathmoveto{\pgfqpoint{1.090745in}{4.051395in}}%
\pgfpathlineto{\pgfqpoint{1.278623in}{3.699473in}}%
\pgfpathlineto{\pgfqpoint{1.466502in}{3.433469in}}%
\pgfpathlineto{\pgfqpoint{1.654381in}{3.150050in}}%
\pgfpathlineto{\pgfqpoint{1.842260in}{2.789944in}}%
\pgfpathlineto{\pgfqpoint{2.030139in}{2.260060in}}%
\pgfpathlineto{\pgfqpoint{2.218017in}{1.909676in}}%
\pgfpathlineto{\pgfqpoint{2.405896in}{1.495144in}}%
\pgfpathlineto{\pgfqpoint{2.593775in}{1.406965in}}%
\pgfpathlineto{\pgfqpoint{2.781654in}{1.403233in}}%
\pgfpathlineto{\pgfqpoint{2.969533in}{1.403154in}}%
\pgfpathlineto{\pgfqpoint{3.157411in}{1.403134in}}%
\pgfpathlineto{\pgfqpoint{3.345290in}{1.403154in}}%
\pgfpathlineto{\pgfqpoint{3.533169in}{1.403193in}}%
\pgfpathlineto{\pgfqpoint{3.721048in}{1.403193in}}%
\pgfpathlineto{\pgfqpoint{3.908927in}{1.403154in}}%
\pgfpathlineto{\pgfqpoint{4.096805in}{1.403154in}}%
\pgfpathlineto{\pgfqpoint{4.284684in}{1.403193in}}%
\pgfpathlineto{\pgfqpoint{4.472563in}{1.403193in}}%
\pgfpathlineto{\pgfqpoint{4.660442in}{1.403173in}}%
\pgfpathlineto{\pgfqpoint{4.848320in}{1.403173in}}%
\pgfpathlineto{\pgfqpoint{5.036199in}{1.403193in}}%
\pgfpathlineto{\pgfqpoint{5.224078in}{1.403193in}}%
\pgfpathlineto{\pgfqpoint{5.411957in}{1.403233in}}%
\pgfpathlineto{\pgfqpoint{5.599836in}{1.403233in}}%
\pgfusepath{stroke}%
\end{pgfscope}%
\begin{pgfscope}%
\pgfsetrectcap%
\pgfsetmiterjoin%
\pgfsetlinewidth{0.803000pt}%
\definecolor{currentstroke}{rgb}{0.000000,0.000000,0.000000}%
\pgfsetstrokecolor{currentstroke}%
\pgfsetdash{}{0pt}%
\pgfpathmoveto{\pgfqpoint{0.865290in}{0.582899in}}%
\pgfpathlineto{\pgfqpoint{0.865290in}{4.278899in}}%
\pgfusepath{stroke}%
\end{pgfscope}%
\begin{pgfscope}%
\pgfsetrectcap%
\pgfsetmiterjoin%
\pgfsetlinewidth{0.803000pt}%
\definecolor{currentstroke}{rgb}{0.000000,0.000000,0.000000}%
\pgfsetstrokecolor{currentstroke}%
\pgfsetdash{}{0pt}%
\pgfpathmoveto{\pgfqpoint{5.825290in}{0.582899in}}%
\pgfpathlineto{\pgfqpoint{5.825290in}{4.278899in}}%
\pgfusepath{stroke}%
\end{pgfscope}%
\begin{pgfscope}%
\pgfsetrectcap%
\pgfsetmiterjoin%
\pgfsetlinewidth{0.803000pt}%
\definecolor{currentstroke}{rgb}{0.000000,0.000000,0.000000}%
\pgfsetstrokecolor{currentstroke}%
\pgfsetdash{}{0pt}%
\pgfpathmoveto{\pgfqpoint{0.865290in}{0.582899in}}%
\pgfpathlineto{\pgfqpoint{5.825290in}{0.582899in}}%
\pgfusepath{stroke}%
\end{pgfscope}%
\begin{pgfscope}%
\pgfsetrectcap%
\pgfsetmiterjoin%
\pgfsetlinewidth{0.803000pt}%
\definecolor{currentstroke}{rgb}{0.000000,0.000000,0.000000}%
\pgfsetstrokecolor{currentstroke}%
\pgfsetdash{}{0pt}%
\pgfpathmoveto{\pgfqpoint{0.865290in}{4.278899in}}%
\pgfpathlineto{\pgfqpoint{5.825290in}{4.278899in}}%
\pgfusepath{stroke}%
\end{pgfscope}%
\begin{pgfscope}%
\pgfsetbuttcap%
\pgfsetmiterjoin%
\definecolor{currentfill}{rgb}{1.000000,1.000000,1.000000}%
\pgfsetfillcolor{currentfill}%
\pgfsetfillopacity{0.800000}%
\pgfsetlinewidth{1.003750pt}%
\definecolor{currentstroke}{rgb}{0.800000,0.800000,0.800000}%
\pgfsetstrokecolor{currentstroke}%
\pgfsetstrokeopacity{0.800000}%
\pgfsetdash{}{0pt}%
\pgfpathmoveto{\pgfqpoint{2.158059in}{2.796434in}}%
\pgfpathlineto{\pgfqpoint{5.669735in}{2.796434in}}%
\pgfpathquadraticcurveto{\pgfqpoint{5.714179in}{2.796434in}}{\pgfqpoint{5.714179in}{2.840879in}}%
\pgfpathlineto{\pgfqpoint{5.714179in}{4.123343in}}%
\pgfpathquadraticcurveto{\pgfqpoint{5.714179in}{4.167788in}}{\pgfqpoint{5.669735in}{4.167788in}}%
\pgfpathlineto{\pgfqpoint{2.158059in}{4.167788in}}%
\pgfpathquadraticcurveto{\pgfqpoint{2.113615in}{4.167788in}}{\pgfqpoint{2.113615in}{4.123343in}}%
\pgfpathlineto{\pgfqpoint{2.113615in}{2.840879in}}%
\pgfpathquadraticcurveto{\pgfqpoint{2.113615in}{2.796434in}}{\pgfqpoint{2.158059in}{2.796434in}}%
\pgfpathlineto{\pgfqpoint{2.158059in}{2.796434in}}%
\pgfpathclose%
\pgfusepath{stroke,fill}%
\end{pgfscope}%
\begin{pgfscope}%
\pgfsetrectcap%
\pgfsetroundjoin%
\pgfsetlinewidth{1.505625pt}%
\definecolor{currentstroke}{rgb}{0.121569,0.466667,0.705882}%
\pgfsetstrokecolor{currentstroke}%
\pgfsetdash{}{0pt}%
\pgfpathmoveto{\pgfqpoint{2.202504in}{3.987840in}}%
\pgfpathlineto{\pgfqpoint{2.424726in}{3.987840in}}%
\pgfpathlineto{\pgfqpoint{2.646948in}{3.987840in}}%
\pgfusepath{stroke}%
\end{pgfscope}%
\begin{pgfscope}%
\definecolor{textcolor}{rgb}{0.000000,0.000000,0.000000}%
\pgfsetstrokecolor{textcolor}%
\pgfsetfillcolor{textcolor}%
\pgftext[x=2.824726in,y=3.910062in,left,base]{\color{textcolor}{\sffamily\fontsize{16.000000}{19.200000}\selectfont\catcode`\^=\active\def^{\ifmmode\sp\else\^{}\fi}\catcode`\%=\active\def
\end{pgfscope}%
\begin{pgfscope}%
\pgfsetrectcap%
\pgfsetroundjoin%
\pgfsetlinewidth{1.505625pt}%
\definecolor{currentstroke}{rgb}{1.000000,0.498039,0.054902}%
\pgfsetstrokecolor{currentstroke}%
\pgfsetdash{}{0pt}%
\pgfpathmoveto{\pgfqpoint{2.202504in}{3.661668in}}%
\pgfpathlineto{\pgfqpoint{2.424726in}{3.661668in}}%
\pgfpathlineto{\pgfqpoint{2.646948in}{3.661668in}}%
\pgfusepath{stroke}%
\end{pgfscope}%
\begin{pgfscope}%
\definecolor{textcolor}{rgb}{0.000000,0.000000,0.000000}%
\pgfsetstrokecolor{textcolor}%
\pgfsetfillcolor{textcolor}%
\pgftext[x=2.824726in,y=3.583890in,left,base]{\color{textcolor}{\sffamily\fontsize{16.000000}{19.200000}\selectfont\catcode`\^=\active\def^{\ifmmode\sp\else\^{}\fi}\catcode`\%=\active\def
\end{pgfscope}%
\begin{pgfscope}%
\pgfsetrectcap%
\pgfsetroundjoin%
\pgfsetlinewidth{1.505625pt}%
\definecolor{currentstroke}{rgb}{0.172549,0.627451,0.172549}%
\pgfsetstrokecolor{currentstroke}%
\pgfsetdash{}{0pt}%
\pgfpathmoveto{\pgfqpoint{2.202504in}{3.335497in}}%
\pgfpathlineto{\pgfqpoint{2.424726in}{3.335497in}}%
\pgfpathlineto{\pgfqpoint{2.646948in}{3.335497in}}%
\pgfusepath{stroke}%
\end{pgfscope}%
\begin{pgfscope}%
\definecolor{textcolor}{rgb}{0.000000,0.000000,0.000000}%
\pgfsetstrokecolor{textcolor}%
\pgfsetfillcolor{textcolor}%
\pgftext[x=2.824726in,y=3.257719in,left,base]{\color{textcolor}{\sffamily\fontsize{16.000000}{19.200000}\selectfont\catcode`\^=\active\def^{\ifmmode\sp\else\^{}\fi}\catcode`\%=\active\def
\end{pgfscope}%
\begin{pgfscope}%
\pgfsetrectcap%
\pgfsetroundjoin%
\pgfsetlinewidth{1.505625pt}%
\definecolor{currentstroke}{rgb}{0.839216,0.152941,0.156863}%
\pgfsetstrokecolor{currentstroke}%
\pgfsetdash{}{0pt}%
\pgfpathmoveto{\pgfqpoint{2.202504in}{3.009325in}}%
\pgfpathlineto{\pgfqpoint{2.424726in}{3.009325in}}%
\pgfpathlineto{\pgfqpoint{2.646948in}{3.009325in}}%
\pgfusepath{stroke}%
\end{pgfscope}%
\begin{pgfscope}%
\definecolor{textcolor}{rgb}{0.000000,0.000000,0.000000}%
\pgfsetstrokecolor{textcolor}%
\pgfsetfillcolor{textcolor}%
\pgftext[x=2.824726in,y=2.931547in,left,base]{\color{textcolor}{\sffamily\fontsize{16.000000}{19.200000}\selectfont\catcode`\^=\active\def^{\ifmmode\sp\else\^{}\fi}\catcode`\%=\active\def
\end{pgfscope}%
\end{pgfpicture}%
\makeatother%
\endgroup%

%% file: figs/network_conv_USPowerGrid.pgf
\begingroup%
\makeatletter%
\begin{pgfpicture}%
\pgfpathrectangle{\pgfpointorigin}{\pgfqpoint{5.825290in}{4.278899in}}%
\pgfusepath{use as bounding box, clip}%
\begin{pgfscope}%
\pgfsetbuttcap%
\pgfsetmiterjoin%
\definecolor{currentfill}{rgb}{1.000000,1.000000,1.000000}%
\pgfsetfillcolor{currentfill}%
\pgfsetlinewidth{0.000000pt}%
\definecolor{currentstroke}{rgb}{1.000000,1.000000,1.000000}%
\pgfsetstrokecolor{currentstroke}%
\pgfsetdash{}{0pt}%
\pgfpathmoveto{\pgfqpoint{0.000000in}{0.000000in}}%
\pgfpathlineto{\pgfqpoint{5.825290in}{0.000000in}}%
\pgfpathlineto{\pgfqpoint{5.825290in}{4.278899in}}%
\pgfpathlineto{\pgfqpoint{0.000000in}{4.278899in}}%
\pgfpathlineto{\pgfqpoint{0.000000in}{0.000000in}}%
\pgfpathclose%
\pgfusepath{fill}%
\end{pgfscope}%
\begin{pgfscope}%
\pgfsetbuttcap%
\pgfsetmiterjoin%
\definecolor{currentfill}{rgb}{1.000000,1.000000,1.000000}%
\pgfsetfillcolor{currentfill}%
\pgfsetlinewidth{0.000000pt}%
\definecolor{currentstroke}{rgb}{0.000000,0.000000,0.000000}%
\pgfsetstrokecolor{currentstroke}%
\pgfsetstrokeopacity{0.000000}%
\pgfsetdash{}{0pt}%
\pgfpathmoveto{\pgfqpoint{0.865290in}{0.582899in}}%
\pgfpathlineto{\pgfqpoint{5.825290in}{0.582899in}}%
\pgfpathlineto{\pgfqpoint{5.825290in}{4.278899in}}%
\pgfpathlineto{\pgfqpoint{0.865290in}{4.278899in}}%
\pgfpathlineto{\pgfqpoint{0.865290in}{0.582899in}}%
\pgfpathclose%
\pgfusepath{fill}%
\end{pgfscope}%
\begin{pgfscope}%
\pgfsetbuttcap%
\pgfsetroundjoin%
\definecolor{currentfill}{rgb}{0.000000,0.000000,0.000000}%
\pgfsetfillcolor{currentfill}%
\pgfsetlinewidth{0.803000pt}%
\definecolor{currentstroke}{rgb}{0.000000,0.000000,0.000000}%
\pgfsetstrokecolor{currentstroke}%
\pgfsetdash{}{0pt}%
\pgfsys@defobject{currentmarker}{\pgfqpoint{0.000000in}{-0.048611in}}{\pgfqpoint{0.000000in}{0.000000in}}{%
\pgfpathmoveto{\pgfqpoint{0.000000in}{0.000000in}}%
\pgfpathlineto{\pgfqpoint{0.000000in}{-0.048611in}}%
\pgfusepath{stroke,fill}%
}%
\begin{pgfscope}%
\pgfsys@transformshift{1.372563in}{0.582899in}%
\pgfsys@useobject{currentmarker}{}%
\end{pgfscope}%
\end{pgfscope}%
\begin{pgfscope}%
\definecolor{textcolor}{rgb}{0.000000,0.000000,0.000000}%
\pgfsetstrokecolor{textcolor}%
\pgfsetfillcolor{textcolor}%
\pgftext[x=1.372563in,y=0.485677in,,top]{\color{textcolor}{\sffamily\fontsize{16.000000}{19.200000}\selectfont\catcode`\^=\active\def^{\ifmmode\sp\else\^{}\fi}\catcode`\%=\active\def
\end{pgfscope}%
\begin{pgfscope}%
\pgfsetbuttcap%
\pgfsetroundjoin%
\definecolor{currentfill}{rgb}{0.000000,0.000000,0.000000}%
\pgfsetfillcolor{currentfill}%
\pgfsetlinewidth{0.803000pt}%
\definecolor{currentstroke}{rgb}{0.000000,0.000000,0.000000}%
\pgfsetstrokecolor{currentstroke}%
\pgfsetdash{}{0pt}%
\pgfsys@defobject{currentmarker}{\pgfqpoint{0.000000in}{-0.048611in}}{\pgfqpoint{0.000000in}{0.000000in}}{%
\pgfpathmoveto{\pgfqpoint{0.000000in}{0.000000in}}%
\pgfpathlineto{\pgfqpoint{0.000000in}{-0.048611in}}%
\pgfusepath{stroke,fill}%
}%
\begin{pgfscope}%
\pgfsys@transformshift{1.842260in}{0.582899in}%
\pgfsys@useobject{currentmarker}{}%
\end{pgfscope}%
\end{pgfscope}%
\begin{pgfscope}%
\definecolor{textcolor}{rgb}{0.000000,0.000000,0.000000}%
\pgfsetstrokecolor{textcolor}%
\pgfsetfillcolor{textcolor}%
\pgftext[x=1.842260in,y=0.485677in,,top]{\color{textcolor}{\sffamily\fontsize{16.000000}{19.200000}\selectfont\catcode`\^=\active\def^{\ifmmode\sp\else\^{}\fi}\catcode`\%=\active\def
\end{pgfscope}%
\begin{pgfscope}%
\pgfsetbuttcap%
\pgfsetroundjoin%
\definecolor{currentfill}{rgb}{0.000000,0.000000,0.000000}%
\pgfsetfillcolor{currentfill}%
\pgfsetlinewidth{0.803000pt}%
\definecolor{currentstroke}{rgb}{0.000000,0.000000,0.000000}%
\pgfsetstrokecolor{currentstroke}%
\pgfsetdash{}{0pt}%
\pgfsys@defobject{currentmarker}{\pgfqpoint{0.000000in}{-0.048611in}}{\pgfqpoint{0.000000in}{0.000000in}}{%
\pgfpathmoveto{\pgfqpoint{0.000000in}{0.000000in}}%
\pgfpathlineto{\pgfqpoint{0.000000in}{-0.048611in}}%
\pgfusepath{stroke,fill}%
}%
\begin{pgfscope}%
\pgfsys@transformshift{2.311957in}{0.582899in}%
\pgfsys@useobject{currentmarker}{}%
\end{pgfscope}%
\end{pgfscope}%
\begin{pgfscope}%
\definecolor{textcolor}{rgb}{0.000000,0.000000,0.000000}%
\pgfsetstrokecolor{textcolor}%
\pgfsetfillcolor{textcolor}%
\pgftext[x=2.311957in,y=0.485677in,,top]{\color{textcolor}{\sffamily\fontsize{16.000000}{19.200000}\selectfont\catcode`\^=\active\def^{\ifmmode\sp\else\^{}\fi}\catcode`\%=\active\def
\end{pgfscope}%
\begin{pgfscope}%
\pgfsetbuttcap%
\pgfsetroundjoin%
\definecolor{currentfill}{rgb}{0.000000,0.000000,0.000000}%
\pgfsetfillcolor{currentfill}%
\pgfsetlinewidth{0.803000pt}%
\definecolor{currentstroke}{rgb}{0.000000,0.000000,0.000000}%
\pgfsetstrokecolor{currentstroke}%
\pgfsetdash{}{0pt}%
\pgfsys@defobject{currentmarker}{\pgfqpoint{0.000000in}{-0.048611in}}{\pgfqpoint{0.000000in}{0.000000in}}{%
\pgfpathmoveto{\pgfqpoint{0.000000in}{0.000000in}}%
\pgfpathlineto{\pgfqpoint{0.000000in}{-0.048611in}}%
\pgfusepath{stroke,fill}%
}%
\begin{pgfscope}%
\pgfsys@transformshift{2.781654in}{0.582899in}%
\pgfsys@useobject{currentmarker}{}%
\end{pgfscope}%
\end{pgfscope}%
\begin{pgfscope}%
\definecolor{textcolor}{rgb}{0.000000,0.000000,0.000000}%
\pgfsetstrokecolor{textcolor}%
\pgfsetfillcolor{textcolor}%
\pgftext[x=2.781654in,y=0.485677in,,top]{\color{textcolor}{\sffamily\fontsize{16.000000}{19.200000}\selectfont\catcode`\^=\active\def^{\ifmmode\sp\else\^{}\fi}\catcode`\%=\active\def
\end{pgfscope}%
\begin{pgfscope}%
\pgfsetbuttcap%
\pgfsetroundjoin%
\definecolor{currentfill}{rgb}{0.000000,0.000000,0.000000}%
\pgfsetfillcolor{currentfill}%
\pgfsetlinewidth{0.803000pt}%
\definecolor{currentstroke}{rgb}{0.000000,0.000000,0.000000}%
\pgfsetstrokecolor{currentstroke}%
\pgfsetdash{}{0pt}%
\pgfsys@defobject{currentmarker}{\pgfqpoint{0.000000in}{-0.048611in}}{\pgfqpoint{0.000000in}{0.000000in}}{%
\pgfpathmoveto{\pgfqpoint{0.000000in}{0.000000in}}%
\pgfpathlineto{\pgfqpoint{0.000000in}{-0.048611in}}%
\pgfusepath{stroke,fill}%
}%
\begin{pgfscope}%
\pgfsys@transformshift{3.251351in}{0.582899in}%
\pgfsys@useobject{currentmarker}{}%
\end{pgfscope}%
\end{pgfscope}%
\begin{pgfscope}%
\definecolor{textcolor}{rgb}{0.000000,0.000000,0.000000}%
\pgfsetstrokecolor{textcolor}%
\pgfsetfillcolor{textcolor}%
\pgftext[x=3.251351in,y=0.485677in,,top]{\color{textcolor}{\sffamily\fontsize{16.000000}{19.200000}\selectfont\catcode`\^=\active\def^{\ifmmode\sp\else\^{}\fi}\catcode`\%=\active\def
\end{pgfscope}%
\begin{pgfscope}%
\pgfsetbuttcap%
\pgfsetroundjoin%
\definecolor{currentfill}{rgb}{0.000000,0.000000,0.000000}%
\pgfsetfillcolor{currentfill}%
\pgfsetlinewidth{0.803000pt}%
\definecolor{currentstroke}{rgb}{0.000000,0.000000,0.000000}%
\pgfsetstrokecolor{currentstroke}%
\pgfsetdash{}{0pt}%
\pgfsys@defobject{currentmarker}{\pgfqpoint{0.000000in}{-0.048611in}}{\pgfqpoint{0.000000in}{0.000000in}}{%
\pgfpathmoveto{\pgfqpoint{0.000000in}{0.000000in}}%
\pgfpathlineto{\pgfqpoint{0.000000in}{-0.048611in}}%
\pgfusepath{stroke,fill}%
}%
\begin{pgfscope}%
\pgfsys@transformshift{3.721048in}{0.582899in}%
\pgfsys@useobject{currentmarker}{}%
\end{pgfscope}%
\end{pgfscope}%
\begin{pgfscope}%
\definecolor{textcolor}{rgb}{0.000000,0.000000,0.000000}%
\pgfsetstrokecolor{textcolor}%
\pgfsetfillcolor{textcolor}%
\pgftext[x=3.721048in,y=0.485677in,,top]{\color{textcolor}{\sffamily\fontsize{16.000000}{19.200000}\selectfont\catcode`\^=\active\def^{\ifmmode\sp\else\^{}\fi}\catcode`\%=\active\def
\end{pgfscope}%
\begin{pgfscope}%
\pgfsetbuttcap%
\pgfsetroundjoin%
\definecolor{currentfill}{rgb}{0.000000,0.000000,0.000000}%
\pgfsetfillcolor{currentfill}%
\pgfsetlinewidth{0.803000pt}%
\definecolor{currentstroke}{rgb}{0.000000,0.000000,0.000000}%
\pgfsetstrokecolor{currentstroke}%
\pgfsetdash{}{0pt}%
\pgfsys@defobject{currentmarker}{\pgfqpoint{0.000000in}{-0.048611in}}{\pgfqpoint{0.000000in}{0.000000in}}{%
\pgfpathmoveto{\pgfqpoint{0.000000in}{0.000000in}}%
\pgfpathlineto{\pgfqpoint{0.000000in}{-0.048611in}}%
\pgfusepath{stroke,fill}%
}%
\begin{pgfscope}%
\pgfsys@transformshift{4.190745in}{0.582899in}%
\pgfsys@useobject{currentmarker}{}%
\end{pgfscope}%
\end{pgfscope}%
\begin{pgfscope}%
\definecolor{textcolor}{rgb}{0.000000,0.000000,0.000000}%
\pgfsetstrokecolor{textcolor}%
\pgfsetfillcolor{textcolor}%
\pgftext[x=4.190745in,y=0.485677in,,top]{\color{textcolor}{\sffamily\fontsize{16.000000}{19.200000}\selectfont\catcode`\^=\active\def^{\ifmmode\sp\else\^{}\fi}\catcode`\%=\active\def
\end{pgfscope}%
\begin{pgfscope}%
\pgfsetbuttcap%
\pgfsetroundjoin%
\definecolor{currentfill}{rgb}{0.000000,0.000000,0.000000}%
\pgfsetfillcolor{currentfill}%
\pgfsetlinewidth{0.803000pt}%
\definecolor{currentstroke}{rgb}{0.000000,0.000000,0.000000}%
\pgfsetstrokecolor{currentstroke}%
\pgfsetdash{}{0pt}%
\pgfsys@defobject{currentmarker}{\pgfqpoint{0.000000in}{-0.048611in}}{\pgfqpoint{0.000000in}{0.000000in}}{%
\pgfpathmoveto{\pgfqpoint{0.000000in}{0.000000in}}%
\pgfpathlineto{\pgfqpoint{0.000000in}{-0.048611in}}%
\pgfusepath{stroke,fill}%
}%
\begin{pgfscope}%
\pgfsys@transformshift{4.660442in}{0.582899in}%
\pgfsys@useobject{currentmarker}{}%
\end{pgfscope}%
\end{pgfscope}%
\begin{pgfscope}%
\definecolor{textcolor}{rgb}{0.000000,0.000000,0.000000}%
\pgfsetstrokecolor{textcolor}%
\pgfsetfillcolor{textcolor}%
\pgftext[x=4.660442in,y=0.485677in,,top]{\color{textcolor}{\sffamily\fontsize{16.000000}{19.200000}\selectfont\catcode`\^=\active\def^{\ifmmode\sp\else\^{}\fi}\catcode`\%=\active\def
\end{pgfscope}%
\begin{pgfscope}%
\pgfsetbuttcap%
\pgfsetroundjoin%
\definecolor{currentfill}{rgb}{0.000000,0.000000,0.000000}%
\pgfsetfillcolor{currentfill}%
\pgfsetlinewidth{0.803000pt}%
\definecolor{currentstroke}{rgb}{0.000000,0.000000,0.000000}%
\pgfsetstrokecolor{currentstroke}%
\pgfsetdash{}{0pt}%
\pgfsys@defobject{currentmarker}{\pgfqpoint{0.000000in}{-0.048611in}}{\pgfqpoint{0.000000in}{0.000000in}}{%
\pgfpathmoveto{\pgfqpoint{0.000000in}{0.000000in}}%
\pgfpathlineto{\pgfqpoint{0.000000in}{-0.048611in}}%
\pgfusepath{stroke,fill}%
}%
\begin{pgfscope}%
\pgfsys@transformshift{5.130139in}{0.582899in}%
\pgfsys@useobject{currentmarker}{}%
\end{pgfscope}%
\end{pgfscope}%
\begin{pgfscope}%
\definecolor{textcolor}{rgb}{0.000000,0.000000,0.000000}%
\pgfsetstrokecolor{textcolor}%
\pgfsetfillcolor{textcolor}%
\pgftext[x=5.130139in,y=0.485677in,,top]{\color{textcolor}{\sffamily\fontsize{16.000000}{19.200000}\selectfont\catcode`\^=\active\def^{\ifmmode\sp\else\^{}\fi}\catcode`\%=\active\def
\end{pgfscope}%
\begin{pgfscope}%
\pgfsetbuttcap%
\pgfsetroundjoin%
\definecolor{currentfill}{rgb}{0.000000,0.000000,0.000000}%
\pgfsetfillcolor{currentfill}%
\pgfsetlinewidth{0.803000pt}%
\definecolor{currentstroke}{rgb}{0.000000,0.000000,0.000000}%
\pgfsetstrokecolor{currentstroke}%
\pgfsetdash{}{0pt}%
\pgfsys@defobject{currentmarker}{\pgfqpoint{0.000000in}{-0.048611in}}{\pgfqpoint{0.000000in}{0.000000in}}{%
\pgfpathmoveto{\pgfqpoint{0.000000in}{0.000000in}}%
\pgfpathlineto{\pgfqpoint{0.000000in}{-0.048611in}}%
\pgfusepath{stroke,fill}%
}%
\begin{pgfscope}%
\pgfsys@transformshift{5.599836in}{0.582899in}%
\pgfsys@useobject{currentmarker}{}%
\end{pgfscope}%
\end{pgfscope}%
\begin{pgfscope}%
\definecolor{textcolor}{rgb}{0.000000,0.000000,0.000000}%
\pgfsetstrokecolor{textcolor}%
\pgfsetfillcolor{textcolor}%
\pgftext[x=5.599836in,y=0.485677in,,top]{\color{textcolor}{\sffamily\fontsize{16.000000}{19.200000}\selectfont\catcode`\^=\active\def^{\ifmmode\sp\else\^{}\fi}\catcode`\%=\active\def
\end{pgfscope}%
\begin{pgfscope}%
\definecolor{textcolor}{rgb}{0.000000,0.000000,0.000000}%
\pgfsetstrokecolor{textcolor}%
\pgfsetfillcolor{textcolor}%
\pgftext[x=3.345290in,y=0.215061in,,top]{\color{textcolor}{\sffamily\fontsize{16.000000}{19.200000}\selectfont\catcode`\^=\active\def^{\ifmmode\sp\else\^{}\fi}\catcode`\%=\active\def
\end{pgfscope}%
\begin{pgfscope}%
\pgfsetbuttcap%
\pgfsetroundjoin%
\definecolor{currentfill}{rgb}{0.000000,0.000000,0.000000}%
\pgfsetfillcolor{currentfill}%
\pgfsetlinewidth{0.803000pt}%
\definecolor{currentstroke}{rgb}{0.000000,0.000000,0.000000}%
\pgfsetstrokecolor{currentstroke}%
\pgfsetdash{}{0pt}%
\pgfsys@defobject{currentmarker}{\pgfqpoint{-0.048611in}{0.000000in}}{\pgfqpoint{-0.000000in}{0.000000in}}{%
\pgfpathmoveto{\pgfqpoint{-0.000000in}{0.000000in}}%
\pgfpathlineto{\pgfqpoint{-0.048611in}{0.000000in}}%
\pgfusepath{stroke,fill}%
}%
\begin{pgfscope}%
\pgfsys@transformshift{0.865290in}{1.073874in}%
\pgfsys@useobject{currentmarker}{}%
\end{pgfscope}%
\end{pgfscope}%
\begin{pgfscope}%
\definecolor{textcolor}{rgb}{0.000000,0.000000,0.000000}%
\pgfsetstrokecolor{textcolor}%
\pgfsetfillcolor{textcolor}%
\pgftext[x=0.270616in, y=0.989456in, left, base]{\color{textcolor}{\sffamily\fontsize{16.000000}{19.200000}\selectfont\catcode`\^=\active\def^{\ifmmode\sp\else\^{}\fi}\catcode`\%=\active\def
\end{pgfscope}%
\begin{pgfscope}%
\pgfsetbuttcap%
\pgfsetroundjoin%
\definecolor{currentfill}{rgb}{0.000000,0.000000,0.000000}%
\pgfsetfillcolor{currentfill}%
\pgfsetlinewidth{0.803000pt}%
\definecolor{currentstroke}{rgb}{0.000000,0.000000,0.000000}%
\pgfsetstrokecolor{currentstroke}%
\pgfsetdash{}{0pt}%
\pgfsys@defobject{currentmarker}{\pgfqpoint{-0.048611in}{0.000000in}}{\pgfqpoint{-0.000000in}{0.000000in}}{%
\pgfpathmoveto{\pgfqpoint{-0.000000in}{0.000000in}}%
\pgfpathlineto{\pgfqpoint{-0.048611in}{0.000000in}}%
\pgfusepath{stroke,fill}%
}%
\begin{pgfscope}%
\pgfsys@transformshift{0.865290in}{1.733011in}%
\pgfsys@useobject{currentmarker}{}%
\end{pgfscope}%
\end{pgfscope}%
\begin{pgfscope}%
\definecolor{textcolor}{rgb}{0.000000,0.000000,0.000000}%
\pgfsetstrokecolor{textcolor}%
\pgfsetfillcolor{textcolor}%
\pgftext[x=0.270616in, y=1.648592in, left, base]{\color{textcolor}{\sffamily\fontsize{16.000000}{19.200000}\selectfont\catcode`\^=\active\def^{\ifmmode\sp\else\^{}\fi}\catcode`\%=\active\def
\end{pgfscope}%
\begin{pgfscope}%
\pgfsetbuttcap%
\pgfsetroundjoin%
\definecolor{currentfill}{rgb}{0.000000,0.000000,0.000000}%
\pgfsetfillcolor{currentfill}%
\pgfsetlinewidth{0.803000pt}%
\definecolor{currentstroke}{rgb}{0.000000,0.000000,0.000000}%
\pgfsetstrokecolor{currentstroke}%
\pgfsetdash{}{0pt}%
\pgfsys@defobject{currentmarker}{\pgfqpoint{-0.048611in}{0.000000in}}{\pgfqpoint{-0.000000in}{0.000000in}}{%
\pgfpathmoveto{\pgfqpoint{-0.000000in}{0.000000in}}%
\pgfpathlineto{\pgfqpoint{-0.048611in}{0.000000in}}%
\pgfusepath{stroke,fill}%
}%
\begin{pgfscope}%
\pgfsys@transformshift{0.865290in}{2.392147in}%
\pgfsys@useobject{currentmarker}{}%
\end{pgfscope}%
\end{pgfscope}%
\begin{pgfscope}%
\definecolor{textcolor}{rgb}{0.000000,0.000000,0.000000}%
\pgfsetstrokecolor{textcolor}%
\pgfsetfillcolor{textcolor}%
\pgftext[x=0.346658in, y=2.307729in, left, base]{\color{textcolor}{\sffamily\fontsize{16.000000}{19.200000}\selectfont\catcode`\^=\active\def^{\ifmmode\sp\else\^{}\fi}\catcode`\%=\active\def
\end{pgfscope}%
\begin{pgfscope}%
\pgfsetbuttcap%
\pgfsetroundjoin%
\definecolor{currentfill}{rgb}{0.000000,0.000000,0.000000}%
\pgfsetfillcolor{currentfill}%
\pgfsetlinewidth{0.803000pt}%
\definecolor{currentstroke}{rgb}{0.000000,0.000000,0.000000}%
\pgfsetstrokecolor{currentstroke}%
\pgfsetdash{}{0pt}%
\pgfsys@defobject{currentmarker}{\pgfqpoint{-0.048611in}{0.000000in}}{\pgfqpoint{-0.000000in}{0.000000in}}{%
\pgfpathmoveto{\pgfqpoint{-0.000000in}{0.000000in}}%
\pgfpathlineto{\pgfqpoint{-0.048611in}{0.000000in}}%
\pgfusepath{stroke,fill}%
}%
\begin{pgfscope}%
\pgfsys@transformshift{0.865290in}{3.051284in}%
\pgfsys@useobject{currentmarker}{}%
\end{pgfscope}%
\end{pgfscope}%
\begin{pgfscope}%
\definecolor{textcolor}{rgb}{0.000000,0.000000,0.000000}%
\pgfsetstrokecolor{textcolor}%
\pgfsetfillcolor{textcolor}%
\pgftext[x=0.346658in, y=2.966866in, left, base]{\color{textcolor}{\sffamily\fontsize{16.000000}{19.200000}\selectfont\catcode`\^=\active\def^{\ifmmode\sp\else\^{}\fi}\catcode`\%=\active\def
\end{pgfscope}%
\begin{pgfscope}%
\pgfsetbuttcap%
\pgfsetroundjoin%
\definecolor{currentfill}{rgb}{0.000000,0.000000,0.000000}%
\pgfsetfillcolor{currentfill}%
\pgfsetlinewidth{0.803000pt}%
\definecolor{currentstroke}{rgb}{0.000000,0.000000,0.000000}%
\pgfsetstrokecolor{currentstroke}%
\pgfsetdash{}{0pt}%
\pgfsys@defobject{currentmarker}{\pgfqpoint{-0.048611in}{0.000000in}}{\pgfqpoint{-0.000000in}{0.000000in}}{%
\pgfpathmoveto{\pgfqpoint{-0.000000in}{0.000000in}}%
\pgfpathlineto{\pgfqpoint{-0.048611in}{0.000000in}}%
\pgfusepath{stroke,fill}%
}%
\begin{pgfscope}%
\pgfsys@transformshift{0.865290in}{3.710421in}%
\pgfsys@useobject{currentmarker}{}%
\end{pgfscope}%
\end{pgfscope}%
\begin{pgfscope}%
\definecolor{textcolor}{rgb}{0.000000,0.000000,0.000000}%
\pgfsetstrokecolor{textcolor}%
\pgfsetfillcolor{textcolor}%
\pgftext[x=0.346658in, y=3.626002in, left, base]{\color{textcolor}{\sffamily\fontsize{16.000000}{19.200000}\selectfont\catcode`\^=\active\def^{\ifmmode\sp\else\^{}\fi}\catcode`\%=\active\def
\end{pgfscope}%
\begin{pgfscope}%
\definecolor{textcolor}{rgb}{0.000000,0.000000,0.000000}%
\pgfsetstrokecolor{textcolor}%
\pgfsetfillcolor{textcolor}%
\pgftext[x=0.215061in,y=2.430899in,,bottom,rotate=90.000000]{\color{textcolor}{\sffamily\fontsize{16.000000}{19.200000}\selectfont\catcode`\^=\active\def^{\ifmmode\sp\else\^{}\fi}\catcode`\%=\active\def
\end{pgfscope}%
\begin{pgfscope}%
\pgfpathrectangle{\pgfqpoint{0.865290in}{0.582899in}}{\pgfqpoint{4.960000in}{3.696000in}}%
\pgfusepath{clip}%
\pgfsetrectcap%
\pgfsetroundjoin%
\pgfsetlinewidth{1.505625pt}%
\definecolor{currentstroke}{rgb}{0.121569,0.466667,0.705882}%
\pgfsetstrokecolor{currentstroke}%
\pgfsetdash{}{0pt}%
\pgfpathmoveto{\pgfqpoint{1.090745in}{3.978845in}}%
\pgfpathlineto{\pgfqpoint{1.278623in}{3.961946in}}%
\pgfpathlineto{\pgfqpoint{1.466502in}{4.030135in}}%
\pgfpathlineto{\pgfqpoint{1.654381in}{3.853185in}}%
\pgfpathlineto{\pgfqpoint{1.842260in}{3.551747in}}%
\pgfpathlineto{\pgfqpoint{2.030139in}{3.232791in}}%
\pgfpathlineto{\pgfqpoint{2.218017in}{3.038856in}}%
\pgfpathlineto{\pgfqpoint{2.405896in}{2.690985in}}%
\pgfpathlineto{\pgfqpoint{2.593775in}{2.168673in}}%
\pgfpathlineto{\pgfqpoint{2.781654in}{1.894146in}}%
\pgfpathlineto{\pgfqpoint{2.969533in}{1.530172in}}%
\pgfpathlineto{\pgfqpoint{3.157411in}{1.073837in}}%
\pgfpathlineto{\pgfqpoint{3.345290in}{0.936577in}}%
\pgfpathlineto{\pgfqpoint{3.533169in}{0.929506in}}%
\pgfpathlineto{\pgfqpoint{3.721048in}{0.929506in}}%
\pgfpathlineto{\pgfqpoint{3.908927in}{0.936577in}}%
\pgfpathlineto{\pgfqpoint{4.096805in}{0.936577in}}%
\pgfpathlineto{\pgfqpoint{4.284684in}{0.936577in}}%
\pgfpathlineto{\pgfqpoint{4.472563in}{0.936577in}}%
\pgfpathlineto{\pgfqpoint{4.660442in}{0.936577in}}%
\pgfpathlineto{\pgfqpoint{4.848320in}{0.936577in}}%
\pgfpathlineto{\pgfqpoint{5.036199in}{0.936577in}}%
\pgfpathlineto{\pgfqpoint{5.224078in}{0.936577in}}%
\pgfpathlineto{\pgfqpoint{5.411957in}{0.929506in}}%
\pgfpathlineto{\pgfqpoint{5.599836in}{0.929506in}}%
\pgfusepath{stroke}%
\end{pgfscope}%
\begin{pgfscope}%
\pgfpathrectangle{\pgfqpoint{0.865290in}{0.582899in}}{\pgfqpoint{4.960000in}{3.696000in}}%
\pgfusepath{clip}%
\pgfsetrectcap%
\pgfsetroundjoin%
\pgfsetlinewidth{1.505625pt}%
\definecolor{currentstroke}{rgb}{1.000000,0.498039,0.054902}%
\pgfsetstrokecolor{currentstroke}%
\pgfsetdash{}{0pt}%
\pgfpathmoveto{\pgfqpoint{1.090745in}{4.085650in}}%
\pgfpathlineto{\pgfqpoint{1.278623in}{3.860210in}}%
\pgfpathlineto{\pgfqpoint{1.466502in}{3.639785in}}%
\pgfpathlineto{\pgfqpoint{1.654381in}{3.363055in}}%
\pgfpathlineto{\pgfqpoint{1.842260in}{2.997408in}}%
\pgfpathlineto{\pgfqpoint{2.030139in}{2.526228in}}%
\pgfpathlineto{\pgfqpoint{2.218017in}{1.819932in}}%
\pgfpathlineto{\pgfqpoint{2.405896in}{1.691474in}}%
\pgfpathlineto{\pgfqpoint{2.593775in}{1.362321in}}%
\pgfpathlineto{\pgfqpoint{2.781654in}{0.929506in}}%
\pgfpathlineto{\pgfqpoint{2.781963in}{0.572899in}}%
\pgfusepath{stroke}%
\end{pgfscope}%
\begin{pgfscope}%
\pgfpathrectangle{\pgfqpoint{0.865290in}{0.582899in}}{\pgfqpoint{4.960000in}{3.696000in}}%
\pgfusepath{clip}%
\pgfsetrectcap%
\pgfsetroundjoin%
\pgfsetlinewidth{1.505625pt}%
\definecolor{currentstroke}{rgb}{0.172549,0.627451,0.172549}%
\pgfsetstrokecolor{currentstroke}%
\pgfsetdash{}{0pt}%
\pgfpathmoveto{\pgfqpoint{1.090745in}{4.110899in}}%
\pgfpathlineto{\pgfqpoint{1.278623in}{3.974379in}}%
\pgfpathlineto{\pgfqpoint{1.466502in}{3.716540in}}%
\pgfpathlineto{\pgfqpoint{1.654381in}{3.569643in}}%
\pgfpathlineto{\pgfqpoint{1.842260in}{3.315567in}}%
\pgfpathlineto{\pgfqpoint{2.030139in}{3.151542in}}%
\pgfpathlineto{\pgfqpoint{2.218017in}{2.832968in}}%
\pgfpathlineto{\pgfqpoint{2.405896in}{2.402464in}}%
\pgfpathlineto{\pgfqpoint{2.593775in}{1.863558in}}%
\pgfpathlineto{\pgfqpoint{2.781654in}{1.538953in}}%
\pgfpathlineto{\pgfqpoint{2.969533in}{1.096733in}}%
\pgfpathlineto{\pgfqpoint{3.157411in}{0.817039in}}%
\pgfpathlineto{\pgfqpoint{3.345290in}{0.838331in}}%
\pgfpathlineto{\pgfqpoint{3.533169in}{0.855728in}}%
\pgfpathlineto{\pgfqpoint{3.721048in}{0.855728in}}%
\pgfpathlineto{\pgfqpoint{3.908927in}{0.838331in}}%
\pgfpathlineto{\pgfqpoint{4.096805in}{0.789588in}}%
\pgfpathlineto{\pgfqpoint{4.284684in}{0.817039in}}%
\pgfpathlineto{\pgfqpoint{4.472563in}{0.855728in}}%
\pgfpathlineto{\pgfqpoint{4.660442in}{0.838331in}}%
\pgfpathlineto{\pgfqpoint{4.848320in}{0.789588in}}%
\pgfpathlineto{\pgfqpoint{5.036199in}{0.817039in}}%
\pgfpathlineto{\pgfqpoint{5.224078in}{0.855728in}}%
\pgfpathlineto{\pgfqpoint{5.411957in}{0.838331in}}%
\pgfpathlineto{\pgfqpoint{5.599836in}{0.750899in}}%
\pgfusepath{stroke}%
\end{pgfscope}%
\begin{pgfscope}%
\pgfpathrectangle{\pgfqpoint{0.865290in}{0.582899in}}{\pgfqpoint{4.960000in}{3.696000in}}%
\pgfusepath{clip}%
\pgfsetrectcap%
\pgfsetroundjoin%
\pgfsetlinewidth{1.505625pt}%
\definecolor{currentstroke}{rgb}{0.839216,0.152941,0.156863}%
\pgfsetstrokecolor{currentstroke}%
\pgfsetdash{}{0pt}%
\pgfpathmoveto{\pgfqpoint{1.090745in}{3.978315in}}%
\pgfpathlineto{\pgfqpoint{1.278623in}{3.893627in}}%
\pgfpathlineto{\pgfqpoint{1.466502in}{3.840331in}}%
\pgfpathlineto{\pgfqpoint{1.654381in}{3.533810in}}%
\pgfpathlineto{\pgfqpoint{1.842260in}{3.089050in}}%
\pgfpathlineto{\pgfqpoint{2.030139in}{2.931037in}}%
\pgfpathlineto{\pgfqpoint{2.218017in}{2.681780in}}%
\pgfpathlineto{\pgfqpoint{2.405896in}{2.241457in}}%
\pgfpathlineto{\pgfqpoint{2.593775in}{1.909416in}}%
\pgfpathlineto{\pgfqpoint{2.781654in}{1.550188in}}%
\pgfpathlineto{\pgfqpoint{2.969533in}{1.289923in}}%
\pgfpathlineto{\pgfqpoint{3.157411in}{1.267702in}}%
\pgfpathlineto{\pgfqpoint{3.345290in}{1.268125in}}%
\pgfpathlineto{\pgfqpoint{3.533169in}{1.268125in}}%
\pgfpathlineto{\pgfqpoint{3.721048in}{1.268125in}}%
\pgfpathlineto{\pgfqpoint{3.908927in}{1.268125in}}%
\pgfpathlineto{\pgfqpoint{4.096805in}{1.268125in}}%
\pgfpathlineto{\pgfqpoint{4.284684in}{1.268125in}}%
\pgfpathlineto{\pgfqpoint{4.472563in}{1.268125in}}%
\pgfpathlineto{\pgfqpoint{4.660442in}{1.267702in}}%
\pgfpathlineto{\pgfqpoint{4.848320in}{1.267702in}}%
\pgfpathlineto{\pgfqpoint{5.036199in}{1.267702in}}%
\pgfpathlineto{\pgfqpoint{5.224078in}{1.268125in}}%
\pgfpathlineto{\pgfqpoint{5.411957in}{1.267702in}}%
\pgfpathlineto{\pgfqpoint{5.599836in}{1.267702in}}%
\pgfusepath{stroke}%
\end{pgfscope}%
\begin{pgfscope}%
\pgfsetrectcap%
\pgfsetmiterjoin%
\pgfsetlinewidth{0.803000pt}%
\definecolor{currentstroke}{rgb}{0.000000,0.000000,0.000000}%
\pgfsetstrokecolor{currentstroke}%
\pgfsetdash{}{0pt}%
\pgfpathmoveto{\pgfqpoint{0.865290in}{0.582899in}}%
\pgfpathlineto{\pgfqpoint{0.865290in}{4.278899in}}%
\pgfusepath{stroke}%
\end{pgfscope}%
\begin{pgfscope}%
\pgfsetrectcap%
\pgfsetmiterjoin%
\pgfsetlinewidth{0.803000pt}%
\definecolor{currentstroke}{rgb}{0.000000,0.000000,0.000000}%
\pgfsetstrokecolor{currentstroke}%
\pgfsetdash{}{0pt}%
\pgfpathmoveto{\pgfqpoint{5.825290in}{0.582899in}}%
\pgfpathlineto{\pgfqpoint{5.825290in}{4.278899in}}%
\pgfusepath{stroke}%
\end{pgfscope}%
\begin{pgfscope}%
\pgfsetrectcap%
\pgfsetmiterjoin%
\pgfsetlinewidth{0.803000pt}%
\definecolor{currentstroke}{rgb}{0.000000,0.000000,0.000000}%
\pgfsetstrokecolor{currentstroke}%
\pgfsetdash{}{0pt}%
\pgfpathmoveto{\pgfqpoint{0.865290in}{0.582899in}}%
\pgfpathlineto{\pgfqpoint{5.825290in}{0.582899in}}%
\pgfusepath{stroke}%
\end{pgfscope}%
\begin{pgfscope}%
\pgfsetrectcap%
\pgfsetmiterjoin%
\pgfsetlinewidth{0.803000pt}%
\definecolor{currentstroke}{rgb}{0.000000,0.000000,0.000000}%
\pgfsetstrokecolor{currentstroke}%
\pgfsetdash{}{0pt}%
\pgfpathmoveto{\pgfqpoint{0.865290in}{4.278899in}}%
\pgfpathlineto{\pgfqpoint{5.825290in}{4.278899in}}%
\pgfusepath{stroke}%
\end{pgfscope}%
\begin{pgfscope}%
\pgfsetbuttcap%
\pgfsetmiterjoin%
\definecolor{currentfill}{rgb}{1.000000,1.000000,1.000000}%
\pgfsetfillcolor{currentfill}%
\pgfsetfillopacity{0.800000}%
\pgfsetlinewidth{1.003750pt}%
\definecolor{currentstroke}{rgb}{0.800000,0.800000,0.800000}%
\pgfsetstrokecolor{currentstroke}%
\pgfsetstrokeopacity{0.800000}%
\pgfsetdash{}{0pt}%
\pgfpathmoveto{\pgfqpoint{2.158059in}{2.796434in}}%
\pgfpathlineto{\pgfqpoint{5.669735in}{2.796434in}}%
\pgfpathquadraticcurveto{\pgfqpoint{5.714179in}{2.796434in}}{\pgfqpoint{5.714179in}{2.840879in}}%
\pgfpathlineto{\pgfqpoint{5.714179in}{4.123343in}}%
\pgfpathquadraticcurveto{\pgfqpoint{5.714179in}{4.167788in}}{\pgfqpoint{5.669735in}{4.167788in}}%
\pgfpathlineto{\pgfqpoint{2.158059in}{4.167788in}}%
\pgfpathquadraticcurveto{\pgfqpoint{2.113615in}{4.167788in}}{\pgfqpoint{2.113615in}{4.123343in}}%
\pgfpathlineto{\pgfqpoint{2.113615in}{2.840879in}}%
\pgfpathquadraticcurveto{\pgfqpoint{2.113615in}{2.796434in}}{\pgfqpoint{2.158059in}{2.796434in}}%
\pgfpathlineto{\pgfqpoint{2.158059in}{2.796434in}}%
\pgfpathclose%
\pgfusepath{stroke,fill}%
\end{pgfscope}%
\begin{pgfscope}%
\pgfsetrectcap%
\pgfsetroundjoin%
\pgfsetlinewidth{1.505625pt}%
\definecolor{currentstroke}{rgb}{0.121569,0.466667,0.705882}%
\pgfsetstrokecolor{currentstroke}%
\pgfsetdash{}{0pt}%
\pgfpathmoveto{\pgfqpoint{2.202504in}{3.987840in}}%
\pgfpathlineto{\pgfqpoint{2.424726in}{3.987840in}}%
\pgfpathlineto{\pgfqpoint{2.646948in}{3.987840in}}%
\pgfusepath{stroke}%
\end{pgfscope}%
\begin{pgfscope}%
\definecolor{textcolor}{rgb}{0.000000,0.000000,0.000000}%
\pgfsetstrokecolor{textcolor}%
\pgfsetfillcolor{textcolor}%
\pgftext[x=2.824726in,y=3.910062in,left,base]{\color{textcolor}{\sffamily\fontsize{16.000000}{19.200000}\selectfont\catcode`\^=\active\def^{\ifmmode\sp\else\^{}\fi}\catcode`\%=\active\def
\end{pgfscope}%
\begin{pgfscope}%
\pgfsetrectcap%
\pgfsetroundjoin%
\pgfsetlinewidth{1.505625pt}%
\definecolor{currentstroke}{rgb}{1.000000,0.498039,0.054902}%
\pgfsetstrokecolor{currentstroke}%
\pgfsetdash{}{0pt}%
\pgfpathmoveto{\pgfqpoint{2.202504in}{3.661668in}}%
\pgfpathlineto{\pgfqpoint{2.424726in}{3.661668in}}%
\pgfpathlineto{\pgfqpoint{2.646948in}{3.661668in}}%
\pgfusepath{stroke}%
\end{pgfscope}%
\begin{pgfscope}%
\definecolor{textcolor}{rgb}{0.000000,0.000000,0.000000}%
\pgfsetstrokecolor{textcolor}%
\pgfsetfillcolor{textcolor}%
\pgftext[x=2.824726in,y=3.583890in,left,base]{\color{textcolor}{\sffamily\fontsize{16.000000}{19.200000}\selectfont\catcode`\^=\active\def^{\ifmmode\sp\else\^{}\fi}\catcode`\%=\active\def
\end{pgfscope}%
\begin{pgfscope}%
\pgfsetrectcap%
\pgfsetroundjoin%
\pgfsetlinewidth{1.505625pt}%
\definecolor{currentstroke}{rgb}{0.172549,0.627451,0.172549}%
\pgfsetstrokecolor{currentstroke}%
\pgfsetdash{}{0pt}%
\pgfpathmoveto{\pgfqpoint{2.202504in}{3.335497in}}%
\pgfpathlineto{\pgfqpoint{2.424726in}{3.335497in}}%
\pgfpathlineto{\pgfqpoint{2.646948in}{3.335497in}}%
\pgfusepath{stroke}%
\end{pgfscope}%
\begin{pgfscope}%
\definecolor{textcolor}{rgb}{0.000000,0.000000,0.000000}%
\pgfsetstrokecolor{textcolor}%
\pgfsetfillcolor{textcolor}%
\pgftext[x=2.824726in,y=3.257719in,left,base]{\color{textcolor}{\sffamily\fontsize{16.000000}{19.200000}\selectfont\catcode`\^=\active\def^{\ifmmode\sp\else\^{}\fi}\catcode`\%=\active\def
\end{pgfscope}%
\begin{pgfscope}%
\pgfsetrectcap%
\pgfsetroundjoin%
\pgfsetlinewidth{1.505625pt}%
\definecolor{currentstroke}{rgb}{0.839216,0.152941,0.156863}%
\pgfsetstrokecolor{currentstroke}%
\pgfsetdash{}{0pt}%
\pgfpathmoveto{\pgfqpoint{2.202504in}{3.009325in}}%
\pgfpathlineto{\pgfqpoint{2.424726in}{3.009325in}}%
\pgfpathlineto{\pgfqpoint{2.646948in}{3.009325in}}%
\pgfusepath{stroke}%
\end{pgfscope}%
\begin{pgfscope}%
\definecolor{textcolor}{rgb}{0.000000,0.000000,0.000000}%
\pgfsetstrokecolor{textcolor}%
\pgfsetfillcolor{textcolor}%
\pgftext[x=2.824726in,y=2.931547in,left,base]{\color{textcolor}{\sffamily\fontsize{16.000000}{19.200000}\selectfont\catcode`\^=\active\def^{\ifmmode\sp\else\^{}\fi}\catcode`\%=\active\def
\end{pgfscope}%
\end{pgfpicture}%
\makeatother%
\endgroup%

%% file: figs/network_conv_as-735.pgf
\begingroup%
\makeatletter%
\begin{pgfpicture}%
\pgfpathrectangle{\pgfpointorigin}{\pgfqpoint{5.825290in}{4.278899in}}%
\pgfusepath{use as bounding box, clip}%
\begin{pgfscope}%
\pgfsetbuttcap%
\pgfsetmiterjoin%
\definecolor{currentfill}{rgb}{1.000000,1.000000,1.000000}%
\pgfsetfillcolor{currentfill}%
\pgfsetlinewidth{0.000000pt}%
\definecolor{currentstroke}{rgb}{1.000000,1.000000,1.000000}%
\pgfsetstrokecolor{currentstroke}%
\pgfsetdash{}{0pt}%
\pgfpathmoveto{\pgfqpoint{0.000000in}{0.000000in}}%
\pgfpathlineto{\pgfqpoint{5.825290in}{0.000000in}}%
\pgfpathlineto{\pgfqpoint{5.825290in}{4.278899in}}%
\pgfpathlineto{\pgfqpoint{0.000000in}{4.278899in}}%
\pgfpathlineto{\pgfqpoint{0.000000in}{0.000000in}}%
\pgfpathclose%
\pgfusepath{fill}%
\end{pgfscope}%
\begin{pgfscope}%
\pgfsetbuttcap%
\pgfsetmiterjoin%
\definecolor{currentfill}{rgb}{1.000000,1.000000,1.000000}%
\pgfsetfillcolor{currentfill}%
\pgfsetlinewidth{0.000000pt}%
\definecolor{currentstroke}{rgb}{0.000000,0.000000,0.000000}%
\pgfsetstrokecolor{currentstroke}%
\pgfsetstrokeopacity{0.000000}%
\pgfsetdash{}{0pt}%
\pgfpathmoveto{\pgfqpoint{0.865290in}{0.582899in}}%
\pgfpathlineto{\pgfqpoint{5.825290in}{0.582899in}}%
\pgfpathlineto{\pgfqpoint{5.825290in}{4.278899in}}%
\pgfpathlineto{\pgfqpoint{0.865290in}{4.278899in}}%
\pgfpathlineto{\pgfqpoint{0.865290in}{0.582899in}}%
\pgfpathclose%
\pgfusepath{fill}%
\end{pgfscope}%
\begin{pgfscope}%
\pgfsetbuttcap%
\pgfsetroundjoin%
\definecolor{currentfill}{rgb}{0.000000,0.000000,0.000000}%
\pgfsetfillcolor{currentfill}%
\pgfsetlinewidth{0.803000pt}%
\definecolor{currentstroke}{rgb}{0.000000,0.000000,0.000000}%
\pgfsetstrokecolor{currentstroke}%
\pgfsetdash{}{0pt}%
\pgfsys@defobject{currentmarker}{\pgfqpoint{0.000000in}{-0.048611in}}{\pgfqpoint{0.000000in}{0.000000in}}{%
\pgfpathmoveto{\pgfqpoint{0.000000in}{0.000000in}}%
\pgfpathlineto{\pgfqpoint{0.000000in}{-0.048611in}}%
\pgfusepath{stroke,fill}%
}%
\begin{pgfscope}%
\pgfsys@transformshift{1.372563in}{0.582899in}%
\pgfsys@useobject{currentmarker}{}%
\end{pgfscope}%
\end{pgfscope}%
\begin{pgfscope}%
\definecolor{textcolor}{rgb}{0.000000,0.000000,0.000000}%
\pgfsetstrokecolor{textcolor}%
\pgfsetfillcolor{textcolor}%
\pgftext[x=1.372563in,y=0.485677in,,top]{\color{textcolor}{\sffamily\fontsize{16.000000}{19.200000}\selectfont\catcode`\^=\active\def^{\ifmmode\sp\else\^{}\fi}\catcode`\%=\active\def
\end{pgfscope}%
\begin{pgfscope}%
\pgfsetbuttcap%
\pgfsetroundjoin%
\definecolor{currentfill}{rgb}{0.000000,0.000000,0.000000}%
\pgfsetfillcolor{currentfill}%
\pgfsetlinewidth{0.803000pt}%
\definecolor{currentstroke}{rgb}{0.000000,0.000000,0.000000}%
\pgfsetstrokecolor{currentstroke}%
\pgfsetdash{}{0pt}%
\pgfsys@defobject{currentmarker}{\pgfqpoint{0.000000in}{-0.048611in}}{\pgfqpoint{0.000000in}{0.000000in}}{%
\pgfpathmoveto{\pgfqpoint{0.000000in}{0.000000in}}%
\pgfpathlineto{\pgfqpoint{0.000000in}{-0.048611in}}%
\pgfusepath{stroke,fill}%
}%
\begin{pgfscope}%
\pgfsys@transformshift{1.842260in}{0.582899in}%
\pgfsys@useobject{currentmarker}{}%
\end{pgfscope}%
\end{pgfscope}%
\begin{pgfscope}%
\definecolor{textcolor}{rgb}{0.000000,0.000000,0.000000}%
\pgfsetstrokecolor{textcolor}%
\pgfsetfillcolor{textcolor}%
\pgftext[x=1.842260in,y=0.485677in,,top]{\color{textcolor}{\sffamily\fontsize{16.000000}{19.200000}\selectfont\catcode`\^=\active\def^{\ifmmode\sp\else\^{}\fi}\catcode`\%=\active\def
\end{pgfscope}%
\begin{pgfscope}%
\pgfsetbuttcap%
\pgfsetroundjoin%
\definecolor{currentfill}{rgb}{0.000000,0.000000,0.000000}%
\pgfsetfillcolor{currentfill}%
\pgfsetlinewidth{0.803000pt}%
\definecolor{currentstroke}{rgb}{0.000000,0.000000,0.000000}%
\pgfsetstrokecolor{currentstroke}%
\pgfsetdash{}{0pt}%
\pgfsys@defobject{currentmarker}{\pgfqpoint{0.000000in}{-0.048611in}}{\pgfqpoint{0.000000in}{0.000000in}}{%
\pgfpathmoveto{\pgfqpoint{0.000000in}{0.000000in}}%
\pgfpathlineto{\pgfqpoint{0.000000in}{-0.048611in}}%
\pgfusepath{stroke,fill}%
}%
\begin{pgfscope}%
\pgfsys@transformshift{2.311957in}{0.582899in}%
\pgfsys@useobject{currentmarker}{}%
\end{pgfscope}%
\end{pgfscope}%
\begin{pgfscope}%
\definecolor{textcolor}{rgb}{0.000000,0.000000,0.000000}%
\pgfsetstrokecolor{textcolor}%
\pgfsetfillcolor{textcolor}%
\pgftext[x=2.311957in,y=0.485677in,,top]{\color{textcolor}{\sffamily\fontsize{16.000000}{19.200000}\selectfont\catcode`\^=\active\def^{\ifmmode\sp\else\^{}\fi}\catcode`\%=\active\def
\end{pgfscope}%
\begin{pgfscope}%
\pgfsetbuttcap%
\pgfsetroundjoin%
\definecolor{currentfill}{rgb}{0.000000,0.000000,0.000000}%
\pgfsetfillcolor{currentfill}%
\pgfsetlinewidth{0.803000pt}%
\definecolor{currentstroke}{rgb}{0.000000,0.000000,0.000000}%
\pgfsetstrokecolor{currentstroke}%
\pgfsetdash{}{0pt}%
\pgfsys@defobject{currentmarker}{\pgfqpoint{0.000000in}{-0.048611in}}{\pgfqpoint{0.000000in}{0.000000in}}{%
\pgfpathmoveto{\pgfqpoint{0.000000in}{0.000000in}}%
\pgfpathlineto{\pgfqpoint{0.000000in}{-0.048611in}}%
\pgfusepath{stroke,fill}%
}%
\begin{pgfscope}%
\pgfsys@transformshift{2.781654in}{0.582899in}%
\pgfsys@useobject{currentmarker}{}%
\end{pgfscope}%
\end{pgfscope}%
\begin{pgfscope}%
\definecolor{textcolor}{rgb}{0.000000,0.000000,0.000000}%
\pgfsetstrokecolor{textcolor}%
\pgfsetfillcolor{textcolor}%
\pgftext[x=2.781654in,y=0.485677in,,top]{\color{textcolor}{\sffamily\fontsize{16.000000}{19.200000}\selectfont\catcode`\^=\active\def^{\ifmmode\sp\else\^{}\fi}\catcode`\%=\active\def
\end{pgfscope}%
\begin{pgfscope}%
\pgfsetbuttcap%
\pgfsetroundjoin%
\definecolor{currentfill}{rgb}{0.000000,0.000000,0.000000}%
\pgfsetfillcolor{currentfill}%
\pgfsetlinewidth{0.803000pt}%
\definecolor{currentstroke}{rgb}{0.000000,0.000000,0.000000}%
\pgfsetstrokecolor{currentstroke}%
\pgfsetdash{}{0pt}%
\pgfsys@defobject{currentmarker}{\pgfqpoint{0.000000in}{-0.048611in}}{\pgfqpoint{0.000000in}{0.000000in}}{%
\pgfpathmoveto{\pgfqpoint{0.000000in}{0.000000in}}%
\pgfpathlineto{\pgfqpoint{0.000000in}{-0.048611in}}%
\pgfusepath{stroke,fill}%
}%
\begin{pgfscope}%
\pgfsys@transformshift{3.251351in}{0.582899in}%
\pgfsys@useobject{currentmarker}{}%
\end{pgfscope}%
\end{pgfscope}%
\begin{pgfscope}%
\definecolor{textcolor}{rgb}{0.000000,0.000000,0.000000}%
\pgfsetstrokecolor{textcolor}%
\pgfsetfillcolor{textcolor}%
\pgftext[x=3.251351in,y=0.485677in,,top]{\color{textcolor}{\sffamily\fontsize{16.000000}{19.200000}\selectfont\catcode`\^=\active\def^{\ifmmode\sp\else\^{}\fi}\catcode`\%=\active\def
\end{pgfscope}%
\begin{pgfscope}%
\pgfsetbuttcap%
\pgfsetroundjoin%
\definecolor{currentfill}{rgb}{0.000000,0.000000,0.000000}%
\pgfsetfillcolor{currentfill}%
\pgfsetlinewidth{0.803000pt}%
\definecolor{currentstroke}{rgb}{0.000000,0.000000,0.000000}%
\pgfsetstrokecolor{currentstroke}%
\pgfsetdash{}{0pt}%
\pgfsys@defobject{currentmarker}{\pgfqpoint{0.000000in}{-0.048611in}}{\pgfqpoint{0.000000in}{0.000000in}}{%
\pgfpathmoveto{\pgfqpoint{0.000000in}{0.000000in}}%
\pgfpathlineto{\pgfqpoint{0.000000in}{-0.048611in}}%
\pgfusepath{stroke,fill}%
}%
\begin{pgfscope}%
\pgfsys@transformshift{3.721048in}{0.582899in}%
\pgfsys@useobject{currentmarker}{}%
\end{pgfscope}%
\end{pgfscope}%
\begin{pgfscope}%
\definecolor{textcolor}{rgb}{0.000000,0.000000,0.000000}%
\pgfsetstrokecolor{textcolor}%
\pgfsetfillcolor{textcolor}%
\pgftext[x=3.721048in,y=0.485677in,,top]{\color{textcolor}{\sffamily\fontsize{16.000000}{19.200000}\selectfont\catcode`\^=\active\def^{\ifmmode\sp\else\^{}\fi}\catcode`\%=\active\def
\end{pgfscope}%
\begin{pgfscope}%
\pgfsetbuttcap%
\pgfsetroundjoin%
\definecolor{currentfill}{rgb}{0.000000,0.000000,0.000000}%
\pgfsetfillcolor{currentfill}%
\pgfsetlinewidth{0.803000pt}%
\definecolor{currentstroke}{rgb}{0.000000,0.000000,0.000000}%
\pgfsetstrokecolor{currentstroke}%
\pgfsetdash{}{0pt}%
\pgfsys@defobject{currentmarker}{\pgfqpoint{0.000000in}{-0.048611in}}{\pgfqpoint{0.000000in}{0.000000in}}{%
\pgfpathmoveto{\pgfqpoint{0.000000in}{0.000000in}}%
\pgfpathlineto{\pgfqpoint{0.000000in}{-0.048611in}}%
\pgfusepath{stroke,fill}%
}%
\begin{pgfscope}%
\pgfsys@transformshift{4.190745in}{0.582899in}%
\pgfsys@useobject{currentmarker}{}%
\end{pgfscope}%
\end{pgfscope}%
\begin{pgfscope}%
\definecolor{textcolor}{rgb}{0.000000,0.000000,0.000000}%
\pgfsetstrokecolor{textcolor}%
\pgfsetfillcolor{textcolor}%
\pgftext[x=4.190745in,y=0.485677in,,top]{\color{textcolor}{\sffamily\fontsize{16.000000}{19.200000}\selectfont\catcode`\^=\active\def^{\ifmmode\sp\else\^{}\fi}\catcode`\%=\active\def
\end{pgfscope}%
\begin{pgfscope}%
\pgfsetbuttcap%
\pgfsetroundjoin%
\definecolor{currentfill}{rgb}{0.000000,0.000000,0.000000}%
\pgfsetfillcolor{currentfill}%
\pgfsetlinewidth{0.803000pt}%
\definecolor{currentstroke}{rgb}{0.000000,0.000000,0.000000}%
\pgfsetstrokecolor{currentstroke}%
\pgfsetdash{}{0pt}%
\pgfsys@defobject{currentmarker}{\pgfqpoint{0.000000in}{-0.048611in}}{\pgfqpoint{0.000000in}{0.000000in}}{%
\pgfpathmoveto{\pgfqpoint{0.000000in}{0.000000in}}%
\pgfpathlineto{\pgfqpoint{0.000000in}{-0.048611in}}%
\pgfusepath{stroke,fill}%
}%
\begin{pgfscope}%
\pgfsys@transformshift{4.660442in}{0.582899in}%
\pgfsys@useobject{currentmarker}{}%
\end{pgfscope}%
\end{pgfscope}%
\begin{pgfscope}%
\definecolor{textcolor}{rgb}{0.000000,0.000000,0.000000}%
\pgfsetstrokecolor{textcolor}%
\pgfsetfillcolor{textcolor}%
\pgftext[x=4.660442in,y=0.485677in,,top]{\color{textcolor}{\sffamily\fontsize{16.000000}{19.200000}\selectfont\catcode`\^=\active\def^{\ifmmode\sp\else\^{}\fi}\catcode`\%=\active\def
\end{pgfscope}%
\begin{pgfscope}%
\pgfsetbuttcap%
\pgfsetroundjoin%
\definecolor{currentfill}{rgb}{0.000000,0.000000,0.000000}%
\pgfsetfillcolor{currentfill}%
\pgfsetlinewidth{0.803000pt}%
\definecolor{currentstroke}{rgb}{0.000000,0.000000,0.000000}%
\pgfsetstrokecolor{currentstroke}%
\pgfsetdash{}{0pt}%
\pgfsys@defobject{currentmarker}{\pgfqpoint{0.000000in}{-0.048611in}}{\pgfqpoint{0.000000in}{0.000000in}}{%
\pgfpathmoveto{\pgfqpoint{0.000000in}{0.000000in}}%
\pgfpathlineto{\pgfqpoint{0.000000in}{-0.048611in}}%
\pgfusepath{stroke,fill}%
}%
\begin{pgfscope}%
\pgfsys@transformshift{5.130139in}{0.582899in}%
\pgfsys@useobject{currentmarker}{}%
\end{pgfscope}%
\end{pgfscope}%
\begin{pgfscope}%
\definecolor{textcolor}{rgb}{0.000000,0.000000,0.000000}%
\pgfsetstrokecolor{textcolor}%
\pgfsetfillcolor{textcolor}%
\pgftext[x=5.130139in,y=0.485677in,,top]{\color{textcolor}{\sffamily\fontsize{16.000000}{19.200000}\selectfont\catcode`\^=\active\def^{\ifmmode\sp\else\^{}\fi}\catcode`\%=\active\def
\end{pgfscope}%
\begin{pgfscope}%
\pgfsetbuttcap%
\pgfsetroundjoin%
\definecolor{currentfill}{rgb}{0.000000,0.000000,0.000000}%
\pgfsetfillcolor{currentfill}%
\pgfsetlinewidth{0.803000pt}%
\definecolor{currentstroke}{rgb}{0.000000,0.000000,0.000000}%
\pgfsetstrokecolor{currentstroke}%
\pgfsetdash{}{0pt}%
\pgfsys@defobject{currentmarker}{\pgfqpoint{0.000000in}{-0.048611in}}{\pgfqpoint{0.000000in}{0.000000in}}{%
\pgfpathmoveto{\pgfqpoint{0.000000in}{0.000000in}}%
\pgfpathlineto{\pgfqpoint{0.000000in}{-0.048611in}}%
\pgfusepath{stroke,fill}%
}%
\begin{pgfscope}%
\pgfsys@transformshift{5.599836in}{0.582899in}%
\pgfsys@useobject{currentmarker}{}%
\end{pgfscope}%
\end{pgfscope}%
\begin{pgfscope}%
\definecolor{textcolor}{rgb}{0.000000,0.000000,0.000000}%
\pgfsetstrokecolor{textcolor}%
\pgfsetfillcolor{textcolor}%
\pgftext[x=5.599836in,y=0.485677in,,top]{\color{textcolor}{\sffamily\fontsize{16.000000}{19.200000}\selectfont\catcode`\^=\active\def^{\ifmmode\sp\else\^{}\fi}\catcode`\%=\active\def
\end{pgfscope}%
\begin{pgfscope}%
\definecolor{textcolor}{rgb}{0.000000,0.000000,0.000000}%
\pgfsetstrokecolor{textcolor}%
\pgfsetfillcolor{textcolor}%
\pgftext[x=3.345290in,y=0.215061in,,top]{\color{textcolor}{\sffamily\fontsize{16.000000}{19.200000}\selectfont\catcode`\^=\active\def^{\ifmmode\sp\else\^{}\fi}\catcode`\%=\active\def
\end{pgfscope}%
\begin{pgfscope}%
\pgfsetbuttcap%
\pgfsetroundjoin%
\definecolor{currentfill}{rgb}{0.000000,0.000000,0.000000}%
\pgfsetfillcolor{currentfill}%
\pgfsetlinewidth{0.803000pt}%
\definecolor{currentstroke}{rgb}{0.000000,0.000000,0.000000}%
\pgfsetstrokecolor{currentstroke}%
\pgfsetdash{}{0pt}%
\pgfsys@defobject{currentmarker}{\pgfqpoint{-0.048611in}{0.000000in}}{\pgfqpoint{-0.000000in}{0.000000in}}{%
\pgfpathmoveto{\pgfqpoint{-0.000000in}{0.000000in}}%
\pgfpathlineto{\pgfqpoint{-0.048611in}{0.000000in}}%
\pgfusepath{stroke,fill}%
}%
\begin{pgfscope}%
\pgfsys@transformshift{0.865290in}{0.903469in}%
\pgfsys@useobject{currentmarker}{}%
\end{pgfscope}%
\end{pgfscope}%
\begin{pgfscope}%
\definecolor{textcolor}{rgb}{0.000000,0.000000,0.000000}%
\pgfsetstrokecolor{textcolor}%
\pgfsetfillcolor{textcolor}%
\pgftext[x=0.270616in, y=0.819051in, left, base]{\color{textcolor}{\sffamily\fontsize{16.000000}{19.200000}\selectfont\catcode`\^=\active\def^{\ifmmode\sp\else\^{}\fi}\catcode`\%=\active\def
\end{pgfscope}%
\begin{pgfscope}%
\pgfsetbuttcap%
\pgfsetroundjoin%
\definecolor{currentfill}{rgb}{0.000000,0.000000,0.000000}%
\pgfsetfillcolor{currentfill}%
\pgfsetlinewidth{0.803000pt}%
\definecolor{currentstroke}{rgb}{0.000000,0.000000,0.000000}%
\pgfsetstrokecolor{currentstroke}%
\pgfsetdash{}{0pt}%
\pgfsys@defobject{currentmarker}{\pgfqpoint{-0.048611in}{0.000000in}}{\pgfqpoint{-0.000000in}{0.000000in}}{%
\pgfpathmoveto{\pgfqpoint{-0.000000in}{0.000000in}}%
\pgfpathlineto{\pgfqpoint{-0.048611in}{0.000000in}}%
\pgfusepath{stroke,fill}%
}%
\begin{pgfscope}%
\pgfsys@transformshift{0.865290in}{1.396920in}%
\pgfsys@useobject{currentmarker}{}%
\end{pgfscope}%
\end{pgfscope}%
\begin{pgfscope}%
\definecolor{textcolor}{rgb}{0.000000,0.000000,0.000000}%
\pgfsetstrokecolor{textcolor}%
\pgfsetfillcolor{textcolor}%
\pgftext[x=0.270616in, y=1.312501in, left, base]{\color{textcolor}{\sffamily\fontsize{16.000000}{19.200000}\selectfont\catcode`\^=\active\def^{\ifmmode\sp\else\^{}\fi}\catcode`\%=\active\def
\end{pgfscope}%
\begin{pgfscope}%
\pgfsetbuttcap%
\pgfsetroundjoin%
\definecolor{currentfill}{rgb}{0.000000,0.000000,0.000000}%
\pgfsetfillcolor{currentfill}%
\pgfsetlinewidth{0.803000pt}%
\definecolor{currentstroke}{rgb}{0.000000,0.000000,0.000000}%
\pgfsetstrokecolor{currentstroke}%
\pgfsetdash{}{0pt}%
\pgfsys@defobject{currentmarker}{\pgfqpoint{-0.048611in}{0.000000in}}{\pgfqpoint{-0.000000in}{0.000000in}}{%
\pgfpathmoveto{\pgfqpoint{-0.000000in}{0.000000in}}%
\pgfpathlineto{\pgfqpoint{-0.048611in}{0.000000in}}%
\pgfusepath{stroke,fill}%
}%
\begin{pgfscope}%
\pgfsys@transformshift{0.865290in}{1.890371in}%
\pgfsys@useobject{currentmarker}{}%
\end{pgfscope}%
\end{pgfscope}%
\begin{pgfscope}%
\definecolor{textcolor}{rgb}{0.000000,0.000000,0.000000}%
\pgfsetstrokecolor{textcolor}%
\pgfsetfillcolor{textcolor}%
\pgftext[x=0.346658in, y=1.805952in, left, base]{\color{textcolor}{\sffamily\fontsize{16.000000}{19.200000}\selectfont\catcode`\^=\active\def^{\ifmmode\sp\else\^{}\fi}\catcode`\%=\active\def
\end{pgfscope}%
\begin{pgfscope}%
\pgfsetbuttcap%
\pgfsetroundjoin%
\definecolor{currentfill}{rgb}{0.000000,0.000000,0.000000}%
\pgfsetfillcolor{currentfill}%
\pgfsetlinewidth{0.803000pt}%
\definecolor{currentstroke}{rgb}{0.000000,0.000000,0.000000}%
\pgfsetstrokecolor{currentstroke}%
\pgfsetdash{}{0pt}%
\pgfsys@defobject{currentmarker}{\pgfqpoint{-0.048611in}{0.000000in}}{\pgfqpoint{-0.000000in}{0.000000in}}{%
\pgfpathmoveto{\pgfqpoint{-0.000000in}{0.000000in}}%
\pgfpathlineto{\pgfqpoint{-0.048611in}{0.000000in}}%
\pgfusepath{stroke,fill}%
}%
\begin{pgfscope}%
\pgfsys@transformshift{0.865290in}{2.383821in}%
\pgfsys@useobject{currentmarker}{}%
\end{pgfscope}%
\end{pgfscope}%
\begin{pgfscope}%
\definecolor{textcolor}{rgb}{0.000000,0.000000,0.000000}%
\pgfsetstrokecolor{textcolor}%
\pgfsetfillcolor{textcolor}%
\pgftext[x=0.346658in, y=2.299403in, left, base]{\color{textcolor}{\sffamily\fontsize{16.000000}{19.200000}\selectfont\catcode`\^=\active\def^{\ifmmode\sp\else\^{}\fi}\catcode`\%=\active\def
\end{pgfscope}%
\begin{pgfscope}%
\pgfsetbuttcap%
\pgfsetroundjoin%
\definecolor{currentfill}{rgb}{0.000000,0.000000,0.000000}%
\pgfsetfillcolor{currentfill}%
\pgfsetlinewidth{0.803000pt}%
\definecolor{currentstroke}{rgb}{0.000000,0.000000,0.000000}%
\pgfsetstrokecolor{currentstroke}%
\pgfsetdash{}{0pt}%
\pgfsys@defobject{currentmarker}{\pgfqpoint{-0.048611in}{0.000000in}}{\pgfqpoint{-0.000000in}{0.000000in}}{%
\pgfpathmoveto{\pgfqpoint{-0.000000in}{0.000000in}}%
\pgfpathlineto{\pgfqpoint{-0.048611in}{0.000000in}}%
\pgfusepath{stroke,fill}%
}%
\begin{pgfscope}%
\pgfsys@transformshift{0.865290in}{2.877272in}%
\pgfsys@useobject{currentmarker}{}%
\end{pgfscope}%
\end{pgfscope}%
\begin{pgfscope}%
\definecolor{textcolor}{rgb}{0.000000,0.000000,0.000000}%
\pgfsetstrokecolor{textcolor}%
\pgfsetfillcolor{textcolor}%
\pgftext[x=0.346658in, y=2.792854in, left, base]{\color{textcolor}{\sffamily\fontsize{16.000000}{19.200000}\selectfont\catcode`\^=\active\def^{\ifmmode\sp\else\^{}\fi}\catcode`\%=\active\def
\end{pgfscope}%
\begin{pgfscope}%
\pgfsetbuttcap%
\pgfsetroundjoin%
\definecolor{currentfill}{rgb}{0.000000,0.000000,0.000000}%
\pgfsetfillcolor{currentfill}%
\pgfsetlinewidth{0.803000pt}%
\definecolor{currentstroke}{rgb}{0.000000,0.000000,0.000000}%
\pgfsetstrokecolor{currentstroke}%
\pgfsetdash{}{0pt}%
\pgfsys@defobject{currentmarker}{\pgfqpoint{-0.048611in}{0.000000in}}{\pgfqpoint{-0.000000in}{0.000000in}}{%
\pgfpathmoveto{\pgfqpoint{-0.000000in}{0.000000in}}%
\pgfpathlineto{\pgfqpoint{-0.048611in}{0.000000in}}%
\pgfusepath{stroke,fill}%
}%
\begin{pgfscope}%
\pgfsys@transformshift{0.865290in}{3.370723in}%
\pgfsys@useobject{currentmarker}{}%
\end{pgfscope}%
\end{pgfscope}%
\begin{pgfscope}%
\definecolor{textcolor}{rgb}{0.000000,0.000000,0.000000}%
\pgfsetstrokecolor{textcolor}%
\pgfsetfillcolor{textcolor}%
\pgftext[x=0.346658in, y=3.286304in, left, base]{\color{textcolor}{\sffamily\fontsize{16.000000}{19.200000}\selectfont\catcode`\^=\active\def^{\ifmmode\sp\else\^{}\fi}\catcode`\%=\active\def
\end{pgfscope}%
\begin{pgfscope}%
\pgfsetbuttcap%
\pgfsetroundjoin%
\definecolor{currentfill}{rgb}{0.000000,0.000000,0.000000}%
\pgfsetfillcolor{currentfill}%
\pgfsetlinewidth{0.803000pt}%
\definecolor{currentstroke}{rgb}{0.000000,0.000000,0.000000}%
\pgfsetstrokecolor{currentstroke}%
\pgfsetdash{}{0pt}%
\pgfsys@defobject{currentmarker}{\pgfqpoint{-0.048611in}{0.000000in}}{\pgfqpoint{-0.000000in}{0.000000in}}{%
\pgfpathmoveto{\pgfqpoint{-0.000000in}{0.000000in}}%
\pgfpathlineto{\pgfqpoint{-0.048611in}{0.000000in}}%
\pgfusepath{stroke,fill}%
}%
\begin{pgfscope}%
\pgfsys@transformshift{0.865290in}{3.864174in}%
\pgfsys@useobject{currentmarker}{}%
\end{pgfscope}%
\end{pgfscope}%
\begin{pgfscope}%
\definecolor{textcolor}{rgb}{0.000000,0.000000,0.000000}%
\pgfsetstrokecolor{textcolor}%
\pgfsetfillcolor{textcolor}%
\pgftext[x=0.346658in, y=3.779755in, left, base]{\color{textcolor}{\sffamily\fontsize{16.000000}{19.200000}\selectfont\catcode`\^=\active\def^{\ifmmode\sp\else\^{}\fi}\catcode`\%=\active\def
\end{pgfscope}%
\begin{pgfscope}%
\definecolor{textcolor}{rgb}{0.000000,0.000000,0.000000}%
\pgfsetstrokecolor{textcolor}%
\pgfsetfillcolor{textcolor}%
\pgftext[x=0.215061in,y=2.430899in,,bottom,rotate=90.000000]{\color{textcolor}{\sffamily\fontsize{16.000000}{19.200000}\selectfont\catcode`\^=\active\def^{\ifmmode\sp\else\^{}\fi}\catcode`\%=\active\def
\end{pgfscope}%
\begin{pgfscope}%
\pgfpathrectangle{\pgfqpoint{0.865290in}{0.582899in}}{\pgfqpoint{4.960000in}{3.696000in}}%
\pgfusepath{clip}%
\pgfsetrectcap%
\pgfsetroundjoin%
\pgfsetlinewidth{1.505625pt}%
\definecolor{currentstroke}{rgb}{0.121569,0.466667,0.705882}%
\pgfsetstrokecolor{currentstroke}%
\pgfsetdash{}{0pt}%
\pgfpathmoveto{\pgfqpoint{1.090745in}{4.110899in}}%
\pgfpathlineto{\pgfqpoint{1.278623in}{4.074469in}}%
\pgfpathlineto{\pgfqpoint{1.466502in}{4.063362in}}%
\pgfpathlineto{\pgfqpoint{1.654381in}{4.056647in}}%
\pgfpathlineto{\pgfqpoint{1.842260in}{4.049569in}}%
\pgfpathlineto{\pgfqpoint{2.030139in}{4.042332in}}%
\pgfpathlineto{\pgfqpoint{2.218017in}{4.033841in}}%
\pgfpathlineto{\pgfqpoint{2.405896in}{3.803239in}}%
\pgfpathlineto{\pgfqpoint{2.593775in}{3.416769in}}%
\pgfpathlineto{\pgfqpoint{2.781654in}{2.794285in}}%
\pgfpathlineto{\pgfqpoint{2.969533in}{2.354700in}}%
\pgfpathlineto{\pgfqpoint{3.157411in}{2.185906in}}%
\pgfpathlineto{\pgfqpoint{3.345290in}{2.038506in}}%
\pgfpathlineto{\pgfqpoint{3.533169in}{1.778580in}}%
\pgfpathlineto{\pgfqpoint{3.721048in}{1.171448in}}%
\pgfpathlineto{\pgfqpoint{3.908927in}{1.133371in}}%
\pgfpathlineto{\pgfqpoint{4.096805in}{1.135103in}}%
\pgfpathlineto{\pgfqpoint{4.284684in}{1.135183in}}%
\pgfpathlineto{\pgfqpoint{4.472563in}{1.135183in}}%
\pgfpathlineto{\pgfqpoint{4.660442in}{1.135199in}}%
\pgfpathlineto{\pgfqpoint{4.848320in}{1.135199in}}%
\pgfpathlineto{\pgfqpoint{5.036199in}{1.135199in}}%
\pgfpathlineto{\pgfqpoint{5.224078in}{1.135199in}}%
\pgfpathlineto{\pgfqpoint{5.411957in}{1.135199in}}%
\pgfpathlineto{\pgfqpoint{5.599836in}{1.135199in}}%
\pgfusepath{stroke}%
\end{pgfscope}%
\begin{pgfscope}%
\pgfpathrectangle{\pgfqpoint{0.865290in}{0.582899in}}{\pgfqpoint{4.960000in}{3.696000in}}%
\pgfusepath{clip}%
\pgfsetrectcap%
\pgfsetroundjoin%
\pgfsetlinewidth{1.505625pt}%
\definecolor{currentstroke}{rgb}{1.000000,0.498039,0.054902}%
\pgfsetstrokecolor{currentstroke}%
\pgfsetdash{}{0pt}%
\pgfpathmoveto{\pgfqpoint{1.090745in}{4.110899in}}%
\pgfpathlineto{\pgfqpoint{1.278623in}{4.110630in}}%
\pgfpathlineto{\pgfqpoint{1.466502in}{3.896423in}}%
\pgfpathlineto{\pgfqpoint{1.654381in}{3.184367in}}%
\pgfpathlineto{\pgfqpoint{1.842260in}{3.062892in}}%
\pgfpathlineto{\pgfqpoint{2.030139in}{2.776118in}}%
\pgfpathlineto{\pgfqpoint{2.218017in}{2.520724in}}%
\pgfpathlineto{\pgfqpoint{2.405896in}{2.104539in}}%
\pgfpathlineto{\pgfqpoint{2.593775in}{1.866315in}}%
\pgfpathlineto{\pgfqpoint{2.781654in}{1.481194in}}%
\pgfpathlineto{\pgfqpoint{2.969533in}{0.837210in}}%
\pgfpathlineto{\pgfqpoint{3.157411in}{0.750899in}}%
\pgfpathlineto{\pgfqpoint{3.345290in}{0.819522in}}%
\pgfpathlineto{\pgfqpoint{3.533169in}{0.824298in}}%
\pgfpathlineto{\pgfqpoint{3.721048in}{0.820434in}}%
\pgfpathlineto{\pgfqpoint{3.908927in}{0.820434in}}%
\pgfpathlineto{\pgfqpoint{4.096805in}{0.820434in}}%
\pgfpathlineto{\pgfqpoint{4.284684in}{0.820434in}}%
\pgfpathlineto{\pgfqpoint{4.472563in}{0.820434in}}%
\pgfpathlineto{\pgfqpoint{4.660442in}{0.820434in}}%
\pgfpathlineto{\pgfqpoint{4.848320in}{0.820434in}}%
\pgfpathlineto{\pgfqpoint{5.036199in}{0.820131in}}%
\pgfpathlineto{\pgfqpoint{5.224078in}{0.820736in}}%
\pgfpathlineto{\pgfqpoint{5.411957in}{0.824006in}}%
\pgfpathlineto{\pgfqpoint{5.599836in}{0.820434in}}%
\pgfusepath{stroke}%
\end{pgfscope}%
\begin{pgfscope}%
\pgfpathrectangle{\pgfqpoint{0.865290in}{0.582899in}}{\pgfqpoint{4.960000in}{3.696000in}}%
\pgfusepath{clip}%
\pgfsetrectcap%
\pgfsetroundjoin%
\pgfsetlinewidth{1.505625pt}%
\definecolor{currentstroke}{rgb}{0.172549,0.627451,0.172549}%
\pgfsetstrokecolor{currentstroke}%
\pgfsetdash{}{0pt}%
\pgfpathmoveto{\pgfqpoint{1.090745in}{4.110899in}}%
\pgfpathlineto{\pgfqpoint{1.278623in}{3.761879in}}%
\pgfpathlineto{\pgfqpoint{1.466502in}{3.681906in}}%
\pgfpathlineto{\pgfqpoint{1.654381in}{3.306856in}}%
\pgfpathlineto{\pgfqpoint{1.842260in}{3.169404in}}%
\pgfpathlineto{\pgfqpoint{2.030139in}{2.803469in}}%
\pgfpathlineto{\pgfqpoint{2.218017in}{2.422654in}}%
\pgfpathlineto{\pgfqpoint{2.405896in}{2.206183in}}%
\pgfpathlineto{\pgfqpoint{2.593775in}{1.806025in}}%
\pgfpathlineto{\pgfqpoint{2.781654in}{1.625463in}}%
\pgfpathlineto{\pgfqpoint{2.969533in}{1.088769in}}%
\pgfpathlineto{\pgfqpoint{3.157411in}{0.906041in}}%
\pgfpathlineto{\pgfqpoint{3.345290in}{0.874623in}}%
\pgfpathlineto{\pgfqpoint{3.533169in}{0.876433in}}%
\pgfpathlineto{\pgfqpoint{3.721048in}{0.876074in}}%
\pgfpathlineto{\pgfqpoint{3.908927in}{0.876433in}}%
\pgfpathlineto{\pgfqpoint{4.096805in}{0.876433in}}%
\pgfpathlineto{\pgfqpoint{4.284684in}{0.876433in}}%
\pgfpathlineto{\pgfqpoint{4.472563in}{0.875713in}}%
\pgfpathlineto{\pgfqpoint{4.660442in}{0.876433in}}%
\pgfpathlineto{\pgfqpoint{4.848320in}{0.876433in}}%
\pgfpathlineto{\pgfqpoint{5.036199in}{0.876433in}}%
\pgfpathlineto{\pgfqpoint{5.224078in}{0.875713in}}%
\pgfpathlineto{\pgfqpoint{5.411957in}{0.876433in}}%
\pgfpathlineto{\pgfqpoint{5.599836in}{0.876433in}}%
\pgfusepath{stroke}%
\end{pgfscope}%
\begin{pgfscope}%
\pgfpathrectangle{\pgfqpoint{0.865290in}{0.582899in}}{\pgfqpoint{4.960000in}{3.696000in}}%
\pgfusepath{clip}%
\pgfsetrectcap%
\pgfsetroundjoin%
\pgfsetlinewidth{1.505625pt}%
\definecolor{currentstroke}{rgb}{0.839216,0.152941,0.156863}%
\pgfsetstrokecolor{currentstroke}%
\pgfsetdash{}{0pt}%
\pgfpathmoveto{\pgfqpoint{1.090745in}{4.110899in}}%
\pgfpathlineto{\pgfqpoint{1.278623in}{4.056669in}}%
\pgfpathlineto{\pgfqpoint{1.466502in}{3.747835in}}%
\pgfpathlineto{\pgfqpoint{1.654381in}{3.384803in}}%
\pgfpathlineto{\pgfqpoint{1.842260in}{3.116247in}}%
\pgfpathlineto{\pgfqpoint{2.030139in}{2.820756in}}%
\pgfpathlineto{\pgfqpoint{2.218017in}{2.532230in}}%
\pgfpathlineto{\pgfqpoint{2.405896in}{1.856504in}}%
\pgfpathlineto{\pgfqpoint{2.593775in}{1.809565in}}%
\pgfpathlineto{\pgfqpoint{2.781654in}{1.213400in}}%
\pgfpathlineto{\pgfqpoint{2.969533in}{0.784441in}}%
\pgfpathlineto{\pgfqpoint{3.157411in}{1.124390in}}%
\pgfpathlineto{\pgfqpoint{3.345290in}{1.115361in}}%
\pgfpathlineto{\pgfqpoint{3.533169in}{1.115323in}}%
\pgfpathlineto{\pgfqpoint{3.721048in}{1.115284in}}%
\pgfpathlineto{\pgfqpoint{3.908927in}{1.115342in}}%
\pgfpathlineto{\pgfqpoint{4.096805in}{1.115342in}}%
\pgfpathlineto{\pgfqpoint{4.284684in}{1.115323in}}%
\pgfpathlineto{\pgfqpoint{4.472563in}{1.115284in}}%
\pgfpathlineto{\pgfqpoint{4.660442in}{1.115342in}}%
\pgfpathlineto{\pgfqpoint{4.848320in}{1.115342in}}%
\pgfpathlineto{\pgfqpoint{5.036199in}{1.115342in}}%
\pgfpathlineto{\pgfqpoint{5.224078in}{1.115284in}}%
\pgfpathlineto{\pgfqpoint{5.411957in}{1.115342in}}%
\pgfpathlineto{\pgfqpoint{5.599836in}{1.115342in}}%
\pgfusepath{stroke}%
\end{pgfscope}%
\begin{pgfscope}%
\pgfsetrectcap%
\pgfsetmiterjoin%
\pgfsetlinewidth{0.803000pt}%
\definecolor{currentstroke}{rgb}{0.000000,0.000000,0.000000}%
\pgfsetstrokecolor{currentstroke}%
\pgfsetdash{}{0pt}%
\pgfpathmoveto{\pgfqpoint{0.865290in}{0.582899in}}%
\pgfpathlineto{\pgfqpoint{0.865290in}{4.278899in}}%
\pgfusepath{stroke}%
\end{pgfscope}%
\begin{pgfscope}%
\pgfsetrectcap%
\pgfsetmiterjoin%
\pgfsetlinewidth{0.803000pt}%
\definecolor{currentstroke}{rgb}{0.000000,0.000000,0.000000}%
\pgfsetstrokecolor{currentstroke}%
\pgfsetdash{}{0pt}%
\pgfpathmoveto{\pgfqpoint{5.825290in}{0.582899in}}%
\pgfpathlineto{\pgfqpoint{5.825290in}{4.278899in}}%
\pgfusepath{stroke}%
\end{pgfscope}%
\begin{pgfscope}%
\pgfsetrectcap%
\pgfsetmiterjoin%
\pgfsetlinewidth{0.803000pt}%
\definecolor{currentstroke}{rgb}{0.000000,0.000000,0.000000}%
\pgfsetstrokecolor{currentstroke}%
\pgfsetdash{}{0pt}%
\pgfpathmoveto{\pgfqpoint{0.865290in}{0.582899in}}%
\pgfpathlineto{\pgfqpoint{5.825290in}{0.582899in}}%
\pgfusepath{stroke}%
\end{pgfscope}%
\begin{pgfscope}%
\pgfsetrectcap%
\pgfsetmiterjoin%
\pgfsetlinewidth{0.803000pt}%
\definecolor{currentstroke}{rgb}{0.000000,0.000000,0.000000}%
\pgfsetstrokecolor{currentstroke}%
\pgfsetdash{}{0pt}%
\pgfpathmoveto{\pgfqpoint{0.865290in}{4.278899in}}%
\pgfpathlineto{\pgfqpoint{5.825290in}{4.278899in}}%
\pgfusepath{stroke}%
\end{pgfscope}%
\begin{pgfscope}%
\pgfsetbuttcap%
\pgfsetmiterjoin%
\definecolor{currentfill}{rgb}{1.000000,1.000000,1.000000}%
\pgfsetfillcolor{currentfill}%
\pgfsetfillopacity{0.800000}%
\pgfsetlinewidth{1.003750pt}%
\definecolor{currentstroke}{rgb}{0.800000,0.800000,0.800000}%
\pgfsetstrokecolor{currentstroke}%
\pgfsetstrokeopacity{0.800000}%
\pgfsetdash{}{0pt}%
\pgfpathmoveto{\pgfqpoint{2.158059in}{2.796434in}}%
\pgfpathlineto{\pgfqpoint{5.669735in}{2.796434in}}%
\pgfpathquadraticcurveto{\pgfqpoint{5.714179in}{2.796434in}}{\pgfqpoint{5.714179in}{2.840879in}}%
\pgfpathlineto{\pgfqpoint{5.714179in}{4.123343in}}%
\pgfpathquadraticcurveto{\pgfqpoint{5.714179in}{4.167788in}}{\pgfqpoint{5.669735in}{4.167788in}}%
\pgfpathlineto{\pgfqpoint{2.158059in}{4.167788in}}%
\pgfpathquadraticcurveto{\pgfqpoint{2.113615in}{4.167788in}}{\pgfqpoint{2.113615in}{4.123343in}}%
\pgfpathlineto{\pgfqpoint{2.113615in}{2.840879in}}%
\pgfpathquadraticcurveto{\pgfqpoint{2.113615in}{2.796434in}}{\pgfqpoint{2.158059in}{2.796434in}}%
\pgfpathlineto{\pgfqpoint{2.158059in}{2.796434in}}%
\pgfpathclose%
\pgfusepath{stroke,fill}%
\end{pgfscope}%
\begin{pgfscope}%
\pgfsetrectcap%
\pgfsetroundjoin%
\pgfsetlinewidth{1.505625pt}%
\definecolor{currentstroke}{rgb}{0.121569,0.466667,0.705882}%
\pgfsetstrokecolor{currentstroke}%
\pgfsetdash{}{0pt}%
\pgfpathmoveto{\pgfqpoint{2.202504in}{3.987840in}}%
\pgfpathlineto{\pgfqpoint{2.424726in}{3.987840in}}%
\pgfpathlineto{\pgfqpoint{2.646948in}{3.987840in}}%
\pgfusepath{stroke}%
\end{pgfscope}%
\begin{pgfscope}%
\definecolor{textcolor}{rgb}{0.000000,0.000000,0.000000}%
\pgfsetstrokecolor{textcolor}%
\pgfsetfillcolor{textcolor}%
\pgftext[x=2.824726in,y=3.910062in,left,base]{\color{textcolor}{\sffamily\fontsize{16.000000}{19.200000}\selectfont\catcode`\^=\active\def^{\ifmmode\sp\else\^{}\fi}\catcode`\%=\active\def
\end{pgfscope}%
\begin{pgfscope}%
\pgfsetrectcap%
\pgfsetroundjoin%
\pgfsetlinewidth{1.505625pt}%
\definecolor{currentstroke}{rgb}{1.000000,0.498039,0.054902}%
\pgfsetstrokecolor{currentstroke}%
\pgfsetdash{}{0pt}%
\pgfpathmoveto{\pgfqpoint{2.202504in}{3.661668in}}%
\pgfpathlineto{\pgfqpoint{2.424726in}{3.661668in}}%
\pgfpathlineto{\pgfqpoint{2.646948in}{3.661668in}}%
\pgfusepath{stroke}%
\end{pgfscope}%
\begin{pgfscope}%
\definecolor{textcolor}{rgb}{0.000000,0.000000,0.000000}%
\pgfsetstrokecolor{textcolor}%
\pgfsetfillcolor{textcolor}%
\pgftext[x=2.824726in,y=3.583890in,left,base]{\color{textcolor}{\sffamily\fontsize{16.000000}{19.200000}\selectfont\catcode`\^=\active\def^{\ifmmode\sp\else\^{}\fi}\catcode`\%=\active\def
\end{pgfscope}%
\begin{pgfscope}%
\pgfsetrectcap%
\pgfsetroundjoin%
\pgfsetlinewidth{1.505625pt}%
\definecolor{currentstroke}{rgb}{0.172549,0.627451,0.172549}%
\pgfsetstrokecolor{currentstroke}%
\pgfsetdash{}{0pt}%
\pgfpathmoveto{\pgfqpoint{2.202504in}{3.335497in}}%
\pgfpathlineto{\pgfqpoint{2.424726in}{3.335497in}}%
\pgfpathlineto{\pgfqpoint{2.646948in}{3.335497in}}%
\pgfusepath{stroke}%
\end{pgfscope}%
\begin{pgfscope}%
\definecolor{textcolor}{rgb}{0.000000,0.000000,0.000000}%
\pgfsetstrokecolor{textcolor}%
\pgfsetfillcolor{textcolor}%
\pgftext[x=2.824726in,y=3.257719in,left,base]{\color{textcolor}{\sffamily\fontsize{16.000000}{19.200000}\selectfont\catcode`\^=\active\def^{\ifmmode\sp\else\^{}\fi}\catcode`\%=\active\def
\end{pgfscope}%
\begin{pgfscope}%
\pgfsetrectcap%
\pgfsetroundjoin%
\pgfsetlinewidth{1.505625pt}%
\definecolor{currentstroke}{rgb}{0.839216,0.152941,0.156863}%
\pgfsetstrokecolor{currentstroke}%
\pgfsetdash{}{0pt}%
\pgfpathmoveto{\pgfqpoint{2.202504in}{3.009325in}}%
\pgfpathlineto{\pgfqpoint{2.424726in}{3.009325in}}%
\pgfpathlineto{\pgfqpoint{2.646948in}{3.009325in}}%
\pgfusepath{stroke}%
\end{pgfscope}%
\begin{pgfscope}%
\definecolor{textcolor}{rgb}{0.000000,0.000000,0.000000}%
\pgfsetstrokecolor{textcolor}%
\pgfsetfillcolor{textcolor}%
\pgftext[x=2.824726in,y=2.931547in,left,base]{\color{textcolor}{\sffamily\fontsize{16.000000}{19.200000}\selectfont\catcode`\^=\active\def^{\ifmmode\sp\else\^{}\fi}\catcode`\%=\active\def
\end{pgfscope}%
\end{pgfpicture}%
\makeatother%
\endgroup%

%% file: figs/network_conv_HepTh.pgf
\begingroup%
\makeatletter%
\begin{pgfpicture}%
\pgfpathrectangle{\pgfpointorigin}{\pgfqpoint{5.825290in}{4.278899in}}%
\pgfusepath{use as bounding box, clip}%
\begin{pgfscope}%
\pgfsetbuttcap%
\pgfsetmiterjoin%
\definecolor{currentfill}{rgb}{1.000000,1.000000,1.000000}%
\pgfsetfillcolor{currentfill}%
\pgfsetlinewidth{0.000000pt}%
\definecolor{currentstroke}{rgb}{1.000000,1.000000,1.000000}%
\pgfsetstrokecolor{currentstroke}%
\pgfsetdash{}{0pt}%
\pgfpathmoveto{\pgfqpoint{0.000000in}{0.000000in}}%
\pgfpathlineto{\pgfqpoint{5.825290in}{0.000000in}}%
\pgfpathlineto{\pgfqpoint{5.825290in}{4.278899in}}%
\pgfpathlineto{\pgfqpoint{0.000000in}{4.278899in}}%
\pgfpathlineto{\pgfqpoint{0.000000in}{0.000000in}}%
\pgfpathclose%
\pgfusepath{fill}%
\end{pgfscope}%
\begin{pgfscope}%
\pgfsetbuttcap%
\pgfsetmiterjoin%
\definecolor{currentfill}{rgb}{1.000000,1.000000,1.000000}%
\pgfsetfillcolor{currentfill}%
\pgfsetlinewidth{0.000000pt}%
\definecolor{currentstroke}{rgb}{0.000000,0.000000,0.000000}%
\pgfsetstrokecolor{currentstroke}%
\pgfsetstrokeopacity{0.000000}%
\pgfsetdash{}{0pt}%
\pgfpathmoveto{\pgfqpoint{0.865290in}{0.582899in}}%
\pgfpathlineto{\pgfqpoint{5.825290in}{0.582899in}}%
\pgfpathlineto{\pgfqpoint{5.825290in}{4.278899in}}%
\pgfpathlineto{\pgfqpoint{0.865290in}{4.278899in}}%
\pgfpathlineto{\pgfqpoint{0.865290in}{0.582899in}}%
\pgfpathclose%
\pgfusepath{fill}%
\end{pgfscope}%
\begin{pgfscope}%
\pgfsetbuttcap%
\pgfsetroundjoin%
\definecolor{currentfill}{rgb}{0.000000,0.000000,0.000000}%
\pgfsetfillcolor{currentfill}%
\pgfsetlinewidth{0.803000pt}%
\definecolor{currentstroke}{rgb}{0.000000,0.000000,0.000000}%
\pgfsetstrokecolor{currentstroke}%
\pgfsetdash{}{0pt}%
\pgfsys@defobject{currentmarker}{\pgfqpoint{0.000000in}{-0.048611in}}{\pgfqpoint{0.000000in}{0.000000in}}{%
\pgfpathmoveto{\pgfqpoint{0.000000in}{0.000000in}}%
\pgfpathlineto{\pgfqpoint{0.000000in}{-0.048611in}}%
\pgfusepath{stroke,fill}%
}%
\begin{pgfscope}%
\pgfsys@transformshift{1.372563in}{0.582899in}%
\pgfsys@useobject{currentmarker}{}%
\end{pgfscope}%
\end{pgfscope}%
\begin{pgfscope}%
\definecolor{textcolor}{rgb}{0.000000,0.000000,0.000000}%
\pgfsetstrokecolor{textcolor}%
\pgfsetfillcolor{textcolor}%
\pgftext[x=1.372563in,y=0.485677in,,top]{\color{textcolor}{\sffamily\fontsize{16.000000}{19.200000}\selectfont\catcode`\^=\active\def^{\ifmmode\sp\else\^{}\fi}\catcode`\%=\active\def
\end{pgfscope}%
\begin{pgfscope}%
\pgfsetbuttcap%
\pgfsetroundjoin%
\definecolor{currentfill}{rgb}{0.000000,0.000000,0.000000}%
\pgfsetfillcolor{currentfill}%
\pgfsetlinewidth{0.803000pt}%
\definecolor{currentstroke}{rgb}{0.000000,0.000000,0.000000}%
\pgfsetstrokecolor{currentstroke}%
\pgfsetdash{}{0pt}%
\pgfsys@defobject{currentmarker}{\pgfqpoint{0.000000in}{-0.048611in}}{\pgfqpoint{0.000000in}{0.000000in}}{%
\pgfpathmoveto{\pgfqpoint{0.000000in}{0.000000in}}%
\pgfpathlineto{\pgfqpoint{0.000000in}{-0.048611in}}%
\pgfusepath{stroke,fill}%
}%
\begin{pgfscope}%
\pgfsys@transformshift{1.842260in}{0.582899in}%
\pgfsys@useobject{currentmarker}{}%
\end{pgfscope}%
\end{pgfscope}%
\begin{pgfscope}%
\definecolor{textcolor}{rgb}{0.000000,0.000000,0.000000}%
\pgfsetstrokecolor{textcolor}%
\pgfsetfillcolor{textcolor}%
\pgftext[x=1.842260in,y=0.485677in,,top]{\color{textcolor}{\sffamily\fontsize{16.000000}{19.200000}\selectfont\catcode`\^=\active\def^{\ifmmode\sp\else\^{}\fi}\catcode`\%=\active\def
\end{pgfscope}%
\begin{pgfscope}%
\pgfsetbuttcap%
\pgfsetroundjoin%
\definecolor{currentfill}{rgb}{0.000000,0.000000,0.000000}%
\pgfsetfillcolor{currentfill}%
\pgfsetlinewidth{0.803000pt}%
\definecolor{currentstroke}{rgb}{0.000000,0.000000,0.000000}%
\pgfsetstrokecolor{currentstroke}%
\pgfsetdash{}{0pt}%
\pgfsys@defobject{currentmarker}{\pgfqpoint{0.000000in}{-0.048611in}}{\pgfqpoint{0.000000in}{0.000000in}}{%
\pgfpathmoveto{\pgfqpoint{0.000000in}{0.000000in}}%
\pgfpathlineto{\pgfqpoint{0.000000in}{-0.048611in}}%
\pgfusepath{stroke,fill}%
}%
\begin{pgfscope}%
\pgfsys@transformshift{2.311957in}{0.582899in}%
\pgfsys@useobject{currentmarker}{}%
\end{pgfscope}%
\end{pgfscope}%
\begin{pgfscope}%
\definecolor{textcolor}{rgb}{0.000000,0.000000,0.000000}%
\pgfsetstrokecolor{textcolor}%
\pgfsetfillcolor{textcolor}%
\pgftext[x=2.311957in,y=0.485677in,,top]{\color{textcolor}{\sffamily\fontsize{16.000000}{19.200000}\selectfont\catcode`\^=\active\def^{\ifmmode\sp\else\^{}\fi}\catcode`\%=\active\def
\end{pgfscope}%
\begin{pgfscope}%
\pgfsetbuttcap%
\pgfsetroundjoin%
\definecolor{currentfill}{rgb}{0.000000,0.000000,0.000000}%
\pgfsetfillcolor{currentfill}%
\pgfsetlinewidth{0.803000pt}%
\definecolor{currentstroke}{rgb}{0.000000,0.000000,0.000000}%
\pgfsetstrokecolor{currentstroke}%
\pgfsetdash{}{0pt}%
\pgfsys@defobject{currentmarker}{\pgfqpoint{0.000000in}{-0.048611in}}{\pgfqpoint{0.000000in}{0.000000in}}{%
\pgfpathmoveto{\pgfqpoint{0.000000in}{0.000000in}}%
\pgfpathlineto{\pgfqpoint{0.000000in}{-0.048611in}}%
\pgfusepath{stroke,fill}%
}%
\begin{pgfscope}%
\pgfsys@transformshift{2.781654in}{0.582899in}%
\pgfsys@useobject{currentmarker}{}%
\end{pgfscope}%
\end{pgfscope}%
\begin{pgfscope}%
\definecolor{textcolor}{rgb}{0.000000,0.000000,0.000000}%
\pgfsetstrokecolor{textcolor}%
\pgfsetfillcolor{textcolor}%
\pgftext[x=2.781654in,y=0.485677in,,top]{\color{textcolor}{\sffamily\fontsize{16.000000}{19.200000}\selectfont\catcode`\^=\active\def^{\ifmmode\sp\else\^{}\fi}\catcode`\%=\active\def
\end{pgfscope}%
\begin{pgfscope}%
\pgfsetbuttcap%
\pgfsetroundjoin%
\definecolor{currentfill}{rgb}{0.000000,0.000000,0.000000}%
\pgfsetfillcolor{currentfill}%
\pgfsetlinewidth{0.803000pt}%
\definecolor{currentstroke}{rgb}{0.000000,0.000000,0.000000}%
\pgfsetstrokecolor{currentstroke}%
\pgfsetdash{}{0pt}%
\pgfsys@defobject{currentmarker}{\pgfqpoint{0.000000in}{-0.048611in}}{\pgfqpoint{0.000000in}{0.000000in}}{%
\pgfpathmoveto{\pgfqpoint{0.000000in}{0.000000in}}%
\pgfpathlineto{\pgfqpoint{0.000000in}{-0.048611in}}%
\pgfusepath{stroke,fill}%
}%
\begin{pgfscope}%
\pgfsys@transformshift{3.251351in}{0.582899in}%
\pgfsys@useobject{currentmarker}{}%
\end{pgfscope}%
\end{pgfscope}%
\begin{pgfscope}%
\definecolor{textcolor}{rgb}{0.000000,0.000000,0.000000}%
\pgfsetstrokecolor{textcolor}%
\pgfsetfillcolor{textcolor}%
\pgftext[x=3.251351in,y=0.485677in,,top]{\color{textcolor}{\sffamily\fontsize{16.000000}{19.200000}\selectfont\catcode`\^=\active\def^{\ifmmode\sp\else\^{}\fi}\catcode`\%=\active\def
\end{pgfscope}%
\begin{pgfscope}%
\pgfsetbuttcap%
\pgfsetroundjoin%
\definecolor{currentfill}{rgb}{0.000000,0.000000,0.000000}%
\pgfsetfillcolor{currentfill}%
\pgfsetlinewidth{0.803000pt}%
\definecolor{currentstroke}{rgb}{0.000000,0.000000,0.000000}%
\pgfsetstrokecolor{currentstroke}%
\pgfsetdash{}{0pt}%
\pgfsys@defobject{currentmarker}{\pgfqpoint{0.000000in}{-0.048611in}}{\pgfqpoint{0.000000in}{0.000000in}}{%
\pgfpathmoveto{\pgfqpoint{0.000000in}{0.000000in}}%
\pgfpathlineto{\pgfqpoint{0.000000in}{-0.048611in}}%
\pgfusepath{stroke,fill}%
}%
\begin{pgfscope}%
\pgfsys@transformshift{3.721048in}{0.582899in}%
\pgfsys@useobject{currentmarker}{}%
\end{pgfscope}%
\end{pgfscope}%
\begin{pgfscope}%
\definecolor{textcolor}{rgb}{0.000000,0.000000,0.000000}%
\pgfsetstrokecolor{textcolor}%
\pgfsetfillcolor{textcolor}%
\pgftext[x=3.721048in,y=0.485677in,,top]{\color{textcolor}{\sffamily\fontsize{16.000000}{19.200000}\selectfont\catcode`\^=\active\def^{\ifmmode\sp\else\^{}\fi}\catcode`\%=\active\def
\end{pgfscope}%
\begin{pgfscope}%
\pgfsetbuttcap%
\pgfsetroundjoin%
\definecolor{currentfill}{rgb}{0.000000,0.000000,0.000000}%
\pgfsetfillcolor{currentfill}%
\pgfsetlinewidth{0.803000pt}%
\definecolor{currentstroke}{rgb}{0.000000,0.000000,0.000000}%
\pgfsetstrokecolor{currentstroke}%
\pgfsetdash{}{0pt}%
\pgfsys@defobject{currentmarker}{\pgfqpoint{0.000000in}{-0.048611in}}{\pgfqpoint{0.000000in}{0.000000in}}{%
\pgfpathmoveto{\pgfqpoint{0.000000in}{0.000000in}}%
\pgfpathlineto{\pgfqpoint{0.000000in}{-0.048611in}}%
\pgfusepath{stroke,fill}%
}%
\begin{pgfscope}%
\pgfsys@transformshift{4.190745in}{0.582899in}%
\pgfsys@useobject{currentmarker}{}%
\end{pgfscope}%
\end{pgfscope}%
\begin{pgfscope}%
\definecolor{textcolor}{rgb}{0.000000,0.000000,0.000000}%
\pgfsetstrokecolor{textcolor}%
\pgfsetfillcolor{textcolor}%
\pgftext[x=4.190745in,y=0.485677in,,top]{\color{textcolor}{\sffamily\fontsize{16.000000}{19.200000}\selectfont\catcode`\^=\active\def^{\ifmmode\sp\else\^{}\fi}\catcode`\%=\active\def
\end{pgfscope}%
\begin{pgfscope}%
\pgfsetbuttcap%
\pgfsetroundjoin%
\definecolor{currentfill}{rgb}{0.000000,0.000000,0.000000}%
\pgfsetfillcolor{currentfill}%
\pgfsetlinewidth{0.803000pt}%
\definecolor{currentstroke}{rgb}{0.000000,0.000000,0.000000}%
\pgfsetstrokecolor{currentstroke}%
\pgfsetdash{}{0pt}%
\pgfsys@defobject{currentmarker}{\pgfqpoint{0.000000in}{-0.048611in}}{\pgfqpoint{0.000000in}{0.000000in}}{%
\pgfpathmoveto{\pgfqpoint{0.000000in}{0.000000in}}%
\pgfpathlineto{\pgfqpoint{0.000000in}{-0.048611in}}%
\pgfusepath{stroke,fill}%
}%
\begin{pgfscope}%
\pgfsys@transformshift{4.660442in}{0.582899in}%
\pgfsys@useobject{currentmarker}{}%
\end{pgfscope}%
\end{pgfscope}%
\begin{pgfscope}%
\definecolor{textcolor}{rgb}{0.000000,0.000000,0.000000}%
\pgfsetstrokecolor{textcolor}%
\pgfsetfillcolor{textcolor}%
\pgftext[x=4.660442in,y=0.485677in,,top]{\color{textcolor}{\sffamily\fontsize{16.000000}{19.200000}\selectfont\catcode`\^=\active\def^{\ifmmode\sp\else\^{}\fi}\catcode`\%=\active\def
\end{pgfscope}%
\begin{pgfscope}%
\pgfsetbuttcap%
\pgfsetroundjoin%
\definecolor{currentfill}{rgb}{0.000000,0.000000,0.000000}%
\pgfsetfillcolor{currentfill}%
\pgfsetlinewidth{0.803000pt}%
\definecolor{currentstroke}{rgb}{0.000000,0.000000,0.000000}%
\pgfsetstrokecolor{currentstroke}%
\pgfsetdash{}{0pt}%
\pgfsys@defobject{currentmarker}{\pgfqpoint{0.000000in}{-0.048611in}}{\pgfqpoint{0.000000in}{0.000000in}}{%
\pgfpathmoveto{\pgfqpoint{0.000000in}{0.000000in}}%
\pgfpathlineto{\pgfqpoint{0.000000in}{-0.048611in}}%
\pgfusepath{stroke,fill}%
}%
\begin{pgfscope}%
\pgfsys@transformshift{5.130139in}{0.582899in}%
\pgfsys@useobject{currentmarker}{}%
\end{pgfscope}%
\end{pgfscope}%
\begin{pgfscope}%
\definecolor{textcolor}{rgb}{0.000000,0.000000,0.000000}%
\pgfsetstrokecolor{textcolor}%
\pgfsetfillcolor{textcolor}%
\pgftext[x=5.130139in,y=0.485677in,,top]{\color{textcolor}{\sffamily\fontsize{16.000000}{19.200000}\selectfont\catcode`\^=\active\def^{\ifmmode\sp\else\^{}\fi}\catcode`\%=\active\def
\end{pgfscope}%
\begin{pgfscope}%
\pgfsetbuttcap%
\pgfsetroundjoin%
\definecolor{currentfill}{rgb}{0.000000,0.000000,0.000000}%
\pgfsetfillcolor{currentfill}%
\pgfsetlinewidth{0.803000pt}%
\definecolor{currentstroke}{rgb}{0.000000,0.000000,0.000000}%
\pgfsetstrokecolor{currentstroke}%
\pgfsetdash{}{0pt}%
\pgfsys@defobject{currentmarker}{\pgfqpoint{0.000000in}{-0.048611in}}{\pgfqpoint{0.000000in}{0.000000in}}{%
\pgfpathmoveto{\pgfqpoint{0.000000in}{0.000000in}}%
\pgfpathlineto{\pgfqpoint{0.000000in}{-0.048611in}}%
\pgfusepath{stroke,fill}%
}%
\begin{pgfscope}%
\pgfsys@transformshift{5.599836in}{0.582899in}%
\pgfsys@useobject{currentmarker}{}%
\end{pgfscope}%
\end{pgfscope}%
\begin{pgfscope}%
\definecolor{textcolor}{rgb}{0.000000,0.000000,0.000000}%
\pgfsetstrokecolor{textcolor}%
\pgfsetfillcolor{textcolor}%
\pgftext[x=5.599836in,y=0.485677in,,top]{\color{textcolor}{\sffamily\fontsize{16.000000}{19.200000}\selectfont\catcode`\^=\active\def^{\ifmmode\sp\else\^{}\fi}\catcode`\%=\active\def
\end{pgfscope}%
\begin{pgfscope}%
\definecolor{textcolor}{rgb}{0.000000,0.000000,0.000000}%
\pgfsetstrokecolor{textcolor}%
\pgfsetfillcolor{textcolor}%
\pgftext[x=3.345290in,y=0.215061in,,top]{\color{textcolor}{\sffamily\fontsize{16.000000}{19.200000}\selectfont\catcode`\^=\active\def^{\ifmmode\sp\else\^{}\fi}\catcode`\%=\active\def
\end{pgfscope}%
\begin{pgfscope}%
\pgfsetbuttcap%
\pgfsetroundjoin%
\definecolor{currentfill}{rgb}{0.000000,0.000000,0.000000}%
\pgfsetfillcolor{currentfill}%
\pgfsetlinewidth{0.803000pt}%
\definecolor{currentstroke}{rgb}{0.000000,0.000000,0.000000}%
\pgfsetstrokecolor{currentstroke}%
\pgfsetdash{}{0pt}%
\pgfsys@defobject{currentmarker}{\pgfqpoint{-0.048611in}{0.000000in}}{\pgfqpoint{-0.000000in}{0.000000in}}{%
\pgfpathmoveto{\pgfqpoint{-0.000000in}{0.000000in}}%
\pgfpathlineto{\pgfqpoint{-0.048611in}{0.000000in}}%
\pgfusepath{stroke,fill}%
}%
\begin{pgfscope}%
\pgfsys@transformshift{0.865290in}{1.073265in}%
\pgfsys@useobject{currentmarker}{}%
\end{pgfscope}%
\end{pgfscope}%
\begin{pgfscope}%
\definecolor{textcolor}{rgb}{0.000000,0.000000,0.000000}%
\pgfsetstrokecolor{textcolor}%
\pgfsetfillcolor{textcolor}%
\pgftext[x=0.270616in, y=0.988847in, left, base]{\color{textcolor}{\sffamily\fontsize{16.000000}{19.200000}\selectfont\catcode`\^=\active\def^{\ifmmode\sp\else\^{}\fi}\catcode`\%=\active\def
\end{pgfscope}%
\begin{pgfscope}%
\pgfsetbuttcap%
\pgfsetroundjoin%
\definecolor{currentfill}{rgb}{0.000000,0.000000,0.000000}%
\pgfsetfillcolor{currentfill}%
\pgfsetlinewidth{0.803000pt}%
\definecolor{currentstroke}{rgb}{0.000000,0.000000,0.000000}%
\pgfsetstrokecolor{currentstroke}%
\pgfsetdash{}{0pt}%
\pgfsys@defobject{currentmarker}{\pgfqpoint{-0.048611in}{0.000000in}}{\pgfqpoint{-0.000000in}{0.000000in}}{%
\pgfpathmoveto{\pgfqpoint{-0.000000in}{0.000000in}}%
\pgfpathlineto{\pgfqpoint{-0.048611in}{0.000000in}}%
\pgfusepath{stroke,fill}%
}%
\begin{pgfscope}%
\pgfsys@transformshift{0.865290in}{1.634203in}%
\pgfsys@useobject{currentmarker}{}%
\end{pgfscope}%
\end{pgfscope}%
\begin{pgfscope}%
\definecolor{textcolor}{rgb}{0.000000,0.000000,0.000000}%
\pgfsetstrokecolor{textcolor}%
\pgfsetfillcolor{textcolor}%
\pgftext[x=0.270616in, y=1.549784in, left, base]{\color{textcolor}{\sffamily\fontsize{16.000000}{19.200000}\selectfont\catcode`\^=\active\def^{\ifmmode\sp\else\^{}\fi}\catcode`\%=\active\def
\end{pgfscope}%
\begin{pgfscope}%
\pgfsetbuttcap%
\pgfsetroundjoin%
\definecolor{currentfill}{rgb}{0.000000,0.000000,0.000000}%
\pgfsetfillcolor{currentfill}%
\pgfsetlinewidth{0.803000pt}%
\definecolor{currentstroke}{rgb}{0.000000,0.000000,0.000000}%
\pgfsetstrokecolor{currentstroke}%
\pgfsetdash{}{0pt}%
\pgfsys@defobject{currentmarker}{\pgfqpoint{-0.048611in}{0.000000in}}{\pgfqpoint{-0.000000in}{0.000000in}}{%
\pgfpathmoveto{\pgfqpoint{-0.000000in}{0.000000in}}%
\pgfpathlineto{\pgfqpoint{-0.048611in}{0.000000in}}%
\pgfusepath{stroke,fill}%
}%
\begin{pgfscope}%
\pgfsys@transformshift{0.865290in}{2.195140in}%
\pgfsys@useobject{currentmarker}{}%
\end{pgfscope}%
\end{pgfscope}%
\begin{pgfscope}%
\definecolor{textcolor}{rgb}{0.000000,0.000000,0.000000}%
\pgfsetstrokecolor{textcolor}%
\pgfsetfillcolor{textcolor}%
\pgftext[x=0.346658in, y=2.110721in, left, base]{\color{textcolor}{\sffamily\fontsize{16.000000}{19.200000}\selectfont\catcode`\^=\active\def^{\ifmmode\sp\else\^{}\fi}\catcode`\%=\active\def
\end{pgfscope}%
\begin{pgfscope}%
\pgfsetbuttcap%
\pgfsetroundjoin%
\definecolor{currentfill}{rgb}{0.000000,0.000000,0.000000}%
\pgfsetfillcolor{currentfill}%
\pgfsetlinewidth{0.803000pt}%
\definecolor{currentstroke}{rgb}{0.000000,0.000000,0.000000}%
\pgfsetstrokecolor{currentstroke}%
\pgfsetdash{}{0pt}%
\pgfsys@defobject{currentmarker}{\pgfqpoint{-0.048611in}{0.000000in}}{\pgfqpoint{-0.000000in}{0.000000in}}{%
\pgfpathmoveto{\pgfqpoint{-0.000000in}{0.000000in}}%
\pgfpathlineto{\pgfqpoint{-0.048611in}{0.000000in}}%
\pgfusepath{stroke,fill}%
}%
\begin{pgfscope}%
\pgfsys@transformshift{0.865290in}{2.756077in}%
\pgfsys@useobject{currentmarker}{}%
\end{pgfscope}%
\end{pgfscope}%
\begin{pgfscope}%
\definecolor{textcolor}{rgb}{0.000000,0.000000,0.000000}%
\pgfsetstrokecolor{textcolor}%
\pgfsetfillcolor{textcolor}%
\pgftext[x=0.346658in, y=2.671659in, left, base]{\color{textcolor}{\sffamily\fontsize{16.000000}{19.200000}\selectfont\catcode`\^=\active\def^{\ifmmode\sp\else\^{}\fi}\catcode`\%=\active\def
\end{pgfscope}%
\begin{pgfscope}%
\pgfsetbuttcap%
\pgfsetroundjoin%
\definecolor{currentfill}{rgb}{0.000000,0.000000,0.000000}%
\pgfsetfillcolor{currentfill}%
\pgfsetlinewidth{0.803000pt}%
\definecolor{currentstroke}{rgb}{0.000000,0.000000,0.000000}%
\pgfsetstrokecolor{currentstroke}%
\pgfsetdash{}{0pt}%
\pgfsys@defobject{currentmarker}{\pgfqpoint{-0.048611in}{0.000000in}}{\pgfqpoint{-0.000000in}{0.000000in}}{%
\pgfpathmoveto{\pgfqpoint{-0.000000in}{0.000000in}}%
\pgfpathlineto{\pgfqpoint{-0.048611in}{0.000000in}}%
\pgfusepath{stroke,fill}%
}%
\begin{pgfscope}%
\pgfsys@transformshift{0.865290in}{3.317014in}%
\pgfsys@useobject{currentmarker}{}%
\end{pgfscope}%
\end{pgfscope}%
\begin{pgfscope}%
\definecolor{textcolor}{rgb}{0.000000,0.000000,0.000000}%
\pgfsetstrokecolor{textcolor}%
\pgfsetfillcolor{textcolor}%
\pgftext[x=0.346658in, y=3.232596in, left, base]{\color{textcolor}{\sffamily\fontsize{16.000000}{19.200000}\selectfont\catcode`\^=\active\def^{\ifmmode\sp\else\^{}\fi}\catcode`\%=\active\def
\end{pgfscope}%
\begin{pgfscope}%
\pgfsetbuttcap%
\pgfsetroundjoin%
\definecolor{currentfill}{rgb}{0.000000,0.000000,0.000000}%
\pgfsetfillcolor{currentfill}%
\pgfsetlinewidth{0.803000pt}%
\definecolor{currentstroke}{rgb}{0.000000,0.000000,0.000000}%
\pgfsetstrokecolor{currentstroke}%
\pgfsetdash{}{0pt}%
\pgfsys@defobject{currentmarker}{\pgfqpoint{-0.048611in}{0.000000in}}{\pgfqpoint{-0.000000in}{0.000000in}}{%
\pgfpathmoveto{\pgfqpoint{-0.000000in}{0.000000in}}%
\pgfpathlineto{\pgfqpoint{-0.048611in}{0.000000in}}%
\pgfusepath{stroke,fill}%
}%
\begin{pgfscope}%
\pgfsys@transformshift{0.865290in}{3.877952in}%
\pgfsys@useobject{currentmarker}{}%
\end{pgfscope}%
\end{pgfscope}%
\begin{pgfscope}%
\definecolor{textcolor}{rgb}{0.000000,0.000000,0.000000}%
\pgfsetstrokecolor{textcolor}%
\pgfsetfillcolor{textcolor}%
\pgftext[x=0.464945in, y=3.793533in, left, base]{\color{textcolor}{\sffamily\fontsize{16.000000}{19.200000}\selectfont\catcode`\^=\active\def^{\ifmmode\sp\else\^{}\fi}\catcode`\%=\active\def
\end{pgfscope}%
\begin{pgfscope}%
\definecolor{textcolor}{rgb}{0.000000,0.000000,0.000000}%
\pgfsetstrokecolor{textcolor}%
\pgfsetfillcolor{textcolor}%
\pgftext[x=0.215061in,y=2.430899in,,bottom,rotate=90.000000]{\color{textcolor}{\sffamily\fontsize{16.000000}{19.200000}\selectfont\catcode`\^=\active\def^{\ifmmode\sp\else\^{}\fi}\catcode`\%=\active\def
\end{pgfscope}%
\begin{pgfscope}%
\pgfpathrectangle{\pgfqpoint{0.865290in}{0.582899in}}{\pgfqpoint{4.960000in}{3.696000in}}%
\pgfusepath{clip}%
\pgfsetrectcap%
\pgfsetroundjoin%
\pgfsetlinewidth{1.505625pt}%
\definecolor{currentstroke}{rgb}{0.121569,0.466667,0.705882}%
\pgfsetstrokecolor{currentstroke}%
\pgfsetdash{}{0pt}%
\pgfpathmoveto{\pgfqpoint{1.090745in}{3.494516in}}%
\pgfpathlineto{\pgfqpoint{1.278623in}{4.110899in}}%
\pgfpathlineto{\pgfqpoint{1.466502in}{3.983538in}}%
\pgfpathlineto{\pgfqpoint{1.654381in}{3.845590in}}%
\pgfpathlineto{\pgfqpoint{1.842260in}{3.779502in}}%
\pgfpathlineto{\pgfqpoint{2.030139in}{3.581801in}}%
\pgfpathlineto{\pgfqpoint{2.218017in}{3.316488in}}%
\pgfpathlineto{\pgfqpoint{2.405896in}{3.206087in}}%
\pgfpathlineto{\pgfqpoint{2.593775in}{3.021308in}}%
\pgfpathlineto{\pgfqpoint{2.781654in}{2.682917in}}%
\pgfpathlineto{\pgfqpoint{2.969533in}{2.514693in}}%
\pgfpathlineto{\pgfqpoint{3.157411in}{2.282576in}}%
\pgfpathlineto{\pgfqpoint{3.345290in}{2.044546in}}%
\pgfpathlineto{\pgfqpoint{3.533169in}{2.026205in}}%
\pgfpathlineto{\pgfqpoint{3.721048in}{1.974879in}}%
\pgfpathlineto{\pgfqpoint{3.908927in}{1.892811in}}%
\pgfpathlineto{\pgfqpoint{4.096805in}{1.565813in}}%
\pgfpathlineto{\pgfqpoint{4.284684in}{1.553871in}}%
\pgfpathlineto{\pgfqpoint{4.472563in}{1.553349in}}%
\pgfpathlineto{\pgfqpoint{4.660442in}{1.553287in}}%
\pgfpathlineto{\pgfqpoint{4.848320in}{1.553284in}}%
\pgfpathlineto{\pgfqpoint{5.036199in}{1.553284in}}%
\pgfpathlineto{\pgfqpoint{5.224078in}{1.553284in}}%
\pgfpathlineto{\pgfqpoint{5.411957in}{1.553305in}}%
\pgfpathlineto{\pgfqpoint{5.599836in}{1.553305in}}%
\pgfusepath{stroke}%
\end{pgfscope}%
\begin{pgfscope}%
\pgfpathrectangle{\pgfqpoint{0.865290in}{0.582899in}}{\pgfqpoint{4.960000in}{3.696000in}}%
\pgfusepath{clip}%
\pgfsetrectcap%
\pgfsetroundjoin%
\pgfsetlinewidth{1.505625pt}%
\definecolor{currentstroke}{rgb}{1.000000,0.498039,0.054902}%
\pgfsetstrokecolor{currentstroke}%
\pgfsetdash{}{0pt}%
\pgfpathmoveto{\pgfqpoint{1.090745in}{3.503993in}}%
\pgfpathlineto{\pgfqpoint{1.278623in}{3.503990in}}%
\pgfpathlineto{\pgfqpoint{1.466502in}{3.503954in}}%
\pgfpathlineto{\pgfqpoint{1.654381in}{3.502661in}}%
\pgfpathlineto{\pgfqpoint{1.842260in}{3.461030in}}%
\pgfpathlineto{\pgfqpoint{2.030139in}{3.016215in}}%
\pgfpathlineto{\pgfqpoint{2.218017in}{2.966535in}}%
\pgfpathlineto{\pgfqpoint{2.405896in}{2.801493in}}%
\pgfpathlineto{\pgfqpoint{2.593775in}{2.495402in}}%
\pgfpathlineto{\pgfqpoint{2.781654in}{2.345342in}}%
\pgfpathlineto{\pgfqpoint{2.969533in}{2.163811in}}%
\pgfpathlineto{\pgfqpoint{3.157411in}{1.987034in}}%
\pgfpathlineto{\pgfqpoint{3.345290in}{1.727182in}}%
\pgfpathlineto{\pgfqpoint{3.533169in}{1.402462in}}%
\pgfpathlineto{\pgfqpoint{3.721048in}{1.085851in}}%
\pgfpathlineto{\pgfqpoint{3.908927in}{0.750899in}}%
\pgfpathlineto{\pgfqpoint{4.096805in}{0.818028in}}%
\pgfpathlineto{\pgfqpoint{4.284684in}{0.818028in}}%
\pgfpathlineto{\pgfqpoint{4.472563in}{0.818028in}}%
\pgfpathlineto{\pgfqpoint{4.660442in}{0.818028in}}%
\pgfpathlineto{\pgfqpoint{4.848320in}{0.818028in}}%
\pgfpathlineto{\pgfqpoint{5.036199in}{0.818028in}}%
\pgfpathlineto{\pgfqpoint{5.224078in}{0.818028in}}%
\pgfpathlineto{\pgfqpoint{5.411957in}{0.818028in}}%
\pgfpathlineto{\pgfqpoint{5.599836in}{0.818028in}}%
\pgfusepath{stroke}%
\end{pgfscope}%
\begin{pgfscope}%
\pgfpathrectangle{\pgfqpoint{0.865290in}{0.582899in}}{\pgfqpoint{4.960000in}{3.696000in}}%
\pgfusepath{clip}%
\pgfsetrectcap%
\pgfsetroundjoin%
\pgfsetlinewidth{1.505625pt}%
\definecolor{currentstroke}{rgb}{0.172549,0.627451,0.172549}%
\pgfsetstrokecolor{currentstroke}%
\pgfsetdash{}{0pt}%
\pgfpathmoveto{\pgfqpoint{1.090745in}{3.501399in}}%
\pgfpathlineto{\pgfqpoint{1.278623in}{3.875741in}}%
\pgfpathlineto{\pgfqpoint{1.466502in}{3.647750in}}%
\pgfpathlineto{\pgfqpoint{1.654381in}{3.798316in}}%
\pgfpathlineto{\pgfqpoint{1.842260in}{3.686224in}}%
\pgfpathlineto{\pgfqpoint{2.030139in}{3.514147in}}%
\pgfpathlineto{\pgfqpoint{2.218017in}{3.315345in}}%
\pgfpathlineto{\pgfqpoint{2.405896in}{2.835530in}}%
\pgfpathlineto{\pgfqpoint{2.593775in}{2.826120in}}%
\pgfpathlineto{\pgfqpoint{2.781654in}{2.593998in}}%
\pgfpathlineto{\pgfqpoint{2.969533in}{2.113886in}}%
\pgfpathlineto{\pgfqpoint{3.157411in}{2.088042in}}%
\pgfpathlineto{\pgfqpoint{3.345290in}{1.811287in}}%
\pgfpathlineto{\pgfqpoint{3.533169in}{1.455432in}}%
\pgfpathlineto{\pgfqpoint{3.721048in}{1.170201in}}%
\pgfpathlineto{\pgfqpoint{3.908927in}{1.146429in}}%
\pgfpathlineto{\pgfqpoint{4.096805in}{1.180882in}}%
\pgfpathlineto{\pgfqpoint{4.284684in}{0.880423in}}%
\pgfpathlineto{\pgfqpoint{4.472563in}{1.172200in}}%
\pgfpathlineto{\pgfqpoint{4.660442in}{1.147270in}}%
\pgfpathlineto{\pgfqpoint{4.848320in}{1.180882in}}%
\pgfpathlineto{\pgfqpoint{5.036199in}{0.880423in}}%
\pgfpathlineto{\pgfqpoint{5.224078in}{1.172200in}}%
\pgfpathlineto{\pgfqpoint{5.411957in}{1.147270in}}%
\pgfpathlineto{\pgfqpoint{5.599836in}{1.173739in}}%
\pgfusepath{stroke}%
\end{pgfscope}%
\begin{pgfscope}%
\pgfpathrectangle{\pgfqpoint{0.865290in}{0.582899in}}{\pgfqpoint{4.960000in}{3.696000in}}%
\pgfusepath{clip}%
\pgfsetrectcap%
\pgfsetroundjoin%
\pgfsetlinewidth{1.505625pt}%
\definecolor{currentstroke}{rgb}{0.839216,0.152941,0.156863}%
\pgfsetstrokecolor{currentstroke}%
\pgfsetdash{}{0pt}%
\pgfpathmoveto{\pgfqpoint{1.090745in}{3.503621in}}%
\pgfpathlineto{\pgfqpoint{1.278623in}{3.893038in}}%
\pgfpathlineto{\pgfqpoint{1.466502in}{3.877877in}}%
\pgfpathlineto{\pgfqpoint{1.654381in}{3.798321in}}%
\pgfpathlineto{\pgfqpoint{1.842260in}{3.646016in}}%
\pgfpathlineto{\pgfqpoint{2.030139in}{3.447071in}}%
\pgfpathlineto{\pgfqpoint{2.218017in}{3.232795in}}%
\pgfpathlineto{\pgfqpoint{2.405896in}{2.711854in}}%
\pgfpathlineto{\pgfqpoint{2.593775in}{2.812917in}}%
\pgfpathlineto{\pgfqpoint{2.781654in}{2.559951in}}%
\pgfpathlineto{\pgfqpoint{2.969533in}{2.217464in}}%
\pgfpathlineto{\pgfqpoint{3.157411in}{2.059945in}}%
\pgfpathlineto{\pgfqpoint{3.345290in}{1.745495in}}%
\pgfpathlineto{\pgfqpoint{3.533169in}{1.586582in}}%
\pgfpathlineto{\pgfqpoint{3.721048in}{1.553583in}}%
\pgfpathlineto{\pgfqpoint{3.908927in}{1.554455in}}%
\pgfpathlineto{\pgfqpoint{4.096805in}{1.553922in}}%
\pgfpathlineto{\pgfqpoint{4.284684in}{1.554498in}}%
\pgfpathlineto{\pgfqpoint{4.472563in}{1.554141in}}%
\pgfpathlineto{\pgfqpoint{4.660442in}{1.554498in}}%
\pgfpathlineto{\pgfqpoint{4.848320in}{1.555483in}}%
\pgfpathlineto{\pgfqpoint{5.036199in}{1.554435in}}%
\pgfpathlineto{\pgfqpoint{5.224078in}{1.554021in}}%
\pgfpathlineto{\pgfqpoint{5.411957in}{1.554292in}}%
\pgfpathlineto{\pgfqpoint{5.599836in}{1.553857in}}%
\pgfusepath{stroke}%
\end{pgfscope}%
\begin{pgfscope}%
\pgfsetrectcap%
\pgfsetmiterjoin%
\pgfsetlinewidth{0.803000pt}%
\definecolor{currentstroke}{rgb}{0.000000,0.000000,0.000000}%
\pgfsetstrokecolor{currentstroke}%
\pgfsetdash{}{0pt}%
\pgfpathmoveto{\pgfqpoint{0.865290in}{0.582899in}}%
\pgfpathlineto{\pgfqpoint{0.865290in}{4.278899in}}%
\pgfusepath{stroke}%
\end{pgfscope}%
\begin{pgfscope}%
\pgfsetrectcap%
\pgfsetmiterjoin%
\pgfsetlinewidth{0.803000pt}%
\definecolor{currentstroke}{rgb}{0.000000,0.000000,0.000000}%
\pgfsetstrokecolor{currentstroke}%
\pgfsetdash{}{0pt}%
\pgfpathmoveto{\pgfqpoint{5.825290in}{0.582899in}}%
\pgfpathlineto{\pgfqpoint{5.825290in}{4.278899in}}%
\pgfusepath{stroke}%
\end{pgfscope}%
\begin{pgfscope}%
\pgfsetrectcap%
\pgfsetmiterjoin%
\pgfsetlinewidth{0.803000pt}%
\definecolor{currentstroke}{rgb}{0.000000,0.000000,0.000000}%
\pgfsetstrokecolor{currentstroke}%
\pgfsetdash{}{0pt}%
\pgfpathmoveto{\pgfqpoint{0.865290in}{0.582899in}}%
\pgfpathlineto{\pgfqpoint{5.825290in}{0.582899in}}%
\pgfusepath{stroke}%
\end{pgfscope}%
\begin{pgfscope}%
\pgfsetrectcap%
\pgfsetmiterjoin%
\pgfsetlinewidth{0.803000pt}%
\definecolor{currentstroke}{rgb}{0.000000,0.000000,0.000000}%
\pgfsetstrokecolor{currentstroke}%
\pgfsetdash{}{0pt}%
\pgfpathmoveto{\pgfqpoint{0.865290in}{4.278899in}}%
\pgfpathlineto{\pgfqpoint{5.825290in}{4.278899in}}%
\pgfusepath{stroke}%
\end{pgfscope}%
\begin{pgfscope}%
\pgfsetbuttcap%
\pgfsetmiterjoin%
\definecolor{currentfill}{rgb}{1.000000,1.000000,1.000000}%
\pgfsetfillcolor{currentfill}%
\pgfsetfillopacity{0.800000}%
\pgfsetlinewidth{1.003750pt}%
\definecolor{currentstroke}{rgb}{0.800000,0.800000,0.800000}%
\pgfsetstrokecolor{currentstroke}%
\pgfsetstrokeopacity{0.800000}%
\pgfsetdash{}{0pt}%
\pgfpathmoveto{\pgfqpoint{2.158059in}{2.796434in}}%
\pgfpathlineto{\pgfqpoint{5.669735in}{2.796434in}}%
\pgfpathquadraticcurveto{\pgfqpoint{5.714179in}{2.796434in}}{\pgfqpoint{5.714179in}{2.840879in}}%
\pgfpathlineto{\pgfqpoint{5.714179in}{4.123343in}}%
\pgfpathquadraticcurveto{\pgfqpoint{5.714179in}{4.167788in}}{\pgfqpoint{5.669735in}{4.167788in}}%
\pgfpathlineto{\pgfqpoint{2.158059in}{4.167788in}}%
\pgfpathquadraticcurveto{\pgfqpoint{2.113615in}{4.167788in}}{\pgfqpoint{2.113615in}{4.123343in}}%
\pgfpathlineto{\pgfqpoint{2.113615in}{2.840879in}}%
\pgfpathquadraticcurveto{\pgfqpoint{2.113615in}{2.796434in}}{\pgfqpoint{2.158059in}{2.796434in}}%
\pgfpathlineto{\pgfqpoint{2.158059in}{2.796434in}}%
\pgfpathclose%
\pgfusepath{stroke,fill}%
\end{pgfscope}%
\begin{pgfscope}%
\pgfsetrectcap%
\pgfsetroundjoin%
\pgfsetlinewidth{1.505625pt}%
\definecolor{currentstroke}{rgb}{0.121569,0.466667,0.705882}%
\pgfsetstrokecolor{currentstroke}%
\pgfsetdash{}{0pt}%
\pgfpathmoveto{\pgfqpoint{2.202504in}{3.987840in}}%
\pgfpathlineto{\pgfqpoint{2.424726in}{3.987840in}}%
\pgfpathlineto{\pgfqpoint{2.646948in}{3.987840in}}%
\pgfusepath{stroke}%
\end{pgfscope}%
\begin{pgfscope}%
\definecolor{textcolor}{rgb}{0.000000,0.000000,0.000000}%
\pgfsetstrokecolor{textcolor}%
\pgfsetfillcolor{textcolor}%
\pgftext[x=2.824726in,y=3.910062in,left,base]{\color{textcolor}{\sffamily\fontsize{16.000000}{19.200000}\selectfont\catcode`\^=\active\def^{\ifmmode\sp\else\^{}\fi}\catcode`\%=\active\def
\end{pgfscope}%
\begin{pgfscope}%
\pgfsetrectcap%
\pgfsetroundjoin%
\pgfsetlinewidth{1.505625pt}%
\definecolor{currentstroke}{rgb}{1.000000,0.498039,0.054902}%
\pgfsetstrokecolor{currentstroke}%
\pgfsetdash{}{0pt}%
\pgfpathmoveto{\pgfqpoint{2.202504in}{3.661668in}}%
\pgfpathlineto{\pgfqpoint{2.424726in}{3.661668in}}%
\pgfpathlineto{\pgfqpoint{2.646948in}{3.661668in}}%
\pgfusepath{stroke}%
\end{pgfscope}%
\begin{pgfscope}%
\definecolor{textcolor}{rgb}{0.000000,0.000000,0.000000}%
\pgfsetstrokecolor{textcolor}%
\pgfsetfillcolor{textcolor}%
\pgftext[x=2.824726in,y=3.583890in,left,base]{\color{textcolor}{\sffamily\fontsize{16.000000}{19.200000}\selectfont\catcode`\^=\active\def^{\ifmmode\sp\else\^{}\fi}\catcode`\%=\active\def
\end{pgfscope}%
\begin{pgfscope}%
\pgfsetrectcap%
\pgfsetroundjoin%
\pgfsetlinewidth{1.505625pt}%
\definecolor{currentstroke}{rgb}{0.172549,0.627451,0.172549}%
\pgfsetstrokecolor{currentstroke}%
\pgfsetdash{}{0pt}%
\pgfpathmoveto{\pgfqpoint{2.202504in}{3.335497in}}%
\pgfpathlineto{\pgfqpoint{2.424726in}{3.335497in}}%
\pgfpathlineto{\pgfqpoint{2.646948in}{3.335497in}}%
\pgfusepath{stroke}%
\end{pgfscope}%
\begin{pgfscope}%
\definecolor{textcolor}{rgb}{0.000000,0.000000,0.000000}%
\pgfsetstrokecolor{textcolor}%
\pgfsetfillcolor{textcolor}%
\pgftext[x=2.824726in,y=3.257719in,left,base]{\color{textcolor}{\sffamily\fontsize{16.000000}{19.200000}\selectfont\catcode`\^=\active\def^{\ifmmode\sp\else\^{}\fi}\catcode`\%=\active\def
\end{pgfscope}%
\begin{pgfscope}%
\pgfsetrectcap%
\pgfsetroundjoin%
\pgfsetlinewidth{1.505625pt}%
\definecolor{currentstroke}{rgb}{0.839216,0.152941,0.156863}%
\pgfsetstrokecolor{currentstroke}%
\pgfsetdash{}{0pt}%
\pgfpathmoveto{\pgfqpoint{2.202504in}{3.009325in}}%
\pgfpathlineto{\pgfqpoint{2.424726in}{3.009325in}}%
\pgfpathlineto{\pgfqpoint{2.646948in}{3.009325in}}%
\pgfusepath{stroke}%
\end{pgfscope}%
\begin{pgfscope}%
\definecolor{textcolor}{rgb}{0.000000,0.000000,0.000000}%
\pgfsetstrokecolor{textcolor}%
\pgfsetfillcolor{textcolor}%
\pgftext[x=2.824726in,y=2.931547in,left,base]{\color{textcolor}{\sffamily\fontsize{16.000000}{19.200000}\selectfont\catcode`\^=\active\def^{\ifmmode\sp\else\^{}\fi}\catcode`\%=\active\def
\end{pgfscope}%
\end{pgfpicture}%
\makeatother%
\endgroup%

%% file: figs/network_conv_ei_Air500.pgf
\begingroup%
\makeatletter%
\begin{pgfpicture}%
\pgfpathrectangle{\pgfpointorigin}{\pgfqpoint{5.825290in}{4.278899in}}%
\pgfusepath{use as bounding box, clip}%
\begin{pgfscope}%
\pgfsetbuttcap%
\pgfsetmiterjoin%
\definecolor{currentfill}{rgb}{1.000000,1.000000,1.000000}%
\pgfsetfillcolor{currentfill}%
\pgfsetlinewidth{0.000000pt}%
\definecolor{currentstroke}{rgb}{1.000000,1.000000,1.000000}%
\pgfsetstrokecolor{currentstroke}%
\pgfsetdash{}{0pt}%
\pgfpathmoveto{\pgfqpoint{0.000000in}{0.000000in}}%
\pgfpathlineto{\pgfqpoint{5.825290in}{0.000000in}}%
\pgfpathlineto{\pgfqpoint{5.825290in}{4.278899in}}%
\pgfpathlineto{\pgfqpoint{0.000000in}{4.278899in}}%
\pgfpathlineto{\pgfqpoint{0.000000in}{0.000000in}}%
\pgfpathclose%
\pgfusepath{fill}%
\end{pgfscope}%
\begin{pgfscope}%
\pgfsetbuttcap%
\pgfsetmiterjoin%
\definecolor{currentfill}{rgb}{1.000000,1.000000,1.000000}%
\pgfsetfillcolor{currentfill}%
\pgfsetlinewidth{0.000000pt}%
\definecolor{currentstroke}{rgb}{0.000000,0.000000,0.000000}%
\pgfsetstrokecolor{currentstroke}%
\pgfsetstrokeopacity{0.000000}%
\pgfsetdash{}{0pt}%
\pgfpathmoveto{\pgfqpoint{0.865290in}{0.582899in}}%
\pgfpathlineto{\pgfqpoint{5.825290in}{0.582899in}}%
\pgfpathlineto{\pgfqpoint{5.825290in}{4.278899in}}%
\pgfpathlineto{\pgfqpoint{0.865290in}{4.278899in}}%
\pgfpathlineto{\pgfqpoint{0.865290in}{0.582899in}}%
\pgfpathclose%
\pgfusepath{fill}%
\end{pgfscope}%
\begin{pgfscope}%
\pgfsetbuttcap%
\pgfsetroundjoin%
\definecolor{currentfill}{rgb}{0.000000,0.000000,0.000000}%
\pgfsetfillcolor{currentfill}%
\pgfsetlinewidth{0.803000pt}%
\definecolor{currentstroke}{rgb}{0.000000,0.000000,0.000000}%
\pgfsetstrokecolor{currentstroke}%
\pgfsetdash{}{0pt}%
\pgfsys@defobject{currentmarker}{\pgfqpoint{0.000000in}{-0.048611in}}{\pgfqpoint{0.000000in}{0.000000in}}{%
\pgfpathmoveto{\pgfqpoint{0.000000in}{0.000000in}}%
\pgfpathlineto{\pgfqpoint{0.000000in}{-0.048611in}}%
\pgfusepath{stroke,fill}%
}%
\begin{pgfscope}%
\pgfsys@transformshift{1.372563in}{0.582899in}%
\pgfsys@useobject{currentmarker}{}%
\end{pgfscope}%
\end{pgfscope}%
\begin{pgfscope}%
\definecolor{textcolor}{rgb}{0.000000,0.000000,0.000000}%
\pgfsetstrokecolor{textcolor}%
\pgfsetfillcolor{textcolor}%
\pgftext[x=1.372563in,y=0.485677in,,top]{\color{textcolor}{\sffamily\fontsize{16.000000}{19.200000}\selectfont\catcode`\^=\active\def^{\ifmmode\sp\else\^{}\fi}\catcode`\%=\active\def
\end{pgfscope}%
\begin{pgfscope}%
\pgfsetbuttcap%
\pgfsetroundjoin%
\definecolor{currentfill}{rgb}{0.000000,0.000000,0.000000}%
\pgfsetfillcolor{currentfill}%
\pgfsetlinewidth{0.803000pt}%
\definecolor{currentstroke}{rgb}{0.000000,0.000000,0.000000}%
\pgfsetstrokecolor{currentstroke}%
\pgfsetdash{}{0pt}%
\pgfsys@defobject{currentmarker}{\pgfqpoint{0.000000in}{-0.048611in}}{\pgfqpoint{0.000000in}{0.000000in}}{%
\pgfpathmoveto{\pgfqpoint{0.000000in}{0.000000in}}%
\pgfpathlineto{\pgfqpoint{0.000000in}{-0.048611in}}%
\pgfusepath{stroke,fill}%
}%
\begin{pgfscope}%
\pgfsys@transformshift{1.842260in}{0.582899in}%
\pgfsys@useobject{currentmarker}{}%
\end{pgfscope}%
\end{pgfscope}%
\begin{pgfscope}%
\definecolor{textcolor}{rgb}{0.000000,0.000000,0.000000}%
\pgfsetstrokecolor{textcolor}%
\pgfsetfillcolor{textcolor}%
\pgftext[x=1.842260in,y=0.485677in,,top]{\color{textcolor}{\sffamily\fontsize{16.000000}{19.200000}\selectfont\catcode`\^=\active\def^{\ifmmode\sp\else\^{}\fi}\catcode`\%=\active\def
\end{pgfscope}%
\begin{pgfscope}%
\pgfsetbuttcap%
\pgfsetroundjoin%
\definecolor{currentfill}{rgb}{0.000000,0.000000,0.000000}%
\pgfsetfillcolor{currentfill}%
\pgfsetlinewidth{0.803000pt}%
\definecolor{currentstroke}{rgb}{0.000000,0.000000,0.000000}%
\pgfsetstrokecolor{currentstroke}%
\pgfsetdash{}{0pt}%
\pgfsys@defobject{currentmarker}{\pgfqpoint{0.000000in}{-0.048611in}}{\pgfqpoint{0.000000in}{0.000000in}}{%
\pgfpathmoveto{\pgfqpoint{0.000000in}{0.000000in}}%
\pgfpathlineto{\pgfqpoint{0.000000in}{-0.048611in}}%
\pgfusepath{stroke,fill}%
}%
\begin{pgfscope}%
\pgfsys@transformshift{2.311957in}{0.582899in}%
\pgfsys@useobject{currentmarker}{}%
\end{pgfscope}%
\end{pgfscope}%
\begin{pgfscope}%
\definecolor{textcolor}{rgb}{0.000000,0.000000,0.000000}%
\pgfsetstrokecolor{textcolor}%
\pgfsetfillcolor{textcolor}%
\pgftext[x=2.311957in,y=0.485677in,,top]{\color{textcolor}{\sffamily\fontsize{16.000000}{19.200000}\selectfont\catcode`\^=\active\def^{\ifmmode\sp\else\^{}\fi}\catcode`\%=\active\def
\end{pgfscope}%
\begin{pgfscope}%
\pgfsetbuttcap%
\pgfsetroundjoin%
\definecolor{currentfill}{rgb}{0.000000,0.000000,0.000000}%
\pgfsetfillcolor{currentfill}%
\pgfsetlinewidth{0.803000pt}%
\definecolor{currentstroke}{rgb}{0.000000,0.000000,0.000000}%
\pgfsetstrokecolor{currentstroke}%
\pgfsetdash{}{0pt}%
\pgfsys@defobject{currentmarker}{\pgfqpoint{0.000000in}{-0.048611in}}{\pgfqpoint{0.000000in}{0.000000in}}{%
\pgfpathmoveto{\pgfqpoint{0.000000in}{0.000000in}}%
\pgfpathlineto{\pgfqpoint{0.000000in}{-0.048611in}}%
\pgfusepath{stroke,fill}%
}%
\begin{pgfscope}%
\pgfsys@transformshift{2.781654in}{0.582899in}%
\pgfsys@useobject{currentmarker}{}%
\end{pgfscope}%
\end{pgfscope}%
\begin{pgfscope}%
\definecolor{textcolor}{rgb}{0.000000,0.000000,0.000000}%
\pgfsetstrokecolor{textcolor}%
\pgfsetfillcolor{textcolor}%
\pgftext[x=2.781654in,y=0.485677in,,top]{\color{textcolor}{\sffamily\fontsize{16.000000}{19.200000}\selectfont\catcode`\^=\active\def^{\ifmmode\sp\else\^{}\fi}\catcode`\%=\active\def
\end{pgfscope}%
\begin{pgfscope}%
\pgfsetbuttcap%
\pgfsetroundjoin%
\definecolor{currentfill}{rgb}{0.000000,0.000000,0.000000}%
\pgfsetfillcolor{currentfill}%
\pgfsetlinewidth{0.803000pt}%
\definecolor{currentstroke}{rgb}{0.000000,0.000000,0.000000}%
\pgfsetstrokecolor{currentstroke}%
\pgfsetdash{}{0pt}%
\pgfsys@defobject{currentmarker}{\pgfqpoint{0.000000in}{-0.048611in}}{\pgfqpoint{0.000000in}{0.000000in}}{%
\pgfpathmoveto{\pgfqpoint{0.000000in}{0.000000in}}%
\pgfpathlineto{\pgfqpoint{0.000000in}{-0.048611in}}%
\pgfusepath{stroke,fill}%
}%
\begin{pgfscope}%
\pgfsys@transformshift{3.251351in}{0.582899in}%
\pgfsys@useobject{currentmarker}{}%
\end{pgfscope}%
\end{pgfscope}%
\begin{pgfscope}%
\definecolor{textcolor}{rgb}{0.000000,0.000000,0.000000}%
\pgfsetstrokecolor{textcolor}%
\pgfsetfillcolor{textcolor}%
\pgftext[x=3.251351in,y=0.485677in,,top]{\color{textcolor}{\sffamily\fontsize{16.000000}{19.200000}\selectfont\catcode`\^=\active\def^{\ifmmode\sp\else\^{}\fi}\catcode`\%=\active\def
\end{pgfscope}%
\begin{pgfscope}%
\pgfsetbuttcap%
\pgfsetroundjoin%
\definecolor{currentfill}{rgb}{0.000000,0.000000,0.000000}%
\pgfsetfillcolor{currentfill}%
\pgfsetlinewidth{0.803000pt}%
\definecolor{currentstroke}{rgb}{0.000000,0.000000,0.000000}%
\pgfsetstrokecolor{currentstroke}%
\pgfsetdash{}{0pt}%
\pgfsys@defobject{currentmarker}{\pgfqpoint{0.000000in}{-0.048611in}}{\pgfqpoint{0.000000in}{0.000000in}}{%
\pgfpathmoveto{\pgfqpoint{0.000000in}{0.000000in}}%
\pgfpathlineto{\pgfqpoint{0.000000in}{-0.048611in}}%
\pgfusepath{stroke,fill}%
}%
\begin{pgfscope}%
\pgfsys@transformshift{3.721048in}{0.582899in}%
\pgfsys@useobject{currentmarker}{}%
\end{pgfscope}%
\end{pgfscope}%
\begin{pgfscope}%
\definecolor{textcolor}{rgb}{0.000000,0.000000,0.000000}%
\pgfsetstrokecolor{textcolor}%
\pgfsetfillcolor{textcolor}%
\pgftext[x=3.721048in,y=0.485677in,,top]{\color{textcolor}{\sffamily\fontsize{16.000000}{19.200000}\selectfont\catcode`\^=\active\def^{\ifmmode\sp\else\^{}\fi}\catcode`\%=\active\def
\end{pgfscope}%
\begin{pgfscope}%
\pgfsetbuttcap%
\pgfsetroundjoin%
\definecolor{currentfill}{rgb}{0.000000,0.000000,0.000000}%
\pgfsetfillcolor{currentfill}%
\pgfsetlinewidth{0.803000pt}%
\definecolor{currentstroke}{rgb}{0.000000,0.000000,0.000000}%
\pgfsetstrokecolor{currentstroke}%
\pgfsetdash{}{0pt}%
\pgfsys@defobject{currentmarker}{\pgfqpoint{0.000000in}{-0.048611in}}{\pgfqpoint{0.000000in}{0.000000in}}{%
\pgfpathmoveto{\pgfqpoint{0.000000in}{0.000000in}}%
\pgfpathlineto{\pgfqpoint{0.000000in}{-0.048611in}}%
\pgfusepath{stroke,fill}%
}%
\begin{pgfscope}%
\pgfsys@transformshift{4.190745in}{0.582899in}%
\pgfsys@useobject{currentmarker}{}%
\end{pgfscope}%
\end{pgfscope}%
\begin{pgfscope}%
\definecolor{textcolor}{rgb}{0.000000,0.000000,0.000000}%
\pgfsetstrokecolor{textcolor}%
\pgfsetfillcolor{textcolor}%
\pgftext[x=4.190745in,y=0.485677in,,top]{\color{textcolor}{\sffamily\fontsize{16.000000}{19.200000}\selectfont\catcode`\^=\active\def^{\ifmmode\sp\else\^{}\fi}\catcode`\%=\active\def
\end{pgfscope}%
\begin{pgfscope}%
\pgfsetbuttcap%
\pgfsetroundjoin%
\definecolor{currentfill}{rgb}{0.000000,0.000000,0.000000}%
\pgfsetfillcolor{currentfill}%
\pgfsetlinewidth{0.803000pt}%
\definecolor{currentstroke}{rgb}{0.000000,0.000000,0.000000}%
\pgfsetstrokecolor{currentstroke}%
\pgfsetdash{}{0pt}%
\pgfsys@defobject{currentmarker}{\pgfqpoint{0.000000in}{-0.048611in}}{\pgfqpoint{0.000000in}{0.000000in}}{%
\pgfpathmoveto{\pgfqpoint{0.000000in}{0.000000in}}%
\pgfpathlineto{\pgfqpoint{0.000000in}{-0.048611in}}%
\pgfusepath{stroke,fill}%
}%
\begin{pgfscope}%
\pgfsys@transformshift{4.660442in}{0.582899in}%
\pgfsys@useobject{currentmarker}{}%
\end{pgfscope}%
\end{pgfscope}%
\begin{pgfscope}%
\definecolor{textcolor}{rgb}{0.000000,0.000000,0.000000}%
\pgfsetstrokecolor{textcolor}%
\pgfsetfillcolor{textcolor}%
\pgftext[x=4.660442in,y=0.485677in,,top]{\color{textcolor}{\sffamily\fontsize{16.000000}{19.200000}\selectfont\catcode`\^=\active\def^{\ifmmode\sp\else\^{}\fi}\catcode`\%=\active\def
\end{pgfscope}%
\begin{pgfscope}%
\pgfsetbuttcap%
\pgfsetroundjoin%
\definecolor{currentfill}{rgb}{0.000000,0.000000,0.000000}%
\pgfsetfillcolor{currentfill}%
\pgfsetlinewidth{0.803000pt}%
\definecolor{currentstroke}{rgb}{0.000000,0.000000,0.000000}%
\pgfsetstrokecolor{currentstroke}%
\pgfsetdash{}{0pt}%
\pgfsys@defobject{currentmarker}{\pgfqpoint{0.000000in}{-0.048611in}}{\pgfqpoint{0.000000in}{0.000000in}}{%
\pgfpathmoveto{\pgfqpoint{0.000000in}{0.000000in}}%
\pgfpathlineto{\pgfqpoint{0.000000in}{-0.048611in}}%
\pgfusepath{stroke,fill}%
}%
\begin{pgfscope}%
\pgfsys@transformshift{5.130139in}{0.582899in}%
\pgfsys@useobject{currentmarker}{}%
\end{pgfscope}%
\end{pgfscope}%
\begin{pgfscope}%
\definecolor{textcolor}{rgb}{0.000000,0.000000,0.000000}%
\pgfsetstrokecolor{textcolor}%
\pgfsetfillcolor{textcolor}%
\pgftext[x=5.130139in,y=0.485677in,,top]{\color{textcolor}{\sffamily\fontsize{16.000000}{19.200000}\selectfont\catcode`\^=\active\def^{\ifmmode\sp\else\^{}\fi}\catcode`\%=\active\def
\end{pgfscope}%
\begin{pgfscope}%
\pgfsetbuttcap%
\pgfsetroundjoin%
\definecolor{currentfill}{rgb}{0.000000,0.000000,0.000000}%
\pgfsetfillcolor{currentfill}%
\pgfsetlinewidth{0.803000pt}%
\definecolor{currentstroke}{rgb}{0.000000,0.000000,0.000000}%
\pgfsetstrokecolor{currentstroke}%
\pgfsetdash{}{0pt}%
\pgfsys@defobject{currentmarker}{\pgfqpoint{0.000000in}{-0.048611in}}{\pgfqpoint{0.000000in}{0.000000in}}{%
\pgfpathmoveto{\pgfqpoint{0.000000in}{0.000000in}}%
\pgfpathlineto{\pgfqpoint{0.000000in}{-0.048611in}}%
\pgfusepath{stroke,fill}%
}%
\begin{pgfscope}%
\pgfsys@transformshift{5.599836in}{0.582899in}%
\pgfsys@useobject{currentmarker}{}%
\end{pgfscope}%
\end{pgfscope}%
\begin{pgfscope}%
\definecolor{textcolor}{rgb}{0.000000,0.000000,0.000000}%
\pgfsetstrokecolor{textcolor}%
\pgfsetfillcolor{textcolor}%
\pgftext[x=5.599836in,y=0.485677in,,top]{\color{textcolor}{\sffamily\fontsize{16.000000}{19.200000}\selectfont\catcode`\^=\active\def^{\ifmmode\sp\else\^{}\fi}\catcode`\%=\active\def
\end{pgfscope}%
\begin{pgfscope}%
\definecolor{textcolor}{rgb}{0.000000,0.000000,0.000000}%
\pgfsetstrokecolor{textcolor}%
\pgfsetfillcolor{textcolor}%
\pgftext[x=3.345290in,y=0.215061in,,top]{\color{textcolor}{\sffamily\fontsize{16.000000}{19.200000}\selectfont\catcode`\^=\active\def^{\ifmmode\sp\else\^{}\fi}\catcode`\%=\active\def
\end{pgfscope}%
\begin{pgfscope}%
\pgfsetbuttcap%
\pgfsetroundjoin%
\definecolor{currentfill}{rgb}{0.000000,0.000000,0.000000}%
\pgfsetfillcolor{currentfill}%
\pgfsetlinewidth{0.803000pt}%
\definecolor{currentstroke}{rgb}{0.000000,0.000000,0.000000}%
\pgfsetstrokecolor{currentstroke}%
\pgfsetdash{}{0pt}%
\pgfsys@defobject{currentmarker}{\pgfqpoint{-0.048611in}{0.000000in}}{\pgfqpoint{-0.000000in}{0.000000in}}{%
\pgfpathmoveto{\pgfqpoint{-0.000000in}{0.000000in}}%
\pgfpathlineto{\pgfqpoint{-0.048611in}{0.000000in}}%
\pgfusepath{stroke,fill}%
}%
\begin{pgfscope}%
\pgfsys@transformshift{0.865290in}{1.079037in}%
\pgfsys@useobject{currentmarker}{}%
\end{pgfscope}%
\end{pgfscope}%
\begin{pgfscope}%
\definecolor{textcolor}{rgb}{0.000000,0.000000,0.000000}%
\pgfsetstrokecolor{textcolor}%
\pgfsetfillcolor{textcolor}%
\pgftext[x=0.270616in, y=0.994619in, left, base]{\color{textcolor}{\sffamily\fontsize{16.000000}{19.200000}\selectfont\catcode`\^=\active\def^{\ifmmode\sp\else\^{}\fi}\catcode`\%=\active\def
\end{pgfscope}%
\begin{pgfscope}%
\pgfsetbuttcap%
\pgfsetroundjoin%
\definecolor{currentfill}{rgb}{0.000000,0.000000,0.000000}%
\pgfsetfillcolor{currentfill}%
\pgfsetlinewidth{0.803000pt}%
\definecolor{currentstroke}{rgb}{0.000000,0.000000,0.000000}%
\pgfsetstrokecolor{currentstroke}%
\pgfsetdash{}{0pt}%
\pgfsys@defobject{currentmarker}{\pgfqpoint{-0.048611in}{0.000000in}}{\pgfqpoint{-0.000000in}{0.000000in}}{%
\pgfpathmoveto{\pgfqpoint{-0.000000in}{0.000000in}}%
\pgfpathlineto{\pgfqpoint{-0.048611in}{0.000000in}}%
\pgfusepath{stroke,fill}%
}%
\begin{pgfscope}%
\pgfsys@transformshift{0.865290in}{1.663064in}%
\pgfsys@useobject{currentmarker}{}%
\end{pgfscope}%
\end{pgfscope}%
\begin{pgfscope}%
\definecolor{textcolor}{rgb}{0.000000,0.000000,0.000000}%
\pgfsetstrokecolor{textcolor}%
\pgfsetfillcolor{textcolor}%
\pgftext[x=0.346658in, y=1.578646in, left, base]{\color{textcolor}{\sffamily\fontsize{16.000000}{19.200000}\selectfont\catcode`\^=\active\def^{\ifmmode\sp\else\^{}\fi}\catcode`\%=\active\def
\end{pgfscope}%
\begin{pgfscope}%
\pgfsetbuttcap%
\pgfsetroundjoin%
\definecolor{currentfill}{rgb}{0.000000,0.000000,0.000000}%
\pgfsetfillcolor{currentfill}%
\pgfsetlinewidth{0.803000pt}%
\definecolor{currentstroke}{rgb}{0.000000,0.000000,0.000000}%
\pgfsetstrokecolor{currentstroke}%
\pgfsetdash{}{0pt}%
\pgfsys@defobject{currentmarker}{\pgfqpoint{-0.048611in}{0.000000in}}{\pgfqpoint{-0.000000in}{0.000000in}}{%
\pgfpathmoveto{\pgfqpoint{-0.000000in}{0.000000in}}%
\pgfpathlineto{\pgfqpoint{-0.048611in}{0.000000in}}%
\pgfusepath{stroke,fill}%
}%
\begin{pgfscope}%
\pgfsys@transformshift{0.865290in}{2.247091in}%
\pgfsys@useobject{currentmarker}{}%
\end{pgfscope}%
\end{pgfscope}%
\begin{pgfscope}%
\definecolor{textcolor}{rgb}{0.000000,0.000000,0.000000}%
\pgfsetstrokecolor{textcolor}%
\pgfsetfillcolor{textcolor}%
\pgftext[x=0.346658in, y=2.162673in, left, base]{\color{textcolor}{\sffamily\fontsize{16.000000}{19.200000}\selectfont\catcode`\^=\active\def^{\ifmmode\sp\else\^{}\fi}\catcode`\%=\active\def
\end{pgfscope}%
\begin{pgfscope}%
\pgfsetbuttcap%
\pgfsetroundjoin%
\definecolor{currentfill}{rgb}{0.000000,0.000000,0.000000}%
\pgfsetfillcolor{currentfill}%
\pgfsetlinewidth{0.803000pt}%
\definecolor{currentstroke}{rgb}{0.000000,0.000000,0.000000}%
\pgfsetstrokecolor{currentstroke}%
\pgfsetdash{}{0pt}%
\pgfsys@defobject{currentmarker}{\pgfqpoint{-0.048611in}{0.000000in}}{\pgfqpoint{-0.000000in}{0.000000in}}{%
\pgfpathmoveto{\pgfqpoint{-0.000000in}{0.000000in}}%
\pgfpathlineto{\pgfqpoint{-0.048611in}{0.000000in}}%
\pgfusepath{stroke,fill}%
}%
\begin{pgfscope}%
\pgfsys@transformshift{0.865290in}{2.831118in}%
\pgfsys@useobject{currentmarker}{}%
\end{pgfscope}%
\end{pgfscope}%
\begin{pgfscope}%
\definecolor{textcolor}{rgb}{0.000000,0.000000,0.000000}%
\pgfsetstrokecolor{textcolor}%
\pgfsetfillcolor{textcolor}%
\pgftext[x=0.464945in, y=2.746700in, left, base]{\color{textcolor}{\sffamily\fontsize{16.000000}{19.200000}\selectfont\catcode`\^=\active\def^{\ifmmode\sp\else\^{}\fi}\catcode`\%=\active\def
\end{pgfscope}%
\begin{pgfscope}%
\pgfsetbuttcap%
\pgfsetroundjoin%
\definecolor{currentfill}{rgb}{0.000000,0.000000,0.000000}%
\pgfsetfillcolor{currentfill}%
\pgfsetlinewidth{0.803000pt}%
\definecolor{currentstroke}{rgb}{0.000000,0.000000,0.000000}%
\pgfsetstrokecolor{currentstroke}%
\pgfsetdash{}{0pt}%
\pgfsys@defobject{currentmarker}{\pgfqpoint{-0.048611in}{0.000000in}}{\pgfqpoint{-0.000000in}{0.000000in}}{%
\pgfpathmoveto{\pgfqpoint{-0.000000in}{0.000000in}}%
\pgfpathlineto{\pgfqpoint{-0.048611in}{0.000000in}}%
\pgfusepath{stroke,fill}%
}%
\begin{pgfscope}%
\pgfsys@transformshift{0.865290in}{3.415145in}%
\pgfsys@useobject{currentmarker}{}%
\end{pgfscope}%
\end{pgfscope}%
\begin{pgfscope}%
\definecolor{textcolor}{rgb}{0.000000,0.000000,0.000000}%
\pgfsetstrokecolor{textcolor}%
\pgfsetfillcolor{textcolor}%
\pgftext[x=0.464945in, y=3.330726in, left, base]{\color{textcolor}{\sffamily\fontsize{16.000000}{19.200000}\selectfont\catcode`\^=\active\def^{\ifmmode\sp\else\^{}\fi}\catcode`\%=\active\def
\end{pgfscope}%
\begin{pgfscope}%
\pgfsetbuttcap%
\pgfsetroundjoin%
\definecolor{currentfill}{rgb}{0.000000,0.000000,0.000000}%
\pgfsetfillcolor{currentfill}%
\pgfsetlinewidth{0.803000pt}%
\definecolor{currentstroke}{rgb}{0.000000,0.000000,0.000000}%
\pgfsetstrokecolor{currentstroke}%
\pgfsetdash{}{0pt}%
\pgfsys@defobject{currentmarker}{\pgfqpoint{-0.048611in}{0.000000in}}{\pgfqpoint{-0.000000in}{0.000000in}}{%
\pgfpathmoveto{\pgfqpoint{-0.000000in}{0.000000in}}%
\pgfpathlineto{\pgfqpoint{-0.048611in}{0.000000in}}%
\pgfusepath{stroke,fill}%
}%
\begin{pgfscope}%
\pgfsys@transformshift{0.865290in}{3.999172in}%
\pgfsys@useobject{currentmarker}{}%
\end{pgfscope}%
\end{pgfscope}%
\begin{pgfscope}%
\definecolor{textcolor}{rgb}{0.000000,0.000000,0.000000}%
\pgfsetstrokecolor{textcolor}%
\pgfsetfillcolor{textcolor}%
\pgftext[x=0.388903in, y=3.914753in, left, base]{\color{textcolor}{\sffamily\fontsize{16.000000}{19.200000}\selectfont\catcode`\^=\active\def^{\ifmmode\sp\else\^{}\fi}\catcode`\%=\active\def
\end{pgfscope}%
\begin{pgfscope}%
\definecolor{textcolor}{rgb}{0.000000,0.000000,0.000000}%
\pgfsetstrokecolor{textcolor}%
\pgfsetfillcolor{textcolor}%
\pgftext[x=0.215061in,y=2.430899in,,bottom,rotate=90.000000]{\color{textcolor}{\sffamily\fontsize{16.000000}{19.200000}\selectfont\catcode`\^=\active\def^{\ifmmode\sp\else\^{}\fi}\catcode`\%=\active\def
\end{pgfscope}%
\begin{pgfscope}%
\pgfpathrectangle{\pgfqpoint{0.865290in}{0.582899in}}{\pgfqpoint{4.960000in}{3.696000in}}%
\pgfusepath{clip}%
\pgfsetrectcap%
\pgfsetroundjoin%
\pgfsetlinewidth{1.505625pt}%
\definecolor{currentstroke}{rgb}{0.121569,0.466667,0.705882}%
\pgfsetstrokecolor{currentstroke}%
\pgfsetdash{}{0pt}%
\pgfpathmoveto{\pgfqpoint{1.090745in}{4.110899in}}%
\pgfpathlineto{\pgfqpoint{1.278623in}{2.549182in}}%
\pgfpathlineto{\pgfqpoint{1.466502in}{2.605174in}}%
\pgfpathlineto{\pgfqpoint{1.654381in}{2.536755in}}%
\pgfpathlineto{\pgfqpoint{1.842260in}{2.295942in}}%
\pgfpathlineto{\pgfqpoint{2.030139in}{1.904189in}}%
\pgfpathlineto{\pgfqpoint{2.218017in}{1.689656in}}%
\pgfpathlineto{\pgfqpoint{2.405896in}{1.410982in}}%
\pgfpathlineto{\pgfqpoint{2.593775in}{1.208031in}}%
\pgfpathlineto{\pgfqpoint{2.781654in}{0.994478in}}%
\pgfpathlineto{\pgfqpoint{2.969533in}{0.802813in}}%
\pgfpathlineto{\pgfqpoint{3.157411in}{0.768331in}}%
\pgfpathlineto{\pgfqpoint{3.345290in}{0.768175in}}%
\pgfpathlineto{\pgfqpoint{3.533169in}{0.768175in}}%
\pgfpathlineto{\pgfqpoint{3.721048in}{0.768175in}}%
\pgfpathlineto{\pgfqpoint{3.908927in}{0.768175in}}%
\pgfpathlineto{\pgfqpoint{4.096805in}{0.768175in}}%
\pgfpathlineto{\pgfqpoint{4.284684in}{0.768175in}}%
\pgfpathlineto{\pgfqpoint{4.472563in}{0.768175in}}%
\pgfpathlineto{\pgfqpoint{4.660442in}{0.768175in}}%
\pgfpathlineto{\pgfqpoint{4.848320in}{0.768175in}}%
\pgfpathlineto{\pgfqpoint{5.036199in}{0.768175in}}%
\pgfpathlineto{\pgfqpoint{5.224078in}{0.768175in}}%
\pgfpathlineto{\pgfqpoint{5.411957in}{0.768175in}}%
\pgfpathlineto{\pgfqpoint{5.599836in}{0.768175in}}%
\pgfusepath{stroke}%
\end{pgfscope}%
\begin{pgfscope}%
\pgfpathrectangle{\pgfqpoint{0.865290in}{0.582899in}}{\pgfqpoint{4.960000in}{3.696000in}}%
\pgfusepath{clip}%
\pgfsetrectcap%
\pgfsetroundjoin%
\pgfsetlinewidth{1.505625pt}%
\definecolor{currentstroke}{rgb}{1.000000,0.498039,0.054902}%
\pgfsetstrokecolor{currentstroke}%
\pgfsetdash{}{0pt}%
\pgfpathmoveto{\pgfqpoint{1.090745in}{2.527058in}}%
\pgfpathlineto{\pgfqpoint{1.278623in}{2.372662in}}%
\pgfpathlineto{\pgfqpoint{1.466502in}{2.028422in}}%
\pgfpathlineto{\pgfqpoint{1.654381in}{1.763959in}}%
\pgfpathlineto{\pgfqpoint{1.842260in}{1.548323in}}%
\pgfpathlineto{\pgfqpoint{2.030139in}{1.266812in}}%
\pgfpathlineto{\pgfqpoint{2.218017in}{0.968120in}}%
\pgfpathlineto{\pgfqpoint{2.405896in}{0.799758in}}%
\pgfpathlineto{\pgfqpoint{2.593775in}{0.754396in}}%
\pgfpathlineto{\pgfqpoint{2.781654in}{0.754282in}}%
\pgfpathlineto{\pgfqpoint{2.969533in}{0.754298in}}%
\pgfpathlineto{\pgfqpoint{3.157411in}{0.750899in}}%
\pgfpathlineto{\pgfqpoint{3.345290in}{0.751104in}}%
\pgfpathlineto{\pgfqpoint{3.533169in}{0.751104in}}%
\pgfpathlineto{\pgfqpoint{3.721048in}{0.751121in}}%
\pgfpathlineto{\pgfqpoint{3.908927in}{0.750899in}}%
\pgfpathlineto{\pgfqpoint{4.096805in}{0.750899in}}%
\pgfpathlineto{\pgfqpoint{4.284684in}{0.750899in}}%
\pgfpathlineto{\pgfqpoint{4.472563in}{0.750916in}}%
\pgfpathlineto{\pgfqpoint{4.660442in}{0.750899in}}%
\pgfpathlineto{\pgfqpoint{4.848320in}{0.750899in}}%
\pgfpathlineto{\pgfqpoint{5.036199in}{0.750899in}}%
\pgfpathlineto{\pgfqpoint{5.224078in}{0.750899in}}%
\pgfpathlineto{\pgfqpoint{5.411957in}{0.750899in}}%
\pgfpathlineto{\pgfqpoint{5.599836in}{0.750899in}}%
\pgfusepath{stroke}%
\end{pgfscope}%
\begin{pgfscope}%
\pgfsetrectcap%
\pgfsetmiterjoin%
\pgfsetlinewidth{0.803000pt}%
\definecolor{currentstroke}{rgb}{0.000000,0.000000,0.000000}%
\pgfsetstrokecolor{currentstroke}%
\pgfsetdash{}{0pt}%
\pgfpathmoveto{\pgfqpoint{0.865290in}{0.582899in}}%
\pgfpathlineto{\pgfqpoint{0.865290in}{4.278899in}}%
\pgfusepath{stroke}%
\end{pgfscope}%
\begin{pgfscope}%
\pgfsetrectcap%
\pgfsetmiterjoin%
\pgfsetlinewidth{0.803000pt}%
\definecolor{currentstroke}{rgb}{0.000000,0.000000,0.000000}%
\pgfsetstrokecolor{currentstroke}%
\pgfsetdash{}{0pt}%
\pgfpathmoveto{\pgfqpoint{5.825290in}{0.582899in}}%
\pgfpathlineto{\pgfqpoint{5.825290in}{4.278899in}}%
\pgfusepath{stroke}%
\end{pgfscope}%
\begin{pgfscope}%
\pgfsetrectcap%
\pgfsetmiterjoin%
\pgfsetlinewidth{0.803000pt}%
\definecolor{currentstroke}{rgb}{0.000000,0.000000,0.000000}%
\pgfsetstrokecolor{currentstroke}%
\pgfsetdash{}{0pt}%
\pgfpathmoveto{\pgfqpoint{0.865290in}{0.582899in}}%
\pgfpathlineto{\pgfqpoint{5.825290in}{0.582899in}}%
\pgfusepath{stroke}%
\end{pgfscope}%
\begin{pgfscope}%
\pgfsetrectcap%
\pgfsetmiterjoin%
\pgfsetlinewidth{0.803000pt}%
\definecolor{currentstroke}{rgb}{0.000000,0.000000,0.000000}%
\pgfsetstrokecolor{currentstroke}%
\pgfsetdash{}{0pt}%
\pgfpathmoveto{\pgfqpoint{0.865290in}{4.278899in}}%
\pgfpathlineto{\pgfqpoint{5.825290in}{4.278899in}}%
\pgfusepath{stroke}%
\end{pgfscope}%
\begin{pgfscope}%
\pgfsetbuttcap%
\pgfsetmiterjoin%
\definecolor{currentfill}{rgb}{1.000000,1.000000,1.000000}%
\pgfsetfillcolor{currentfill}%
\pgfsetfillopacity{0.800000}%
\pgfsetlinewidth{1.003750pt}%
\definecolor{currentstroke}{rgb}{0.800000,0.800000,0.800000}%
\pgfsetstrokecolor{currentstroke}%
\pgfsetstrokeopacity{0.800000}%
\pgfsetdash{}{0pt}%
\pgfpathmoveto{\pgfqpoint{3.158927in}{3.448778in}}%
\pgfpathlineto{\pgfqpoint{5.669735in}{3.448778in}}%
\pgfpathquadraticcurveto{\pgfqpoint{5.714179in}{3.448778in}}{\pgfqpoint{5.714179in}{3.493222in}}%
\pgfpathlineto{\pgfqpoint{5.714179in}{4.123343in}}%
\pgfpathquadraticcurveto{\pgfqpoint{5.714179in}{4.167788in}}{\pgfqpoint{5.669735in}{4.167788in}}%
\pgfpathlineto{\pgfqpoint{3.158927in}{4.167788in}}%
\pgfpathquadraticcurveto{\pgfqpoint{3.114483in}{4.167788in}}{\pgfqpoint{3.114483in}{4.123343in}}%
\pgfpathlineto{\pgfqpoint{3.114483in}{3.493222in}}%
\pgfpathquadraticcurveto{\pgfqpoint{3.114483in}{3.448778in}}{\pgfqpoint{3.158927in}{3.448778in}}%
\pgfpathlineto{\pgfqpoint{3.158927in}{3.448778in}}%
\pgfpathclose%
\pgfusepath{stroke,fill}%
\end{pgfscope}%
\begin{pgfscope}%
\pgfsetrectcap%
\pgfsetroundjoin%
\pgfsetlinewidth{1.505625pt}%
\definecolor{currentstroke}{rgb}{0.121569,0.466667,0.705882}%
\pgfsetstrokecolor{currentstroke}%
\pgfsetdash{}{0pt}%
\pgfpathmoveto{\pgfqpoint{3.203372in}{3.987840in}}%
\pgfpathlineto{\pgfqpoint{3.425594in}{3.987840in}}%
\pgfpathlineto{\pgfqpoint{3.647816in}{3.987840in}}%
\pgfusepath{stroke}%
\end{pgfscope}%
\begin{pgfscope}%
\definecolor{textcolor}{rgb}{0.000000,0.000000,0.000000}%
\pgfsetstrokecolor{textcolor}%
\pgfsetfillcolor{textcolor}%
\pgftext[x=3.825594in,y=3.910062in,left,base]{\color{textcolor}{\sffamily\fontsize{16.000000}{19.200000}\selectfont\catcode`\^=\active\def^{\ifmmode\sp\else\^{}\fi}\catcode`\%=\active\def
\end{pgfscope}%
\begin{pgfscope}%
\pgfsetrectcap%
\pgfsetroundjoin%
\pgfsetlinewidth{1.505625pt}%
\definecolor{currentstroke}{rgb}{1.000000,0.498039,0.054902}%
\pgfsetstrokecolor{currentstroke}%
\pgfsetdash{}{0pt}%
\pgfpathmoveto{\pgfqpoint{3.203372in}{3.661668in}}%
\pgfpathlineto{\pgfqpoint{3.425594in}{3.661668in}}%
\pgfpathlineto{\pgfqpoint{3.647816in}{3.661668in}}%
\pgfusepath{stroke}%
\end{pgfscope}%
\begin{pgfscope}%
\definecolor{textcolor}{rgb}{0.000000,0.000000,0.000000}%
\pgfsetstrokecolor{textcolor}%
\pgfsetfillcolor{textcolor}%
\pgftext[x=3.825594in,y=3.583890in,left,base]{\color{textcolor}{\sffamily\fontsize{16.000000}{19.200000}\selectfont\catcode`\^=\active\def^{\ifmmode\sp\else\^{}\fi}\catcode`\%=\active\def
\end{pgfscope}%
\end{pgfpicture}%
\makeatother%
\endgroup%

%% file: figs/network_conv_ei_Autobahn.pgf
\begingroup%
\makeatletter%
\begin{pgfpicture}%
\pgfpathrectangle{\pgfpointorigin}{\pgfqpoint{5.825290in}{4.278899in}}%
\pgfusepath{use as bounding box, clip}%
\begin{pgfscope}%
\pgfsetbuttcap%
\pgfsetmiterjoin%
\definecolor{currentfill}{rgb}{1.000000,1.000000,1.000000}%
\pgfsetfillcolor{currentfill}%
\pgfsetlinewidth{0.000000pt}%
\definecolor{currentstroke}{rgb}{1.000000,1.000000,1.000000}%
\pgfsetstrokecolor{currentstroke}%
\pgfsetdash{}{0pt}%
\pgfpathmoveto{\pgfqpoint{0.000000in}{0.000000in}}%
\pgfpathlineto{\pgfqpoint{5.825290in}{0.000000in}}%
\pgfpathlineto{\pgfqpoint{5.825290in}{4.278899in}}%
\pgfpathlineto{\pgfqpoint{0.000000in}{4.278899in}}%
\pgfpathlineto{\pgfqpoint{0.000000in}{0.000000in}}%
\pgfpathclose%
\pgfusepath{fill}%
\end{pgfscope}%
\begin{pgfscope}%
\pgfsetbuttcap%
\pgfsetmiterjoin%
\definecolor{currentfill}{rgb}{1.000000,1.000000,1.000000}%
\pgfsetfillcolor{currentfill}%
\pgfsetlinewidth{0.000000pt}%
\definecolor{currentstroke}{rgb}{0.000000,0.000000,0.000000}%
\pgfsetstrokecolor{currentstroke}%
\pgfsetstrokeopacity{0.000000}%
\pgfsetdash{}{0pt}%
\pgfpathmoveto{\pgfqpoint{0.865290in}{0.582899in}}%
\pgfpathlineto{\pgfqpoint{5.825290in}{0.582899in}}%
\pgfpathlineto{\pgfqpoint{5.825290in}{4.278899in}}%
\pgfpathlineto{\pgfqpoint{0.865290in}{4.278899in}}%
\pgfpathlineto{\pgfqpoint{0.865290in}{0.582899in}}%
\pgfpathclose%
\pgfusepath{fill}%
\end{pgfscope}%
\begin{pgfscope}%
\pgfsetbuttcap%
\pgfsetroundjoin%
\definecolor{currentfill}{rgb}{0.000000,0.000000,0.000000}%
\pgfsetfillcolor{currentfill}%
\pgfsetlinewidth{0.803000pt}%
\definecolor{currentstroke}{rgb}{0.000000,0.000000,0.000000}%
\pgfsetstrokecolor{currentstroke}%
\pgfsetdash{}{0pt}%
\pgfsys@defobject{currentmarker}{\pgfqpoint{0.000000in}{-0.048611in}}{\pgfqpoint{0.000000in}{0.000000in}}{%
\pgfpathmoveto{\pgfqpoint{0.000000in}{0.000000in}}%
\pgfpathlineto{\pgfqpoint{0.000000in}{-0.048611in}}%
\pgfusepath{stroke,fill}%
}%
\begin{pgfscope}%
\pgfsys@transformshift{1.372563in}{0.582899in}%
\pgfsys@useobject{currentmarker}{}%
\end{pgfscope}%
\end{pgfscope}%
\begin{pgfscope}%
\definecolor{textcolor}{rgb}{0.000000,0.000000,0.000000}%
\pgfsetstrokecolor{textcolor}%
\pgfsetfillcolor{textcolor}%
\pgftext[x=1.372563in,y=0.485677in,,top]{\color{textcolor}{\sffamily\fontsize{16.000000}{19.200000}\selectfont\catcode`\^=\active\def^{\ifmmode\sp\else\^{}\fi}\catcode`\%=\active\def
\end{pgfscope}%
\begin{pgfscope}%
\pgfsetbuttcap%
\pgfsetroundjoin%
\definecolor{currentfill}{rgb}{0.000000,0.000000,0.000000}%
\pgfsetfillcolor{currentfill}%
\pgfsetlinewidth{0.803000pt}%
\definecolor{currentstroke}{rgb}{0.000000,0.000000,0.000000}%
\pgfsetstrokecolor{currentstroke}%
\pgfsetdash{}{0pt}%
\pgfsys@defobject{currentmarker}{\pgfqpoint{0.000000in}{-0.048611in}}{\pgfqpoint{0.000000in}{0.000000in}}{%
\pgfpathmoveto{\pgfqpoint{0.000000in}{0.000000in}}%
\pgfpathlineto{\pgfqpoint{0.000000in}{-0.048611in}}%
\pgfusepath{stroke,fill}%
}%
\begin{pgfscope}%
\pgfsys@transformshift{1.842260in}{0.582899in}%
\pgfsys@useobject{currentmarker}{}%
\end{pgfscope}%
\end{pgfscope}%
\begin{pgfscope}%
\definecolor{textcolor}{rgb}{0.000000,0.000000,0.000000}%
\pgfsetstrokecolor{textcolor}%
\pgfsetfillcolor{textcolor}%
\pgftext[x=1.842260in,y=0.485677in,,top]{\color{textcolor}{\sffamily\fontsize{16.000000}{19.200000}\selectfont\catcode`\^=\active\def^{\ifmmode\sp\else\^{}\fi}\catcode`\%=\active\def
\end{pgfscope}%
\begin{pgfscope}%
\pgfsetbuttcap%
\pgfsetroundjoin%
\definecolor{currentfill}{rgb}{0.000000,0.000000,0.000000}%
\pgfsetfillcolor{currentfill}%
\pgfsetlinewidth{0.803000pt}%
\definecolor{currentstroke}{rgb}{0.000000,0.000000,0.000000}%
\pgfsetstrokecolor{currentstroke}%
\pgfsetdash{}{0pt}%
\pgfsys@defobject{currentmarker}{\pgfqpoint{0.000000in}{-0.048611in}}{\pgfqpoint{0.000000in}{0.000000in}}{%
\pgfpathmoveto{\pgfqpoint{0.000000in}{0.000000in}}%
\pgfpathlineto{\pgfqpoint{0.000000in}{-0.048611in}}%
\pgfusepath{stroke,fill}%
}%
\begin{pgfscope}%
\pgfsys@transformshift{2.311957in}{0.582899in}%
\pgfsys@useobject{currentmarker}{}%
\end{pgfscope}%
\end{pgfscope}%
\begin{pgfscope}%
\definecolor{textcolor}{rgb}{0.000000,0.000000,0.000000}%
\pgfsetstrokecolor{textcolor}%
\pgfsetfillcolor{textcolor}%
\pgftext[x=2.311957in,y=0.485677in,,top]{\color{textcolor}{\sffamily\fontsize{16.000000}{19.200000}\selectfont\catcode`\^=\active\def^{\ifmmode\sp\else\^{}\fi}\catcode`\%=\active\def
\end{pgfscope}%
\begin{pgfscope}%
\pgfsetbuttcap%
\pgfsetroundjoin%
\definecolor{currentfill}{rgb}{0.000000,0.000000,0.000000}%
\pgfsetfillcolor{currentfill}%
\pgfsetlinewidth{0.803000pt}%
\definecolor{currentstroke}{rgb}{0.000000,0.000000,0.000000}%
\pgfsetstrokecolor{currentstroke}%
\pgfsetdash{}{0pt}%
\pgfsys@defobject{currentmarker}{\pgfqpoint{0.000000in}{-0.048611in}}{\pgfqpoint{0.000000in}{0.000000in}}{%
\pgfpathmoveto{\pgfqpoint{0.000000in}{0.000000in}}%
\pgfpathlineto{\pgfqpoint{0.000000in}{-0.048611in}}%
\pgfusepath{stroke,fill}%
}%
\begin{pgfscope}%
\pgfsys@transformshift{2.781654in}{0.582899in}%
\pgfsys@useobject{currentmarker}{}%
\end{pgfscope}%
\end{pgfscope}%
\begin{pgfscope}%
\definecolor{textcolor}{rgb}{0.000000,0.000000,0.000000}%
\pgfsetstrokecolor{textcolor}%
\pgfsetfillcolor{textcolor}%
\pgftext[x=2.781654in,y=0.485677in,,top]{\color{textcolor}{\sffamily\fontsize{16.000000}{19.200000}\selectfont\catcode`\^=\active\def^{\ifmmode\sp\else\^{}\fi}\catcode`\%=\active\def
\end{pgfscope}%
\begin{pgfscope}%
\pgfsetbuttcap%
\pgfsetroundjoin%
\definecolor{currentfill}{rgb}{0.000000,0.000000,0.000000}%
\pgfsetfillcolor{currentfill}%
\pgfsetlinewidth{0.803000pt}%
\definecolor{currentstroke}{rgb}{0.000000,0.000000,0.000000}%
\pgfsetstrokecolor{currentstroke}%
\pgfsetdash{}{0pt}%
\pgfsys@defobject{currentmarker}{\pgfqpoint{0.000000in}{-0.048611in}}{\pgfqpoint{0.000000in}{0.000000in}}{%
\pgfpathmoveto{\pgfqpoint{0.000000in}{0.000000in}}%
\pgfpathlineto{\pgfqpoint{0.000000in}{-0.048611in}}%
\pgfusepath{stroke,fill}%
}%
\begin{pgfscope}%
\pgfsys@transformshift{3.251351in}{0.582899in}%
\pgfsys@useobject{currentmarker}{}%
\end{pgfscope}%
\end{pgfscope}%
\begin{pgfscope}%
\definecolor{textcolor}{rgb}{0.000000,0.000000,0.000000}%
\pgfsetstrokecolor{textcolor}%
\pgfsetfillcolor{textcolor}%
\pgftext[x=3.251351in,y=0.485677in,,top]{\color{textcolor}{\sffamily\fontsize{16.000000}{19.200000}\selectfont\catcode`\^=\active\def^{\ifmmode\sp\else\^{}\fi}\catcode`\%=\active\def
\end{pgfscope}%
\begin{pgfscope}%
\pgfsetbuttcap%
\pgfsetroundjoin%
\definecolor{currentfill}{rgb}{0.000000,0.000000,0.000000}%
\pgfsetfillcolor{currentfill}%
\pgfsetlinewidth{0.803000pt}%
\definecolor{currentstroke}{rgb}{0.000000,0.000000,0.000000}%
\pgfsetstrokecolor{currentstroke}%
\pgfsetdash{}{0pt}%
\pgfsys@defobject{currentmarker}{\pgfqpoint{0.000000in}{-0.048611in}}{\pgfqpoint{0.000000in}{0.000000in}}{%
\pgfpathmoveto{\pgfqpoint{0.000000in}{0.000000in}}%
\pgfpathlineto{\pgfqpoint{0.000000in}{-0.048611in}}%
\pgfusepath{stroke,fill}%
}%
\begin{pgfscope}%
\pgfsys@transformshift{3.721048in}{0.582899in}%
\pgfsys@useobject{currentmarker}{}%
\end{pgfscope}%
\end{pgfscope}%
\begin{pgfscope}%
\definecolor{textcolor}{rgb}{0.000000,0.000000,0.000000}%
\pgfsetstrokecolor{textcolor}%
\pgfsetfillcolor{textcolor}%
\pgftext[x=3.721048in,y=0.485677in,,top]{\color{textcolor}{\sffamily\fontsize{16.000000}{19.200000}\selectfont\catcode`\^=\active\def^{\ifmmode\sp\else\^{}\fi}\catcode`\%=\active\def
\end{pgfscope}%
\begin{pgfscope}%
\pgfsetbuttcap%
\pgfsetroundjoin%
\definecolor{currentfill}{rgb}{0.000000,0.000000,0.000000}%
\pgfsetfillcolor{currentfill}%
\pgfsetlinewidth{0.803000pt}%
\definecolor{currentstroke}{rgb}{0.000000,0.000000,0.000000}%
\pgfsetstrokecolor{currentstroke}%
\pgfsetdash{}{0pt}%
\pgfsys@defobject{currentmarker}{\pgfqpoint{0.000000in}{-0.048611in}}{\pgfqpoint{0.000000in}{0.000000in}}{%
\pgfpathmoveto{\pgfqpoint{0.000000in}{0.000000in}}%
\pgfpathlineto{\pgfqpoint{0.000000in}{-0.048611in}}%
\pgfusepath{stroke,fill}%
}%
\begin{pgfscope}%
\pgfsys@transformshift{4.190745in}{0.582899in}%
\pgfsys@useobject{currentmarker}{}%
\end{pgfscope}%
\end{pgfscope}%
\begin{pgfscope}%
\definecolor{textcolor}{rgb}{0.000000,0.000000,0.000000}%
\pgfsetstrokecolor{textcolor}%
\pgfsetfillcolor{textcolor}%
\pgftext[x=4.190745in,y=0.485677in,,top]{\color{textcolor}{\sffamily\fontsize{16.000000}{19.200000}\selectfont\catcode`\^=\active\def^{\ifmmode\sp\else\^{}\fi}\catcode`\%=\active\def
\end{pgfscope}%
\begin{pgfscope}%
\pgfsetbuttcap%
\pgfsetroundjoin%
\definecolor{currentfill}{rgb}{0.000000,0.000000,0.000000}%
\pgfsetfillcolor{currentfill}%
\pgfsetlinewidth{0.803000pt}%
\definecolor{currentstroke}{rgb}{0.000000,0.000000,0.000000}%
\pgfsetstrokecolor{currentstroke}%
\pgfsetdash{}{0pt}%
\pgfsys@defobject{currentmarker}{\pgfqpoint{0.000000in}{-0.048611in}}{\pgfqpoint{0.000000in}{0.000000in}}{%
\pgfpathmoveto{\pgfqpoint{0.000000in}{0.000000in}}%
\pgfpathlineto{\pgfqpoint{0.000000in}{-0.048611in}}%
\pgfusepath{stroke,fill}%
}%
\begin{pgfscope}%
\pgfsys@transformshift{4.660442in}{0.582899in}%
\pgfsys@useobject{currentmarker}{}%
\end{pgfscope}%
\end{pgfscope}%
\begin{pgfscope}%
\definecolor{textcolor}{rgb}{0.000000,0.000000,0.000000}%
\pgfsetstrokecolor{textcolor}%
\pgfsetfillcolor{textcolor}%
\pgftext[x=4.660442in,y=0.485677in,,top]{\color{textcolor}{\sffamily\fontsize{16.000000}{19.200000}\selectfont\catcode`\^=\active\def^{\ifmmode\sp\else\^{}\fi}\catcode`\%=\active\def
\end{pgfscope}%
\begin{pgfscope}%
\pgfsetbuttcap%
\pgfsetroundjoin%
\definecolor{currentfill}{rgb}{0.000000,0.000000,0.000000}%
\pgfsetfillcolor{currentfill}%
\pgfsetlinewidth{0.803000pt}%
\definecolor{currentstroke}{rgb}{0.000000,0.000000,0.000000}%
\pgfsetstrokecolor{currentstroke}%
\pgfsetdash{}{0pt}%
\pgfsys@defobject{currentmarker}{\pgfqpoint{0.000000in}{-0.048611in}}{\pgfqpoint{0.000000in}{0.000000in}}{%
\pgfpathmoveto{\pgfqpoint{0.000000in}{0.000000in}}%
\pgfpathlineto{\pgfqpoint{0.000000in}{-0.048611in}}%
\pgfusepath{stroke,fill}%
}%
\begin{pgfscope}%
\pgfsys@transformshift{5.130139in}{0.582899in}%
\pgfsys@useobject{currentmarker}{}%
\end{pgfscope}%
\end{pgfscope}%
\begin{pgfscope}%
\definecolor{textcolor}{rgb}{0.000000,0.000000,0.000000}%
\pgfsetstrokecolor{textcolor}%
\pgfsetfillcolor{textcolor}%
\pgftext[x=5.130139in,y=0.485677in,,top]{\color{textcolor}{\sffamily\fontsize{16.000000}{19.200000}\selectfont\catcode`\^=\active\def^{\ifmmode\sp\else\^{}\fi}\catcode`\%=\active\def
\end{pgfscope}%
\begin{pgfscope}%
\pgfsetbuttcap%
\pgfsetroundjoin%
\definecolor{currentfill}{rgb}{0.000000,0.000000,0.000000}%
\pgfsetfillcolor{currentfill}%
\pgfsetlinewidth{0.803000pt}%
\definecolor{currentstroke}{rgb}{0.000000,0.000000,0.000000}%
\pgfsetstrokecolor{currentstroke}%
\pgfsetdash{}{0pt}%
\pgfsys@defobject{currentmarker}{\pgfqpoint{0.000000in}{-0.048611in}}{\pgfqpoint{0.000000in}{0.000000in}}{%
\pgfpathmoveto{\pgfqpoint{0.000000in}{0.000000in}}%
\pgfpathlineto{\pgfqpoint{0.000000in}{-0.048611in}}%
\pgfusepath{stroke,fill}%
}%
\begin{pgfscope}%
\pgfsys@transformshift{5.599836in}{0.582899in}%
\pgfsys@useobject{currentmarker}{}%
\end{pgfscope}%
\end{pgfscope}%
\begin{pgfscope}%
\definecolor{textcolor}{rgb}{0.000000,0.000000,0.000000}%
\pgfsetstrokecolor{textcolor}%
\pgfsetfillcolor{textcolor}%
\pgftext[x=5.599836in,y=0.485677in,,top]{\color{textcolor}{\sffamily\fontsize{16.000000}{19.200000}\selectfont\catcode`\^=\active\def^{\ifmmode\sp\else\^{}\fi}\catcode`\%=\active\def
\end{pgfscope}%
\begin{pgfscope}%
\definecolor{textcolor}{rgb}{0.000000,0.000000,0.000000}%
\pgfsetstrokecolor{textcolor}%
\pgfsetfillcolor{textcolor}%
\pgftext[x=3.345290in,y=0.215061in,,top]{\color{textcolor}{\sffamily\fontsize{16.000000}{19.200000}\selectfont\catcode`\^=\active\def^{\ifmmode\sp\else\^{}\fi}\catcode`\%=\active\def
\end{pgfscope}%
\begin{pgfscope}%
\pgfsetbuttcap%
\pgfsetroundjoin%
\definecolor{currentfill}{rgb}{0.000000,0.000000,0.000000}%
\pgfsetfillcolor{currentfill}%
\pgfsetlinewidth{0.803000pt}%
\definecolor{currentstroke}{rgb}{0.000000,0.000000,0.000000}%
\pgfsetstrokecolor{currentstroke}%
\pgfsetdash{}{0pt}%
\pgfsys@defobject{currentmarker}{\pgfqpoint{-0.048611in}{0.000000in}}{\pgfqpoint{-0.000000in}{0.000000in}}{%
\pgfpathmoveto{\pgfqpoint{-0.000000in}{0.000000in}}%
\pgfpathlineto{\pgfqpoint{-0.048611in}{0.000000in}}%
\pgfusepath{stroke,fill}%
}%
\begin{pgfscope}%
\pgfsys@transformshift{0.865290in}{1.063533in}%
\pgfsys@useobject{currentmarker}{}%
\end{pgfscope}%
\end{pgfscope}%
\begin{pgfscope}%
\definecolor{textcolor}{rgb}{0.000000,0.000000,0.000000}%
\pgfsetstrokecolor{textcolor}%
\pgfsetfillcolor{textcolor}%
\pgftext[x=0.270616in, y=0.979114in, left, base]{\color{textcolor}{\sffamily\fontsize{16.000000}{19.200000}\selectfont\catcode`\^=\active\def^{\ifmmode\sp\else\^{}\fi}\catcode`\%=\active\def
\end{pgfscope}%
\begin{pgfscope}%
\pgfsetbuttcap%
\pgfsetroundjoin%
\definecolor{currentfill}{rgb}{0.000000,0.000000,0.000000}%
\pgfsetfillcolor{currentfill}%
\pgfsetlinewidth{0.803000pt}%
\definecolor{currentstroke}{rgb}{0.000000,0.000000,0.000000}%
\pgfsetstrokecolor{currentstroke}%
\pgfsetdash{}{0pt}%
\pgfsys@defobject{currentmarker}{\pgfqpoint{-0.048611in}{0.000000in}}{\pgfqpoint{-0.000000in}{0.000000in}}{%
\pgfpathmoveto{\pgfqpoint{-0.000000in}{0.000000in}}%
\pgfpathlineto{\pgfqpoint{-0.048611in}{0.000000in}}%
\pgfusepath{stroke,fill}%
}%
\begin{pgfscope}%
\pgfsys@transformshift{0.865290in}{1.561240in}%
\pgfsys@useobject{currentmarker}{}%
\end{pgfscope}%
\end{pgfscope}%
\begin{pgfscope}%
\definecolor{textcolor}{rgb}{0.000000,0.000000,0.000000}%
\pgfsetstrokecolor{textcolor}%
\pgfsetfillcolor{textcolor}%
\pgftext[x=0.270616in, y=1.476821in, left, base]{\color{textcolor}{\sffamily\fontsize{16.000000}{19.200000}\selectfont\catcode`\^=\active\def^{\ifmmode\sp\else\^{}\fi}\catcode`\%=\active\def
\end{pgfscope}%
\begin{pgfscope}%
\pgfsetbuttcap%
\pgfsetroundjoin%
\definecolor{currentfill}{rgb}{0.000000,0.000000,0.000000}%
\pgfsetfillcolor{currentfill}%
\pgfsetlinewidth{0.803000pt}%
\definecolor{currentstroke}{rgb}{0.000000,0.000000,0.000000}%
\pgfsetstrokecolor{currentstroke}%
\pgfsetdash{}{0pt}%
\pgfsys@defobject{currentmarker}{\pgfqpoint{-0.048611in}{0.000000in}}{\pgfqpoint{-0.000000in}{0.000000in}}{%
\pgfpathmoveto{\pgfqpoint{-0.000000in}{0.000000in}}%
\pgfpathlineto{\pgfqpoint{-0.048611in}{0.000000in}}%
\pgfusepath{stroke,fill}%
}%
\begin{pgfscope}%
\pgfsys@transformshift{0.865290in}{2.058947in}%
\pgfsys@useobject{currentmarker}{}%
\end{pgfscope}%
\end{pgfscope}%
\begin{pgfscope}%
\definecolor{textcolor}{rgb}{0.000000,0.000000,0.000000}%
\pgfsetstrokecolor{textcolor}%
\pgfsetfillcolor{textcolor}%
\pgftext[x=0.346658in, y=1.974529in, left, base]{\color{textcolor}{\sffamily\fontsize{16.000000}{19.200000}\selectfont\catcode`\^=\active\def^{\ifmmode\sp\else\^{}\fi}\catcode`\%=\active\def
\end{pgfscope}%
\begin{pgfscope}%
\pgfsetbuttcap%
\pgfsetroundjoin%
\definecolor{currentfill}{rgb}{0.000000,0.000000,0.000000}%
\pgfsetfillcolor{currentfill}%
\pgfsetlinewidth{0.803000pt}%
\definecolor{currentstroke}{rgb}{0.000000,0.000000,0.000000}%
\pgfsetstrokecolor{currentstroke}%
\pgfsetdash{}{0pt}%
\pgfsys@defobject{currentmarker}{\pgfqpoint{-0.048611in}{0.000000in}}{\pgfqpoint{-0.000000in}{0.000000in}}{%
\pgfpathmoveto{\pgfqpoint{-0.000000in}{0.000000in}}%
\pgfpathlineto{\pgfqpoint{-0.048611in}{0.000000in}}%
\pgfusepath{stroke,fill}%
}%
\begin{pgfscope}%
\pgfsys@transformshift{0.865290in}{2.556654in}%
\pgfsys@useobject{currentmarker}{}%
\end{pgfscope}%
\end{pgfscope}%
\begin{pgfscope}%
\definecolor{textcolor}{rgb}{0.000000,0.000000,0.000000}%
\pgfsetstrokecolor{textcolor}%
\pgfsetfillcolor{textcolor}%
\pgftext[x=0.346658in, y=2.472236in, left, base]{\color{textcolor}{\sffamily\fontsize{16.000000}{19.200000}\selectfont\catcode`\^=\active\def^{\ifmmode\sp\else\^{}\fi}\catcode`\%=\active\def
\end{pgfscope}%
\begin{pgfscope}%
\pgfsetbuttcap%
\pgfsetroundjoin%
\definecolor{currentfill}{rgb}{0.000000,0.000000,0.000000}%
\pgfsetfillcolor{currentfill}%
\pgfsetlinewidth{0.803000pt}%
\definecolor{currentstroke}{rgb}{0.000000,0.000000,0.000000}%
\pgfsetstrokecolor{currentstroke}%
\pgfsetdash{}{0pt}%
\pgfsys@defobject{currentmarker}{\pgfqpoint{-0.048611in}{0.000000in}}{\pgfqpoint{-0.000000in}{0.000000in}}{%
\pgfpathmoveto{\pgfqpoint{-0.000000in}{0.000000in}}%
\pgfpathlineto{\pgfqpoint{-0.048611in}{0.000000in}}%
\pgfusepath{stroke,fill}%
}%
\begin{pgfscope}%
\pgfsys@transformshift{0.865290in}{3.054361in}%
\pgfsys@useobject{currentmarker}{}%
\end{pgfscope}%
\end{pgfscope}%
\begin{pgfscope}%
\definecolor{textcolor}{rgb}{0.000000,0.000000,0.000000}%
\pgfsetstrokecolor{textcolor}%
\pgfsetfillcolor{textcolor}%
\pgftext[x=0.346658in, y=2.969943in, left, base]{\color{textcolor}{\sffamily\fontsize{16.000000}{19.200000}\selectfont\catcode`\^=\active\def^{\ifmmode\sp\else\^{}\fi}\catcode`\%=\active\def
\end{pgfscope}%
\begin{pgfscope}%
\pgfsetbuttcap%
\pgfsetroundjoin%
\definecolor{currentfill}{rgb}{0.000000,0.000000,0.000000}%
\pgfsetfillcolor{currentfill}%
\pgfsetlinewidth{0.803000pt}%
\definecolor{currentstroke}{rgb}{0.000000,0.000000,0.000000}%
\pgfsetstrokecolor{currentstroke}%
\pgfsetdash{}{0pt}%
\pgfsys@defobject{currentmarker}{\pgfqpoint{-0.048611in}{0.000000in}}{\pgfqpoint{-0.000000in}{0.000000in}}{%
\pgfpathmoveto{\pgfqpoint{-0.000000in}{0.000000in}}%
\pgfpathlineto{\pgfqpoint{-0.048611in}{0.000000in}}%
\pgfusepath{stroke,fill}%
}%
\begin{pgfscope}%
\pgfsys@transformshift{0.865290in}{3.552068in}%
\pgfsys@useobject{currentmarker}{}%
\end{pgfscope}%
\end{pgfscope}%
\begin{pgfscope}%
\definecolor{textcolor}{rgb}{0.000000,0.000000,0.000000}%
\pgfsetstrokecolor{textcolor}%
\pgfsetfillcolor{textcolor}%
\pgftext[x=0.346658in, y=3.467650in, left, base]{\color{textcolor}{\sffamily\fontsize{16.000000}{19.200000}\selectfont\catcode`\^=\active\def^{\ifmmode\sp\else\^{}\fi}\catcode`\%=\active\def
\end{pgfscope}%
\begin{pgfscope}%
\pgfsetbuttcap%
\pgfsetroundjoin%
\definecolor{currentfill}{rgb}{0.000000,0.000000,0.000000}%
\pgfsetfillcolor{currentfill}%
\pgfsetlinewidth{0.803000pt}%
\definecolor{currentstroke}{rgb}{0.000000,0.000000,0.000000}%
\pgfsetstrokecolor{currentstroke}%
\pgfsetdash{}{0pt}%
\pgfsys@defobject{currentmarker}{\pgfqpoint{-0.048611in}{0.000000in}}{\pgfqpoint{-0.000000in}{0.000000in}}{%
\pgfpathmoveto{\pgfqpoint{-0.000000in}{0.000000in}}%
\pgfpathlineto{\pgfqpoint{-0.048611in}{0.000000in}}%
\pgfusepath{stroke,fill}%
}%
\begin{pgfscope}%
\pgfsys@transformshift{0.865290in}{4.049775in}%
\pgfsys@useobject{currentmarker}{}%
\end{pgfscope}%
\end{pgfscope}%
\begin{pgfscope}%
\definecolor{textcolor}{rgb}{0.000000,0.000000,0.000000}%
\pgfsetstrokecolor{textcolor}%
\pgfsetfillcolor{textcolor}%
\pgftext[x=0.346658in, y=3.965357in, left, base]{\color{textcolor}{\sffamily\fontsize{16.000000}{19.200000}\selectfont\catcode`\^=\active\def^{\ifmmode\sp\else\^{}\fi}\catcode`\%=\active\def
\end{pgfscope}%
\begin{pgfscope}%
\definecolor{textcolor}{rgb}{0.000000,0.000000,0.000000}%
\pgfsetstrokecolor{textcolor}%
\pgfsetfillcolor{textcolor}%
\pgftext[x=0.215061in,y=2.430899in,,bottom,rotate=90.000000]{\color{textcolor}{\sffamily\fontsize{16.000000}{19.200000}\selectfont\catcode`\^=\active\def^{\ifmmode\sp\else\^{}\fi}\catcode`\%=\active\def
\end{pgfscope}%
\begin{pgfscope}%
\pgfpathrectangle{\pgfqpoint{0.865290in}{0.582899in}}{\pgfqpoint{4.960000in}{3.696000in}}%
\pgfusepath{clip}%
\pgfsetrectcap%
\pgfsetroundjoin%
\pgfsetlinewidth{1.505625pt}%
\definecolor{currentstroke}{rgb}{0.121569,0.466667,0.705882}%
\pgfsetstrokecolor{currentstroke}%
\pgfsetdash{}{0pt}%
\pgfpathmoveto{\pgfqpoint{1.090745in}{4.110899in}}%
\pgfpathlineto{\pgfqpoint{1.278623in}{3.893968in}}%
\pgfpathlineto{\pgfqpoint{1.466502in}{3.633285in}}%
\pgfpathlineto{\pgfqpoint{1.654381in}{3.296421in}}%
\pgfpathlineto{\pgfqpoint{1.842260in}{2.897834in}}%
\pgfpathlineto{\pgfqpoint{2.030139in}{2.461877in}}%
\pgfpathlineto{\pgfqpoint{2.218017in}{1.990708in}}%
\pgfpathlineto{\pgfqpoint{2.405896in}{1.542759in}}%
\pgfpathlineto{\pgfqpoint{2.593775in}{1.030377in}}%
\pgfpathlineto{\pgfqpoint{2.781654in}{0.750899in}}%
\pgfpathlineto{\pgfqpoint{2.969533in}{0.750899in}}%
\pgfpathlineto{\pgfqpoint{3.157411in}{0.750899in}}%
\pgfpathlineto{\pgfqpoint{3.345290in}{0.750899in}}%
\pgfpathlineto{\pgfqpoint{3.533169in}{0.750899in}}%
\pgfpathlineto{\pgfqpoint{3.721048in}{0.750899in}}%
\pgfpathlineto{\pgfqpoint{3.908927in}{0.750899in}}%
\pgfpathlineto{\pgfqpoint{4.096805in}{0.750899in}}%
\pgfpathlineto{\pgfqpoint{4.284684in}{0.750899in}}%
\pgfpathlineto{\pgfqpoint{4.472563in}{0.750899in}}%
\pgfpathlineto{\pgfqpoint{4.660442in}{0.750899in}}%
\pgfpathlineto{\pgfqpoint{4.848320in}{0.750899in}}%
\pgfpathlineto{\pgfqpoint{5.036199in}{0.750899in}}%
\pgfpathlineto{\pgfqpoint{5.224078in}{0.750899in}}%
\pgfpathlineto{\pgfqpoint{5.411957in}{0.750899in}}%
\pgfpathlineto{\pgfqpoint{5.599836in}{0.750899in}}%
\pgfusepath{stroke}%
\end{pgfscope}%
\begin{pgfscope}%
\pgfpathrectangle{\pgfqpoint{0.865290in}{0.582899in}}{\pgfqpoint{4.960000in}{3.696000in}}%
\pgfusepath{clip}%
\pgfsetrectcap%
\pgfsetroundjoin%
\pgfsetlinewidth{1.505625pt}%
\definecolor{currentstroke}{rgb}{1.000000,0.498039,0.054902}%
\pgfsetstrokecolor{currentstroke}%
\pgfsetdash{}{0pt}%
\pgfpathmoveto{\pgfqpoint{1.090745in}{3.912611in}}%
\pgfpathlineto{\pgfqpoint{1.278623in}{3.788227in}}%
\pgfpathlineto{\pgfqpoint{1.466502in}{3.476553in}}%
\pgfpathlineto{\pgfqpoint{1.654381in}{3.140780in}}%
\pgfpathlineto{\pgfqpoint{1.842260in}{2.761343in}}%
\pgfpathlineto{\pgfqpoint{2.030139in}{2.327407in}}%
\pgfpathlineto{\pgfqpoint{2.218017in}{1.872088in}}%
\pgfpathlineto{\pgfqpoint{2.405896in}{1.383155in}}%
\pgfpathlineto{\pgfqpoint{2.593775in}{0.966908in}}%
\pgfpathlineto{\pgfqpoint{2.781654in}{0.915504in}}%
\pgfpathlineto{\pgfqpoint{2.969533in}{0.913866in}}%
\pgfpathlineto{\pgfqpoint{3.157411in}{0.913866in}}%
\pgfpathlineto{\pgfqpoint{3.345290in}{0.913866in}}%
\pgfpathlineto{\pgfqpoint{3.533169in}{0.913866in}}%
\pgfpathlineto{\pgfqpoint{3.721048in}{0.913866in}}%
\pgfpathlineto{\pgfqpoint{3.908927in}{0.913866in}}%
\pgfpathlineto{\pgfqpoint{4.096805in}{0.913866in}}%
\pgfpathlineto{\pgfqpoint{4.284684in}{0.913866in}}%
\pgfpathlineto{\pgfqpoint{4.472563in}{0.913866in}}%
\pgfpathlineto{\pgfqpoint{4.660442in}{0.913866in}}%
\pgfpathlineto{\pgfqpoint{4.848320in}{0.913866in}}%
\pgfpathlineto{\pgfqpoint{5.036199in}{0.914688in}}%
\pgfpathlineto{\pgfqpoint{5.224078in}{0.913866in}}%
\pgfpathlineto{\pgfqpoint{5.411957in}{0.913866in}}%
\pgfpathlineto{\pgfqpoint{5.599836in}{0.913866in}}%
\pgfusepath{stroke}%
\end{pgfscope}%
\begin{pgfscope}%
\pgfsetrectcap%
\pgfsetmiterjoin%
\pgfsetlinewidth{0.803000pt}%
\definecolor{currentstroke}{rgb}{0.000000,0.000000,0.000000}%
\pgfsetstrokecolor{currentstroke}%
\pgfsetdash{}{0pt}%
\pgfpathmoveto{\pgfqpoint{0.865290in}{0.582899in}}%
\pgfpathlineto{\pgfqpoint{0.865290in}{4.278899in}}%
\pgfusepath{stroke}%
\end{pgfscope}%
\begin{pgfscope}%
\pgfsetrectcap%
\pgfsetmiterjoin%
\pgfsetlinewidth{0.803000pt}%
\definecolor{currentstroke}{rgb}{0.000000,0.000000,0.000000}%
\pgfsetstrokecolor{currentstroke}%
\pgfsetdash{}{0pt}%
\pgfpathmoveto{\pgfqpoint{5.825290in}{0.582899in}}%
\pgfpathlineto{\pgfqpoint{5.825290in}{4.278899in}}%
\pgfusepath{stroke}%
\end{pgfscope}%
\begin{pgfscope}%
\pgfsetrectcap%
\pgfsetmiterjoin%
\pgfsetlinewidth{0.803000pt}%
\definecolor{currentstroke}{rgb}{0.000000,0.000000,0.000000}%
\pgfsetstrokecolor{currentstroke}%
\pgfsetdash{}{0pt}%
\pgfpathmoveto{\pgfqpoint{0.865290in}{0.582899in}}%
\pgfpathlineto{\pgfqpoint{5.825290in}{0.582899in}}%
\pgfusepath{stroke}%
\end{pgfscope}%
\begin{pgfscope}%
\pgfsetrectcap%
\pgfsetmiterjoin%
\pgfsetlinewidth{0.803000pt}%
\definecolor{currentstroke}{rgb}{0.000000,0.000000,0.000000}%
\pgfsetstrokecolor{currentstroke}%
\pgfsetdash{}{0pt}%
\pgfpathmoveto{\pgfqpoint{0.865290in}{4.278899in}}%
\pgfpathlineto{\pgfqpoint{5.825290in}{4.278899in}}%
\pgfusepath{stroke}%
\end{pgfscope}%
\begin{pgfscope}%
\pgfsetbuttcap%
\pgfsetmiterjoin%
\definecolor{currentfill}{rgb}{1.000000,1.000000,1.000000}%
\pgfsetfillcolor{currentfill}%
\pgfsetfillopacity{0.800000}%
\pgfsetlinewidth{1.003750pt}%
\definecolor{currentstroke}{rgb}{0.800000,0.800000,0.800000}%
\pgfsetstrokecolor{currentstroke}%
\pgfsetstrokeopacity{0.800000}%
\pgfsetdash{}{0pt}%
\pgfpathmoveto{\pgfqpoint{3.158927in}{3.448778in}}%
\pgfpathlineto{\pgfqpoint{5.669735in}{3.448778in}}%
\pgfpathquadraticcurveto{\pgfqpoint{5.714179in}{3.448778in}}{\pgfqpoint{5.714179in}{3.493222in}}%
\pgfpathlineto{\pgfqpoint{5.714179in}{4.123343in}}%
\pgfpathquadraticcurveto{\pgfqpoint{5.714179in}{4.167788in}}{\pgfqpoint{5.669735in}{4.167788in}}%
\pgfpathlineto{\pgfqpoint{3.158927in}{4.167788in}}%
\pgfpathquadraticcurveto{\pgfqpoint{3.114483in}{4.167788in}}{\pgfqpoint{3.114483in}{4.123343in}}%
\pgfpathlineto{\pgfqpoint{3.114483in}{3.493222in}}%
\pgfpathquadraticcurveto{\pgfqpoint{3.114483in}{3.448778in}}{\pgfqpoint{3.158927in}{3.448778in}}%
\pgfpathlineto{\pgfqpoint{3.158927in}{3.448778in}}%
\pgfpathclose%
\pgfusepath{stroke,fill}%
\end{pgfscope}%
\begin{pgfscope}%
\pgfsetrectcap%
\pgfsetroundjoin%
\pgfsetlinewidth{1.505625pt}%
\definecolor{currentstroke}{rgb}{0.121569,0.466667,0.705882}%
\pgfsetstrokecolor{currentstroke}%
\pgfsetdash{}{0pt}%
\pgfpathmoveto{\pgfqpoint{3.203372in}{3.987840in}}%
\pgfpathlineto{\pgfqpoint{3.425594in}{3.987840in}}%
\pgfpathlineto{\pgfqpoint{3.647816in}{3.987840in}}%
\pgfusepath{stroke}%
\end{pgfscope}%
\begin{pgfscope}%
\definecolor{textcolor}{rgb}{0.000000,0.000000,0.000000}%
\pgfsetstrokecolor{textcolor}%
\pgfsetfillcolor{textcolor}%
\pgftext[x=3.825594in,y=3.910062in,left,base]{\color{textcolor}{\sffamily\fontsize{16.000000}{19.200000}\selectfont\catcode`\^=\active\def^{\ifmmode\sp\else\^{}\fi}\catcode`\%=\active\def
\end{pgfscope}%
\begin{pgfscope}%
\pgfsetrectcap%
\pgfsetroundjoin%
\pgfsetlinewidth{1.505625pt}%
\definecolor{currentstroke}{rgb}{1.000000,0.498039,0.054902}%
\pgfsetstrokecolor{currentstroke}%
\pgfsetdash{}{0pt}%
\pgfpathmoveto{\pgfqpoint{3.203372in}{3.661668in}}%
\pgfpathlineto{\pgfqpoint{3.425594in}{3.661668in}}%
\pgfpathlineto{\pgfqpoint{3.647816in}{3.661668in}}%
\pgfusepath{stroke}%
\end{pgfscope}%
\begin{pgfscope}%
\definecolor{textcolor}{rgb}{0.000000,0.000000,0.000000}%
\pgfsetstrokecolor{textcolor}%
\pgfsetfillcolor{textcolor}%
\pgftext[x=3.825594in,y=3.583890in,left,base]{\color{textcolor}{\sffamily\fontsize{16.000000}{19.200000}\selectfont\catcode`\^=\active\def^{\ifmmode\sp\else\^{}\fi}\catcode`\%=\active\def
\end{pgfscope}%
\end{pgfpicture}%
\makeatother%
\endgroup%

%% file: figs/network_conv_ei_USPowerGrid.pgf
\begingroup%
\makeatletter%
\begin{pgfpicture}%
\pgfpathrectangle{\pgfpointorigin}{\pgfqpoint{5.825290in}{4.278899in}}%
\pgfusepath{use as bounding box, clip}%
\begin{pgfscope}%
\pgfsetbuttcap%
\pgfsetmiterjoin%
\definecolor{currentfill}{rgb}{1.000000,1.000000,1.000000}%
\pgfsetfillcolor{currentfill}%
\pgfsetlinewidth{0.000000pt}%
\definecolor{currentstroke}{rgb}{1.000000,1.000000,1.000000}%
\pgfsetstrokecolor{currentstroke}%
\pgfsetdash{}{0pt}%
\pgfpathmoveto{\pgfqpoint{0.000000in}{0.000000in}}%
\pgfpathlineto{\pgfqpoint{5.825290in}{0.000000in}}%
\pgfpathlineto{\pgfqpoint{5.825290in}{4.278899in}}%
\pgfpathlineto{\pgfqpoint{0.000000in}{4.278899in}}%
\pgfpathlineto{\pgfqpoint{0.000000in}{0.000000in}}%
\pgfpathclose%
\pgfusepath{fill}%
\end{pgfscope}%
\begin{pgfscope}%
\pgfsetbuttcap%
\pgfsetmiterjoin%
\definecolor{currentfill}{rgb}{1.000000,1.000000,1.000000}%
\pgfsetfillcolor{currentfill}%
\pgfsetlinewidth{0.000000pt}%
\definecolor{currentstroke}{rgb}{0.000000,0.000000,0.000000}%
\pgfsetstrokecolor{currentstroke}%
\pgfsetstrokeopacity{0.000000}%
\pgfsetdash{}{0pt}%
\pgfpathmoveto{\pgfqpoint{0.865290in}{0.582899in}}%
\pgfpathlineto{\pgfqpoint{5.825290in}{0.582899in}}%
\pgfpathlineto{\pgfqpoint{5.825290in}{4.278899in}}%
\pgfpathlineto{\pgfqpoint{0.865290in}{4.278899in}}%
\pgfpathlineto{\pgfqpoint{0.865290in}{0.582899in}}%
\pgfpathclose%
\pgfusepath{fill}%
\end{pgfscope}%
\begin{pgfscope}%
\pgfsetbuttcap%
\pgfsetroundjoin%
\definecolor{currentfill}{rgb}{0.000000,0.000000,0.000000}%
\pgfsetfillcolor{currentfill}%
\pgfsetlinewidth{0.803000pt}%
\definecolor{currentstroke}{rgb}{0.000000,0.000000,0.000000}%
\pgfsetstrokecolor{currentstroke}%
\pgfsetdash{}{0pt}%
\pgfsys@defobject{currentmarker}{\pgfqpoint{0.000000in}{-0.048611in}}{\pgfqpoint{0.000000in}{0.000000in}}{%
\pgfpathmoveto{\pgfqpoint{0.000000in}{0.000000in}}%
\pgfpathlineto{\pgfqpoint{0.000000in}{-0.048611in}}%
\pgfusepath{stroke,fill}%
}%
\begin{pgfscope}%
\pgfsys@transformshift{1.372563in}{0.582899in}%
\pgfsys@useobject{currentmarker}{}%
\end{pgfscope}%
\end{pgfscope}%
\begin{pgfscope}%
\definecolor{textcolor}{rgb}{0.000000,0.000000,0.000000}%
\pgfsetstrokecolor{textcolor}%
\pgfsetfillcolor{textcolor}%
\pgftext[x=1.372563in,y=0.485677in,,top]{\color{textcolor}{\sffamily\fontsize{16.000000}{19.200000}\selectfont\catcode`\^=\active\def^{\ifmmode\sp\else\^{}\fi}\catcode`\%=\active\def
\end{pgfscope}%
\begin{pgfscope}%
\pgfsetbuttcap%
\pgfsetroundjoin%
\definecolor{currentfill}{rgb}{0.000000,0.000000,0.000000}%
\pgfsetfillcolor{currentfill}%
\pgfsetlinewidth{0.803000pt}%
\definecolor{currentstroke}{rgb}{0.000000,0.000000,0.000000}%
\pgfsetstrokecolor{currentstroke}%
\pgfsetdash{}{0pt}%
\pgfsys@defobject{currentmarker}{\pgfqpoint{0.000000in}{-0.048611in}}{\pgfqpoint{0.000000in}{0.000000in}}{%
\pgfpathmoveto{\pgfqpoint{0.000000in}{0.000000in}}%
\pgfpathlineto{\pgfqpoint{0.000000in}{-0.048611in}}%
\pgfusepath{stroke,fill}%
}%
\begin{pgfscope}%
\pgfsys@transformshift{1.842260in}{0.582899in}%
\pgfsys@useobject{currentmarker}{}%
\end{pgfscope}%
\end{pgfscope}%
\begin{pgfscope}%
\definecolor{textcolor}{rgb}{0.000000,0.000000,0.000000}%
\pgfsetstrokecolor{textcolor}%
\pgfsetfillcolor{textcolor}%
\pgftext[x=1.842260in,y=0.485677in,,top]{\color{textcolor}{\sffamily\fontsize{16.000000}{19.200000}\selectfont\catcode`\^=\active\def^{\ifmmode\sp\else\^{}\fi}\catcode`\%=\active\def
\end{pgfscope}%
\begin{pgfscope}%
\pgfsetbuttcap%
\pgfsetroundjoin%
\definecolor{currentfill}{rgb}{0.000000,0.000000,0.000000}%
\pgfsetfillcolor{currentfill}%
\pgfsetlinewidth{0.803000pt}%
\definecolor{currentstroke}{rgb}{0.000000,0.000000,0.000000}%
\pgfsetstrokecolor{currentstroke}%
\pgfsetdash{}{0pt}%
\pgfsys@defobject{currentmarker}{\pgfqpoint{0.000000in}{-0.048611in}}{\pgfqpoint{0.000000in}{0.000000in}}{%
\pgfpathmoveto{\pgfqpoint{0.000000in}{0.000000in}}%
\pgfpathlineto{\pgfqpoint{0.000000in}{-0.048611in}}%
\pgfusepath{stroke,fill}%
}%
\begin{pgfscope}%
\pgfsys@transformshift{2.311957in}{0.582899in}%
\pgfsys@useobject{currentmarker}{}%
\end{pgfscope}%
\end{pgfscope}%
\begin{pgfscope}%
\definecolor{textcolor}{rgb}{0.000000,0.000000,0.000000}%
\pgfsetstrokecolor{textcolor}%
\pgfsetfillcolor{textcolor}%
\pgftext[x=2.311957in,y=0.485677in,,top]{\color{textcolor}{\sffamily\fontsize{16.000000}{19.200000}\selectfont\catcode`\^=\active\def^{\ifmmode\sp\else\^{}\fi}\catcode`\%=\active\def
\end{pgfscope}%
\begin{pgfscope}%
\pgfsetbuttcap%
\pgfsetroundjoin%
\definecolor{currentfill}{rgb}{0.000000,0.000000,0.000000}%
\pgfsetfillcolor{currentfill}%
\pgfsetlinewidth{0.803000pt}%
\definecolor{currentstroke}{rgb}{0.000000,0.000000,0.000000}%
\pgfsetstrokecolor{currentstroke}%
\pgfsetdash{}{0pt}%
\pgfsys@defobject{currentmarker}{\pgfqpoint{0.000000in}{-0.048611in}}{\pgfqpoint{0.000000in}{0.000000in}}{%
\pgfpathmoveto{\pgfqpoint{0.000000in}{0.000000in}}%
\pgfpathlineto{\pgfqpoint{0.000000in}{-0.048611in}}%
\pgfusepath{stroke,fill}%
}%
\begin{pgfscope}%
\pgfsys@transformshift{2.781654in}{0.582899in}%
\pgfsys@useobject{currentmarker}{}%
\end{pgfscope}%
\end{pgfscope}%
\begin{pgfscope}%
\definecolor{textcolor}{rgb}{0.000000,0.000000,0.000000}%
\pgfsetstrokecolor{textcolor}%
\pgfsetfillcolor{textcolor}%
\pgftext[x=2.781654in,y=0.485677in,,top]{\color{textcolor}{\sffamily\fontsize{16.000000}{19.200000}\selectfont\catcode`\^=\active\def^{\ifmmode\sp\else\^{}\fi}\catcode`\%=\active\def
\end{pgfscope}%
\begin{pgfscope}%
\pgfsetbuttcap%
\pgfsetroundjoin%
\definecolor{currentfill}{rgb}{0.000000,0.000000,0.000000}%
\pgfsetfillcolor{currentfill}%
\pgfsetlinewidth{0.803000pt}%
\definecolor{currentstroke}{rgb}{0.000000,0.000000,0.000000}%
\pgfsetstrokecolor{currentstroke}%
\pgfsetdash{}{0pt}%
\pgfsys@defobject{currentmarker}{\pgfqpoint{0.000000in}{-0.048611in}}{\pgfqpoint{0.000000in}{0.000000in}}{%
\pgfpathmoveto{\pgfqpoint{0.000000in}{0.000000in}}%
\pgfpathlineto{\pgfqpoint{0.000000in}{-0.048611in}}%
\pgfusepath{stroke,fill}%
}%
\begin{pgfscope}%
\pgfsys@transformshift{3.251351in}{0.582899in}%
\pgfsys@useobject{currentmarker}{}%
\end{pgfscope}%
\end{pgfscope}%
\begin{pgfscope}%
\definecolor{textcolor}{rgb}{0.000000,0.000000,0.000000}%
\pgfsetstrokecolor{textcolor}%
\pgfsetfillcolor{textcolor}%
\pgftext[x=3.251351in,y=0.485677in,,top]{\color{textcolor}{\sffamily\fontsize{16.000000}{19.200000}\selectfont\catcode`\^=\active\def^{\ifmmode\sp\else\^{}\fi}\catcode`\%=\active\def
\end{pgfscope}%
\begin{pgfscope}%
\pgfsetbuttcap%
\pgfsetroundjoin%
\definecolor{currentfill}{rgb}{0.000000,0.000000,0.000000}%
\pgfsetfillcolor{currentfill}%
\pgfsetlinewidth{0.803000pt}%
\definecolor{currentstroke}{rgb}{0.000000,0.000000,0.000000}%
\pgfsetstrokecolor{currentstroke}%
\pgfsetdash{}{0pt}%
\pgfsys@defobject{currentmarker}{\pgfqpoint{0.000000in}{-0.048611in}}{\pgfqpoint{0.000000in}{0.000000in}}{%
\pgfpathmoveto{\pgfqpoint{0.000000in}{0.000000in}}%
\pgfpathlineto{\pgfqpoint{0.000000in}{-0.048611in}}%
\pgfusepath{stroke,fill}%
}%
\begin{pgfscope}%
\pgfsys@transformshift{3.721048in}{0.582899in}%
\pgfsys@useobject{currentmarker}{}%
\end{pgfscope}%
\end{pgfscope}%
\begin{pgfscope}%
\definecolor{textcolor}{rgb}{0.000000,0.000000,0.000000}%
\pgfsetstrokecolor{textcolor}%
\pgfsetfillcolor{textcolor}%
\pgftext[x=3.721048in,y=0.485677in,,top]{\color{textcolor}{\sffamily\fontsize{16.000000}{19.200000}\selectfont\catcode`\^=\active\def^{\ifmmode\sp\else\^{}\fi}\catcode`\%=\active\def
\end{pgfscope}%
\begin{pgfscope}%
\pgfsetbuttcap%
\pgfsetroundjoin%
\definecolor{currentfill}{rgb}{0.000000,0.000000,0.000000}%
\pgfsetfillcolor{currentfill}%
\pgfsetlinewidth{0.803000pt}%
\definecolor{currentstroke}{rgb}{0.000000,0.000000,0.000000}%
\pgfsetstrokecolor{currentstroke}%
\pgfsetdash{}{0pt}%
\pgfsys@defobject{currentmarker}{\pgfqpoint{0.000000in}{-0.048611in}}{\pgfqpoint{0.000000in}{0.000000in}}{%
\pgfpathmoveto{\pgfqpoint{0.000000in}{0.000000in}}%
\pgfpathlineto{\pgfqpoint{0.000000in}{-0.048611in}}%
\pgfusepath{stroke,fill}%
}%
\begin{pgfscope}%
\pgfsys@transformshift{4.190745in}{0.582899in}%
\pgfsys@useobject{currentmarker}{}%
\end{pgfscope}%
\end{pgfscope}%
\begin{pgfscope}%
\definecolor{textcolor}{rgb}{0.000000,0.000000,0.000000}%
\pgfsetstrokecolor{textcolor}%
\pgfsetfillcolor{textcolor}%
\pgftext[x=4.190745in,y=0.485677in,,top]{\color{textcolor}{\sffamily\fontsize{16.000000}{19.200000}\selectfont\catcode`\^=\active\def^{\ifmmode\sp\else\^{}\fi}\catcode`\%=\active\def
\end{pgfscope}%
\begin{pgfscope}%
\pgfsetbuttcap%
\pgfsetroundjoin%
\definecolor{currentfill}{rgb}{0.000000,0.000000,0.000000}%
\pgfsetfillcolor{currentfill}%
\pgfsetlinewidth{0.803000pt}%
\definecolor{currentstroke}{rgb}{0.000000,0.000000,0.000000}%
\pgfsetstrokecolor{currentstroke}%
\pgfsetdash{}{0pt}%
\pgfsys@defobject{currentmarker}{\pgfqpoint{0.000000in}{-0.048611in}}{\pgfqpoint{0.000000in}{0.000000in}}{%
\pgfpathmoveto{\pgfqpoint{0.000000in}{0.000000in}}%
\pgfpathlineto{\pgfqpoint{0.000000in}{-0.048611in}}%
\pgfusepath{stroke,fill}%
}%
\begin{pgfscope}%
\pgfsys@transformshift{4.660442in}{0.582899in}%
\pgfsys@useobject{currentmarker}{}%
\end{pgfscope}%
\end{pgfscope}%
\begin{pgfscope}%
\definecolor{textcolor}{rgb}{0.000000,0.000000,0.000000}%
\pgfsetstrokecolor{textcolor}%
\pgfsetfillcolor{textcolor}%
\pgftext[x=4.660442in,y=0.485677in,,top]{\color{textcolor}{\sffamily\fontsize{16.000000}{19.200000}\selectfont\catcode`\^=\active\def^{\ifmmode\sp\else\^{}\fi}\catcode`\%=\active\def
\end{pgfscope}%
\begin{pgfscope}%
\pgfsetbuttcap%
\pgfsetroundjoin%
\definecolor{currentfill}{rgb}{0.000000,0.000000,0.000000}%
\pgfsetfillcolor{currentfill}%
\pgfsetlinewidth{0.803000pt}%
\definecolor{currentstroke}{rgb}{0.000000,0.000000,0.000000}%
\pgfsetstrokecolor{currentstroke}%
\pgfsetdash{}{0pt}%
\pgfsys@defobject{currentmarker}{\pgfqpoint{0.000000in}{-0.048611in}}{\pgfqpoint{0.000000in}{0.000000in}}{%
\pgfpathmoveto{\pgfqpoint{0.000000in}{0.000000in}}%
\pgfpathlineto{\pgfqpoint{0.000000in}{-0.048611in}}%
\pgfusepath{stroke,fill}%
}%
\begin{pgfscope}%
\pgfsys@transformshift{5.130139in}{0.582899in}%
\pgfsys@useobject{currentmarker}{}%
\end{pgfscope}%
\end{pgfscope}%
\begin{pgfscope}%
\definecolor{textcolor}{rgb}{0.000000,0.000000,0.000000}%
\pgfsetstrokecolor{textcolor}%
\pgfsetfillcolor{textcolor}%
\pgftext[x=5.130139in,y=0.485677in,,top]{\color{textcolor}{\sffamily\fontsize{16.000000}{19.200000}\selectfont\catcode`\^=\active\def^{\ifmmode\sp\else\^{}\fi}\catcode`\%=\active\def
\end{pgfscope}%
\begin{pgfscope}%
\pgfsetbuttcap%
\pgfsetroundjoin%
\definecolor{currentfill}{rgb}{0.000000,0.000000,0.000000}%
\pgfsetfillcolor{currentfill}%
\pgfsetlinewidth{0.803000pt}%
\definecolor{currentstroke}{rgb}{0.000000,0.000000,0.000000}%
\pgfsetstrokecolor{currentstroke}%
\pgfsetdash{}{0pt}%
\pgfsys@defobject{currentmarker}{\pgfqpoint{0.000000in}{-0.048611in}}{\pgfqpoint{0.000000in}{0.000000in}}{%
\pgfpathmoveto{\pgfqpoint{0.000000in}{0.000000in}}%
\pgfpathlineto{\pgfqpoint{0.000000in}{-0.048611in}}%
\pgfusepath{stroke,fill}%
}%
\begin{pgfscope}%
\pgfsys@transformshift{5.599836in}{0.582899in}%
\pgfsys@useobject{currentmarker}{}%
\end{pgfscope}%
\end{pgfscope}%
\begin{pgfscope}%
\definecolor{textcolor}{rgb}{0.000000,0.000000,0.000000}%
\pgfsetstrokecolor{textcolor}%
\pgfsetfillcolor{textcolor}%
\pgftext[x=5.599836in,y=0.485677in,,top]{\color{textcolor}{\sffamily\fontsize{16.000000}{19.200000}\selectfont\catcode`\^=\active\def^{\ifmmode\sp\else\^{}\fi}\catcode`\%=\active\def
\end{pgfscope}%
\begin{pgfscope}%
\definecolor{textcolor}{rgb}{0.000000,0.000000,0.000000}%
\pgfsetstrokecolor{textcolor}%
\pgfsetfillcolor{textcolor}%
\pgftext[x=3.345290in,y=0.215061in,,top]{\color{textcolor}{\sffamily\fontsize{16.000000}{19.200000}\selectfont\catcode`\^=\active\def^{\ifmmode\sp\else\^{}\fi}\catcode`\%=\active\def
\end{pgfscope}%
\begin{pgfscope}%
\pgfsetbuttcap%
\pgfsetroundjoin%
\definecolor{currentfill}{rgb}{0.000000,0.000000,0.000000}%
\pgfsetfillcolor{currentfill}%
\pgfsetlinewidth{0.803000pt}%
\definecolor{currentstroke}{rgb}{0.000000,0.000000,0.000000}%
\pgfsetstrokecolor{currentstroke}%
\pgfsetdash{}{0pt}%
\pgfsys@defobject{currentmarker}{\pgfqpoint{-0.048611in}{0.000000in}}{\pgfqpoint{-0.000000in}{0.000000in}}{%
\pgfpathmoveto{\pgfqpoint{-0.000000in}{0.000000in}}%
\pgfpathlineto{\pgfqpoint{-0.048611in}{0.000000in}}%
\pgfusepath{stroke,fill}%
}%
\begin{pgfscope}%
\pgfsys@transformshift{0.865290in}{1.043928in}%
\pgfsys@useobject{currentmarker}{}%
\end{pgfscope}%
\end{pgfscope}%
\begin{pgfscope}%
\definecolor{textcolor}{rgb}{0.000000,0.000000,0.000000}%
\pgfsetstrokecolor{textcolor}%
\pgfsetfillcolor{textcolor}%
\pgftext[x=0.270616in, y=0.959509in, left, base]{\color{textcolor}{\sffamily\fontsize{16.000000}{19.200000}\selectfont\catcode`\^=\active\def^{\ifmmode\sp\else\^{}\fi}\catcode`\%=\active\def
\end{pgfscope}%
\begin{pgfscope}%
\pgfsetbuttcap%
\pgfsetroundjoin%
\definecolor{currentfill}{rgb}{0.000000,0.000000,0.000000}%
\pgfsetfillcolor{currentfill}%
\pgfsetlinewidth{0.803000pt}%
\definecolor{currentstroke}{rgb}{0.000000,0.000000,0.000000}%
\pgfsetstrokecolor{currentstroke}%
\pgfsetdash{}{0pt}%
\pgfsys@defobject{currentmarker}{\pgfqpoint{-0.048611in}{0.000000in}}{\pgfqpoint{-0.000000in}{0.000000in}}{%
\pgfpathmoveto{\pgfqpoint{-0.000000in}{0.000000in}}%
\pgfpathlineto{\pgfqpoint{-0.048611in}{0.000000in}}%
\pgfusepath{stroke,fill}%
}%
\begin{pgfscope}%
\pgfsys@transformshift{0.865290in}{1.527596in}%
\pgfsys@useobject{currentmarker}{}%
\end{pgfscope}%
\end{pgfscope}%
\begin{pgfscope}%
\definecolor{textcolor}{rgb}{0.000000,0.000000,0.000000}%
\pgfsetstrokecolor{textcolor}%
\pgfsetfillcolor{textcolor}%
\pgftext[x=0.270616in, y=1.443178in, left, base]{\color{textcolor}{\sffamily\fontsize{16.000000}{19.200000}\selectfont\catcode`\^=\active\def^{\ifmmode\sp\else\^{}\fi}\catcode`\%=\active\def
\end{pgfscope}%
\begin{pgfscope}%
\pgfsetbuttcap%
\pgfsetroundjoin%
\definecolor{currentfill}{rgb}{0.000000,0.000000,0.000000}%
\pgfsetfillcolor{currentfill}%
\pgfsetlinewidth{0.803000pt}%
\definecolor{currentstroke}{rgb}{0.000000,0.000000,0.000000}%
\pgfsetstrokecolor{currentstroke}%
\pgfsetdash{}{0pt}%
\pgfsys@defobject{currentmarker}{\pgfqpoint{-0.048611in}{0.000000in}}{\pgfqpoint{-0.000000in}{0.000000in}}{%
\pgfpathmoveto{\pgfqpoint{-0.000000in}{0.000000in}}%
\pgfpathlineto{\pgfqpoint{-0.048611in}{0.000000in}}%
\pgfusepath{stroke,fill}%
}%
\begin{pgfscope}%
\pgfsys@transformshift{0.865290in}{2.011264in}%
\pgfsys@useobject{currentmarker}{}%
\end{pgfscope}%
\end{pgfscope}%
\begin{pgfscope}%
\definecolor{textcolor}{rgb}{0.000000,0.000000,0.000000}%
\pgfsetstrokecolor{textcolor}%
\pgfsetfillcolor{textcolor}%
\pgftext[x=0.346658in, y=1.926846in, left, base]{\color{textcolor}{\sffamily\fontsize{16.000000}{19.200000}\selectfont\catcode`\^=\active\def^{\ifmmode\sp\else\^{}\fi}\catcode`\%=\active\def
\end{pgfscope}%
\begin{pgfscope}%
\pgfsetbuttcap%
\pgfsetroundjoin%
\definecolor{currentfill}{rgb}{0.000000,0.000000,0.000000}%
\pgfsetfillcolor{currentfill}%
\pgfsetlinewidth{0.803000pt}%
\definecolor{currentstroke}{rgb}{0.000000,0.000000,0.000000}%
\pgfsetstrokecolor{currentstroke}%
\pgfsetdash{}{0pt}%
\pgfsys@defobject{currentmarker}{\pgfqpoint{-0.048611in}{0.000000in}}{\pgfqpoint{-0.000000in}{0.000000in}}{%
\pgfpathmoveto{\pgfqpoint{-0.000000in}{0.000000in}}%
\pgfpathlineto{\pgfqpoint{-0.048611in}{0.000000in}}%
\pgfusepath{stroke,fill}%
}%
\begin{pgfscope}%
\pgfsys@transformshift{0.865290in}{2.494933in}%
\pgfsys@useobject{currentmarker}{}%
\end{pgfscope}%
\end{pgfscope}%
\begin{pgfscope}%
\definecolor{textcolor}{rgb}{0.000000,0.000000,0.000000}%
\pgfsetstrokecolor{textcolor}%
\pgfsetfillcolor{textcolor}%
\pgftext[x=0.346658in, y=2.410514in, left, base]{\color{textcolor}{\sffamily\fontsize{16.000000}{19.200000}\selectfont\catcode`\^=\active\def^{\ifmmode\sp\else\^{}\fi}\catcode`\%=\active\def
\end{pgfscope}%
\begin{pgfscope}%
\pgfsetbuttcap%
\pgfsetroundjoin%
\definecolor{currentfill}{rgb}{0.000000,0.000000,0.000000}%
\pgfsetfillcolor{currentfill}%
\pgfsetlinewidth{0.803000pt}%
\definecolor{currentstroke}{rgb}{0.000000,0.000000,0.000000}%
\pgfsetstrokecolor{currentstroke}%
\pgfsetdash{}{0pt}%
\pgfsys@defobject{currentmarker}{\pgfqpoint{-0.048611in}{0.000000in}}{\pgfqpoint{-0.000000in}{0.000000in}}{%
\pgfpathmoveto{\pgfqpoint{-0.000000in}{0.000000in}}%
\pgfpathlineto{\pgfqpoint{-0.048611in}{0.000000in}}%
\pgfusepath{stroke,fill}%
}%
\begin{pgfscope}%
\pgfsys@transformshift{0.865290in}{2.978601in}%
\pgfsys@useobject{currentmarker}{}%
\end{pgfscope}%
\end{pgfscope}%
\begin{pgfscope}%
\definecolor{textcolor}{rgb}{0.000000,0.000000,0.000000}%
\pgfsetstrokecolor{textcolor}%
\pgfsetfillcolor{textcolor}%
\pgftext[x=0.346658in, y=2.894182in, left, base]{\color{textcolor}{\sffamily\fontsize{16.000000}{19.200000}\selectfont\catcode`\^=\active\def^{\ifmmode\sp\else\^{}\fi}\catcode`\%=\active\def
\end{pgfscope}%
\begin{pgfscope}%
\pgfsetbuttcap%
\pgfsetroundjoin%
\definecolor{currentfill}{rgb}{0.000000,0.000000,0.000000}%
\pgfsetfillcolor{currentfill}%
\pgfsetlinewidth{0.803000pt}%
\definecolor{currentstroke}{rgb}{0.000000,0.000000,0.000000}%
\pgfsetstrokecolor{currentstroke}%
\pgfsetdash{}{0pt}%
\pgfsys@defobject{currentmarker}{\pgfqpoint{-0.048611in}{0.000000in}}{\pgfqpoint{-0.000000in}{0.000000in}}{%
\pgfpathmoveto{\pgfqpoint{-0.000000in}{0.000000in}}%
\pgfpathlineto{\pgfqpoint{-0.048611in}{0.000000in}}%
\pgfusepath{stroke,fill}%
}%
\begin{pgfscope}%
\pgfsys@transformshift{0.865290in}{3.462269in}%
\pgfsys@useobject{currentmarker}{}%
\end{pgfscope}%
\end{pgfscope}%
\begin{pgfscope}%
\definecolor{textcolor}{rgb}{0.000000,0.000000,0.000000}%
\pgfsetstrokecolor{textcolor}%
\pgfsetfillcolor{textcolor}%
\pgftext[x=0.346658in, y=3.377851in, left, base]{\color{textcolor}{\sffamily\fontsize{16.000000}{19.200000}\selectfont\catcode`\^=\active\def^{\ifmmode\sp\else\^{}\fi}\catcode`\%=\active\def
\end{pgfscope}%
\begin{pgfscope}%
\pgfsetbuttcap%
\pgfsetroundjoin%
\definecolor{currentfill}{rgb}{0.000000,0.000000,0.000000}%
\pgfsetfillcolor{currentfill}%
\pgfsetlinewidth{0.803000pt}%
\definecolor{currentstroke}{rgb}{0.000000,0.000000,0.000000}%
\pgfsetstrokecolor{currentstroke}%
\pgfsetdash{}{0pt}%
\pgfsys@defobject{currentmarker}{\pgfqpoint{-0.048611in}{0.000000in}}{\pgfqpoint{-0.000000in}{0.000000in}}{%
\pgfpathmoveto{\pgfqpoint{-0.000000in}{0.000000in}}%
\pgfpathlineto{\pgfqpoint{-0.048611in}{0.000000in}}%
\pgfusepath{stroke,fill}%
}%
\begin{pgfscope}%
\pgfsys@transformshift{0.865290in}{3.945937in}%
\pgfsys@useobject{currentmarker}{}%
\end{pgfscope}%
\end{pgfscope}%
\begin{pgfscope}%
\definecolor{textcolor}{rgb}{0.000000,0.000000,0.000000}%
\pgfsetstrokecolor{textcolor}%
\pgfsetfillcolor{textcolor}%
\pgftext[x=0.346658in, y=3.861519in, left, base]{\color{textcolor}{\sffamily\fontsize{16.000000}{19.200000}\selectfont\catcode`\^=\active\def^{\ifmmode\sp\else\^{}\fi}\catcode`\%=\active\def
\end{pgfscope}%
\begin{pgfscope}%
\definecolor{textcolor}{rgb}{0.000000,0.000000,0.000000}%
\pgfsetstrokecolor{textcolor}%
\pgfsetfillcolor{textcolor}%
\pgftext[x=0.215061in,y=2.430899in,,bottom,rotate=90.000000]{\color{textcolor}{\sffamily\fontsize{16.000000}{19.200000}\selectfont\catcode`\^=\active\def^{\ifmmode\sp\else\^{}\fi}\catcode`\%=\active\def
\end{pgfscope}%
\begin{pgfscope}%
\pgfpathrectangle{\pgfqpoint{0.865290in}{0.582899in}}{\pgfqpoint{4.960000in}{3.696000in}}%
\pgfusepath{clip}%
\pgfsetrectcap%
\pgfsetroundjoin%
\pgfsetlinewidth{1.505625pt}%
\definecolor{currentstroke}{rgb}{0.121569,0.466667,0.705882}%
\pgfsetstrokecolor{currentstroke}%
\pgfsetdash{}{0pt}%
\pgfpathmoveto{\pgfqpoint{1.090745in}{4.110899in}}%
\pgfpathlineto{\pgfqpoint{1.278623in}{3.966596in}}%
\pgfpathlineto{\pgfqpoint{1.466502in}{3.818472in}}%
\pgfpathlineto{\pgfqpoint{1.654381in}{3.442468in}}%
\pgfpathlineto{\pgfqpoint{1.842260in}{3.295645in}}%
\pgfpathlineto{\pgfqpoint{2.030139in}{2.985088in}}%
\pgfpathlineto{\pgfqpoint{2.218017in}{2.567408in}}%
\pgfpathlineto{\pgfqpoint{2.405896in}{2.226538in}}%
\pgfpathlineto{\pgfqpoint{2.593775in}{1.684115in}}%
\pgfpathlineto{\pgfqpoint{2.781654in}{1.144562in}}%
\pgfpathlineto{\pgfqpoint{2.969533in}{0.787481in}}%
\pgfpathlineto{\pgfqpoint{3.157411in}{0.750899in}}%
\pgfpathlineto{\pgfqpoint{3.345290in}{0.750899in}}%
\pgfpathlineto{\pgfqpoint{3.533169in}{0.750899in}}%
\pgfpathlineto{\pgfqpoint{3.721048in}{0.750899in}}%
\pgfpathlineto{\pgfqpoint{3.908927in}{0.750899in}}%
\pgfpathlineto{\pgfqpoint{4.096805in}{0.750899in}}%
\pgfpathlineto{\pgfqpoint{4.284684in}{0.750899in}}%
\pgfpathlineto{\pgfqpoint{4.472563in}{0.750899in}}%
\pgfpathlineto{\pgfqpoint{4.660442in}{0.750899in}}%
\pgfpathlineto{\pgfqpoint{4.848320in}{0.750899in}}%
\pgfpathlineto{\pgfqpoint{5.036199in}{0.750899in}}%
\pgfpathlineto{\pgfqpoint{5.224078in}{0.750899in}}%
\pgfpathlineto{\pgfqpoint{5.411957in}{0.750899in}}%
\pgfpathlineto{\pgfqpoint{5.599836in}{0.750899in}}%
\pgfusepath{stroke}%
\end{pgfscope}%
\begin{pgfscope}%
\pgfpathrectangle{\pgfqpoint{0.865290in}{0.582899in}}{\pgfqpoint{4.960000in}{3.696000in}}%
\pgfusepath{clip}%
\pgfsetrectcap%
\pgfsetroundjoin%
\pgfsetlinewidth{1.505625pt}%
\definecolor{currentstroke}{rgb}{1.000000,0.498039,0.054902}%
\pgfsetstrokecolor{currentstroke}%
\pgfsetdash{}{0pt}%
\pgfpathmoveto{\pgfqpoint{1.090745in}{4.069125in}}%
\pgfpathlineto{\pgfqpoint{1.278623in}{3.935037in}}%
\pgfpathlineto{\pgfqpoint{1.466502in}{3.626345in}}%
\pgfpathlineto{\pgfqpoint{1.654381in}{3.290505in}}%
\pgfpathlineto{\pgfqpoint{1.842260in}{2.994723in}}%
\pgfpathlineto{\pgfqpoint{2.030139in}{2.780719in}}%
\pgfpathlineto{\pgfqpoint{2.218017in}{2.436756in}}%
\pgfpathlineto{\pgfqpoint{2.405896in}{2.036026in}}%
\pgfpathlineto{\pgfqpoint{2.593775in}{1.591215in}}%
\pgfpathlineto{\pgfqpoint{2.781654in}{1.171672in}}%
\pgfpathlineto{\pgfqpoint{2.969533in}{1.045681in}}%
\pgfpathlineto{\pgfqpoint{3.157411in}{1.048292in}}%
\pgfpathlineto{\pgfqpoint{3.345290in}{1.048292in}}%
\pgfpathlineto{\pgfqpoint{3.533169in}{1.048120in}}%
\pgfpathlineto{\pgfqpoint{3.721048in}{1.048120in}}%
\pgfpathlineto{\pgfqpoint{3.908927in}{1.048120in}}%
\pgfpathlineto{\pgfqpoint{4.096805in}{1.048120in}}%
\pgfpathlineto{\pgfqpoint{4.284684in}{1.048120in}}%
\pgfpathlineto{\pgfqpoint{4.472563in}{1.048120in}}%
\pgfpathlineto{\pgfqpoint{4.660442in}{1.048292in}}%
\pgfpathlineto{\pgfqpoint{4.848320in}{1.048292in}}%
\pgfpathlineto{\pgfqpoint{5.036199in}{1.048120in}}%
\pgfpathlineto{\pgfqpoint{5.224078in}{1.048292in}}%
\pgfpathlineto{\pgfqpoint{5.411957in}{1.048292in}}%
\pgfpathlineto{\pgfqpoint{5.599836in}{1.048292in}}%
\pgfusepath{stroke}%
\end{pgfscope}%
\begin{pgfscope}%
\pgfsetrectcap%
\pgfsetmiterjoin%
\pgfsetlinewidth{0.803000pt}%
\definecolor{currentstroke}{rgb}{0.000000,0.000000,0.000000}%
\pgfsetstrokecolor{currentstroke}%
\pgfsetdash{}{0pt}%
\pgfpathmoveto{\pgfqpoint{0.865290in}{0.582899in}}%
\pgfpathlineto{\pgfqpoint{0.865290in}{4.278899in}}%
\pgfusepath{stroke}%
\end{pgfscope}%
\begin{pgfscope}%
\pgfsetrectcap%
\pgfsetmiterjoin%
\pgfsetlinewidth{0.803000pt}%
\definecolor{currentstroke}{rgb}{0.000000,0.000000,0.000000}%
\pgfsetstrokecolor{currentstroke}%
\pgfsetdash{}{0pt}%
\pgfpathmoveto{\pgfqpoint{5.825290in}{0.582899in}}%
\pgfpathlineto{\pgfqpoint{5.825290in}{4.278899in}}%
\pgfusepath{stroke}%
\end{pgfscope}%
\begin{pgfscope}%
\pgfsetrectcap%
\pgfsetmiterjoin%
\pgfsetlinewidth{0.803000pt}%
\definecolor{currentstroke}{rgb}{0.000000,0.000000,0.000000}%
\pgfsetstrokecolor{currentstroke}%
\pgfsetdash{}{0pt}%
\pgfpathmoveto{\pgfqpoint{0.865290in}{0.582899in}}%
\pgfpathlineto{\pgfqpoint{5.825290in}{0.582899in}}%
\pgfusepath{stroke}%
\end{pgfscope}%
\begin{pgfscope}%
\pgfsetrectcap%
\pgfsetmiterjoin%
\pgfsetlinewidth{0.803000pt}%
\definecolor{currentstroke}{rgb}{0.000000,0.000000,0.000000}%
\pgfsetstrokecolor{currentstroke}%
\pgfsetdash{}{0pt}%
\pgfpathmoveto{\pgfqpoint{0.865290in}{4.278899in}}%
\pgfpathlineto{\pgfqpoint{5.825290in}{4.278899in}}%
\pgfusepath{stroke}%
\end{pgfscope}%
\begin{pgfscope}%
\pgfsetbuttcap%
\pgfsetmiterjoin%
\definecolor{currentfill}{rgb}{1.000000,1.000000,1.000000}%
\pgfsetfillcolor{currentfill}%
\pgfsetfillopacity{0.800000}%
\pgfsetlinewidth{1.003750pt}%
\definecolor{currentstroke}{rgb}{0.800000,0.800000,0.800000}%
\pgfsetstrokecolor{currentstroke}%
\pgfsetstrokeopacity{0.800000}%
\pgfsetdash{}{0pt}%
\pgfpathmoveto{\pgfqpoint{3.158927in}{3.448778in}}%
\pgfpathlineto{\pgfqpoint{5.669735in}{3.448778in}}%
\pgfpathquadraticcurveto{\pgfqpoint{5.714179in}{3.448778in}}{\pgfqpoint{5.714179in}{3.493222in}}%
\pgfpathlineto{\pgfqpoint{5.714179in}{4.123343in}}%
\pgfpathquadraticcurveto{\pgfqpoint{5.714179in}{4.167788in}}{\pgfqpoint{5.669735in}{4.167788in}}%
\pgfpathlineto{\pgfqpoint{3.158927in}{4.167788in}}%
\pgfpathquadraticcurveto{\pgfqpoint{3.114483in}{4.167788in}}{\pgfqpoint{3.114483in}{4.123343in}}%
\pgfpathlineto{\pgfqpoint{3.114483in}{3.493222in}}%
\pgfpathquadraticcurveto{\pgfqpoint{3.114483in}{3.448778in}}{\pgfqpoint{3.158927in}{3.448778in}}%
\pgfpathlineto{\pgfqpoint{3.158927in}{3.448778in}}%
\pgfpathclose%
\pgfusepath{stroke,fill}%
\end{pgfscope}%
\begin{pgfscope}%
\pgfsetrectcap%
\pgfsetroundjoin%
\pgfsetlinewidth{1.505625pt}%
\definecolor{currentstroke}{rgb}{0.121569,0.466667,0.705882}%
\pgfsetstrokecolor{currentstroke}%
\pgfsetdash{}{0pt}%
\pgfpathmoveto{\pgfqpoint{3.203372in}{3.987840in}}%
\pgfpathlineto{\pgfqpoint{3.425594in}{3.987840in}}%
\pgfpathlineto{\pgfqpoint{3.647816in}{3.987840in}}%
\pgfusepath{stroke}%
\end{pgfscope}%
\begin{pgfscope}%
\definecolor{textcolor}{rgb}{0.000000,0.000000,0.000000}%
\pgfsetstrokecolor{textcolor}%
\pgfsetfillcolor{textcolor}%
\pgftext[x=3.825594in,y=3.910062in,left,base]{\color{textcolor}{\sffamily\fontsize{16.000000}{19.200000}\selectfont\catcode`\^=\active\def^{\ifmmode\sp\else\^{}\fi}\catcode`\%=\active\def
\end{pgfscope}%
\begin{pgfscope}%
\pgfsetrectcap%
\pgfsetroundjoin%
\pgfsetlinewidth{1.505625pt}%
\definecolor{currentstroke}{rgb}{1.000000,0.498039,0.054902}%
\pgfsetstrokecolor{currentstroke}%
\pgfsetdash{}{0pt}%
\pgfpathmoveto{\pgfqpoint{3.203372in}{3.661668in}}%
\pgfpathlineto{\pgfqpoint{3.425594in}{3.661668in}}%
\pgfpathlineto{\pgfqpoint{3.647816in}{3.661668in}}%
\pgfusepath{stroke}%
\end{pgfscope}%
\begin{pgfscope}%
\definecolor{textcolor}{rgb}{0.000000,0.000000,0.000000}%
\pgfsetstrokecolor{textcolor}%
\pgfsetfillcolor{textcolor}%
\pgftext[x=3.825594in,y=3.583890in,left,base]{\color{textcolor}{\sffamily\fontsize{16.000000}{19.200000}\selectfont\catcode`\^=\active\def^{\ifmmode\sp\else\^{}\fi}\catcode`\%=\active\def
\end{pgfscope}%
\end{pgfpicture}%
\makeatother%
\endgroup%

%% file: figs/network_conv_ei_as-735.pgf
\begingroup%
\makeatletter%
\begin{pgfpicture}%
\pgfpathrectangle{\pgfpointorigin}{\pgfqpoint{5.825290in}{4.278899in}}%
\pgfusepath{use as bounding box, clip}%
\begin{pgfscope}%
\pgfsetbuttcap%
\pgfsetmiterjoin%
\definecolor{currentfill}{rgb}{1.000000,1.000000,1.000000}%
\pgfsetfillcolor{currentfill}%
\pgfsetlinewidth{0.000000pt}%
\definecolor{currentstroke}{rgb}{1.000000,1.000000,1.000000}%
\pgfsetstrokecolor{currentstroke}%
\pgfsetdash{}{0pt}%
\pgfpathmoveto{\pgfqpoint{0.000000in}{0.000000in}}%
\pgfpathlineto{\pgfqpoint{5.825290in}{0.000000in}}%
\pgfpathlineto{\pgfqpoint{5.825290in}{4.278899in}}%
\pgfpathlineto{\pgfqpoint{0.000000in}{4.278899in}}%
\pgfpathlineto{\pgfqpoint{0.000000in}{0.000000in}}%
\pgfpathclose%
\pgfusepath{fill}%
\end{pgfscope}%
\begin{pgfscope}%
\pgfsetbuttcap%
\pgfsetmiterjoin%
\definecolor{currentfill}{rgb}{1.000000,1.000000,1.000000}%
\pgfsetfillcolor{currentfill}%
\pgfsetlinewidth{0.000000pt}%
\definecolor{currentstroke}{rgb}{0.000000,0.000000,0.000000}%
\pgfsetstrokecolor{currentstroke}%
\pgfsetstrokeopacity{0.000000}%
\pgfsetdash{}{0pt}%
\pgfpathmoveto{\pgfqpoint{0.865290in}{0.582899in}}%
\pgfpathlineto{\pgfqpoint{5.825290in}{0.582899in}}%
\pgfpathlineto{\pgfqpoint{5.825290in}{4.278899in}}%
\pgfpathlineto{\pgfqpoint{0.865290in}{4.278899in}}%
\pgfpathlineto{\pgfqpoint{0.865290in}{0.582899in}}%
\pgfpathclose%
\pgfusepath{fill}%
\end{pgfscope}%
\begin{pgfscope}%
\pgfsetbuttcap%
\pgfsetroundjoin%
\definecolor{currentfill}{rgb}{0.000000,0.000000,0.000000}%
\pgfsetfillcolor{currentfill}%
\pgfsetlinewidth{0.803000pt}%
\definecolor{currentstroke}{rgb}{0.000000,0.000000,0.000000}%
\pgfsetstrokecolor{currentstroke}%
\pgfsetdash{}{0pt}%
\pgfsys@defobject{currentmarker}{\pgfqpoint{0.000000in}{-0.048611in}}{\pgfqpoint{0.000000in}{0.000000in}}{%
\pgfpathmoveto{\pgfqpoint{0.000000in}{0.000000in}}%
\pgfpathlineto{\pgfqpoint{0.000000in}{-0.048611in}}%
\pgfusepath{stroke,fill}%
}%
\begin{pgfscope}%
\pgfsys@transformshift{1.372563in}{0.582899in}%
\pgfsys@useobject{currentmarker}{}%
\end{pgfscope}%
\end{pgfscope}%
\begin{pgfscope}%
\definecolor{textcolor}{rgb}{0.000000,0.000000,0.000000}%
\pgfsetstrokecolor{textcolor}%
\pgfsetfillcolor{textcolor}%
\pgftext[x=1.372563in,y=0.485677in,,top]{\color{textcolor}{\sffamily\fontsize{16.000000}{19.200000}\selectfont\catcode`\^=\active\def^{\ifmmode\sp\else\^{}\fi}\catcode`\%=\active\def
\end{pgfscope}%
\begin{pgfscope}%
\pgfsetbuttcap%
\pgfsetroundjoin%
\definecolor{currentfill}{rgb}{0.000000,0.000000,0.000000}%
\pgfsetfillcolor{currentfill}%
\pgfsetlinewidth{0.803000pt}%
\definecolor{currentstroke}{rgb}{0.000000,0.000000,0.000000}%
\pgfsetstrokecolor{currentstroke}%
\pgfsetdash{}{0pt}%
\pgfsys@defobject{currentmarker}{\pgfqpoint{0.000000in}{-0.048611in}}{\pgfqpoint{0.000000in}{0.000000in}}{%
\pgfpathmoveto{\pgfqpoint{0.000000in}{0.000000in}}%
\pgfpathlineto{\pgfqpoint{0.000000in}{-0.048611in}}%
\pgfusepath{stroke,fill}%
}%
\begin{pgfscope}%
\pgfsys@transformshift{1.842260in}{0.582899in}%
\pgfsys@useobject{currentmarker}{}%
\end{pgfscope}%
\end{pgfscope}%
\begin{pgfscope}%
\definecolor{textcolor}{rgb}{0.000000,0.000000,0.000000}%
\pgfsetstrokecolor{textcolor}%
\pgfsetfillcolor{textcolor}%
\pgftext[x=1.842260in,y=0.485677in,,top]{\color{textcolor}{\sffamily\fontsize{16.000000}{19.200000}\selectfont\catcode`\^=\active\def^{\ifmmode\sp\else\^{}\fi}\catcode`\%=\active\def
\end{pgfscope}%
\begin{pgfscope}%
\pgfsetbuttcap%
\pgfsetroundjoin%
\definecolor{currentfill}{rgb}{0.000000,0.000000,0.000000}%
\pgfsetfillcolor{currentfill}%
\pgfsetlinewidth{0.803000pt}%
\definecolor{currentstroke}{rgb}{0.000000,0.000000,0.000000}%
\pgfsetstrokecolor{currentstroke}%
\pgfsetdash{}{0pt}%
\pgfsys@defobject{currentmarker}{\pgfqpoint{0.000000in}{-0.048611in}}{\pgfqpoint{0.000000in}{0.000000in}}{%
\pgfpathmoveto{\pgfqpoint{0.000000in}{0.000000in}}%
\pgfpathlineto{\pgfqpoint{0.000000in}{-0.048611in}}%
\pgfusepath{stroke,fill}%
}%
\begin{pgfscope}%
\pgfsys@transformshift{2.311957in}{0.582899in}%
\pgfsys@useobject{currentmarker}{}%
\end{pgfscope}%
\end{pgfscope}%
\begin{pgfscope}%
\definecolor{textcolor}{rgb}{0.000000,0.000000,0.000000}%
\pgfsetstrokecolor{textcolor}%
\pgfsetfillcolor{textcolor}%
\pgftext[x=2.311957in,y=0.485677in,,top]{\color{textcolor}{\sffamily\fontsize{16.000000}{19.200000}\selectfont\catcode`\^=\active\def^{\ifmmode\sp\else\^{}\fi}\catcode`\%=\active\def
\end{pgfscope}%
\begin{pgfscope}%
\pgfsetbuttcap%
\pgfsetroundjoin%
\definecolor{currentfill}{rgb}{0.000000,0.000000,0.000000}%
\pgfsetfillcolor{currentfill}%
\pgfsetlinewidth{0.803000pt}%
\definecolor{currentstroke}{rgb}{0.000000,0.000000,0.000000}%
\pgfsetstrokecolor{currentstroke}%
\pgfsetdash{}{0pt}%
\pgfsys@defobject{currentmarker}{\pgfqpoint{0.000000in}{-0.048611in}}{\pgfqpoint{0.000000in}{0.000000in}}{%
\pgfpathmoveto{\pgfqpoint{0.000000in}{0.000000in}}%
\pgfpathlineto{\pgfqpoint{0.000000in}{-0.048611in}}%
\pgfusepath{stroke,fill}%
}%
\begin{pgfscope}%
\pgfsys@transformshift{2.781654in}{0.582899in}%
\pgfsys@useobject{currentmarker}{}%
\end{pgfscope}%
\end{pgfscope}%
\begin{pgfscope}%
\definecolor{textcolor}{rgb}{0.000000,0.000000,0.000000}%
\pgfsetstrokecolor{textcolor}%
\pgfsetfillcolor{textcolor}%
\pgftext[x=2.781654in,y=0.485677in,,top]{\color{textcolor}{\sffamily\fontsize{16.000000}{19.200000}\selectfont\catcode`\^=\active\def^{\ifmmode\sp\else\^{}\fi}\catcode`\%=\active\def
\end{pgfscope}%
\begin{pgfscope}%
\pgfsetbuttcap%
\pgfsetroundjoin%
\definecolor{currentfill}{rgb}{0.000000,0.000000,0.000000}%
\pgfsetfillcolor{currentfill}%
\pgfsetlinewidth{0.803000pt}%
\definecolor{currentstroke}{rgb}{0.000000,0.000000,0.000000}%
\pgfsetstrokecolor{currentstroke}%
\pgfsetdash{}{0pt}%
\pgfsys@defobject{currentmarker}{\pgfqpoint{0.000000in}{-0.048611in}}{\pgfqpoint{0.000000in}{0.000000in}}{%
\pgfpathmoveto{\pgfqpoint{0.000000in}{0.000000in}}%
\pgfpathlineto{\pgfqpoint{0.000000in}{-0.048611in}}%
\pgfusepath{stroke,fill}%
}%
\begin{pgfscope}%
\pgfsys@transformshift{3.251351in}{0.582899in}%
\pgfsys@useobject{currentmarker}{}%
\end{pgfscope}%
\end{pgfscope}%
\begin{pgfscope}%
\definecolor{textcolor}{rgb}{0.000000,0.000000,0.000000}%
\pgfsetstrokecolor{textcolor}%
\pgfsetfillcolor{textcolor}%
\pgftext[x=3.251351in,y=0.485677in,,top]{\color{textcolor}{\sffamily\fontsize{16.000000}{19.200000}\selectfont\catcode`\^=\active\def^{\ifmmode\sp\else\^{}\fi}\catcode`\%=\active\def
\end{pgfscope}%
\begin{pgfscope}%
\pgfsetbuttcap%
\pgfsetroundjoin%
\definecolor{currentfill}{rgb}{0.000000,0.000000,0.000000}%
\pgfsetfillcolor{currentfill}%
\pgfsetlinewidth{0.803000pt}%
\definecolor{currentstroke}{rgb}{0.000000,0.000000,0.000000}%
\pgfsetstrokecolor{currentstroke}%
\pgfsetdash{}{0pt}%
\pgfsys@defobject{currentmarker}{\pgfqpoint{0.000000in}{-0.048611in}}{\pgfqpoint{0.000000in}{0.000000in}}{%
\pgfpathmoveto{\pgfqpoint{0.000000in}{0.000000in}}%
\pgfpathlineto{\pgfqpoint{0.000000in}{-0.048611in}}%
\pgfusepath{stroke,fill}%
}%
\begin{pgfscope}%
\pgfsys@transformshift{3.721048in}{0.582899in}%
\pgfsys@useobject{currentmarker}{}%
\end{pgfscope}%
\end{pgfscope}%
\begin{pgfscope}%
\definecolor{textcolor}{rgb}{0.000000,0.000000,0.000000}%
\pgfsetstrokecolor{textcolor}%
\pgfsetfillcolor{textcolor}%
\pgftext[x=3.721048in,y=0.485677in,,top]{\color{textcolor}{\sffamily\fontsize{16.000000}{19.200000}\selectfont\catcode`\^=\active\def^{\ifmmode\sp\else\^{}\fi}\catcode`\%=\active\def
\end{pgfscope}%
\begin{pgfscope}%
\pgfsetbuttcap%
\pgfsetroundjoin%
\definecolor{currentfill}{rgb}{0.000000,0.000000,0.000000}%
\pgfsetfillcolor{currentfill}%
\pgfsetlinewidth{0.803000pt}%
\definecolor{currentstroke}{rgb}{0.000000,0.000000,0.000000}%
\pgfsetstrokecolor{currentstroke}%
\pgfsetdash{}{0pt}%
\pgfsys@defobject{currentmarker}{\pgfqpoint{0.000000in}{-0.048611in}}{\pgfqpoint{0.000000in}{0.000000in}}{%
\pgfpathmoveto{\pgfqpoint{0.000000in}{0.000000in}}%
\pgfpathlineto{\pgfqpoint{0.000000in}{-0.048611in}}%
\pgfusepath{stroke,fill}%
}%
\begin{pgfscope}%
\pgfsys@transformshift{4.190745in}{0.582899in}%
\pgfsys@useobject{currentmarker}{}%
\end{pgfscope}%
\end{pgfscope}%
\begin{pgfscope}%
\definecolor{textcolor}{rgb}{0.000000,0.000000,0.000000}%
\pgfsetstrokecolor{textcolor}%
\pgfsetfillcolor{textcolor}%
\pgftext[x=4.190745in,y=0.485677in,,top]{\color{textcolor}{\sffamily\fontsize{16.000000}{19.200000}\selectfont\catcode`\^=\active\def^{\ifmmode\sp\else\^{}\fi}\catcode`\%=\active\def
\end{pgfscope}%
\begin{pgfscope}%
\pgfsetbuttcap%
\pgfsetroundjoin%
\definecolor{currentfill}{rgb}{0.000000,0.000000,0.000000}%
\pgfsetfillcolor{currentfill}%
\pgfsetlinewidth{0.803000pt}%
\definecolor{currentstroke}{rgb}{0.000000,0.000000,0.000000}%
\pgfsetstrokecolor{currentstroke}%
\pgfsetdash{}{0pt}%
\pgfsys@defobject{currentmarker}{\pgfqpoint{0.000000in}{-0.048611in}}{\pgfqpoint{0.000000in}{0.000000in}}{%
\pgfpathmoveto{\pgfqpoint{0.000000in}{0.000000in}}%
\pgfpathlineto{\pgfqpoint{0.000000in}{-0.048611in}}%
\pgfusepath{stroke,fill}%
}%
\begin{pgfscope}%
\pgfsys@transformshift{4.660442in}{0.582899in}%
\pgfsys@useobject{currentmarker}{}%
\end{pgfscope}%
\end{pgfscope}%
\begin{pgfscope}%
\definecolor{textcolor}{rgb}{0.000000,0.000000,0.000000}%
\pgfsetstrokecolor{textcolor}%
\pgfsetfillcolor{textcolor}%
\pgftext[x=4.660442in,y=0.485677in,,top]{\color{textcolor}{\sffamily\fontsize{16.000000}{19.200000}\selectfont\catcode`\^=\active\def^{\ifmmode\sp\else\^{}\fi}\catcode`\%=\active\def
\end{pgfscope}%
\begin{pgfscope}%
\pgfsetbuttcap%
\pgfsetroundjoin%
\definecolor{currentfill}{rgb}{0.000000,0.000000,0.000000}%
\pgfsetfillcolor{currentfill}%
\pgfsetlinewidth{0.803000pt}%
\definecolor{currentstroke}{rgb}{0.000000,0.000000,0.000000}%
\pgfsetstrokecolor{currentstroke}%
\pgfsetdash{}{0pt}%
\pgfsys@defobject{currentmarker}{\pgfqpoint{0.000000in}{-0.048611in}}{\pgfqpoint{0.000000in}{0.000000in}}{%
\pgfpathmoveto{\pgfqpoint{0.000000in}{0.000000in}}%
\pgfpathlineto{\pgfqpoint{0.000000in}{-0.048611in}}%
\pgfusepath{stroke,fill}%
}%
\begin{pgfscope}%
\pgfsys@transformshift{5.130139in}{0.582899in}%
\pgfsys@useobject{currentmarker}{}%
\end{pgfscope}%
\end{pgfscope}%
\begin{pgfscope}%
\definecolor{textcolor}{rgb}{0.000000,0.000000,0.000000}%
\pgfsetstrokecolor{textcolor}%
\pgfsetfillcolor{textcolor}%
\pgftext[x=5.130139in,y=0.485677in,,top]{\color{textcolor}{\sffamily\fontsize{16.000000}{19.200000}\selectfont\catcode`\^=\active\def^{\ifmmode\sp\else\^{}\fi}\catcode`\%=\active\def
\end{pgfscope}%
\begin{pgfscope}%
\pgfsetbuttcap%
\pgfsetroundjoin%
\definecolor{currentfill}{rgb}{0.000000,0.000000,0.000000}%
\pgfsetfillcolor{currentfill}%
\pgfsetlinewidth{0.803000pt}%
\definecolor{currentstroke}{rgb}{0.000000,0.000000,0.000000}%
\pgfsetstrokecolor{currentstroke}%
\pgfsetdash{}{0pt}%
\pgfsys@defobject{currentmarker}{\pgfqpoint{0.000000in}{-0.048611in}}{\pgfqpoint{0.000000in}{0.000000in}}{%
\pgfpathmoveto{\pgfqpoint{0.000000in}{0.000000in}}%
\pgfpathlineto{\pgfqpoint{0.000000in}{-0.048611in}}%
\pgfusepath{stroke,fill}%
}%
\begin{pgfscope}%
\pgfsys@transformshift{5.599836in}{0.582899in}%
\pgfsys@useobject{currentmarker}{}%
\end{pgfscope}%
\end{pgfscope}%
\begin{pgfscope}%
\definecolor{textcolor}{rgb}{0.000000,0.000000,0.000000}%
\pgfsetstrokecolor{textcolor}%
\pgfsetfillcolor{textcolor}%
\pgftext[x=5.599836in,y=0.485677in,,top]{\color{textcolor}{\sffamily\fontsize{16.000000}{19.200000}\selectfont\catcode`\^=\active\def^{\ifmmode\sp\else\^{}\fi}\catcode`\%=\active\def
\end{pgfscope}%
\begin{pgfscope}%
\definecolor{textcolor}{rgb}{0.000000,0.000000,0.000000}%
\pgfsetstrokecolor{textcolor}%
\pgfsetfillcolor{textcolor}%
\pgftext[x=3.345290in,y=0.215061in,,top]{\color{textcolor}{\sffamily\fontsize{16.000000}{19.200000}\selectfont\catcode`\^=\active\def^{\ifmmode\sp\else\^{}\fi}\catcode`\%=\active\def
\end{pgfscope}%
\begin{pgfscope}%
\pgfsetbuttcap%
\pgfsetroundjoin%
\definecolor{currentfill}{rgb}{0.000000,0.000000,0.000000}%
\pgfsetfillcolor{currentfill}%
\pgfsetlinewidth{0.803000pt}%
\definecolor{currentstroke}{rgb}{0.000000,0.000000,0.000000}%
\pgfsetstrokecolor{currentstroke}%
\pgfsetdash{}{0pt}%
\pgfsys@defobject{currentmarker}{\pgfqpoint{-0.048611in}{0.000000in}}{\pgfqpoint{-0.000000in}{0.000000in}}{%
\pgfpathmoveto{\pgfqpoint{-0.000000in}{0.000000in}}%
\pgfpathlineto{\pgfqpoint{-0.048611in}{0.000000in}}%
\pgfusepath{stroke,fill}%
}%
\begin{pgfscope}%
\pgfsys@transformshift{0.865290in}{1.096263in}%
\pgfsys@useobject{currentmarker}{}%
\end{pgfscope}%
\end{pgfscope}%
\begin{pgfscope}%
\definecolor{textcolor}{rgb}{0.000000,0.000000,0.000000}%
\pgfsetstrokecolor{textcolor}%
\pgfsetfillcolor{textcolor}%
\pgftext[x=0.270616in, y=1.011844in, left, base]{\color{textcolor}{\sffamily\fontsize{16.000000}{19.200000}\selectfont\catcode`\^=\active\def^{\ifmmode\sp\else\^{}\fi}\catcode`\%=\active\def
\end{pgfscope}%
\begin{pgfscope}%
\pgfsetbuttcap%
\pgfsetroundjoin%
\definecolor{currentfill}{rgb}{0.000000,0.000000,0.000000}%
\pgfsetfillcolor{currentfill}%
\pgfsetlinewidth{0.803000pt}%
\definecolor{currentstroke}{rgb}{0.000000,0.000000,0.000000}%
\pgfsetstrokecolor{currentstroke}%
\pgfsetdash{}{0pt}%
\pgfsys@defobject{currentmarker}{\pgfqpoint{-0.048611in}{0.000000in}}{\pgfqpoint{-0.000000in}{0.000000in}}{%
\pgfpathmoveto{\pgfqpoint{-0.000000in}{0.000000in}}%
\pgfpathlineto{\pgfqpoint{-0.048611in}{0.000000in}}%
\pgfusepath{stroke,fill}%
}%
\begin{pgfscope}%
\pgfsys@transformshift{0.865290in}{1.706466in}%
\pgfsys@useobject{currentmarker}{}%
\end{pgfscope}%
\end{pgfscope}%
\begin{pgfscope}%
\definecolor{textcolor}{rgb}{0.000000,0.000000,0.000000}%
\pgfsetstrokecolor{textcolor}%
\pgfsetfillcolor{textcolor}%
\pgftext[x=0.346658in, y=1.622048in, left, base]{\color{textcolor}{\sffamily\fontsize{16.000000}{19.200000}\selectfont\catcode`\^=\active\def^{\ifmmode\sp\else\^{}\fi}\catcode`\%=\active\def
\end{pgfscope}%
\begin{pgfscope}%
\pgfsetbuttcap%
\pgfsetroundjoin%
\definecolor{currentfill}{rgb}{0.000000,0.000000,0.000000}%
\pgfsetfillcolor{currentfill}%
\pgfsetlinewidth{0.803000pt}%
\definecolor{currentstroke}{rgb}{0.000000,0.000000,0.000000}%
\pgfsetstrokecolor{currentstroke}%
\pgfsetdash{}{0pt}%
\pgfsys@defobject{currentmarker}{\pgfqpoint{-0.048611in}{0.000000in}}{\pgfqpoint{-0.000000in}{0.000000in}}{%
\pgfpathmoveto{\pgfqpoint{-0.000000in}{0.000000in}}%
\pgfpathlineto{\pgfqpoint{-0.048611in}{0.000000in}}%
\pgfusepath{stroke,fill}%
}%
\begin{pgfscope}%
\pgfsys@transformshift{0.865290in}{2.316670in}%
\pgfsys@useobject{currentmarker}{}%
\end{pgfscope}%
\end{pgfscope}%
\begin{pgfscope}%
\definecolor{textcolor}{rgb}{0.000000,0.000000,0.000000}%
\pgfsetstrokecolor{textcolor}%
\pgfsetfillcolor{textcolor}%
\pgftext[x=0.346658in, y=2.232252in, left, base]{\color{textcolor}{\sffamily\fontsize{16.000000}{19.200000}\selectfont\catcode`\^=\active\def^{\ifmmode\sp\else\^{}\fi}\catcode`\%=\active\def
\end{pgfscope}%
\begin{pgfscope}%
\pgfsetbuttcap%
\pgfsetroundjoin%
\definecolor{currentfill}{rgb}{0.000000,0.000000,0.000000}%
\pgfsetfillcolor{currentfill}%
\pgfsetlinewidth{0.803000pt}%
\definecolor{currentstroke}{rgb}{0.000000,0.000000,0.000000}%
\pgfsetstrokecolor{currentstroke}%
\pgfsetdash{}{0pt}%
\pgfsys@defobject{currentmarker}{\pgfqpoint{-0.048611in}{0.000000in}}{\pgfqpoint{-0.000000in}{0.000000in}}{%
\pgfpathmoveto{\pgfqpoint{-0.000000in}{0.000000in}}%
\pgfpathlineto{\pgfqpoint{-0.048611in}{0.000000in}}%
\pgfusepath{stroke,fill}%
}%
\begin{pgfscope}%
\pgfsys@transformshift{0.865290in}{2.926874in}%
\pgfsys@useobject{currentmarker}{}%
\end{pgfscope}%
\end{pgfscope}%
\begin{pgfscope}%
\definecolor{textcolor}{rgb}{0.000000,0.000000,0.000000}%
\pgfsetstrokecolor{textcolor}%
\pgfsetfillcolor{textcolor}%
\pgftext[x=0.346658in, y=2.842456in, left, base]{\color{textcolor}{\sffamily\fontsize{16.000000}{19.200000}\selectfont\catcode`\^=\active\def^{\ifmmode\sp\else\^{}\fi}\catcode`\%=\active\def
\end{pgfscope}%
\begin{pgfscope}%
\pgfsetbuttcap%
\pgfsetroundjoin%
\definecolor{currentfill}{rgb}{0.000000,0.000000,0.000000}%
\pgfsetfillcolor{currentfill}%
\pgfsetlinewidth{0.803000pt}%
\definecolor{currentstroke}{rgb}{0.000000,0.000000,0.000000}%
\pgfsetstrokecolor{currentstroke}%
\pgfsetdash{}{0pt}%
\pgfsys@defobject{currentmarker}{\pgfqpoint{-0.048611in}{0.000000in}}{\pgfqpoint{-0.000000in}{0.000000in}}{%
\pgfpathmoveto{\pgfqpoint{-0.000000in}{0.000000in}}%
\pgfpathlineto{\pgfqpoint{-0.048611in}{0.000000in}}%
\pgfusepath{stroke,fill}%
}%
\begin{pgfscope}%
\pgfsys@transformshift{0.865290in}{3.537078in}%
\pgfsys@useobject{currentmarker}{}%
\end{pgfscope}%
\end{pgfscope}%
\begin{pgfscope}%
\definecolor{textcolor}{rgb}{0.000000,0.000000,0.000000}%
\pgfsetstrokecolor{textcolor}%
\pgfsetfillcolor{textcolor}%
\pgftext[x=0.464945in, y=3.452660in, left, base]{\color{textcolor}{\sffamily\fontsize{16.000000}{19.200000}\selectfont\catcode`\^=\active\def^{\ifmmode\sp\else\^{}\fi}\catcode`\%=\active\def
\end{pgfscope}%
\begin{pgfscope}%
\pgfsetbuttcap%
\pgfsetroundjoin%
\definecolor{currentfill}{rgb}{0.000000,0.000000,0.000000}%
\pgfsetfillcolor{currentfill}%
\pgfsetlinewidth{0.803000pt}%
\definecolor{currentstroke}{rgb}{0.000000,0.000000,0.000000}%
\pgfsetstrokecolor{currentstroke}%
\pgfsetdash{}{0pt}%
\pgfsys@defobject{currentmarker}{\pgfqpoint{-0.048611in}{0.000000in}}{\pgfqpoint{-0.000000in}{0.000000in}}{%
\pgfpathmoveto{\pgfqpoint{-0.000000in}{0.000000in}}%
\pgfpathlineto{\pgfqpoint{-0.048611in}{0.000000in}}%
\pgfusepath{stroke,fill}%
}%
\begin{pgfscope}%
\pgfsys@transformshift{0.865290in}{4.147282in}%
\pgfsys@useobject{currentmarker}{}%
\end{pgfscope}%
\end{pgfscope}%
\begin{pgfscope}%
\definecolor{textcolor}{rgb}{0.000000,0.000000,0.000000}%
\pgfsetstrokecolor{textcolor}%
\pgfsetfillcolor{textcolor}%
\pgftext[x=0.464945in, y=4.062864in, left, base]{\color{textcolor}{\sffamily\fontsize{16.000000}{19.200000}\selectfont\catcode`\^=\active\def^{\ifmmode\sp\else\^{}\fi}\catcode`\%=\active\def
\end{pgfscope}%
\begin{pgfscope}%
\definecolor{textcolor}{rgb}{0.000000,0.000000,0.000000}%
\pgfsetstrokecolor{textcolor}%
\pgfsetfillcolor{textcolor}%
\pgftext[x=0.215061in,y=2.430899in,,bottom,rotate=90.000000]{\color{textcolor}{\sffamily\fontsize{16.000000}{19.200000}\selectfont\catcode`\^=\active\def^{\ifmmode\sp\else\^{}\fi}\catcode`\%=\active\def
\end{pgfscope}%
\begin{pgfscope}%
\pgfpathrectangle{\pgfqpoint{0.865290in}{0.582899in}}{\pgfqpoint{4.960000in}{3.696000in}}%
\pgfusepath{clip}%
\pgfsetrectcap%
\pgfsetroundjoin%
\pgfsetlinewidth{1.505625pt}%
\definecolor{currentstroke}{rgb}{0.121569,0.466667,0.705882}%
\pgfsetstrokecolor{currentstroke}%
\pgfsetdash{}{0pt}%
\pgfpathmoveto{\pgfqpoint{1.090745in}{3.333677in}}%
\pgfpathlineto{\pgfqpoint{1.278623in}{4.110899in}}%
\pgfpathlineto{\pgfqpoint{1.466502in}{3.354735in}}%
\pgfpathlineto{\pgfqpoint{1.654381in}{3.129340in}}%
\pgfpathlineto{\pgfqpoint{1.842260in}{2.888178in}}%
\pgfpathlineto{\pgfqpoint{2.030139in}{2.531758in}}%
\pgfpathlineto{\pgfqpoint{2.218017in}{2.152902in}}%
\pgfpathlineto{\pgfqpoint{2.405896in}{1.906874in}}%
\pgfpathlineto{\pgfqpoint{2.593775in}{1.637077in}}%
\pgfpathlineto{\pgfqpoint{2.781654in}{1.420358in}}%
\pgfpathlineto{\pgfqpoint{2.969533in}{1.246084in}}%
\pgfpathlineto{\pgfqpoint{3.157411in}{0.996319in}}%
\pgfpathlineto{\pgfqpoint{3.345290in}{0.822414in}}%
\pgfpathlineto{\pgfqpoint{3.533169in}{0.864418in}}%
\pgfpathlineto{\pgfqpoint{3.721048in}{0.862430in}}%
\pgfpathlineto{\pgfqpoint{3.908927in}{0.862491in}}%
\pgfpathlineto{\pgfqpoint{4.096805in}{0.862491in}}%
\pgfpathlineto{\pgfqpoint{4.284684in}{0.862491in}}%
\pgfpathlineto{\pgfqpoint{4.472563in}{0.862491in}}%
\pgfpathlineto{\pgfqpoint{4.660442in}{0.862491in}}%
\pgfpathlineto{\pgfqpoint{4.848320in}{0.862491in}}%
\pgfpathlineto{\pgfqpoint{5.036199in}{0.862491in}}%
\pgfpathlineto{\pgfqpoint{5.224078in}{0.862491in}}%
\pgfpathlineto{\pgfqpoint{5.411957in}{0.862491in}}%
\pgfpathlineto{\pgfqpoint{5.599836in}{0.862491in}}%
\pgfusepath{stroke}%
\end{pgfscope}%
\begin{pgfscope}%
\pgfpathrectangle{\pgfqpoint{0.865290in}{0.582899in}}{\pgfqpoint{4.960000in}{3.696000in}}%
\pgfusepath{clip}%
\pgfsetrectcap%
\pgfsetroundjoin%
\pgfsetlinewidth{1.505625pt}%
\definecolor{currentstroke}{rgb}{1.000000,0.498039,0.054902}%
\pgfsetstrokecolor{currentstroke}%
\pgfsetdash{}{0pt}%
\pgfpathmoveto{\pgfqpoint{1.090745in}{3.329811in}}%
\pgfpathlineto{\pgfqpoint{1.278623in}{3.029017in}}%
\pgfpathlineto{\pgfqpoint{1.466502in}{2.576642in}}%
\pgfpathlineto{\pgfqpoint{1.654381in}{2.426479in}}%
\pgfpathlineto{\pgfqpoint{1.842260in}{2.190699in}}%
\pgfpathlineto{\pgfqpoint{2.030139in}{1.891948in}}%
\pgfpathlineto{\pgfqpoint{2.218017in}{1.583537in}}%
\pgfpathlineto{\pgfqpoint{2.405896in}{1.449268in}}%
\pgfpathlineto{\pgfqpoint{2.593775in}{1.187045in}}%
\pgfpathlineto{\pgfqpoint{2.781654in}{0.750899in}}%
\pgfpathlineto{\pgfqpoint{2.969533in}{0.892768in}}%
\pgfpathlineto{\pgfqpoint{3.157411in}{0.888686in}}%
\pgfpathlineto{\pgfqpoint{3.345290in}{0.888867in}}%
\pgfpathlineto{\pgfqpoint{3.533169in}{0.888852in}}%
\pgfpathlineto{\pgfqpoint{3.721048in}{0.888822in}}%
\pgfpathlineto{\pgfqpoint{3.908927in}{0.888852in}}%
\pgfpathlineto{\pgfqpoint{4.096805in}{0.888852in}}%
\pgfpathlineto{\pgfqpoint{4.284684in}{0.888852in}}%
\pgfpathlineto{\pgfqpoint{4.472563in}{0.888822in}}%
\pgfpathlineto{\pgfqpoint{4.660442in}{0.888852in}}%
\pgfpathlineto{\pgfqpoint{4.848320in}{0.888852in}}%
\pgfpathlineto{\pgfqpoint{5.036199in}{0.888852in}}%
\pgfpathlineto{\pgfqpoint{5.224078in}{0.888822in}}%
\pgfpathlineto{\pgfqpoint{5.411957in}{0.888852in}}%
\pgfpathlineto{\pgfqpoint{5.599836in}{0.888852in}}%
\pgfusepath{stroke}%
\end{pgfscope}%
\begin{pgfscope}%
\pgfsetrectcap%
\pgfsetmiterjoin%
\pgfsetlinewidth{0.803000pt}%
\definecolor{currentstroke}{rgb}{0.000000,0.000000,0.000000}%
\pgfsetstrokecolor{currentstroke}%
\pgfsetdash{}{0pt}%
\pgfpathmoveto{\pgfqpoint{0.865290in}{0.582899in}}%
\pgfpathlineto{\pgfqpoint{0.865290in}{4.278899in}}%
\pgfusepath{stroke}%
\end{pgfscope}%
\begin{pgfscope}%
\pgfsetrectcap%
\pgfsetmiterjoin%
\pgfsetlinewidth{0.803000pt}%
\definecolor{currentstroke}{rgb}{0.000000,0.000000,0.000000}%
\pgfsetstrokecolor{currentstroke}%
\pgfsetdash{}{0pt}%
\pgfpathmoveto{\pgfqpoint{5.825290in}{0.582899in}}%
\pgfpathlineto{\pgfqpoint{5.825290in}{4.278899in}}%
\pgfusepath{stroke}%
\end{pgfscope}%
\begin{pgfscope}%
\pgfsetrectcap%
\pgfsetmiterjoin%
\pgfsetlinewidth{0.803000pt}%
\definecolor{currentstroke}{rgb}{0.000000,0.000000,0.000000}%
\pgfsetstrokecolor{currentstroke}%
\pgfsetdash{}{0pt}%
\pgfpathmoveto{\pgfqpoint{0.865290in}{0.582899in}}%
\pgfpathlineto{\pgfqpoint{5.825290in}{0.582899in}}%
\pgfusepath{stroke}%
\end{pgfscope}%
\begin{pgfscope}%
\pgfsetrectcap%
\pgfsetmiterjoin%
\pgfsetlinewidth{0.803000pt}%
\definecolor{currentstroke}{rgb}{0.000000,0.000000,0.000000}%
\pgfsetstrokecolor{currentstroke}%
\pgfsetdash{}{0pt}%
\pgfpathmoveto{\pgfqpoint{0.865290in}{4.278899in}}%
\pgfpathlineto{\pgfqpoint{5.825290in}{4.278899in}}%
\pgfusepath{stroke}%
\end{pgfscope}%
\begin{pgfscope}%
\pgfsetbuttcap%
\pgfsetmiterjoin%
\definecolor{currentfill}{rgb}{1.000000,1.000000,1.000000}%
\pgfsetfillcolor{currentfill}%
\pgfsetfillopacity{0.800000}%
\pgfsetlinewidth{1.003750pt}%
\definecolor{currentstroke}{rgb}{0.800000,0.800000,0.800000}%
\pgfsetstrokecolor{currentstroke}%
\pgfsetstrokeopacity{0.800000}%
\pgfsetdash{}{0pt}%
\pgfpathmoveto{\pgfqpoint{3.158927in}{3.448778in}}%
\pgfpathlineto{\pgfqpoint{5.669735in}{3.448778in}}%
\pgfpathquadraticcurveto{\pgfqpoint{5.714179in}{3.448778in}}{\pgfqpoint{5.714179in}{3.493222in}}%
\pgfpathlineto{\pgfqpoint{5.714179in}{4.123343in}}%
\pgfpathquadraticcurveto{\pgfqpoint{5.714179in}{4.167788in}}{\pgfqpoint{5.669735in}{4.167788in}}%
\pgfpathlineto{\pgfqpoint{3.158927in}{4.167788in}}%
\pgfpathquadraticcurveto{\pgfqpoint{3.114483in}{4.167788in}}{\pgfqpoint{3.114483in}{4.123343in}}%
\pgfpathlineto{\pgfqpoint{3.114483in}{3.493222in}}%
\pgfpathquadraticcurveto{\pgfqpoint{3.114483in}{3.448778in}}{\pgfqpoint{3.158927in}{3.448778in}}%
\pgfpathlineto{\pgfqpoint{3.158927in}{3.448778in}}%
\pgfpathclose%
\pgfusepath{stroke,fill}%
\end{pgfscope}%
\begin{pgfscope}%
\pgfsetrectcap%
\pgfsetroundjoin%
\pgfsetlinewidth{1.505625pt}%
\definecolor{currentstroke}{rgb}{0.121569,0.466667,0.705882}%
\pgfsetstrokecolor{currentstroke}%
\pgfsetdash{}{0pt}%
\pgfpathmoveto{\pgfqpoint{3.203372in}{3.987840in}}%
\pgfpathlineto{\pgfqpoint{3.425594in}{3.987840in}}%
\pgfpathlineto{\pgfqpoint{3.647816in}{3.987840in}}%
\pgfusepath{stroke}%
\end{pgfscope}%
\begin{pgfscope}%
\definecolor{textcolor}{rgb}{0.000000,0.000000,0.000000}%
\pgfsetstrokecolor{textcolor}%
\pgfsetfillcolor{textcolor}%
\pgftext[x=3.825594in,y=3.910062in,left,base]{\color{textcolor}{\sffamily\fontsize{16.000000}{19.200000}\selectfont\catcode`\^=\active\def^{\ifmmode\sp\else\^{}\fi}\catcode`\%=\active\def
\end{pgfscope}%
\begin{pgfscope}%
\pgfsetrectcap%
\pgfsetroundjoin%
\pgfsetlinewidth{1.505625pt}%
\definecolor{currentstroke}{rgb}{1.000000,0.498039,0.054902}%
\pgfsetstrokecolor{currentstroke}%
\pgfsetdash{}{0pt}%
\pgfpathmoveto{\pgfqpoint{3.203372in}{3.661668in}}%
\pgfpathlineto{\pgfqpoint{3.425594in}{3.661668in}}%
\pgfpathlineto{\pgfqpoint{3.647816in}{3.661668in}}%
\pgfusepath{stroke}%
\end{pgfscope}%
\begin{pgfscope}%
\definecolor{textcolor}{rgb}{0.000000,0.000000,0.000000}%
\pgfsetstrokecolor{textcolor}%
\pgfsetfillcolor{textcolor}%
\pgftext[x=3.825594in,y=3.583890in,left,base]{\color{textcolor}{\sffamily\fontsize{16.000000}{19.200000}\selectfont\catcode`\^=\active\def^{\ifmmode\sp\else\^{}\fi}\catcode`\%=\active\def
\end{pgfscope}%
\end{pgfpicture}%
\makeatother%
\endgroup%

%% file: figs/network_conv_ei_HepTh.pgf
\begingroup%
\makeatletter%
\begin{pgfpicture}%
\pgfpathrectangle{\pgfpointorigin}{\pgfqpoint{5.825290in}{4.327488in}}%
\pgfusepath{use as bounding box, clip}%
\begin{pgfscope}%
\pgfsetbuttcap%
\pgfsetmiterjoin%
\definecolor{currentfill}{rgb}{1.000000,1.000000,1.000000}%
\pgfsetfillcolor{currentfill}%
\pgfsetlinewidth{0.000000pt}%
\definecolor{currentstroke}{rgb}{1.000000,1.000000,1.000000}%
\pgfsetstrokecolor{currentstroke}%
\pgfsetdash{}{0pt}%
\pgfpathmoveto{\pgfqpoint{0.000000in}{0.000000in}}%
\pgfpathlineto{\pgfqpoint{5.825290in}{0.000000in}}%
\pgfpathlineto{\pgfqpoint{5.825290in}{4.327488in}}%
\pgfpathlineto{\pgfqpoint{0.000000in}{4.327488in}}%
\pgfpathlineto{\pgfqpoint{0.000000in}{0.000000in}}%
\pgfpathclose%
\pgfusepath{fill}%
\end{pgfscope}%
\begin{pgfscope}%
\pgfsetbuttcap%
\pgfsetmiterjoin%
\definecolor{currentfill}{rgb}{1.000000,1.000000,1.000000}%
\pgfsetfillcolor{currentfill}%
\pgfsetlinewidth{0.000000pt}%
\definecolor{currentstroke}{rgb}{0.000000,0.000000,0.000000}%
\pgfsetstrokecolor{currentstroke}%
\pgfsetstrokeopacity{0.000000}%
\pgfsetdash{}{0pt}%
\pgfpathmoveto{\pgfqpoint{0.865290in}{0.582899in}}%
\pgfpathlineto{\pgfqpoint{5.825290in}{0.582899in}}%
\pgfpathlineto{\pgfqpoint{5.825290in}{4.278899in}}%
\pgfpathlineto{\pgfqpoint{0.865290in}{4.278899in}}%
\pgfpathlineto{\pgfqpoint{0.865290in}{0.582899in}}%
\pgfpathclose%
\pgfusepath{fill}%
\end{pgfscope}%
\begin{pgfscope}%
\pgfsetbuttcap%
\pgfsetroundjoin%
\definecolor{currentfill}{rgb}{0.000000,0.000000,0.000000}%
\pgfsetfillcolor{currentfill}%
\pgfsetlinewidth{0.803000pt}%
\definecolor{currentstroke}{rgb}{0.000000,0.000000,0.000000}%
\pgfsetstrokecolor{currentstroke}%
\pgfsetdash{}{0pt}%
\pgfsys@defobject{currentmarker}{\pgfqpoint{0.000000in}{-0.048611in}}{\pgfqpoint{0.000000in}{0.000000in}}{%
\pgfpathmoveto{\pgfqpoint{0.000000in}{0.000000in}}%
\pgfpathlineto{\pgfqpoint{0.000000in}{-0.048611in}}%
\pgfusepath{stroke,fill}%
}%
\begin{pgfscope}%
\pgfsys@transformshift{1.372563in}{0.582899in}%
\pgfsys@useobject{currentmarker}{}%
\end{pgfscope}%
\end{pgfscope}%
\begin{pgfscope}%
\definecolor{textcolor}{rgb}{0.000000,0.000000,0.000000}%
\pgfsetstrokecolor{textcolor}%
\pgfsetfillcolor{textcolor}%
\pgftext[x=1.372563in,y=0.485677in,,top]{\color{textcolor}{\sffamily\fontsize{16.000000}{19.200000}\selectfont\catcode`\^=\active\def^{\ifmmode\sp\else\^{}\fi}\catcode`\%=\active\def
\end{pgfscope}%
\begin{pgfscope}%
\pgfsetbuttcap%
\pgfsetroundjoin%
\definecolor{currentfill}{rgb}{0.000000,0.000000,0.000000}%
\pgfsetfillcolor{currentfill}%
\pgfsetlinewidth{0.803000pt}%
\definecolor{currentstroke}{rgb}{0.000000,0.000000,0.000000}%
\pgfsetstrokecolor{currentstroke}%
\pgfsetdash{}{0pt}%
\pgfsys@defobject{currentmarker}{\pgfqpoint{0.000000in}{-0.048611in}}{\pgfqpoint{0.000000in}{0.000000in}}{%
\pgfpathmoveto{\pgfqpoint{0.000000in}{0.000000in}}%
\pgfpathlineto{\pgfqpoint{0.000000in}{-0.048611in}}%
\pgfusepath{stroke,fill}%
}%
\begin{pgfscope}%
\pgfsys@transformshift{1.842260in}{0.582899in}%
\pgfsys@useobject{currentmarker}{}%
\end{pgfscope}%
\end{pgfscope}%
\begin{pgfscope}%
\definecolor{textcolor}{rgb}{0.000000,0.000000,0.000000}%
\pgfsetstrokecolor{textcolor}%
\pgfsetfillcolor{textcolor}%
\pgftext[x=1.842260in,y=0.485677in,,top]{\color{textcolor}{\sffamily\fontsize{16.000000}{19.200000}\selectfont\catcode`\^=\active\def^{\ifmmode\sp\else\^{}\fi}\catcode`\%=\active\def
\end{pgfscope}%
\begin{pgfscope}%
\pgfsetbuttcap%
\pgfsetroundjoin%
\definecolor{currentfill}{rgb}{0.000000,0.000000,0.000000}%
\pgfsetfillcolor{currentfill}%
\pgfsetlinewidth{0.803000pt}%
\definecolor{currentstroke}{rgb}{0.000000,0.000000,0.000000}%
\pgfsetstrokecolor{currentstroke}%
\pgfsetdash{}{0pt}%
\pgfsys@defobject{currentmarker}{\pgfqpoint{0.000000in}{-0.048611in}}{\pgfqpoint{0.000000in}{0.000000in}}{%
\pgfpathmoveto{\pgfqpoint{0.000000in}{0.000000in}}%
\pgfpathlineto{\pgfqpoint{0.000000in}{-0.048611in}}%
\pgfusepath{stroke,fill}%
}%
\begin{pgfscope}%
\pgfsys@transformshift{2.311957in}{0.582899in}%
\pgfsys@useobject{currentmarker}{}%
\end{pgfscope}%
\end{pgfscope}%
\begin{pgfscope}%
\definecolor{textcolor}{rgb}{0.000000,0.000000,0.000000}%
\pgfsetstrokecolor{textcolor}%
\pgfsetfillcolor{textcolor}%
\pgftext[x=2.311957in,y=0.485677in,,top]{\color{textcolor}{\sffamily\fontsize{16.000000}{19.200000}\selectfont\catcode`\^=\active\def^{\ifmmode\sp\else\^{}\fi}\catcode`\%=\active\def
\end{pgfscope}%
\begin{pgfscope}%
\pgfsetbuttcap%
\pgfsetroundjoin%
\definecolor{currentfill}{rgb}{0.000000,0.000000,0.000000}%
\pgfsetfillcolor{currentfill}%
\pgfsetlinewidth{0.803000pt}%
\definecolor{currentstroke}{rgb}{0.000000,0.000000,0.000000}%
\pgfsetstrokecolor{currentstroke}%
\pgfsetdash{}{0pt}%
\pgfsys@defobject{currentmarker}{\pgfqpoint{0.000000in}{-0.048611in}}{\pgfqpoint{0.000000in}{0.000000in}}{%
\pgfpathmoveto{\pgfqpoint{0.000000in}{0.000000in}}%
\pgfpathlineto{\pgfqpoint{0.000000in}{-0.048611in}}%
\pgfusepath{stroke,fill}%
}%
\begin{pgfscope}%
\pgfsys@transformshift{2.781654in}{0.582899in}%
\pgfsys@useobject{currentmarker}{}%
\end{pgfscope}%
\end{pgfscope}%
\begin{pgfscope}%
\definecolor{textcolor}{rgb}{0.000000,0.000000,0.000000}%
\pgfsetstrokecolor{textcolor}%
\pgfsetfillcolor{textcolor}%
\pgftext[x=2.781654in,y=0.485677in,,top]{\color{textcolor}{\sffamily\fontsize{16.000000}{19.200000}\selectfont\catcode`\^=\active\def^{\ifmmode\sp\else\^{}\fi}\catcode`\%=\active\def
\end{pgfscope}%
\begin{pgfscope}%
\pgfsetbuttcap%
\pgfsetroundjoin%
\definecolor{currentfill}{rgb}{0.000000,0.000000,0.000000}%
\pgfsetfillcolor{currentfill}%
\pgfsetlinewidth{0.803000pt}%
\definecolor{currentstroke}{rgb}{0.000000,0.000000,0.000000}%
\pgfsetstrokecolor{currentstroke}%
\pgfsetdash{}{0pt}%
\pgfsys@defobject{currentmarker}{\pgfqpoint{0.000000in}{-0.048611in}}{\pgfqpoint{0.000000in}{0.000000in}}{%
\pgfpathmoveto{\pgfqpoint{0.000000in}{0.000000in}}%
\pgfpathlineto{\pgfqpoint{0.000000in}{-0.048611in}}%
\pgfusepath{stroke,fill}%
}%
\begin{pgfscope}%
\pgfsys@transformshift{3.251351in}{0.582899in}%
\pgfsys@useobject{currentmarker}{}%
\end{pgfscope}%
\end{pgfscope}%
\begin{pgfscope}%
\definecolor{textcolor}{rgb}{0.000000,0.000000,0.000000}%
\pgfsetstrokecolor{textcolor}%
\pgfsetfillcolor{textcolor}%
\pgftext[x=3.251351in,y=0.485677in,,top]{\color{textcolor}{\sffamily\fontsize{16.000000}{19.200000}\selectfont\catcode`\^=\active\def^{\ifmmode\sp\else\^{}\fi}\catcode`\%=\active\def
\end{pgfscope}%
\begin{pgfscope}%
\pgfsetbuttcap%
\pgfsetroundjoin%
\definecolor{currentfill}{rgb}{0.000000,0.000000,0.000000}%
\pgfsetfillcolor{currentfill}%
\pgfsetlinewidth{0.803000pt}%
\definecolor{currentstroke}{rgb}{0.000000,0.000000,0.000000}%
\pgfsetstrokecolor{currentstroke}%
\pgfsetdash{}{0pt}%
\pgfsys@defobject{currentmarker}{\pgfqpoint{0.000000in}{-0.048611in}}{\pgfqpoint{0.000000in}{0.000000in}}{%
\pgfpathmoveto{\pgfqpoint{0.000000in}{0.000000in}}%
\pgfpathlineto{\pgfqpoint{0.000000in}{-0.048611in}}%
\pgfusepath{stroke,fill}%
}%
\begin{pgfscope}%
\pgfsys@transformshift{3.721048in}{0.582899in}%
\pgfsys@useobject{currentmarker}{}%
\end{pgfscope}%
\end{pgfscope}%
\begin{pgfscope}%
\definecolor{textcolor}{rgb}{0.000000,0.000000,0.000000}%
\pgfsetstrokecolor{textcolor}%
\pgfsetfillcolor{textcolor}%
\pgftext[x=3.721048in,y=0.485677in,,top]{\color{textcolor}{\sffamily\fontsize{16.000000}{19.200000}\selectfont\catcode`\^=\active\def^{\ifmmode\sp\else\^{}\fi}\catcode`\%=\active\def
\end{pgfscope}%
\begin{pgfscope}%
\pgfsetbuttcap%
\pgfsetroundjoin%
\definecolor{currentfill}{rgb}{0.000000,0.000000,0.000000}%
\pgfsetfillcolor{currentfill}%
\pgfsetlinewidth{0.803000pt}%
\definecolor{currentstroke}{rgb}{0.000000,0.000000,0.000000}%
\pgfsetstrokecolor{currentstroke}%
\pgfsetdash{}{0pt}%
\pgfsys@defobject{currentmarker}{\pgfqpoint{0.000000in}{-0.048611in}}{\pgfqpoint{0.000000in}{0.000000in}}{%
\pgfpathmoveto{\pgfqpoint{0.000000in}{0.000000in}}%
\pgfpathlineto{\pgfqpoint{0.000000in}{-0.048611in}}%
\pgfusepath{stroke,fill}%
}%
\begin{pgfscope}%
\pgfsys@transformshift{4.190745in}{0.582899in}%
\pgfsys@useobject{currentmarker}{}%
\end{pgfscope}%
\end{pgfscope}%
\begin{pgfscope}%
\definecolor{textcolor}{rgb}{0.000000,0.000000,0.000000}%
\pgfsetstrokecolor{textcolor}%
\pgfsetfillcolor{textcolor}%
\pgftext[x=4.190745in,y=0.485677in,,top]{\color{textcolor}{\sffamily\fontsize{16.000000}{19.200000}\selectfont\catcode`\^=\active\def^{\ifmmode\sp\else\^{}\fi}\catcode`\%=\active\def
\end{pgfscope}%
\begin{pgfscope}%
\pgfsetbuttcap%
\pgfsetroundjoin%
\definecolor{currentfill}{rgb}{0.000000,0.000000,0.000000}%
\pgfsetfillcolor{currentfill}%
\pgfsetlinewidth{0.803000pt}%
\definecolor{currentstroke}{rgb}{0.000000,0.000000,0.000000}%
\pgfsetstrokecolor{currentstroke}%
\pgfsetdash{}{0pt}%
\pgfsys@defobject{currentmarker}{\pgfqpoint{0.000000in}{-0.048611in}}{\pgfqpoint{0.000000in}{0.000000in}}{%
\pgfpathmoveto{\pgfqpoint{0.000000in}{0.000000in}}%
\pgfpathlineto{\pgfqpoint{0.000000in}{-0.048611in}}%
\pgfusepath{stroke,fill}%
}%
\begin{pgfscope}%
\pgfsys@transformshift{4.660442in}{0.582899in}%
\pgfsys@useobject{currentmarker}{}%
\end{pgfscope}%
\end{pgfscope}%
\begin{pgfscope}%
\definecolor{textcolor}{rgb}{0.000000,0.000000,0.000000}%
\pgfsetstrokecolor{textcolor}%
\pgfsetfillcolor{textcolor}%
\pgftext[x=4.660442in,y=0.485677in,,top]{\color{textcolor}{\sffamily\fontsize{16.000000}{19.200000}\selectfont\catcode`\^=\active\def^{\ifmmode\sp\else\^{}\fi}\catcode`\%=\active\def
\end{pgfscope}%
\begin{pgfscope}%
\pgfsetbuttcap%
\pgfsetroundjoin%
\definecolor{currentfill}{rgb}{0.000000,0.000000,0.000000}%
\pgfsetfillcolor{currentfill}%
\pgfsetlinewidth{0.803000pt}%
\definecolor{currentstroke}{rgb}{0.000000,0.000000,0.000000}%
\pgfsetstrokecolor{currentstroke}%
\pgfsetdash{}{0pt}%
\pgfsys@defobject{currentmarker}{\pgfqpoint{0.000000in}{-0.048611in}}{\pgfqpoint{0.000000in}{0.000000in}}{%
\pgfpathmoveto{\pgfqpoint{0.000000in}{0.000000in}}%
\pgfpathlineto{\pgfqpoint{0.000000in}{-0.048611in}}%
\pgfusepath{stroke,fill}%
}%
\begin{pgfscope}%
\pgfsys@transformshift{5.130139in}{0.582899in}%
\pgfsys@useobject{currentmarker}{}%
\end{pgfscope}%
\end{pgfscope}%
\begin{pgfscope}%
\definecolor{textcolor}{rgb}{0.000000,0.000000,0.000000}%
\pgfsetstrokecolor{textcolor}%
\pgfsetfillcolor{textcolor}%
\pgftext[x=5.130139in,y=0.485677in,,top]{\color{textcolor}{\sffamily\fontsize{16.000000}{19.200000}\selectfont\catcode`\^=\active\def^{\ifmmode\sp\else\^{}\fi}\catcode`\%=\active\def
\end{pgfscope}%
\begin{pgfscope}%
\pgfsetbuttcap%
\pgfsetroundjoin%
\definecolor{currentfill}{rgb}{0.000000,0.000000,0.000000}%
\pgfsetfillcolor{currentfill}%
\pgfsetlinewidth{0.803000pt}%
\definecolor{currentstroke}{rgb}{0.000000,0.000000,0.000000}%
\pgfsetstrokecolor{currentstroke}%
\pgfsetdash{}{0pt}%
\pgfsys@defobject{currentmarker}{\pgfqpoint{0.000000in}{-0.048611in}}{\pgfqpoint{0.000000in}{0.000000in}}{%
\pgfpathmoveto{\pgfqpoint{0.000000in}{0.000000in}}%
\pgfpathlineto{\pgfqpoint{0.000000in}{-0.048611in}}%
\pgfusepath{stroke,fill}%
}%
\begin{pgfscope}%
\pgfsys@transformshift{5.599836in}{0.582899in}%
\pgfsys@useobject{currentmarker}{}%
\end{pgfscope}%
\end{pgfscope}%
\begin{pgfscope}%
\definecolor{textcolor}{rgb}{0.000000,0.000000,0.000000}%
\pgfsetstrokecolor{textcolor}%
\pgfsetfillcolor{textcolor}%
\pgftext[x=5.599836in,y=0.485677in,,top]{\color{textcolor}{\sffamily\fontsize{16.000000}{19.200000}\selectfont\catcode`\^=\active\def^{\ifmmode\sp\else\^{}\fi}\catcode`\%=\active\def
\end{pgfscope}%
\begin{pgfscope}%
\definecolor{textcolor}{rgb}{0.000000,0.000000,0.000000}%
\pgfsetstrokecolor{textcolor}%
\pgfsetfillcolor{textcolor}%
\pgftext[x=3.345290in,y=0.215061in,,top]{\color{textcolor}{\sffamily\fontsize{16.000000}{19.200000}\selectfont\catcode`\^=\active\def^{\ifmmode\sp\else\^{}\fi}\catcode`\%=\active\def
\end{pgfscope}%
\begin{pgfscope}%
\pgfsetbuttcap%
\pgfsetroundjoin%
\definecolor{currentfill}{rgb}{0.000000,0.000000,0.000000}%
\pgfsetfillcolor{currentfill}%
\pgfsetlinewidth{0.803000pt}%
\definecolor{currentstroke}{rgb}{0.000000,0.000000,0.000000}%
\pgfsetstrokecolor{currentstroke}%
\pgfsetdash{}{0pt}%
\pgfsys@defobject{currentmarker}{\pgfqpoint{-0.048611in}{0.000000in}}{\pgfqpoint{-0.000000in}{0.000000in}}{%
\pgfpathmoveto{\pgfqpoint{-0.000000in}{0.000000in}}%
\pgfpathlineto{\pgfqpoint{-0.048611in}{0.000000in}}%
\pgfusepath{stroke,fill}%
}%
\begin{pgfscope}%
\pgfsys@transformshift{0.865290in}{0.868416in}%
\pgfsys@useobject{currentmarker}{}%
\end{pgfscope}%
\end{pgfscope}%
\begin{pgfscope}%
\definecolor{textcolor}{rgb}{0.000000,0.000000,0.000000}%
\pgfsetstrokecolor{textcolor}%
\pgfsetfillcolor{textcolor}%
\pgftext[x=0.270616in, y=0.783997in, left, base]{\color{textcolor}{\sffamily\fontsize{16.000000}{19.200000}\selectfont\catcode`\^=\active\def^{\ifmmode\sp\else\^{}\fi}\catcode`\%=\active\def
\end{pgfscope}%
\begin{pgfscope}%
\pgfsetbuttcap%
\pgfsetroundjoin%
\definecolor{currentfill}{rgb}{0.000000,0.000000,0.000000}%
\pgfsetfillcolor{currentfill}%
\pgfsetlinewidth{0.803000pt}%
\definecolor{currentstroke}{rgb}{0.000000,0.000000,0.000000}%
\pgfsetstrokecolor{currentstroke}%
\pgfsetdash{}{0pt}%
\pgfsys@defobject{currentmarker}{\pgfqpoint{-0.048611in}{0.000000in}}{\pgfqpoint{-0.000000in}{0.000000in}}{%
\pgfpathmoveto{\pgfqpoint{-0.000000in}{0.000000in}}%
\pgfpathlineto{\pgfqpoint{-0.048611in}{0.000000in}}%
\pgfusepath{stroke,fill}%
}%
\begin{pgfscope}%
\pgfsys@transformshift{0.865290in}{1.350509in}%
\pgfsys@useobject{currentmarker}{}%
\end{pgfscope}%
\end{pgfscope}%
\begin{pgfscope}%
\definecolor{textcolor}{rgb}{0.000000,0.000000,0.000000}%
\pgfsetstrokecolor{textcolor}%
\pgfsetfillcolor{textcolor}%
\pgftext[x=0.270616in, y=1.266091in, left, base]{\color{textcolor}{\sffamily\fontsize{16.000000}{19.200000}\selectfont\catcode`\^=\active\def^{\ifmmode\sp\else\^{}\fi}\catcode`\%=\active\def
\end{pgfscope}%
\begin{pgfscope}%
\pgfsetbuttcap%
\pgfsetroundjoin%
\definecolor{currentfill}{rgb}{0.000000,0.000000,0.000000}%
\pgfsetfillcolor{currentfill}%
\pgfsetlinewidth{0.803000pt}%
\definecolor{currentstroke}{rgb}{0.000000,0.000000,0.000000}%
\pgfsetstrokecolor{currentstroke}%
\pgfsetdash{}{0pt}%
\pgfsys@defobject{currentmarker}{\pgfqpoint{-0.048611in}{0.000000in}}{\pgfqpoint{-0.000000in}{0.000000in}}{%
\pgfpathmoveto{\pgfqpoint{-0.000000in}{0.000000in}}%
\pgfpathlineto{\pgfqpoint{-0.048611in}{0.000000in}}%
\pgfusepath{stroke,fill}%
}%
\begin{pgfscope}%
\pgfsys@transformshift{0.865290in}{1.832603in}%
\pgfsys@useobject{currentmarker}{}%
\end{pgfscope}%
\end{pgfscope}%
\begin{pgfscope}%
\definecolor{textcolor}{rgb}{0.000000,0.000000,0.000000}%
\pgfsetstrokecolor{textcolor}%
\pgfsetfillcolor{textcolor}%
\pgftext[x=0.346658in, y=1.748184in, left, base]{\color{textcolor}{\sffamily\fontsize{16.000000}{19.200000}\selectfont\catcode`\^=\active\def^{\ifmmode\sp\else\^{}\fi}\catcode`\%=\active\def
\end{pgfscope}%
\begin{pgfscope}%
\pgfsetbuttcap%
\pgfsetroundjoin%
\definecolor{currentfill}{rgb}{0.000000,0.000000,0.000000}%
\pgfsetfillcolor{currentfill}%
\pgfsetlinewidth{0.803000pt}%
\definecolor{currentstroke}{rgb}{0.000000,0.000000,0.000000}%
\pgfsetstrokecolor{currentstroke}%
\pgfsetdash{}{0pt}%
\pgfsys@defobject{currentmarker}{\pgfqpoint{-0.048611in}{0.000000in}}{\pgfqpoint{-0.000000in}{0.000000in}}{%
\pgfpathmoveto{\pgfqpoint{-0.000000in}{0.000000in}}%
\pgfpathlineto{\pgfqpoint{-0.048611in}{0.000000in}}%
\pgfusepath{stroke,fill}%
}%
\begin{pgfscope}%
\pgfsys@transformshift{0.865290in}{2.314696in}%
\pgfsys@useobject{currentmarker}{}%
\end{pgfscope}%
\end{pgfscope}%
\begin{pgfscope}%
\definecolor{textcolor}{rgb}{0.000000,0.000000,0.000000}%
\pgfsetstrokecolor{textcolor}%
\pgfsetfillcolor{textcolor}%
\pgftext[x=0.346658in, y=2.230278in, left, base]{\color{textcolor}{\sffamily\fontsize{16.000000}{19.200000}\selectfont\catcode`\^=\active\def^{\ifmmode\sp\else\^{}\fi}\catcode`\%=\active\def
\end{pgfscope}%
\begin{pgfscope}%
\pgfsetbuttcap%
\pgfsetroundjoin%
\definecolor{currentfill}{rgb}{0.000000,0.000000,0.000000}%
\pgfsetfillcolor{currentfill}%
\pgfsetlinewidth{0.803000pt}%
\definecolor{currentstroke}{rgb}{0.000000,0.000000,0.000000}%
\pgfsetstrokecolor{currentstroke}%
\pgfsetdash{}{0pt}%
\pgfsys@defobject{currentmarker}{\pgfqpoint{-0.048611in}{0.000000in}}{\pgfqpoint{-0.000000in}{0.000000in}}{%
\pgfpathmoveto{\pgfqpoint{-0.000000in}{0.000000in}}%
\pgfpathlineto{\pgfqpoint{-0.048611in}{0.000000in}}%
\pgfusepath{stroke,fill}%
}%
\begin{pgfscope}%
\pgfsys@transformshift{0.865290in}{2.796790in}%
\pgfsys@useobject{currentmarker}{}%
\end{pgfscope}%
\end{pgfscope}%
\begin{pgfscope}%
\definecolor{textcolor}{rgb}{0.000000,0.000000,0.000000}%
\pgfsetstrokecolor{textcolor}%
\pgfsetfillcolor{textcolor}%
\pgftext[x=0.346658in, y=2.712371in, left, base]{\color{textcolor}{\sffamily\fontsize{16.000000}{19.200000}\selectfont\catcode`\^=\active\def^{\ifmmode\sp\else\^{}\fi}\catcode`\%=\active\def
\end{pgfscope}%
\begin{pgfscope}%
\pgfsetbuttcap%
\pgfsetroundjoin%
\definecolor{currentfill}{rgb}{0.000000,0.000000,0.000000}%
\pgfsetfillcolor{currentfill}%
\pgfsetlinewidth{0.803000pt}%
\definecolor{currentstroke}{rgb}{0.000000,0.000000,0.000000}%
\pgfsetstrokecolor{currentstroke}%
\pgfsetdash{}{0pt}%
\pgfsys@defobject{currentmarker}{\pgfqpoint{-0.048611in}{0.000000in}}{\pgfqpoint{-0.000000in}{0.000000in}}{%
\pgfpathmoveto{\pgfqpoint{-0.000000in}{0.000000in}}%
\pgfpathlineto{\pgfqpoint{-0.048611in}{0.000000in}}%
\pgfusepath{stroke,fill}%
}%
\begin{pgfscope}%
\pgfsys@transformshift{0.865290in}{3.278883in}%
\pgfsys@useobject{currentmarker}{}%
\end{pgfscope}%
\end{pgfscope}%
\begin{pgfscope}%
\definecolor{textcolor}{rgb}{0.000000,0.000000,0.000000}%
\pgfsetstrokecolor{textcolor}%
\pgfsetfillcolor{textcolor}%
\pgftext[x=0.346658in, y=3.194465in, left, base]{\color{textcolor}{\sffamily\fontsize{16.000000}{19.200000}\selectfont\catcode`\^=\active\def^{\ifmmode\sp\else\^{}\fi}\catcode`\%=\active\def
\end{pgfscope}%
\begin{pgfscope}%
\pgfsetbuttcap%
\pgfsetroundjoin%
\definecolor{currentfill}{rgb}{0.000000,0.000000,0.000000}%
\pgfsetfillcolor{currentfill}%
\pgfsetlinewidth{0.803000pt}%
\definecolor{currentstroke}{rgb}{0.000000,0.000000,0.000000}%
\pgfsetstrokecolor{currentstroke}%
\pgfsetdash{}{0pt}%
\pgfsys@defobject{currentmarker}{\pgfqpoint{-0.048611in}{0.000000in}}{\pgfqpoint{-0.000000in}{0.000000in}}{%
\pgfpathmoveto{\pgfqpoint{-0.000000in}{0.000000in}}%
\pgfpathlineto{\pgfqpoint{-0.048611in}{0.000000in}}%
\pgfusepath{stroke,fill}%
}%
\begin{pgfscope}%
\pgfsys@transformshift{0.865290in}{3.760976in}%
\pgfsys@useobject{currentmarker}{}%
\end{pgfscope}%
\end{pgfscope}%
\begin{pgfscope}%
\definecolor{textcolor}{rgb}{0.000000,0.000000,0.000000}%
\pgfsetstrokecolor{textcolor}%
\pgfsetfillcolor{textcolor}%
\pgftext[x=0.346658in, y=3.676558in, left, base]{\color{textcolor}{\sffamily\fontsize{16.000000}{19.200000}\selectfont\catcode`\^=\active\def^{\ifmmode\sp\else\^{}\fi}\catcode`\%=\active\def
\end{pgfscope}%
\begin{pgfscope}%
\pgfsetbuttcap%
\pgfsetroundjoin%
\definecolor{currentfill}{rgb}{0.000000,0.000000,0.000000}%
\pgfsetfillcolor{currentfill}%
\pgfsetlinewidth{0.803000pt}%
\definecolor{currentstroke}{rgb}{0.000000,0.000000,0.000000}%
\pgfsetstrokecolor{currentstroke}%
\pgfsetdash{}{0pt}%
\pgfsys@defobject{currentmarker}{\pgfqpoint{-0.048611in}{0.000000in}}{\pgfqpoint{-0.000000in}{0.000000in}}{%
\pgfpathmoveto{\pgfqpoint{-0.000000in}{0.000000in}}%
\pgfpathlineto{\pgfqpoint{-0.048611in}{0.000000in}}%
\pgfusepath{stroke,fill}%
}%
\begin{pgfscope}%
\pgfsys@transformshift{0.865290in}{4.243070in}%
\pgfsys@useobject{currentmarker}{}%
\end{pgfscope}%
\end{pgfscope}%
\begin{pgfscope}%
\definecolor{textcolor}{rgb}{0.000000,0.000000,0.000000}%
\pgfsetstrokecolor{textcolor}%
\pgfsetfillcolor{textcolor}%
\pgftext[x=0.464945in, y=4.158652in, left, base]{\color{textcolor}{\sffamily\fontsize{16.000000}{19.200000}\selectfont\catcode`\^=\active\def^{\ifmmode\sp\else\^{}\fi}\catcode`\%=\active\def
\end{pgfscope}%
\begin{pgfscope}%
\definecolor{textcolor}{rgb}{0.000000,0.000000,0.000000}%
\pgfsetstrokecolor{textcolor}%
\pgfsetfillcolor{textcolor}%
\pgftext[x=0.215061in,y=2.430899in,,bottom,rotate=90.000000]{\color{textcolor}{\sffamily\fontsize{16.000000}{19.200000}\selectfont\catcode`\^=\active\def^{\ifmmode\sp\else\^{}\fi}\catcode`\%=\active\def
\end{pgfscope}%
\begin{pgfscope}%
\pgfpathrectangle{\pgfqpoint{0.865290in}{0.582899in}}{\pgfqpoint{4.960000in}{3.696000in}}%
\pgfusepath{clip}%
\pgfsetrectcap%
\pgfsetroundjoin%
\pgfsetlinewidth{1.505625pt}%
\definecolor{currentstroke}{rgb}{0.121569,0.466667,0.705882}%
\pgfsetstrokecolor{currentstroke}%
\pgfsetdash{}{0pt}%
\pgfpathmoveto{\pgfqpoint{1.090745in}{4.001979in}}%
\pgfpathlineto{\pgfqpoint{1.278623in}{4.001545in}}%
\pgfpathlineto{\pgfqpoint{1.466502in}{4.110899in}}%
\pgfpathlineto{\pgfqpoint{1.654381in}{3.964024in}}%
\pgfpathlineto{\pgfqpoint{1.842260in}{3.777065in}}%
\pgfpathlineto{\pgfqpoint{2.030139in}{3.362323in}}%
\pgfpathlineto{\pgfqpoint{2.218017in}{3.066025in}}%
\pgfpathlineto{\pgfqpoint{2.405896in}{2.727400in}}%
\pgfpathlineto{\pgfqpoint{2.593775in}{2.582290in}}%
\pgfpathlineto{\pgfqpoint{2.781654in}{2.204650in}}%
\pgfpathlineto{\pgfqpoint{2.969533in}{1.977905in}}%
\pgfpathlineto{\pgfqpoint{3.157411in}{1.678536in}}%
\pgfpathlineto{\pgfqpoint{3.345290in}{1.294167in}}%
\pgfpathlineto{\pgfqpoint{3.533169in}{0.932996in}}%
\pgfpathlineto{\pgfqpoint{3.721048in}{0.785557in}}%
\pgfpathlineto{\pgfqpoint{3.908927in}{0.750899in}}%
\pgfpathlineto{\pgfqpoint{4.096805in}{0.750899in}}%
\pgfpathlineto{\pgfqpoint{4.284684in}{0.750899in}}%
\pgfpathlineto{\pgfqpoint{4.472563in}{0.750899in}}%
\pgfpathlineto{\pgfqpoint{4.660442in}{0.750899in}}%
\pgfpathlineto{\pgfqpoint{4.848320in}{0.750899in}}%
\pgfpathlineto{\pgfqpoint{5.036199in}{0.750899in}}%
\pgfpathlineto{\pgfqpoint{5.224078in}{0.750899in}}%
\pgfpathlineto{\pgfqpoint{5.411957in}{0.752856in}}%
\pgfpathlineto{\pgfqpoint{5.599836in}{0.752856in}}%
\pgfusepath{stroke}%
\end{pgfscope}%
\begin{pgfscope}%
\pgfpathrectangle{\pgfqpoint{0.865290in}{0.582899in}}{\pgfqpoint{4.960000in}{3.696000in}}%
\pgfusepath{clip}%
\pgfsetrectcap%
\pgfsetroundjoin%
\pgfsetlinewidth{1.505625pt}%
\definecolor{currentstroke}{rgb}{1.000000,0.498039,0.054902}%
\pgfsetstrokecolor{currentstroke}%
\pgfsetdash{}{0pt}%
\pgfpathmoveto{\pgfqpoint{1.090745in}{4.001994in}}%
\pgfpathlineto{\pgfqpoint{1.278623in}{3.921993in}}%
\pgfpathlineto{\pgfqpoint{1.466502in}{3.698241in}}%
\pgfpathlineto{\pgfqpoint{1.654381in}{3.210147in}}%
\pgfpathlineto{\pgfqpoint{1.842260in}{3.149137in}}%
\pgfpathlineto{\pgfqpoint{2.030139in}{2.979710in}}%
\pgfpathlineto{\pgfqpoint{2.218017in}{2.662991in}}%
\pgfpathlineto{\pgfqpoint{2.405896in}{2.002418in}}%
\pgfpathlineto{\pgfqpoint{2.593775in}{2.063054in}}%
\pgfpathlineto{\pgfqpoint{2.781654in}{1.787473in}}%
\pgfpathlineto{\pgfqpoint{2.969533in}{1.416704in}}%
\pgfpathlineto{\pgfqpoint{3.157411in}{1.027282in}}%
\pgfpathlineto{\pgfqpoint{3.345290in}{0.894832in}}%
\pgfpathlineto{\pgfqpoint{3.533169in}{0.896122in}}%
\pgfpathlineto{\pgfqpoint{3.721048in}{0.896122in}}%
\pgfpathlineto{\pgfqpoint{3.908927in}{0.896023in}}%
\pgfpathlineto{\pgfqpoint{4.096805in}{0.896023in}}%
\pgfpathlineto{\pgfqpoint{4.284684in}{0.896319in}}%
\pgfpathlineto{\pgfqpoint{4.472563in}{0.896023in}}%
\pgfpathlineto{\pgfqpoint{4.660442in}{0.896023in}}%
\pgfpathlineto{\pgfqpoint{4.848320in}{0.896023in}}%
\pgfpathlineto{\pgfqpoint{5.036199in}{0.896319in}}%
\pgfpathlineto{\pgfqpoint{5.224078in}{0.896122in}}%
\pgfpathlineto{\pgfqpoint{5.411957in}{0.896122in}}%
\pgfpathlineto{\pgfqpoint{5.599836in}{0.896023in}}%
\pgfusepath{stroke}%
\end{pgfscope}%
\begin{pgfscope}%
\pgfsetrectcap%
\pgfsetmiterjoin%
\pgfsetlinewidth{0.803000pt}%
\definecolor{currentstroke}{rgb}{0.000000,0.000000,0.000000}%
\pgfsetstrokecolor{currentstroke}%
\pgfsetdash{}{0pt}%
\pgfpathmoveto{\pgfqpoint{0.865290in}{0.582899in}}%
\pgfpathlineto{\pgfqpoint{0.865290in}{4.278899in}}%
\pgfusepath{stroke}%
\end{pgfscope}%
\begin{pgfscope}%
\pgfsetrectcap%
\pgfsetmiterjoin%
\pgfsetlinewidth{0.803000pt}%
\definecolor{currentstroke}{rgb}{0.000000,0.000000,0.000000}%
\pgfsetstrokecolor{currentstroke}%
\pgfsetdash{}{0pt}%
\pgfpathmoveto{\pgfqpoint{5.825290in}{0.582899in}}%
\pgfpathlineto{\pgfqpoint{5.825290in}{4.278899in}}%
\pgfusepath{stroke}%
\end{pgfscope}%
\begin{pgfscope}%
\pgfsetrectcap%
\pgfsetmiterjoin%
\pgfsetlinewidth{0.803000pt}%
\definecolor{currentstroke}{rgb}{0.000000,0.000000,0.000000}%
\pgfsetstrokecolor{currentstroke}%
\pgfsetdash{}{0pt}%
\pgfpathmoveto{\pgfqpoint{0.865290in}{0.582899in}}%
\pgfpathlineto{\pgfqpoint{5.825290in}{0.582899in}}%
\pgfusepath{stroke}%
\end{pgfscope}%
\begin{pgfscope}%
\pgfsetrectcap%
\pgfsetmiterjoin%
\pgfsetlinewidth{0.803000pt}%
\definecolor{currentstroke}{rgb}{0.000000,0.000000,0.000000}%
\pgfsetstrokecolor{currentstroke}%
\pgfsetdash{}{0pt}%
\pgfpathmoveto{\pgfqpoint{0.865290in}{4.278899in}}%
\pgfpathlineto{\pgfqpoint{5.825290in}{4.278899in}}%
\pgfusepath{stroke}%
\end{pgfscope}%
\begin{pgfscope}%
\pgfsetbuttcap%
\pgfsetmiterjoin%
\definecolor{currentfill}{rgb}{1.000000,1.000000,1.000000}%
\pgfsetfillcolor{currentfill}%
\pgfsetfillopacity{0.800000}%
\pgfsetlinewidth{1.003750pt}%
\definecolor{currentstroke}{rgb}{0.800000,0.800000,0.800000}%
\pgfsetstrokecolor{currentstroke}%
\pgfsetstrokeopacity{0.800000}%
\pgfsetdash{}{0pt}%
\pgfpathmoveto{\pgfqpoint{3.158927in}{3.448778in}}%
\pgfpathlineto{\pgfqpoint{5.669735in}{3.448778in}}%
\pgfpathquadraticcurveto{\pgfqpoint{5.714179in}{3.448778in}}{\pgfqpoint{5.714179in}{3.493222in}}%
\pgfpathlineto{\pgfqpoint{5.714179in}{4.123343in}}%
\pgfpathquadraticcurveto{\pgfqpoint{5.714179in}{4.167788in}}{\pgfqpoint{5.669735in}{4.167788in}}%
\pgfpathlineto{\pgfqpoint{3.158927in}{4.167788in}}%
\pgfpathquadraticcurveto{\pgfqpoint{3.114483in}{4.167788in}}{\pgfqpoint{3.114483in}{4.123343in}}%
\pgfpathlineto{\pgfqpoint{3.114483in}{3.493222in}}%
\pgfpathquadraticcurveto{\pgfqpoint{3.114483in}{3.448778in}}{\pgfqpoint{3.158927in}{3.448778in}}%
\pgfpathlineto{\pgfqpoint{3.158927in}{3.448778in}}%
\pgfpathclose%
\pgfusepath{stroke,fill}%
\end{pgfscope}%
\begin{pgfscope}%
\pgfsetrectcap%
\pgfsetroundjoin%
\pgfsetlinewidth{1.505625pt}%
\definecolor{currentstroke}{rgb}{0.121569,0.466667,0.705882}%
\pgfsetstrokecolor{currentstroke}%
\pgfsetdash{}{0pt}%
\pgfpathmoveto{\pgfqpoint{3.203372in}{3.987840in}}%
\pgfpathlineto{\pgfqpoint{3.425594in}{3.987840in}}%
\pgfpathlineto{\pgfqpoint{3.647816in}{3.987840in}}%
\pgfusepath{stroke}%
\end{pgfscope}%
\begin{pgfscope}%
\definecolor{textcolor}{rgb}{0.000000,0.000000,0.000000}%
\pgfsetstrokecolor{textcolor}%
\pgfsetfillcolor{textcolor}%
\pgftext[x=3.825594in,y=3.910062in,left,base]{\color{textcolor}{\sffamily\fontsize{16.000000}{19.200000}\selectfont\catcode`\^=\active\def^{\ifmmode\sp\else\^{}\fi}\catcode`\%=\active\def
\end{pgfscope}%
\begin{pgfscope}%
\pgfsetrectcap%
\pgfsetroundjoin%
\pgfsetlinewidth{1.505625pt}%
\definecolor{currentstroke}{rgb}{1.000000,0.498039,0.054902}%
\pgfsetstrokecolor{currentstroke}%
\pgfsetdash{}{0pt}%
\pgfpathmoveto{\pgfqpoint{3.203372in}{3.661668in}}%
\pgfpathlineto{\pgfqpoint{3.425594in}{3.661668in}}%
\pgfpathlineto{\pgfqpoint{3.647816in}{3.661668in}}%
\pgfusepath{stroke}%
\end{pgfscope}%
\begin{pgfscope}%
\definecolor{textcolor}{rgb}{0.000000,0.000000,0.000000}%
\pgfsetstrokecolor{textcolor}%
\pgfsetfillcolor{textcolor}%
\pgftext[x=3.825594in,y=3.583890in,left,base]{\color{textcolor}{\sffamily\fontsize{16.000000}{19.200000}\selectfont\catcode`\^=\active\def^{\ifmmode\sp\else\^{}\fi}\catcode`\%=\active\def
\end{pgfscope}%
\end{pgfpicture}%
\makeatother%
\endgroup%

%% file: figs/heat_conv.pgf
\begingroup%
\makeatletter%
\begin{pgfpicture}%
\pgfpathrectangle{\pgfpointorigin}{\pgfqpoint{6.400000in}{4.800000in}}%
\pgfusepath{use as bounding box, clip}%
\begin{pgfscope}%
\pgfsetbuttcap%
\pgfsetmiterjoin%
\definecolor{currentfill}{rgb}{1.000000,1.000000,1.000000}%
\pgfsetfillcolor{currentfill}%
\pgfsetlinewidth{0.000000pt}%
\definecolor{currentstroke}{rgb}{1.000000,1.000000,1.000000}%
\pgfsetstrokecolor{currentstroke}%
\pgfsetdash{}{0pt}%
\pgfpathmoveto{\pgfqpoint{0.000000in}{0.000000in}}%
\pgfpathlineto{\pgfqpoint{6.400000in}{0.000000in}}%
\pgfpathlineto{\pgfqpoint{6.400000in}{4.800000in}}%
\pgfpathlineto{\pgfqpoint{0.000000in}{4.800000in}}%
\pgfpathlineto{\pgfqpoint{0.000000in}{0.000000in}}%
\pgfpathclose%
\pgfusepath{fill}%
\end{pgfscope}%
\begin{pgfscope}%
\pgfsetbuttcap%
\pgfsetmiterjoin%
\definecolor{currentfill}{rgb}{1.000000,1.000000,1.000000}%
\pgfsetfillcolor{currentfill}%
\pgfsetlinewidth{0.000000pt}%
\definecolor{currentstroke}{rgb}{0.000000,0.000000,0.000000}%
\pgfsetstrokecolor{currentstroke}%
\pgfsetstrokeopacity{0.000000}%
\pgfsetdash{}{0pt}%
\pgfpathmoveto{\pgfqpoint{0.800000in}{0.528000in}}%
\pgfpathlineto{\pgfqpoint{5.760000in}{0.528000in}}%
\pgfpathlineto{\pgfqpoint{5.760000in}{4.224000in}}%
\pgfpathlineto{\pgfqpoint{0.800000in}{4.224000in}}%
\pgfpathlineto{\pgfqpoint{0.800000in}{0.528000in}}%
\pgfpathclose%
\pgfusepath{fill}%
\end{pgfscope}%
\begin{pgfscope}%
\pgfsetbuttcap%
\pgfsetroundjoin%
\definecolor{currentfill}{rgb}{0.000000,0.000000,0.000000}%
\pgfsetfillcolor{currentfill}%
\pgfsetlinewidth{0.803000pt}%
\definecolor{currentstroke}{rgb}{0.000000,0.000000,0.000000}%
\pgfsetstrokecolor{currentstroke}%
\pgfsetdash{}{0pt}%
\pgfsys@defobject{currentmarker}{\pgfqpoint{0.000000in}{-0.048611in}}{\pgfqpoint{0.000000in}{0.000000in}}{%
\pgfpathmoveto{\pgfqpoint{0.000000in}{0.000000in}}%
\pgfpathlineto{\pgfqpoint{0.000000in}{-0.048611in}}%
\pgfusepath{stroke,fill}%
}%
\begin{pgfscope}%
\pgfsys@transformshift{1.476364in}{0.528000in}%
\pgfsys@useobject{currentmarker}{}%
\end{pgfscope}%
\end{pgfscope}%
\begin{pgfscope}%
\definecolor{textcolor}{rgb}{0.000000,0.000000,0.000000}%
\pgfsetstrokecolor{textcolor}%
\pgfsetfillcolor{textcolor}%
\pgftext[x=1.476364in,y=0.430778in,,top]{\color{textcolor}{\sffamily\fontsize{16.000000}{19.200000}\selectfont\catcode`\^=\active\def^{\ifmmode\sp\else\^{}\fi}\catcode`\%=\active\def
\end{pgfscope}%
\begin{pgfscope}%
\pgfsetbuttcap%
\pgfsetroundjoin%
\definecolor{currentfill}{rgb}{0.000000,0.000000,0.000000}%
\pgfsetfillcolor{currentfill}%
\pgfsetlinewidth{0.803000pt}%
\definecolor{currentstroke}{rgb}{0.000000,0.000000,0.000000}%
\pgfsetstrokecolor{currentstroke}%
\pgfsetdash{}{0pt}%
\pgfsys@defobject{currentmarker}{\pgfqpoint{0.000000in}{-0.048611in}}{\pgfqpoint{0.000000in}{0.000000in}}{%
\pgfpathmoveto{\pgfqpoint{0.000000in}{0.000000in}}%
\pgfpathlineto{\pgfqpoint{0.000000in}{-0.048611in}}%
\pgfusepath{stroke,fill}%
}%
\begin{pgfscope}%
\pgfsys@transformshift{2.378182in}{0.528000in}%
\pgfsys@useobject{currentmarker}{}%
\end{pgfscope}%
\end{pgfscope}%
\begin{pgfscope}%
\definecolor{textcolor}{rgb}{0.000000,0.000000,0.000000}%
\pgfsetstrokecolor{textcolor}%
\pgfsetfillcolor{textcolor}%
\pgftext[x=2.378182in,y=0.430778in,,top]{\color{textcolor}{\sffamily\fontsize{16.000000}{19.200000}\selectfont\catcode`\^=\active\def^{\ifmmode\sp\else\^{}\fi}\catcode`\%=\active\def
\end{pgfscope}%
\begin{pgfscope}%
\pgfsetbuttcap%
\pgfsetroundjoin%
\definecolor{currentfill}{rgb}{0.000000,0.000000,0.000000}%
\pgfsetfillcolor{currentfill}%
\pgfsetlinewidth{0.803000pt}%
\definecolor{currentstroke}{rgb}{0.000000,0.000000,0.000000}%
\pgfsetstrokecolor{currentstroke}%
\pgfsetdash{}{0pt}%
\pgfsys@defobject{currentmarker}{\pgfqpoint{0.000000in}{-0.048611in}}{\pgfqpoint{0.000000in}{0.000000in}}{%
\pgfpathmoveto{\pgfqpoint{0.000000in}{0.000000in}}%
\pgfpathlineto{\pgfqpoint{0.000000in}{-0.048611in}}%
\pgfusepath{stroke,fill}%
}%
\begin{pgfscope}%
\pgfsys@transformshift{3.280000in}{0.528000in}%
\pgfsys@useobject{currentmarker}{}%
\end{pgfscope}%
\end{pgfscope}%
\begin{pgfscope}%
\definecolor{textcolor}{rgb}{0.000000,0.000000,0.000000}%
\pgfsetstrokecolor{textcolor}%
\pgfsetfillcolor{textcolor}%
\pgftext[x=3.280000in,y=0.430778in,,top]{\color{textcolor}{\sffamily\fontsize{16.000000}{19.200000}\selectfont\catcode`\^=\active\def^{\ifmmode\sp\else\^{}\fi}\catcode`\%=\active\def
\end{pgfscope}%
\begin{pgfscope}%
\pgfsetbuttcap%
\pgfsetroundjoin%
\definecolor{currentfill}{rgb}{0.000000,0.000000,0.000000}%
\pgfsetfillcolor{currentfill}%
\pgfsetlinewidth{0.803000pt}%
\definecolor{currentstroke}{rgb}{0.000000,0.000000,0.000000}%
\pgfsetstrokecolor{currentstroke}%
\pgfsetdash{}{0pt}%
\pgfsys@defobject{currentmarker}{\pgfqpoint{0.000000in}{-0.048611in}}{\pgfqpoint{0.000000in}{0.000000in}}{%
\pgfpathmoveto{\pgfqpoint{0.000000in}{0.000000in}}%
\pgfpathlineto{\pgfqpoint{0.000000in}{-0.048611in}}%
\pgfusepath{stroke,fill}%
}%
\begin{pgfscope}%
\pgfsys@transformshift{4.181818in}{0.528000in}%
\pgfsys@useobject{currentmarker}{}%
\end{pgfscope}%
\end{pgfscope}%
\begin{pgfscope}%
\definecolor{textcolor}{rgb}{0.000000,0.000000,0.000000}%
\pgfsetstrokecolor{textcolor}%
\pgfsetfillcolor{textcolor}%
\pgftext[x=4.181818in,y=0.430778in,,top]{\color{textcolor}{\sffamily\fontsize{16.000000}{19.200000}\selectfont\catcode`\^=\active\def^{\ifmmode\sp\else\^{}\fi}\catcode`\%=\active\def
\end{pgfscope}%
\begin{pgfscope}%
\pgfsetbuttcap%
\pgfsetroundjoin%
\definecolor{currentfill}{rgb}{0.000000,0.000000,0.000000}%
\pgfsetfillcolor{currentfill}%
\pgfsetlinewidth{0.803000pt}%
\definecolor{currentstroke}{rgb}{0.000000,0.000000,0.000000}%
\pgfsetstrokecolor{currentstroke}%
\pgfsetdash{}{0pt}%
\pgfsys@defobject{currentmarker}{\pgfqpoint{0.000000in}{-0.048611in}}{\pgfqpoint{0.000000in}{0.000000in}}{%
\pgfpathmoveto{\pgfqpoint{0.000000in}{0.000000in}}%
\pgfpathlineto{\pgfqpoint{0.000000in}{-0.048611in}}%
\pgfusepath{stroke,fill}%
}%
\begin{pgfscope}%
\pgfsys@transformshift{5.083636in}{0.528000in}%
\pgfsys@useobject{currentmarker}{}%
\end{pgfscope}%
\end{pgfscope}%
\begin{pgfscope}%
\definecolor{textcolor}{rgb}{0.000000,0.000000,0.000000}%
\pgfsetstrokecolor{textcolor}%
\pgfsetfillcolor{textcolor}%
\pgftext[x=5.083636in,y=0.430778in,,top]{\color{textcolor}{\sffamily\fontsize{16.000000}{19.200000}\selectfont\catcode`\^=\active\def^{\ifmmode\sp\else\^{}\fi}\catcode`\%=\active\def
\end{pgfscope}%
\begin{pgfscope}%
\definecolor{textcolor}{rgb}{0.000000,0.000000,0.000000}%
\pgfsetstrokecolor{textcolor}%
\pgfsetfillcolor{textcolor}%
\pgftext[x=3.280000in,y=0.160162in,,top]{\color{textcolor}{\sffamily\fontsize{16.000000}{19.200000}\selectfont\catcode`\^=\active\def^{\ifmmode\sp\else\^{}\fi}\catcode`\%=\active\def
\end{pgfscope}%
\begin{pgfscope}%
\pgfsetbuttcap%
\pgfsetroundjoin%
\definecolor{currentfill}{rgb}{0.000000,0.000000,0.000000}%
\pgfsetfillcolor{currentfill}%
\pgfsetlinewidth{0.803000pt}%
\definecolor{currentstroke}{rgb}{0.000000,0.000000,0.000000}%
\pgfsetstrokecolor{currentstroke}%
\pgfsetdash{}{0pt}%
\pgfsys@defobject{currentmarker}{\pgfqpoint{-0.048611in}{0.000000in}}{\pgfqpoint{-0.000000in}{0.000000in}}{%
\pgfpathmoveto{\pgfqpoint{-0.000000in}{0.000000in}}%
\pgfpathlineto{\pgfqpoint{-0.048611in}{0.000000in}}%
\pgfusepath{stroke,fill}%
}%
\begin{pgfscope}%
\pgfsys@transformshift{0.800000in}{1.299866in}%
\pgfsys@useobject{currentmarker}{}%
\end{pgfscope}%
\end{pgfscope}%
\begin{pgfscope}%
\definecolor{textcolor}{rgb}{0.000000,0.000000,0.000000}%
\pgfsetstrokecolor{textcolor}%
\pgfsetfillcolor{textcolor}%
\pgftext[x=0.281368in, y=1.215448in, left, base]{\color{textcolor}{\sffamily\fontsize{16.000000}{19.200000}\selectfont\catcode`\^=\active\def^{\ifmmode\sp\else\^{}\fi}\catcode`\%=\active\def
\end{pgfscope}%
\begin{pgfscope}%
\pgfsetbuttcap%
\pgfsetroundjoin%
\definecolor{currentfill}{rgb}{0.000000,0.000000,0.000000}%
\pgfsetfillcolor{currentfill}%
\pgfsetlinewidth{0.803000pt}%
\definecolor{currentstroke}{rgb}{0.000000,0.000000,0.000000}%
\pgfsetstrokecolor{currentstroke}%
\pgfsetdash{}{0pt}%
\pgfsys@defobject{currentmarker}{\pgfqpoint{-0.048611in}{0.000000in}}{\pgfqpoint{-0.000000in}{0.000000in}}{%
\pgfpathmoveto{\pgfqpoint{-0.000000in}{0.000000in}}%
\pgfpathlineto{\pgfqpoint{-0.048611in}{0.000000in}}%
\pgfusepath{stroke,fill}%
}%
\begin{pgfscope}%
\pgfsys@transformshift{0.800000in}{2.130727in}%
\pgfsys@useobject{currentmarker}{}%
\end{pgfscope}%
\end{pgfscope}%
\begin{pgfscope}%
\definecolor{textcolor}{rgb}{0.000000,0.000000,0.000000}%
\pgfsetstrokecolor{textcolor}%
\pgfsetfillcolor{textcolor}%
\pgftext[x=0.281368in, y=2.046308in, left, base]{\color{textcolor}{\sffamily\fontsize{16.000000}{19.200000}\selectfont\catcode`\^=\active\def^{\ifmmode\sp\else\^{}\fi}\catcode`\%=\active\def
\end{pgfscope}%
\begin{pgfscope}%
\pgfsetbuttcap%
\pgfsetroundjoin%
\definecolor{currentfill}{rgb}{0.000000,0.000000,0.000000}%
\pgfsetfillcolor{currentfill}%
\pgfsetlinewidth{0.803000pt}%
\definecolor{currentstroke}{rgb}{0.000000,0.000000,0.000000}%
\pgfsetstrokecolor{currentstroke}%
\pgfsetdash{}{0pt}%
\pgfsys@defobject{currentmarker}{\pgfqpoint{-0.048611in}{0.000000in}}{\pgfqpoint{-0.000000in}{0.000000in}}{%
\pgfpathmoveto{\pgfqpoint{-0.000000in}{0.000000in}}%
\pgfpathlineto{\pgfqpoint{-0.048611in}{0.000000in}}%
\pgfusepath{stroke,fill}%
}%
\begin{pgfscope}%
\pgfsys@transformshift{0.800000in}{2.961587in}%
\pgfsys@useobject{currentmarker}{}%
\end{pgfscope}%
\end{pgfscope}%
\begin{pgfscope}%
\definecolor{textcolor}{rgb}{0.000000,0.000000,0.000000}%
\pgfsetstrokecolor{textcolor}%
\pgfsetfillcolor{textcolor}%
\pgftext[x=0.281368in, y=2.877169in, left, base]{\color{textcolor}{\sffamily\fontsize{16.000000}{19.200000}\selectfont\catcode`\^=\active\def^{\ifmmode\sp\else\^{}\fi}\catcode`\%=\active\def
\end{pgfscope}%
\begin{pgfscope}%
\pgfsetbuttcap%
\pgfsetroundjoin%
\definecolor{currentfill}{rgb}{0.000000,0.000000,0.000000}%
\pgfsetfillcolor{currentfill}%
\pgfsetlinewidth{0.803000pt}%
\definecolor{currentstroke}{rgb}{0.000000,0.000000,0.000000}%
\pgfsetstrokecolor{currentstroke}%
\pgfsetdash{}{0pt}%
\pgfsys@defobject{currentmarker}{\pgfqpoint{-0.048611in}{0.000000in}}{\pgfqpoint{-0.000000in}{0.000000in}}{%
\pgfpathmoveto{\pgfqpoint{-0.000000in}{0.000000in}}%
\pgfpathlineto{\pgfqpoint{-0.048611in}{0.000000in}}%
\pgfusepath{stroke,fill}%
}%
\begin{pgfscope}%
\pgfsys@transformshift{0.800000in}{3.792448in}%
\pgfsys@useobject{currentmarker}{}%
\end{pgfscope}%
\end{pgfscope}%
\begin{pgfscope}%
\definecolor{textcolor}{rgb}{0.000000,0.000000,0.000000}%
\pgfsetstrokecolor{textcolor}%
\pgfsetfillcolor{textcolor}%
\pgftext[x=0.281368in, y=3.708029in, left, base]{\color{textcolor}{\sffamily\fontsize{16.000000}{19.200000}\selectfont\catcode`\^=\active\def^{\ifmmode\sp\else\^{}\fi}\catcode`\%=\active\def
\end{pgfscope}%
\begin{pgfscope}%
\definecolor{textcolor}{rgb}{0.000000,0.000000,0.000000}%
\pgfsetstrokecolor{textcolor}%
\pgfsetfillcolor{textcolor}%
\pgftext[x=0.225812in,y=2.376000in,,bottom,rotate=90.000000]{\color{textcolor}{\sffamily\fontsize{16.000000}{19.200000}\selectfont\catcode`\^=\active\def^{\ifmmode\sp\else\^{}\fi}\catcode`\%=\active\def
\end{pgfscope}%
\begin{pgfscope}%
\pgfpathrectangle{\pgfqpoint{0.800000in}{0.528000in}}{\pgfqpoint{4.960000in}{3.696000in}}%
\pgfusepath{clip}%
\pgfsetrectcap%
\pgfsetroundjoin%
\pgfsetlinewidth{1.505625pt}%
\definecolor{currentstroke}{rgb}{0.121569,0.466667,0.705882}%
\pgfsetstrokecolor{currentstroke}%
\pgfsetdash{}{0pt}%
\pgfpathmoveto{\pgfqpoint{1.025455in}{4.056000in}}%
\pgfpathlineto{\pgfqpoint{1.476364in}{3.308602in}}%
\pgfpathlineto{\pgfqpoint{1.927273in}{3.406233in}}%
\pgfpathlineto{\pgfqpoint{2.378182in}{2.668844in}}%
\pgfpathlineto{\pgfqpoint{2.829091in}{2.402521in}}%
\pgfpathlineto{\pgfqpoint{3.280000in}{2.130165in}}%
\pgfpathlineto{\pgfqpoint{3.730909in}{1.854839in}}%
\pgfpathlineto{\pgfqpoint{4.181818in}{1.578097in}}%
\pgfpathlineto{\pgfqpoint{4.632727in}{1.300691in}}%
\pgfpathlineto{\pgfqpoint{5.083636in}{1.022974in}}%
\pgfpathlineto{\pgfqpoint{5.534545in}{0.745114in}}%
\pgfusepath{stroke}%
\end{pgfscope}%
\begin{pgfscope}%
\pgfpathrectangle{\pgfqpoint{0.800000in}{0.528000in}}{\pgfqpoint{4.960000in}{3.696000in}}%
\pgfusepath{clip}%
\pgfsetrectcap%
\pgfsetroundjoin%
\pgfsetlinewidth{1.505625pt}%
\definecolor{currentstroke}{rgb}{1.000000,0.498039,0.054902}%
\pgfsetstrokecolor{currentstroke}%
\pgfsetdash{}{0pt}%
\pgfpathmoveto{\pgfqpoint{1.025455in}{4.056000in}}%
\pgfpathlineto{\pgfqpoint{1.476364in}{3.299915in}}%
\pgfpathlineto{\pgfqpoint{1.927273in}{2.701694in}}%
\pgfpathlineto{\pgfqpoint{2.378182in}{2.858348in}}%
\pgfpathlineto{\pgfqpoint{2.829091in}{2.509917in}}%
\pgfpathlineto{\pgfqpoint{3.280000in}{2.151433in}}%
\pgfpathlineto{\pgfqpoint{3.730909in}{1.789005in}}%
\pgfpathlineto{\pgfqpoint{4.181818in}{1.425099in}}%
\pgfpathlineto{\pgfqpoint{4.632727in}{1.060649in}}%
\pgfpathlineto{\pgfqpoint{5.083636in}{0.696000in}}%
\pgfusepath{stroke}%
\end{pgfscope}%
\begin{pgfscope}%
\pgfsetrectcap%
\pgfsetmiterjoin%
\pgfsetlinewidth{0.803000pt}%
\definecolor{currentstroke}{rgb}{0.000000,0.000000,0.000000}%
\pgfsetstrokecolor{currentstroke}%
\pgfsetdash{}{0pt}%
\pgfpathmoveto{\pgfqpoint{0.800000in}{0.528000in}}%
\pgfpathlineto{\pgfqpoint{0.800000in}{4.224000in}}%
\pgfusepath{stroke}%
\end{pgfscope}%
\begin{pgfscope}%
\pgfsetrectcap%
\pgfsetmiterjoin%
\pgfsetlinewidth{0.803000pt}%
\definecolor{currentstroke}{rgb}{0.000000,0.000000,0.000000}%
\pgfsetstrokecolor{currentstroke}%
\pgfsetdash{}{0pt}%
\pgfpathmoveto{\pgfqpoint{5.760000in}{0.528000in}}%
\pgfpathlineto{\pgfqpoint{5.760000in}{4.224000in}}%
\pgfusepath{stroke}%
\end{pgfscope}%
\begin{pgfscope}%
\pgfsetrectcap%
\pgfsetmiterjoin%
\pgfsetlinewidth{0.803000pt}%
\definecolor{currentstroke}{rgb}{0.000000,0.000000,0.000000}%
\pgfsetstrokecolor{currentstroke}%
\pgfsetdash{}{0pt}%
\pgfpathmoveto{\pgfqpoint{0.800000in}{0.528000in}}%
\pgfpathlineto{\pgfqpoint{5.760000in}{0.528000in}}%
\pgfusepath{stroke}%
\end{pgfscope}%
\begin{pgfscope}%
\pgfsetrectcap%
\pgfsetmiterjoin%
\pgfsetlinewidth{0.803000pt}%
\definecolor{currentstroke}{rgb}{0.000000,0.000000,0.000000}%
\pgfsetstrokecolor{currentstroke}%
\pgfsetdash{}{0pt}%
\pgfpathmoveto{\pgfqpoint{0.800000in}{4.224000in}}%
\pgfpathlineto{\pgfqpoint{5.760000in}{4.224000in}}%
\pgfusepath{stroke}%
\end{pgfscope}%
\begin{pgfscope}%
\pgfsetbuttcap%
\pgfsetmiterjoin%
\definecolor{currentfill}{rgb}{1.000000,1.000000,1.000000}%
\pgfsetfillcolor{currentfill}%
\pgfsetfillopacity{0.800000}%
\pgfsetlinewidth{1.003750pt}%
\definecolor{currentstroke}{rgb}{0.800000,0.800000,0.800000}%
\pgfsetstrokecolor{currentstroke}%
\pgfsetstrokeopacity{0.800000}%
\pgfsetdash{}{0pt}%
\pgfpathmoveto{\pgfqpoint{3.093637in}{3.393879in}}%
\pgfpathlineto{\pgfqpoint{5.604444in}{3.393879in}}%
\pgfpathquadraticcurveto{\pgfqpoint{5.648889in}{3.393879in}}{\pgfqpoint{5.648889in}{3.438323in}}%
\pgfpathlineto{\pgfqpoint{5.648889in}{4.068444in}}%
\pgfpathquadraticcurveto{\pgfqpoint{5.648889in}{4.112889in}}{\pgfqpoint{5.604444in}{4.112889in}}%
\pgfpathlineto{\pgfqpoint{3.093637in}{4.112889in}}%
\pgfpathquadraticcurveto{\pgfqpoint{3.049193in}{4.112889in}}{\pgfqpoint{3.049193in}{4.068444in}}%
\pgfpathlineto{\pgfqpoint{3.049193in}{3.438323in}}%
\pgfpathquadraticcurveto{\pgfqpoint{3.049193in}{3.393879in}}{\pgfqpoint{3.093637in}{3.393879in}}%
\pgfpathlineto{\pgfqpoint{3.093637in}{3.393879in}}%
\pgfpathclose%
\pgfusepath{stroke,fill}%
\end{pgfscope}%
\begin{pgfscope}%
\pgfsetrectcap%
\pgfsetroundjoin%
\pgfsetlinewidth{1.505625pt}%
\definecolor{currentstroke}{rgb}{0.121569,0.466667,0.705882}%
\pgfsetstrokecolor{currentstroke}%
\pgfsetdash{}{0pt}%
\pgfpathmoveto{\pgfqpoint{3.138082in}{3.932941in}}%
\pgfpathlineto{\pgfqpoint{3.360304in}{3.932941in}}%
\pgfpathlineto{\pgfqpoint{3.582526in}{3.932941in}}%
\pgfusepath{stroke}%
\end{pgfscope}%
\begin{pgfscope}%
\definecolor{textcolor}{rgb}{0.000000,0.000000,0.000000}%
\pgfsetstrokecolor{textcolor}%
\pgfsetfillcolor{textcolor}%
\pgftext[x=3.760304in,y=3.855163in,left,base]{\color{textcolor}{\sffamily\fontsize{16.000000}{19.200000}\selectfont\catcode`\^=\active\def^{\ifmmode\sp\else\^{}\fi}\catcode`\%=\active\def
\end{pgfscope}%
\begin{pgfscope}%
\pgfsetrectcap%
\pgfsetroundjoin%
\pgfsetlinewidth{1.505625pt}%
\definecolor{currentstroke}{rgb}{1.000000,0.498039,0.054902}%
\pgfsetstrokecolor{currentstroke}%
\pgfsetdash{}{0pt}%
\pgfpathmoveto{\pgfqpoint{3.138082in}{3.606769in}}%
\pgfpathlineto{\pgfqpoint{3.360304in}{3.606769in}}%
\pgfpathlineto{\pgfqpoint{3.582526in}{3.606769in}}%
\pgfusepath{stroke}%
\end{pgfscope}%
\begin{pgfscope}%
\definecolor{textcolor}{rgb}{0.000000,0.000000,0.000000}%
\pgfsetstrokecolor{textcolor}%
\pgfsetfillcolor{textcolor}%
\pgftext[x=3.760304in,y=3.528992in,left,base]{\color{textcolor}{\sffamily\fontsize{16.000000}{19.200000}\selectfont\catcode`\^=\active\def^{\ifmmode\sp\else\^{}\fi}\catcode`\%=\active\def
\end{pgfscope}%
\end{pgfpicture}%
\makeatother%
\endgroup%

%% file: figs/heat_param.pgf
\begingroup%
\makeatletter%
\begin{pgfpicture}%
\pgfpathrectangle{\pgfpointorigin}{\pgfqpoint{6.164015in}{4.478899in}}%
\pgfusepath{use as bounding box, clip}%
\begin{pgfscope}%
\pgfsetbuttcap%
\pgfsetmiterjoin%
\definecolor{currentfill}{rgb}{1.000000,1.000000,1.000000}%
\pgfsetfillcolor{currentfill}%
\pgfsetlinewidth{0.000000pt}%
\definecolor{currentstroke}{rgb}{1.000000,1.000000,1.000000}%
\pgfsetstrokecolor{currentstroke}%
\pgfsetdash{}{0pt}%
\pgfpathmoveto{\pgfqpoint{0.000000in}{0.000000in}}%
\pgfpathlineto{\pgfqpoint{6.164015in}{0.000000in}}%
\pgfpathlineto{\pgfqpoint{6.164015in}{4.478899in}}%
\pgfpathlineto{\pgfqpoint{0.000000in}{4.478899in}}%
\pgfpathlineto{\pgfqpoint{0.000000in}{0.000000in}}%
\pgfpathclose%
\pgfusepath{fill}%
\end{pgfscope}%
\begin{pgfscope}%
\pgfsetbuttcap%
\pgfsetmiterjoin%
\definecolor{currentfill}{rgb}{1.000000,1.000000,1.000000}%
\pgfsetfillcolor{currentfill}%
\pgfsetlinewidth{0.000000pt}%
\definecolor{currentstroke}{rgb}{0.000000,0.000000,0.000000}%
\pgfsetstrokecolor{currentstroke}%
\pgfsetstrokeopacity{0.000000}%
\pgfsetdash{}{0pt}%
\pgfpathmoveto{\pgfqpoint{1.104015in}{0.682899in}}%
\pgfpathlineto{\pgfqpoint{6.064015in}{0.682899in}}%
\pgfpathlineto{\pgfqpoint{6.064015in}{4.378899in}}%
\pgfpathlineto{\pgfqpoint{1.104015in}{4.378899in}}%
\pgfpathlineto{\pgfqpoint{1.104015in}{0.682899in}}%
\pgfpathclose%
\pgfusepath{fill}%
\end{pgfscope}%
\begin{pgfscope}%
\pgfsetbuttcap%
\pgfsetroundjoin%
\definecolor{currentfill}{rgb}{0.000000,0.000000,0.000000}%
\pgfsetfillcolor{currentfill}%
\pgfsetlinewidth{0.803000pt}%
\definecolor{currentstroke}{rgb}{0.000000,0.000000,0.000000}%
\pgfsetstrokecolor{currentstroke}%
\pgfsetdash{}{0pt}%
\pgfsys@defobject{currentmarker}{\pgfqpoint{0.000000in}{-0.048611in}}{\pgfqpoint{0.000000in}{0.000000in}}{%
\pgfpathmoveto{\pgfqpoint{0.000000in}{0.000000in}}%
\pgfpathlineto{\pgfqpoint{0.000000in}{-0.048611in}}%
\pgfusepath{stroke,fill}%
}%
\begin{pgfscope}%
\pgfsys@transformshift{1.780378in}{0.682899in}%
\pgfsys@useobject{currentmarker}{}%
\end{pgfscope}%
\end{pgfscope}%
\begin{pgfscope}%
\definecolor{textcolor}{rgb}{0.000000,0.000000,0.000000}%
\pgfsetstrokecolor{textcolor}%
\pgfsetfillcolor{textcolor}%
\pgftext[x=1.780378in,y=0.585677in,,top]{\color{textcolor}{\sffamily\fontsize{16.000000}{19.200000}\selectfont\catcode`\^=\active\def^{\ifmmode\sp\else\^{}\fi}\catcode`\%=\active\def
\end{pgfscope}%
\begin{pgfscope}%
\pgfsetbuttcap%
\pgfsetroundjoin%
\definecolor{currentfill}{rgb}{0.000000,0.000000,0.000000}%
\pgfsetfillcolor{currentfill}%
\pgfsetlinewidth{0.803000pt}%
\definecolor{currentstroke}{rgb}{0.000000,0.000000,0.000000}%
\pgfsetstrokecolor{currentstroke}%
\pgfsetdash{}{0pt}%
\pgfsys@defobject{currentmarker}{\pgfqpoint{0.000000in}{-0.048611in}}{\pgfqpoint{0.000000in}{0.000000in}}{%
\pgfpathmoveto{\pgfqpoint{0.000000in}{0.000000in}}%
\pgfpathlineto{\pgfqpoint{0.000000in}{-0.048611in}}%
\pgfusepath{stroke,fill}%
}%
\begin{pgfscope}%
\pgfsys@transformshift{2.682196in}{0.682899in}%
\pgfsys@useobject{currentmarker}{}%
\end{pgfscope}%
\end{pgfscope}%
\begin{pgfscope}%
\definecolor{textcolor}{rgb}{0.000000,0.000000,0.000000}%
\pgfsetstrokecolor{textcolor}%
\pgfsetfillcolor{textcolor}%
\pgftext[x=2.682196in,y=0.585677in,,top]{\color{textcolor}{\sffamily\fontsize{16.000000}{19.200000}\selectfont\catcode`\^=\active\def^{\ifmmode\sp\else\^{}\fi}\catcode`\%=\active\def
\end{pgfscope}%
\begin{pgfscope}%
\pgfsetbuttcap%
\pgfsetroundjoin%
\definecolor{currentfill}{rgb}{0.000000,0.000000,0.000000}%
\pgfsetfillcolor{currentfill}%
\pgfsetlinewidth{0.803000pt}%
\definecolor{currentstroke}{rgb}{0.000000,0.000000,0.000000}%
\pgfsetstrokecolor{currentstroke}%
\pgfsetdash{}{0pt}%
\pgfsys@defobject{currentmarker}{\pgfqpoint{0.000000in}{-0.048611in}}{\pgfqpoint{0.000000in}{0.000000in}}{%
\pgfpathmoveto{\pgfqpoint{0.000000in}{0.000000in}}%
\pgfpathlineto{\pgfqpoint{0.000000in}{-0.048611in}}%
\pgfusepath{stroke,fill}%
}%
\begin{pgfscope}%
\pgfsys@transformshift{3.584015in}{0.682899in}%
\pgfsys@useobject{currentmarker}{}%
\end{pgfscope}%
\end{pgfscope}%
\begin{pgfscope}%
\definecolor{textcolor}{rgb}{0.000000,0.000000,0.000000}%
\pgfsetstrokecolor{textcolor}%
\pgfsetfillcolor{textcolor}%
\pgftext[x=3.584015in,y=0.585677in,,top]{\color{textcolor}{\sffamily\fontsize{16.000000}{19.200000}\selectfont\catcode`\^=\active\def^{\ifmmode\sp\else\^{}\fi}\catcode`\%=\active\def
\end{pgfscope}%
\begin{pgfscope}%
\pgfsetbuttcap%
\pgfsetroundjoin%
\definecolor{currentfill}{rgb}{0.000000,0.000000,0.000000}%
\pgfsetfillcolor{currentfill}%
\pgfsetlinewidth{0.803000pt}%
\definecolor{currentstroke}{rgb}{0.000000,0.000000,0.000000}%
\pgfsetstrokecolor{currentstroke}%
\pgfsetdash{}{0pt}%
\pgfsys@defobject{currentmarker}{\pgfqpoint{0.000000in}{-0.048611in}}{\pgfqpoint{0.000000in}{0.000000in}}{%
\pgfpathmoveto{\pgfqpoint{0.000000in}{0.000000in}}%
\pgfpathlineto{\pgfqpoint{0.000000in}{-0.048611in}}%
\pgfusepath{stroke,fill}%
}%
\begin{pgfscope}%
\pgfsys@transformshift{4.485833in}{0.682899in}%
\pgfsys@useobject{currentmarker}{}%
\end{pgfscope}%
\end{pgfscope}%
\begin{pgfscope}%
\definecolor{textcolor}{rgb}{0.000000,0.000000,0.000000}%
\pgfsetstrokecolor{textcolor}%
\pgfsetfillcolor{textcolor}%
\pgftext[x=4.485833in,y=0.585677in,,top]{\color{textcolor}{\sffamily\fontsize{16.000000}{19.200000}\selectfont\catcode`\^=\active\def^{\ifmmode\sp\else\^{}\fi}\catcode`\%=\active\def
\end{pgfscope}%
\begin{pgfscope}%
\pgfsetbuttcap%
\pgfsetroundjoin%
\definecolor{currentfill}{rgb}{0.000000,0.000000,0.000000}%
\pgfsetfillcolor{currentfill}%
\pgfsetlinewidth{0.803000pt}%
\definecolor{currentstroke}{rgb}{0.000000,0.000000,0.000000}%
\pgfsetstrokecolor{currentstroke}%
\pgfsetdash{}{0pt}%
\pgfsys@defobject{currentmarker}{\pgfqpoint{0.000000in}{-0.048611in}}{\pgfqpoint{0.000000in}{0.000000in}}{%
\pgfpathmoveto{\pgfqpoint{0.000000in}{0.000000in}}%
\pgfpathlineto{\pgfqpoint{0.000000in}{-0.048611in}}%
\pgfusepath{stroke,fill}%
}%
\begin{pgfscope}%
\pgfsys@transformshift{5.387651in}{0.682899in}%
\pgfsys@useobject{currentmarker}{}%
\end{pgfscope}%
\end{pgfscope}%
\begin{pgfscope}%
\definecolor{textcolor}{rgb}{0.000000,0.000000,0.000000}%
\pgfsetstrokecolor{textcolor}%
\pgfsetfillcolor{textcolor}%
\pgftext[x=5.387651in,y=0.585677in,,top]{\color{textcolor}{\sffamily\fontsize{16.000000}{19.200000}\selectfont\catcode`\^=\active\def^{\ifmmode\sp\else\^{}\fi}\catcode`\%=\active\def
\end{pgfscope}%
\begin{pgfscope}%
\definecolor{textcolor}{rgb}{0.000000,0.000000,0.000000}%
\pgfsetstrokecolor{textcolor}%
\pgfsetfillcolor{textcolor}%
\pgftext[x=3.584015in,y=0.315061in,,top]{\color{textcolor}{\sffamily\fontsize{16.000000}{19.200000}\selectfont\catcode`\^=\active\def^{\ifmmode\sp\else\^{}\fi}\catcode`\%=\active\def
\end{pgfscope}%
\begin{pgfscope}%
\pgfsetbuttcap%
\pgfsetroundjoin%
\definecolor{currentfill}{rgb}{0.000000,0.000000,0.000000}%
\pgfsetfillcolor{currentfill}%
\pgfsetlinewidth{0.803000pt}%
\definecolor{currentstroke}{rgb}{0.000000,0.000000,0.000000}%
\pgfsetstrokecolor{currentstroke}%
\pgfsetdash{}{0pt}%
\pgfsys@defobject{currentmarker}{\pgfqpoint{-0.048611in}{0.000000in}}{\pgfqpoint{-0.000000in}{0.000000in}}{%
\pgfpathmoveto{\pgfqpoint{-0.000000in}{0.000000in}}%
\pgfpathlineto{\pgfqpoint{-0.048611in}{0.000000in}}%
\pgfusepath{stroke,fill}%
}%
\begin{pgfscope}%
\pgfsys@transformshift{1.104015in}{1.051511in}%
\pgfsys@useobject{currentmarker}{}%
\end{pgfscope}%
\end{pgfscope}%
\begin{pgfscope}%
\definecolor{textcolor}{rgb}{0.000000,0.000000,0.000000}%
\pgfsetstrokecolor{textcolor}%
\pgfsetfillcolor{textcolor}%
\pgftext[x=0.370616in, y=0.967092in, left, base]{\color{textcolor}{\sffamily\fontsize{16.000000}{19.200000}\selectfont\catcode`\^=\active\def^{\ifmmode\sp\else\^{}\fi}\catcode`\%=\active\def
\end{pgfscope}%
\begin{pgfscope}%
\pgfsetbuttcap%
\pgfsetroundjoin%
\definecolor{currentfill}{rgb}{0.000000,0.000000,0.000000}%
\pgfsetfillcolor{currentfill}%
\pgfsetlinewidth{0.803000pt}%
\definecolor{currentstroke}{rgb}{0.000000,0.000000,0.000000}%
\pgfsetstrokecolor{currentstroke}%
\pgfsetdash{}{0pt}%
\pgfsys@defobject{currentmarker}{\pgfqpoint{-0.048611in}{0.000000in}}{\pgfqpoint{-0.000000in}{0.000000in}}{%
\pgfpathmoveto{\pgfqpoint{-0.000000in}{0.000000in}}%
\pgfpathlineto{\pgfqpoint{-0.048611in}{0.000000in}}%
\pgfusepath{stroke,fill}%
}%
\begin{pgfscope}%
\pgfsys@transformshift{1.104015in}{1.578075in}%
\pgfsys@useobject{currentmarker}{}%
\end{pgfscope}%
\end{pgfscope}%
\begin{pgfscope}%
\definecolor{textcolor}{rgb}{0.000000,0.000000,0.000000}%
\pgfsetstrokecolor{textcolor}%
\pgfsetfillcolor{textcolor}%
\pgftext[x=0.370616in, y=1.493657in, left, base]{\color{textcolor}{\sffamily\fontsize{16.000000}{19.200000}\selectfont\catcode`\^=\active\def^{\ifmmode\sp\else\^{}\fi}\catcode`\%=\active\def
\end{pgfscope}%
\begin{pgfscope}%
\pgfsetbuttcap%
\pgfsetroundjoin%
\definecolor{currentfill}{rgb}{0.000000,0.000000,0.000000}%
\pgfsetfillcolor{currentfill}%
\pgfsetlinewidth{0.803000pt}%
\definecolor{currentstroke}{rgb}{0.000000,0.000000,0.000000}%
\pgfsetstrokecolor{currentstroke}%
\pgfsetdash{}{0pt}%
\pgfsys@defobject{currentmarker}{\pgfqpoint{-0.048611in}{0.000000in}}{\pgfqpoint{-0.000000in}{0.000000in}}{%
\pgfpathmoveto{\pgfqpoint{-0.000000in}{0.000000in}}%
\pgfpathlineto{\pgfqpoint{-0.048611in}{0.000000in}}%
\pgfusepath{stroke,fill}%
}%
\begin{pgfscope}%
\pgfsys@transformshift{1.104015in}{2.104640in}%
\pgfsys@useobject{currentmarker}{}%
\end{pgfscope}%
\end{pgfscope}%
\begin{pgfscope}%
\definecolor{textcolor}{rgb}{0.000000,0.000000,0.000000}%
\pgfsetstrokecolor{textcolor}%
\pgfsetfillcolor{textcolor}%
\pgftext[x=0.370616in, y=2.020222in, left, base]{\color{textcolor}{\sffamily\fontsize{16.000000}{19.200000}\selectfont\catcode`\^=\active\def^{\ifmmode\sp\else\^{}\fi}\catcode`\%=\active\def
\end{pgfscope}%
\begin{pgfscope}%
\pgfsetbuttcap%
\pgfsetroundjoin%
\definecolor{currentfill}{rgb}{0.000000,0.000000,0.000000}%
\pgfsetfillcolor{currentfill}%
\pgfsetlinewidth{0.803000pt}%
\definecolor{currentstroke}{rgb}{0.000000,0.000000,0.000000}%
\pgfsetstrokecolor{currentstroke}%
\pgfsetdash{}{0pt}%
\pgfsys@defobject{currentmarker}{\pgfqpoint{-0.048611in}{0.000000in}}{\pgfqpoint{-0.000000in}{0.000000in}}{%
\pgfpathmoveto{\pgfqpoint{-0.000000in}{0.000000in}}%
\pgfpathlineto{\pgfqpoint{-0.048611in}{0.000000in}}%
\pgfusepath{stroke,fill}%
}%
\begin{pgfscope}%
\pgfsys@transformshift{1.104015in}{2.631205in}%
\pgfsys@useobject{currentmarker}{}%
\end{pgfscope}%
\end{pgfscope}%
\begin{pgfscope}%
\definecolor{textcolor}{rgb}{0.000000,0.000000,0.000000}%
\pgfsetstrokecolor{textcolor}%
\pgfsetfillcolor{textcolor}%
\pgftext[x=0.370616in, y=2.546786in, left, base]{\color{textcolor}{\sffamily\fontsize{16.000000}{19.200000}\selectfont\catcode`\^=\active\def^{\ifmmode\sp\else\^{}\fi}\catcode`\%=\active\def
\end{pgfscope}%
\begin{pgfscope}%
\pgfsetbuttcap%
\pgfsetroundjoin%
\definecolor{currentfill}{rgb}{0.000000,0.000000,0.000000}%
\pgfsetfillcolor{currentfill}%
\pgfsetlinewidth{0.803000pt}%
\definecolor{currentstroke}{rgb}{0.000000,0.000000,0.000000}%
\pgfsetstrokecolor{currentstroke}%
\pgfsetdash{}{0pt}%
\pgfsys@defobject{currentmarker}{\pgfqpoint{-0.048611in}{0.000000in}}{\pgfqpoint{-0.000000in}{0.000000in}}{%
\pgfpathmoveto{\pgfqpoint{-0.000000in}{0.000000in}}%
\pgfpathlineto{\pgfqpoint{-0.048611in}{0.000000in}}%
\pgfusepath{stroke,fill}%
}%
\begin{pgfscope}%
\pgfsys@transformshift{1.104015in}{3.157769in}%
\pgfsys@useobject{currentmarker}{}%
\end{pgfscope}%
\end{pgfscope}%
\begin{pgfscope}%
\definecolor{textcolor}{rgb}{0.000000,0.000000,0.000000}%
\pgfsetstrokecolor{textcolor}%
\pgfsetfillcolor{textcolor}%
\pgftext[x=0.370616in, y=3.073351in, left, base]{\color{textcolor}{\sffamily\fontsize{16.000000}{19.200000}\selectfont\catcode`\^=\active\def^{\ifmmode\sp\else\^{}\fi}\catcode`\%=\active\def
\end{pgfscope}%
\begin{pgfscope}%
\pgfsetbuttcap%
\pgfsetroundjoin%
\definecolor{currentfill}{rgb}{0.000000,0.000000,0.000000}%
\pgfsetfillcolor{currentfill}%
\pgfsetlinewidth{0.803000pt}%
\definecolor{currentstroke}{rgb}{0.000000,0.000000,0.000000}%
\pgfsetstrokecolor{currentstroke}%
\pgfsetdash{}{0pt}%
\pgfsys@defobject{currentmarker}{\pgfqpoint{-0.048611in}{0.000000in}}{\pgfqpoint{-0.000000in}{0.000000in}}{%
\pgfpathmoveto{\pgfqpoint{-0.000000in}{0.000000in}}%
\pgfpathlineto{\pgfqpoint{-0.048611in}{0.000000in}}%
\pgfusepath{stroke,fill}%
}%
\begin{pgfscope}%
\pgfsys@transformshift{1.104015in}{3.684334in}%
\pgfsys@useobject{currentmarker}{}%
\end{pgfscope}%
\end{pgfscope}%
\begin{pgfscope}%
\definecolor{textcolor}{rgb}{0.000000,0.000000,0.000000}%
\pgfsetstrokecolor{textcolor}%
\pgfsetfillcolor{textcolor}%
\pgftext[x=0.370616in, y=3.599916in, left, base]{\color{textcolor}{\sffamily\fontsize{16.000000}{19.200000}\selectfont\catcode`\^=\active\def^{\ifmmode\sp\else\^{}\fi}\catcode`\%=\active\def
\end{pgfscope}%
\begin{pgfscope}%
\pgfsetbuttcap%
\pgfsetroundjoin%
\definecolor{currentfill}{rgb}{0.000000,0.000000,0.000000}%
\pgfsetfillcolor{currentfill}%
\pgfsetlinewidth{0.803000pt}%
\definecolor{currentstroke}{rgb}{0.000000,0.000000,0.000000}%
\pgfsetstrokecolor{currentstroke}%
\pgfsetdash{}{0pt}%
\pgfsys@defobject{currentmarker}{\pgfqpoint{-0.048611in}{0.000000in}}{\pgfqpoint{-0.000000in}{0.000000in}}{%
\pgfpathmoveto{\pgfqpoint{-0.000000in}{0.000000in}}%
\pgfpathlineto{\pgfqpoint{-0.048611in}{0.000000in}}%
\pgfusepath{stroke,fill}%
}%
\begin{pgfscope}%
\pgfsys@transformshift{1.104015in}{4.210899in}%
\pgfsys@useobject{currentmarker}{}%
\end{pgfscope}%
\end{pgfscope}%
\begin{pgfscope}%
\definecolor{textcolor}{rgb}{0.000000,0.000000,0.000000}%
\pgfsetstrokecolor{textcolor}%
\pgfsetfillcolor{textcolor}%
\pgftext[x=0.370616in, y=4.126480in, left, base]{\color{textcolor}{\sffamily\fontsize{16.000000}{19.200000}\selectfont\catcode`\^=\active\def^{\ifmmode\sp\else\^{}\fi}\catcode`\%=\active\def
\end{pgfscope}%
\begin{pgfscope}%
\definecolor{textcolor}{rgb}{0.000000,0.000000,0.000000}%
\pgfsetstrokecolor{textcolor}%
\pgfsetfillcolor{textcolor}%
\pgftext[x=0.315061in,y=2.530899in,,bottom,rotate=90.000000]{\color{textcolor}{\sffamily\fontsize{16.000000}{19.200000}\selectfont\catcode`\^=\active\def^{\ifmmode\sp\else\^{}\fi}\catcode`\%=\active\def
\end{pgfscope}%
\begin{pgfscope}%
\pgfpathrectangle{\pgfqpoint{1.104015in}{0.682899in}}{\pgfqpoint{4.960000in}{3.696000in}}%
\pgfusepath{clip}%
\pgfsetrectcap%
\pgfsetroundjoin%
\pgfsetlinewidth{1.505625pt}%
\definecolor{currentstroke}{rgb}{0.121569,0.466667,0.705882}%
\pgfsetstrokecolor{currentstroke}%
\pgfsetdash{}{0pt}%
\pgfpathmoveto{\pgfqpoint{1.329469in}{4.210899in}}%
\pgfpathlineto{\pgfqpoint{1.780378in}{0.850899in}}%
\pgfpathlineto{\pgfqpoint{2.231287in}{1.345981in}}%
\pgfpathlineto{\pgfqpoint{2.682196in}{1.087494in}}%
\pgfpathlineto{\pgfqpoint{3.133105in}{1.068638in}}%
\pgfpathlineto{\pgfqpoint{3.584015in}{1.059545in}}%
\pgfpathlineto{\pgfqpoint{4.034924in}{1.055253in}}%
\pgfpathlineto{\pgfqpoint{4.485833in}{1.053248in}}%
\pgfpathlineto{\pgfqpoint{4.936742in}{1.052316in}}%
\pgfpathlineto{\pgfqpoint{5.387651in}{1.051884in}}%
\pgfpathlineto{\pgfqpoint{5.838560in}{1.051683in}}%
\pgfusepath{stroke}%
\end{pgfscope}%
\begin{pgfscope}%
\pgfpathrectangle{\pgfqpoint{1.104015in}{0.682899in}}{\pgfqpoint{4.960000in}{3.696000in}}%
\pgfusepath{clip}%
\pgfsetrectcap%
\pgfsetroundjoin%
\pgfsetlinewidth{1.505625pt}%
\definecolor{currentstroke}{rgb}{1.000000,0.498039,0.054902}%
\pgfsetstrokecolor{currentstroke}%
\pgfsetdash{}{0pt}%
\pgfpathmoveto{\pgfqpoint{1.329469in}{4.210899in}}%
\pgfpathlineto{\pgfqpoint{1.780378in}{1.266966in}}%
\pgfpathlineto{\pgfqpoint{2.231287in}{1.012768in}}%
\pgfpathlineto{\pgfqpoint{2.682196in}{1.112701in}}%
\pgfpathlineto{\pgfqpoint{3.133105in}{1.074607in}}%
\pgfpathlineto{\pgfqpoint{3.584015in}{1.060034in}}%
\pgfpathlineto{\pgfqpoint{4.034924in}{1.054629in}}%
\pgfpathlineto{\pgfqpoint{4.485833in}{1.052647in}}%
\pgfpathlineto{\pgfqpoint{4.936742in}{1.051925in}}%
\pgfpathlineto{\pgfqpoint{5.387651in}{1.051661in}}%
\pgfusepath{stroke}%
\end{pgfscope}%
\begin{pgfscope}%
\pgfpathrectangle{\pgfqpoint{1.104015in}{0.682899in}}{\pgfqpoint{4.960000in}{3.696000in}}%
\pgfusepath{clip}%
\pgfsetbuttcap%
\pgfsetroundjoin%
\pgfsetlinewidth{1.505625pt}%
\definecolor{currentstroke}{rgb}{0.000000,0.000000,0.000000}%
\pgfsetstrokecolor{currentstroke}%
\pgfsetdash{{5.550000pt}{2.400000pt}}{0.000000pt}%
\pgfpathmoveto{\pgfqpoint{1.104015in}{1.051511in}}%
\pgfpathlineto{\pgfqpoint{6.064015in}{1.051511in}}%
\pgfusepath{stroke}%
\end{pgfscope}%
\begin{pgfscope}%
\pgfsetrectcap%
\pgfsetmiterjoin%
\pgfsetlinewidth{0.803000pt}%
\definecolor{currentstroke}{rgb}{0.000000,0.000000,0.000000}%
\pgfsetstrokecolor{currentstroke}%
\pgfsetdash{}{0pt}%
\pgfpathmoveto{\pgfqpoint{1.104015in}{0.682899in}}%
\pgfpathlineto{\pgfqpoint{1.104015in}{4.378899in}}%
\pgfusepath{stroke}%
\end{pgfscope}%
\begin{pgfscope}%
\pgfsetrectcap%
\pgfsetmiterjoin%
\pgfsetlinewidth{0.803000pt}%
\definecolor{currentstroke}{rgb}{0.000000,0.000000,0.000000}%
\pgfsetstrokecolor{currentstroke}%
\pgfsetdash{}{0pt}%
\pgfpathmoveto{\pgfqpoint{6.064015in}{0.682899in}}%
\pgfpathlineto{\pgfqpoint{6.064015in}{4.378899in}}%
\pgfusepath{stroke}%
\end{pgfscope}%
\begin{pgfscope}%
\pgfsetrectcap%
\pgfsetmiterjoin%
\pgfsetlinewidth{0.803000pt}%
\definecolor{currentstroke}{rgb}{0.000000,0.000000,0.000000}%
\pgfsetstrokecolor{currentstroke}%
\pgfsetdash{}{0pt}%
\pgfpathmoveto{\pgfqpoint{1.104015in}{0.682899in}}%
\pgfpathlineto{\pgfqpoint{6.064015in}{0.682899in}}%
\pgfusepath{stroke}%
\end{pgfscope}%
\begin{pgfscope}%
\pgfsetrectcap%
\pgfsetmiterjoin%
\pgfsetlinewidth{0.803000pt}%
\definecolor{currentstroke}{rgb}{0.000000,0.000000,0.000000}%
\pgfsetstrokecolor{currentstroke}%
\pgfsetdash{}{0pt}%
\pgfpathmoveto{\pgfqpoint{1.104015in}{4.378899in}}%
\pgfpathlineto{\pgfqpoint{6.064015in}{4.378899in}}%
\pgfusepath{stroke}%
\end{pgfscope}%
\begin{pgfscope}%
\pgfsetbuttcap%
\pgfsetmiterjoin%
\definecolor{currentfill}{rgb}{1.000000,1.000000,1.000000}%
\pgfsetfillcolor{currentfill}%
\pgfsetfillopacity{0.800000}%
\pgfsetlinewidth{1.003750pt}%
\definecolor{currentstroke}{rgb}{0.800000,0.800000,0.800000}%
\pgfsetstrokecolor{currentstroke}%
\pgfsetstrokeopacity{0.800000}%
\pgfsetdash{}{0pt}%
\pgfpathmoveto{\pgfqpoint{3.397652in}{3.222606in}}%
\pgfpathlineto{\pgfqpoint{5.908459in}{3.222606in}}%
\pgfpathquadraticcurveto{\pgfqpoint{5.952903in}{3.222606in}}{\pgfqpoint{5.952903in}{3.267051in}}%
\pgfpathlineto{\pgfqpoint{5.952903in}{4.223343in}}%
\pgfpathquadraticcurveto{\pgfqpoint{5.952903in}{4.267788in}}{\pgfqpoint{5.908459in}{4.267788in}}%
\pgfpathlineto{\pgfqpoint{3.397652in}{4.267788in}}%
\pgfpathquadraticcurveto{\pgfqpoint{3.353207in}{4.267788in}}{\pgfqpoint{3.353207in}{4.223343in}}%
\pgfpathlineto{\pgfqpoint{3.353207in}{3.267051in}}%
\pgfpathquadraticcurveto{\pgfqpoint{3.353207in}{3.222606in}}{\pgfqpoint{3.397652in}{3.222606in}}%
\pgfpathlineto{\pgfqpoint{3.397652in}{3.222606in}}%
\pgfpathclose%
\pgfusepath{stroke,fill}%
\end{pgfscope}%
\begin{pgfscope}%
\pgfsetrectcap%
\pgfsetroundjoin%
\pgfsetlinewidth{1.505625pt}%
\definecolor{currentstroke}{rgb}{0.121569,0.466667,0.705882}%
\pgfsetstrokecolor{currentstroke}%
\pgfsetdash{}{0pt}%
\pgfpathmoveto{\pgfqpoint{3.442096in}{4.087840in}}%
\pgfpathlineto{\pgfqpoint{3.664318in}{4.087840in}}%
\pgfpathlineto{\pgfqpoint{3.886541in}{4.087840in}}%
\pgfusepath{stroke}%
\end{pgfscope}%
\begin{pgfscope}%
\definecolor{textcolor}{rgb}{0.000000,0.000000,0.000000}%
\pgfsetstrokecolor{textcolor}%
\pgfsetfillcolor{textcolor}%
\pgftext[x=4.064318in,y=4.010062in,left,base]{\color{textcolor}{\sffamily\fontsize{16.000000}{19.200000}\selectfont\catcode`\^=\active\def^{\ifmmode\sp\else\^{}\fi}\catcode`\%=\active\def
\end{pgfscope}%
\begin{pgfscope}%
\pgfsetrectcap%
\pgfsetroundjoin%
\pgfsetlinewidth{1.505625pt}%
\definecolor{currentstroke}{rgb}{1.000000,0.498039,0.054902}%
\pgfsetstrokecolor{currentstroke}%
\pgfsetdash{}{0pt}%
\pgfpathmoveto{\pgfqpoint{3.442096in}{3.761668in}}%
\pgfpathlineto{\pgfqpoint{3.664318in}{3.761668in}}%
\pgfpathlineto{\pgfqpoint{3.886541in}{3.761668in}}%
\pgfusepath{stroke}%
\end{pgfscope}%
\begin{pgfscope}%
\definecolor{textcolor}{rgb}{0.000000,0.000000,0.000000}%
\pgfsetstrokecolor{textcolor}%
\pgfsetfillcolor{textcolor}%
\pgftext[x=4.064318in,y=3.683890in,left,base]{\color{textcolor}{\sffamily\fontsize{16.000000}{19.200000}\selectfont\catcode`\^=\active\def^{\ifmmode\sp\else\^{}\fi}\catcode`\%=\active\def
\end{pgfscope}%
\begin{pgfscope}%
\pgfsetbuttcap%
\pgfsetroundjoin%
\pgfsetlinewidth{1.505625pt}%
\definecolor{currentstroke}{rgb}{0.000000,0.000000,0.000000}%
\pgfsetstrokecolor{currentstroke}%
\pgfsetdash{{5.550000pt}{2.400000pt}}{0.000000pt}%
\pgfpathmoveto{\pgfqpoint{3.442096in}{3.435497in}}%
\pgfpathlineto{\pgfqpoint{3.664318in}{3.435497in}}%
\pgfpathlineto{\pgfqpoint{3.886541in}{3.435497in}}%
\pgfusepath{stroke}%
\end{pgfscope}%
\begin{pgfscope}%
\definecolor{textcolor}{rgb}{0.000000,0.000000,0.000000}%
\pgfsetstrokecolor{textcolor}%
\pgfsetfillcolor{textcolor}%
\pgftext[x=4.064318in,y=3.357719in,left,base]{\color{textcolor}{\sffamily\fontsize{16.000000}{19.200000}\selectfont\catcode`\^=\active\def^{\ifmmode\sp\else\^{}\fi}\catcode`\%=\active\def
\end{pgfscope}%
\end{pgfpicture}%
\makeatother%
\endgroup%